\documentclass[a4paper,12pt]{article}

\usepackage[notext]{kpfonts}
\usepackage{baskervald}

\usepackage{amsfonts,amssymb,amsmath,amsthm,abstract,color}
\usepackage[mathscr]{euscript}
\usepackage[ps,all,arc,rotate]{xy}
\usepackage[lmargin=1in,rmargin=1in,tmargin=1in,bmargin=1in]{geometry}
\usepackage{fancyhdr}
\usepackage{dsfont}
\usepackage{pb-diagram}
\usepackage{mathtools}
\usepackage{enumitem}
\usepackage[backref=page]{hyperref}
\usepackage[usenames,dvipsnames]{xcolor}
\usepackage{tikz-cd}
\usepackage{extpfeil}
\usepackage{float}
\hypersetup{colorlinks=true,citecolor=NavyBlue,linkcolor=Brown,urlcolor=Orange}

\usepackage{stmaryrd}
\usepackage[capitalise]{cleveref}
\usepackage{float}
\usetikzlibrary{arrows.meta, positioning, shapes.misc}
\usepackage{comment}
\numberwithin{equation}{section}
\usepackage{appendix}


\usepackage{titlesec}

\titlespacing*{\section}
{0pt}{3ex plus 1ex minus .2ex}{2.5ex plus .2ex}
\titlespacing*{\subsection}
{0pt}{2ex plus 1ex minus .2ex}{1.5ex plus .2ex}
\titlespacing*{\subsubsection}
{0pt}{2ex plus 1ex minus .2ex}{1.5ex}

\newcommand{\sectionpoint}{\if \@empty\titlesec \else}              

\newcommand*{\justifyheading}{\raggedright}
\titleformat{\section}{\normalfont \Large \bfseries \justifyheading}{\thesection.}{0.5em}{}
\titleformat{\subsection}{\normalfont \large \bfseries \justifyheading}{\thesubsection.}{0.5em}{}
\titleformat{\subsubsection}[runin]{\normalfont \normalsize \bfseries \justifyheading}{\thesubsubsection.}{0.5em}{}[  \protect{\rule[3pt]{10pt}{0.5pt}}]

\theoremstyle{plain}
\newtheorem*{theorem*}{Theorem}
\newtheorem{theorem}{Theorem}[section]
\newtheorem{maintheorem}{Theorem}

\newtheorem{conjecture}[theorem]{Conjecture}
\newtheorem{lemma}[theorem]{Lemma}

\newtheorem{corollary}[theorem]{Corollary}
\newtheorem{proposition}[theorem]{Proposition}

\theoremstyle{definition}
\newtheorem{definition}[theorem]{Definition}

\theoremstyle{remark} 
 \newtheorem{remark}[theorem]{Remark}
\newtheorem{example}[theorem]{Example}
\newtheorem{notation}[theorem]{Notation}

\renewcommand{\tilde}{\widetilde}

\renewcommand{\geq}{\geqslant}
\renewcommand{\leq}{\leqslant}

\newcommand{\CC}{\mathbb{C}}
\newcommand{\RR}{\mathbb{R}}

\newcommand{\FF}{\mathbb{F}}
\newcommand{\PP}{\mathbb{P}}
\newcommand{\QQ}{\mathbb{Q}}

\newcommand{\ZZ}{\mathbb{Z}}

\newcommand{\h}{\mathfrak{h}}

\renewcommand{\1}{\mathds{1}}

\DeclareMathOperator{\CH}{CH}

\DeclareMathOperator{\Db}{\mathrm{D}^\mathrm{b}}

\DeclareMathOperator{\End}{End}
\DeclareMathOperator{\Gal}{Gal}
\DeclareMathOperator{\GL}{GL}

\DeclareMathOperator{\Hom}{Hom}
\DeclareMathOperator{\Hilb}{Hilb}
\DeclareMathOperator{\id}{id}

\DeclareMathOperator{\pt}{pt}
\DeclareMathOperator{\rk}{rk}
\DeclareMathOperator{\SO}{SO}

\DeclareMathOperator{\Sym}{Sym}

\DeclareMathOperator{\tr}{tr}

\DeclareMathOperator{\Bl}{Bl}
\DeclareMathOperator{\Ima}{Im}

\DeclareMathOperator{\MT}{MT}

\DeclareMathOperator{\Aut}{Aut}

\DeclareMathOperator{\Oo}{O}

\DeclareMathOperator{\et}{\mathrm{\acute{e}t}}

\DeclareMathOperator{\NS}{\mathrm{NS}} 
\DeclareMathOperator{\Km}{\mathrm{Km}} 
\DeclareMathOperator{\Pic}{\mathrm{Pic}}

\newcommand{\Sing}{\mathsf{Sing}}

\newcommand{\orb}{\mathrm{orb}}

\newcommand{\CHM}{\mathsf{CHM}}

\newcommand{\Mot}{\mathsf{Mot}}

\begin{document}

\title{\textbf{The hyper-Kummer construction}}

\author{Salvatore Floccari, Lie Fu}
\date{ }
\maketitle

\begin{abstract}
We investigate a construction which associates a hyper-K\"ahler manifold of $\mathrm{K}3^{[3]}$-type with any hyper-K\"ahler sixfold of generalized Kummer type, thus relating the two most studied deformation types of hyper-K\"ahler manifolds of dimension~$6$. This construction, discovered by the first-named author in a previous work, is named \textit{hyper-Kummer} in the present paper, as it is one of our main points that it is a higher-dimensional analog of the classical Kummer construction of K3 surfaces from abelian surfaces. In this spirit, we prove several results which parallel those known for the classical Kummer construction: we characterize the hyper-Kummer manifolds of $\mathrm{K}3^{[3]}$-type up to birational equivalence in terms of their Hodge lattices, and establish a McKay correspondence for their derived categories and Chow motives. Reversing the hyper-Kummer construction, we propose a recipe to construct locally complete families of projective varieties of $\mathrm{Kum}^3$-type starting from a family of varieties of $\mathrm{K}3^{[3]}$-type equipped with 16 prime divisors in a certain Kummer lattice configuration.
We further compare the hyper-Kummer $\mathrm{K}3^{[3]}$-type manifolds with the Mongardi--Rapagnetta--Sacc\`a double covers of O'Grady's six-dimensional hyper-K\"ahler manifolds, thus relating all three known deformation types of hyper-K\"ahler manifolds of dimension $6$. The hyper-Kummer construction produces a rich configuration of hyper-K\"ahler manifolds of $\mathrm{K}3^{[2]}$-type and K3 surfaces canonically associated with a manifold of $\mathrm{Kum}^3$-type, in particular the \textit{hyper-Kummer K3 surfaces}. In the projective case, such K3 surfaces form countably many $4$-dimensional families of generic Picard rank 16, and are generalizations of the classical Kummer K3 surfaces. 
We prove abelianity of Chow motives for infinitely many 4-dimensional families of hyper-Kummer K3 surfaces, thereby proving Kimura--O'Sullivan finite-dimensionality conjecture for many new K3 surfaces of Picard rank $16$. As applications, we show how the hyper-Kummer construction can be used to prove Beauville's weak splitting conjecture for all varieties of $\mathrm{Kum}^3$-type. Building on previous results, we moreover establish the Hodge and Tate conjectures for all powers of any of the varieties involved in the hyper-Kummer construction. 
\end{abstract}

\renewcommand{\thefootnote}{}
\null\footnotetext{ S.F.\ is supported by the ERC Advanced Grant SYZYGY and the ERC Advanced Grant TameHodge. L.F.\ is supported by the University of Strasbourg Institute for Advanced Study (USIAS), by the Agence Nationale de la Recherche (ANR) under projects ANR-20-CE40-0023 and ANR-24-CE40-4098, and by the International Emerging Actions (IEA) project of CNRS.

\noindent{\textbf{Keywords:} Hyper-K\"ahler manifolds, K3 surfaces, Kummer surfaces, Hodge theory, moduli spaces, Hilbert schemes, Hodge conjecture, Tate conjecture, motives, derived categories, lattices, McKay correspondence.}\\
\noindent{\textbf{2020 MSC:} 14J42, 14J28, 14C30, 14C15, 14D22, 14F08, 14J50, 14C05.}
}

\renewcommand{\thefootnote}{\arabic{footnote}} 
\newpage
\small{\tableofcontents}
\newpage

\section{Introduction}

In his seminal papers \cite{Kummer-surfaceII-1864, Kummer-surfaceI-1864}\footnote{Both are collected in the more readily available volume \cite{Kummer-CollectedWorksII} edited by Andr\'e Weil.} published in 1864, Ernst Kummer constructed and studied quartic surfaces with exactly 16 nodes in the 3-dimensional projective space $\mathbb{P}^3$, and discovered profound relations with abelian functions in two variables. Kummer's quartic surface plays an important role in the history of mathematics because it reveals, in a single explicit object, the deep interaction between geometry, complex analysis and algebraic equations. In the modern language of algebraic geometry, quartic surfaces with 16 isolated $\mathrm{A}_1$-singularities arise as quotients of principally polarized abelian surfaces $A$ by the involution $-1: x\mapsto -x$. The 16 singular points are the images of the 16 fixed points of $-1$, while the embedding $A/{-1} \hookrightarrow\mathbb{P}^3$ is via the complete linear system of the square of the principal polarization. Ever since its discovery, the rich and beautiful projective geometry of Kummer’s quartic surfaces has been the topic of much research; see for example Hudson's \cite{Hudson}.
\begin{figure}[H]
    \centering
    \includegraphics[width=0.5\linewidth]{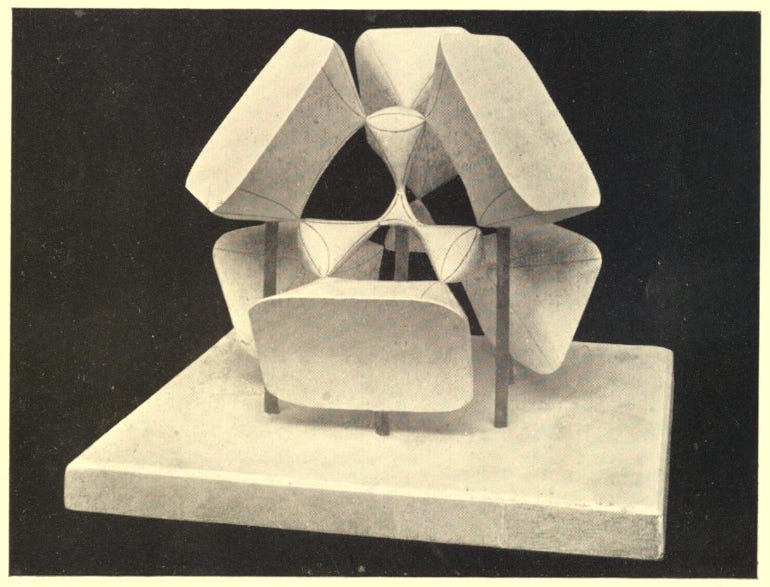}
    \caption{The real locus of Kummer's quartic surface. The picture is taken from \cite{Hudson}.
}
\end{figure}

Forgetting the polarization and the projective embedding, the same construction can be performed more generally for any 2-dimensional complex torus $A$, resulting in a singular complex surface $A/{-1}$, with 16 nodes corresponding to the 2-torsion points for the group structure on $A$.
Moreover, blowing up these singular points gives rise to a crepant resolution of singularities
\begin{equation*}
    \Km(A):=\Bl_{A[2]}(A/{-1}) \longrightarrow A/{-1},
\end{equation*}
producing one of the most important examples of K3 surfaces, namely, \textit{the Kummer surface} associated with $A$, denoted by $\Km(A)$ in this paper. Alternatively, let $p'\colon  \widetilde{A}\to A$ be the blow-up of $A$ along the locus of 2-torsion points $A[2]$; then the involution $-1$ lifts to an involution on $\widetilde{A}$ with fixed locus being the disjoint union of the 16 exceptional curves, and $\Km(A)$ is isomorphic to the quotient of $\widetilde{A}$ by the lifted involution. We summarize the situation in the following commutative diagram:
\begin{equation}
\label{diag:KummerConstruction}
    \begin{tikzcd}
    &\widetilde{A} \arrow[swap]{dl}{p'} \arrow{dr}{q'}&\\
    A  \arrow{dr}[swap]{q} \arrow[dashed]{rr}{r}&& \Km(A) \arrow{dl}{p}\\
    &A/{-1}&
\end{tikzcd}
\end{equation}
where $r$ is the degree-2 rational map, $p, p'$ are the blow-ups described above, and  $q, q'$ are quotient maps by involutions.

The \textit{hyper-Kummer construction}, discovered in \cite{floccariKum3}, is a six-dimensional analog of the Kummer construction in the realm of compact hyper-K\"ahler manifolds (see \cite{beauville1983varietes} for generalities on such manifolds). Like the Kummer construction attaches a K3 surface to any 2-dimensional complex torus, the hyper-Kummer construction attaches a hyper-K\"ahler manifold of $\mathrm{K}3^{[3]}$-type to any hyper-K\"ahler sixfold of generalized Kummer type ($\mathrm{Kum}^3$-type). More precisely, let $K$ be any hyper-K\"ahler manifold of $\mathrm{Kum}^3$-type. There is a natural action of a group $G\cong (\ZZ/2\ZZ)^5$ on $K$, where
\begin{equation}
\label{eqn:DefinitionG}
    G \coloneqq \{g\in \Aut(K)~|~ g|_{H^2(K,\QQ)}=\id~;~g|_{H^4(K,\QQ)}=\id\}.
\end{equation}

As proved in \cite{floccariKum3}, blowing up the quotient $K/G$ along its singular locus gives rise to a crepant resolution
$$p\colon Y_K\longrightarrow K/G,$$ 
producing a hyper-K\"ahler manifold $Y_K$ of $\mathrm{K}3^{[3]}$-type.
Alternatively, let $\widetilde{K}$ be the blow-up of $K$ along the total fixed locus $\bigcup_{1\neq g\in G} K^g$; then the $G$-action lifts to $\widetilde{K}$ and the quotient is isomorphic to $Y_K$. We summarize the situation in the following commutative diagram:
\begin{equation}\label{Diag:hyperKummerConstruction}
\begin{tikzcd}
    & \widetilde{K} \arrow{rd}{q'} \arrow[swap]{ld}{p'} \\
    K \arrow[swap]{rd}{q} \arrow[dashed]{rr}{r} && Y_K \arrow{ld}{p} \\
     & K/G   
\end{tikzcd}
\end{equation}
where $r$ is the degree-32 rational map, $p,p'$ are the blow-ups described above, and $q,q'$ are the quotient maps with respect to the $G$-actions. 
We call the $\mathrm{K}3^{[3]}$-type hyper-K\"ahler manifold $Y_K$ the \textit{hyper-Kummer sixfold} associated with $K$. 

The hyper-Kummer construction is quite explicit for the generalized Kummer 6-folds associated with an abelian surface, at least up to birational transformations. 

\begin{example}[Section \ref{subsec:GeneralizedKummer}]
\label{example:hyperKummerConstruction}
    Let $A$ be an abelian surface and $K=K^3(A)$ the associated 6-dimensional generalized Kummer variety. In this case, $G$ is naturally identified with the group $A[2]\times \{\pm 1\}\simeq (\mathbb{Z}/2\mathbb{Z})^{5}$ generated by $-1_A$ and the translations by the 2-torsion points of $A$. The hyper-Kummer manifold $Y_K$ associated with $K$ is birational to  $\Km(A)^{[3]}$, the Hilbert cube of the Kummer surface associated with $A$, and the hyper-Kummer resolution is identified with the inverse of the birational map
    \begin{align*}
       K^3(A)/G &\dashrightarrow \Km(A)^{[3]}\\
        \{a_1, a_2, a_3, a_4\} &\mapsto \{\overline{a_1+a_2}, \overline{a_1+a_3}, \overline{a_1+a_4}\},
    \end{align*}
    where $\{a_1, a_2, a_3, a_4\}$ denotes the $G$-orbit of a generic point of $K^3(A)$ given by an unordered set $\{a_1, a_2, a_3, a_4\}$ of four points in $A$ such that $a_1+a_2+a_3+a_4=0$, and $\overline{a}$ denotes the class in $A/{-1}$ of a given point $a\in A$. The key point is that the above formula yields an isomorphism of orbifolds $A^{(4)}_0/G\simeq (A/\pm1)^{(3)}$; see Lemma \ref{lem:isomorphismOrbifolds}.
\end{example}

\subsection{Facts on the Kummer construction}

To motivate our results and fix some notation, we present here a collection of known results about the classical Kummer construction. They will serve as the paradigm for our investigation of the hyper-Kummer construction. Consider the classical Kummer construction as in Diagram \eqref{diag:KummerConstruction}.

\subsubsection*{Hodge lattices} 
The second cohomology of a smooth compact complex surface has a natural structure of lattice given by the intersection pairing. 
The natural morphism 
\begin{equation*}
    r_*:=q'_*p'^*\colon H^2(A, \ZZ)(2) \hookrightarrow H^2(\Km(A), \ZZ)
\end{equation*}
is a primitive embedding of Hodge lattices, where $(2)$ in the source means that the quadratic form is multiplied by 2. As lattices, $\Ima(r_*)\simeq \mathrm{U}(2)^{\oplus 3}$ and its orthogonal complement $\Ima(r_*)^{\perp}$ is a negative definite lattice of rank 16 with discriminant group $(\ZZ/2\ZZ)^6$, known as the \textit{Kummer lattice}, denoted by $L_{\Km}$ in this paper. Geometrically, $L_{\Km}$ is the saturation of the sublattice of $H^2(\Km(A), \ZZ)$ generated by the classes of the 16 rational curves introduced by the resolution $\Km(A)\to A/\pm 1$.

\subsubsection*{Characterization (Nikulin \cite{Nikulin})} Let $S$ be a K3 surface. Let $\NS(S)$ be its N\'eron--Severi lattice and let $H^2_{\mathrm{tr}}(S, \ZZ)\coloneqq \NS(S)^{\perp}\subset H^2(S,\ZZ)$ be its transcendental Hodge lattice.  The following conditions are equivalent:
\begin{enumerate}[label=(\roman*)]
    \item $S$ is a Kummer surface, i.e.~$S\simeq \Km(A)$ for a 2-dimensional complex torus $A$;
    \item there exists a primitive embedding of lattices $L_{\Km}\hookrightarrow \NS(S)$;
    \item $S$ contains 16 disjoint smooth rational curves.
\end{enumerate}
If moreover $S$ is projective, the above conditions are also equivalent to:
\begin{enumerate}[label=(\roman*)]
   \item[(iv)] there exists a primitive embedding of lattices $H^2_{\mathrm{tr}}(S, \ZZ)\hookrightarrow \mathrm{U}(2)^{\oplus 3}$.
\end{enumerate}

\subsubsection*{Torelli-type Theorems} 
Let $A$ and $A'$ be two abelian surfaces. 
\begin{itemize}
    \item (Shioda \cite{Shioda-PeriodAbelianSurfaces}). If $H^2(A,\ZZ)$ and $H^2(A',\ZZ)$ are Hodge isometric, then $A'\simeq A$ or $A'\simeq \widehat{A}$, hence $\Km(A)\simeq \Km(A')$.
    \item (Huybrechts \cite{huybrechtsMotives}, Fu--Vial \cite{FuVial}). The following conditions are equivalent:
        \begin{enumerate}[label=(\roman*)]
            \item $H^2(A,\QQ)$ and $H^2(A',\QQ)$ are Hodge isometric; 
            \item $\Km(A)$ and $\Km(A')$ are connected by a chain of twisted derived equivalences;
            \item there is an isomorphism of rational motives $\h(\Km(A))\simeq \h(\Km(A'))$ as Frobenius algebra objects.
        \end{enumerate}
\end{itemize}
More recently, Li--Zou \cite[Corollary 1.2.2]{LiZou-DerivedIsogeniesAbSurfaces} showed that the above three conditions are also equivalent to the condition that $A$ and $A'$ are connected by a chain of twisted derived equivalences and to the condition that there is an isogeny $A\to A'$ of square degree.

\vspace{0.5cm}

The next two facts are instances of the McKay correspondence. 
Let $A$ be an abelian surface. Let $G$ be the cyclic group of order 2 acting on $A$ by the involution $-1_A$. Let $[A/G]$ be the quotient Deligne--Mumford stack. 

\subsubsection*{Derived categories}
Let $\Db(-)$ be the derived category of perfect complexes. Then $\Db([A/G])$ is canonically equivalent to the equivariant category $\mathrm{D}^{\mathrm{b}}_G(A)$.
The Fourier--Mukai functor 
\begin{equation}\label{eqn:DerivedEqKummer}
Rp'_* \circ {q'}^* \colon  \Db(\Km(A))\xrightarrow{\simeq} \Db([A/G])
\end{equation}
induces an equivalence of triangulated categories, with inverse equivalent to the functor $(Rq'_*\circ L{p'}^*)^{G}$.  To see \eqref{eqn:DerivedEqKummer}, we identify $\Km(A)$ with the $G$-equivariant Hilbert scheme of $A$ and then it is an easy example of \cite[Theorem 1.1]{BridgelandKingReid}.
The equivalence \eqref{eqn:DerivedEqKummer} implies relations between derived invariants of $A$ and $\Km(A)$, like K-theory, Hochschild (co)homology, etc.

\subsubsection*{Motives}
In the category of rational Chow motives $\CHM_{\QQ}$, by the blow-up formula, we have a canonical isomorphism 
\begin{equation*}
        \h(\Km(A)) \simeq \h(A)^G \oplus \1(-1)^{\oplus 16}.
\end{equation*}
This isomorphism can be reformulated as an isomorphism between the motive of the Kummer surface $\Km(A)$ and the orbifold motive (see \cite{FuVialTian-MHRC}) of the quotient stack $[A/G]$:
\begin{equation}
\label{eqn:IsomMotKummer}
    \h(\Km(A)) \simeq \h_{\orb}([A/G]).
\end{equation}
Moreover, \eqref{eqn:IsomMotKummer} can be promoted to an isomorphism of algebra objects in $\CHM_{\QQ}$ by \cite[Proposition 3.9]{FuVialTian-MHRC} or \cite[Theorem 1.2]{Fu-Tian-MultiplicativeMcKaySurface}. Applying various realization functors, the isomorphism \eqref{eqn:IsomMotKummer} encompasses most topological and cohomological relations between $A$ and $\Km(A)$, e.g. for Hodge structures, Chow groups, motivic cohomology and so on, with rational coefficients. 

\subsection{Results on the hyper-Kummer construction}
\label{subsec:Intro-ResultsHK}

A main goal of the present work is to establish properties of the hyper-Kummer construction that parallel those of the classical Kummer construction recalled above. Let $K$ be a compact hyper-K\"ahler manifold of $\mathrm{Kum}^3$-type, let $G\simeq (\ZZ/2\ZZ)^5$ be the group defined in \eqref{eqn:DefinitionG}, and let $Y_K$ be the associated hyper-Kummer $\mathrm{K}3^{[3]}$-type manifold. Consider the hyper-Kummer construction as in Diagram \eqref{Diag:hyperKummerConstruction}:
\begin{equation*}
\begin{tikzcd}
    & \widetilde{K} \arrow{rd}{q'} \arrow[swap]{ld}{p'} \\
    K \arrow[swap]{rd}{q} \arrow[dashed]{rr}{r} && Y_K \arrow{ld}{p} \\
     & K/G   
\end{tikzcd}
\end{equation*}
\subsubsection*{Hodge lattices}
The second cohomology of a compact hyper-K\"ahler manifold has a natural structure of Hodge lattice given by the Beauville--Bogomolov quadratic form \cite{beauville1983varietes}. For a $\mathrm{Kum}^3$-type manifold $K$, the lattice $H^2(K, \ZZ)$ is isometric to  $\mathrm{U}^{\oplus 3}\oplus \langle -8 \rangle$, which admits a unique overlattice of index $2$, isometric to $\mathrm{U}^{\oplus 3}\oplus \langle -2 \rangle$. We define the Hodge lattice $\widehat{H}^2(K,\ZZ)$ to be this index-2 overlattice of $H^2(K, \ZZ)$, equipped with the Hodge structure extended from $H^2(K, \ZZ)$. 

\begin{maintheorem}[Hodge lattice, Proposition \ref{prop:cohomologyY_K}]
\label{thm-intro:HodgeLattices}
    With the above notation, we have: 
    \begin{enumerate}[label=(\roman*)]
        \item The classes of the $16$ exceptional divisors of the blow-up $p$ are pairwise orthogonal $(-2)$-classes in $H^2(Y_K,\ZZ)$, and the saturation of the sublattice generated by them is isometric to the Kummer lattice. 
        \item The morphism of Hodge structures $r_*\coloneqq q'_*p'^* \colon H^2(K,\ZZ)\to H^2(Y_K,\ZZ)$ is injective, divisible by $2^4$, and multiplies the quadratic form by $2^9$.
        \item The map $\tfrac{r_*}{2^4}\colon H^2(K,\ZZ)(2)\to H^2(Y_K,\ZZ)$ is an embedding of Hodge lattices which extends to a primitive embedding of Hodge lattices:
        \begin{equation*}
            {\frac{\hat{r}_*}{2^4}}\colon \widehat{H}^2(K,\ZZ)(2)\hookrightarrow  H^2(Y_K,\ZZ),
        \end{equation*} 
        where $(2)$ means that the quadratic form is multiplied by a factor $2$. Hence, the image of $\frac{\hat{r}}{2^4}$ is isometric to $\mathrm{U}(2)^{\oplus 3}\oplus \langle -4\rangle$, while $\Ima\left({\frac{\hat{r}_*}{2^4}}\right)^{\perp}\simeq L_{\Km}$ is the Kummer lattice generated up to finite index by the $16$ exceptional divisors of $p$ as in (i).
    \end{enumerate}
\end{maintheorem}

\subsubsection*{Characterization of hyper-Kummer manifolds}
We establish the following birational characterization of hyper-Kummer sixfolds in terms of the Hodge lattices on the second cohomology, in analogy with the aforementioned Nikulin's characterization of Kummer surfaces.

\begin{maintheorem}[Birational characterization, Theorem \ref{thm:birationalCharacterization}]
\label{thm-intro:BirationalCharacterization}
Let $Y$ be a hyper-K\"ahler manifold of $\mathrm{K}3^{[3]}$-type. 
The following conditions are equivalent:
\begin{enumerate}[label=(\roman*)]
\item $Y$ is bimeromorphic to the hyper-Kummer manifold $Y_K$ associated with some hyper-K\"ahler manifold $K$ of $\mathrm{Kum}^3$-type.
\item There exists a primitive embedding of lattices $L_{\Km}\hookrightarrow \NS(Y)$ such that the orthogonal complement of the induced embedding $L_{\Km}\hookrightarrow H^2(Y,\ZZ)$ is isometric to $\mathrm{U}(2)^{\oplus 3}\oplus \langle -4\rangle$.
\end{enumerate}
\end{maintheorem} 
\begin{remark}
    Up to isometry, there are two primitive embeddings of the Kummer lattice $L_{\Km}$ into the $\mathrm{K}3^{[3]}$-lattice $\Lambda_{\mathrm{K}3^{[3]}}=\mathrm{U}^{\oplus 3}\oplus \mathrm{E}_8(-1)^{\oplus 2}\oplus \langle -4\rangle$, one with orthogonal complement $\mathrm{U}(2)^{\oplus 3}\oplus \langle -4\rangle$ and the other with orthogonal complement $\mathrm{U}(2)^{\oplus 2}\oplus \mathrm{U} \oplus \langle -4\rangle$.
\end{remark}

\subsubsection*{Torelli-type theorems}
As was discovered by Namikawa \cite{Namikawa02}, the naive generalization of the global Torelli theorem for K3 surfaces does not hold for hyper-K\"ahler manifolds of generalized Kummer type: two hyper-K\"ahler manifolds of $\mathrm{Kum}^n$-type with Hodge isometric second cohomology are not necessarily bimeromorphic. We establish the following Torelli-type results, to be compared to the aforementioned results of Shioda \cite{Shioda-PeriodAbelianSurfaces} and of Huybrechts \cite{huybrechtsMotives} and Fu--Vial \cite{FuVial}.

\begin{maintheorem}[Proposition \ref{prop:HodgeIsometricKummer} and Corollary \ref{cor:HomologicalMotiveIso}]
\label{thm-intro:Torelli-type}
    Let $K$ and $K'$ be two projective hyper-K\"ahler manifolds of $\mathrm{Kum}^3$-type. 
    \begin{enumerate}[label=(\roman*)]
        \item If $H^2(K, \ZZ)$ and $H^2(K', \ZZ)$ are Hodge isometric, then the associated hyper-Kummer sixfolds $Y_K$ and $Y_{K'}$ are birational.
        \item If  $H^2(K, \QQ)$ and $H^2(K', \QQ)$ are Hodge isometric, we have isomorphisms of homological motives as Frobenius algebra objects $\mathsf h(K)\simeq \mathsf h(K')$  and $\mathsf h(Y_K)\simeq \mathsf h(Y_{K'})$.
    \end{enumerate}
\end{maintheorem}

Statement (ii) confirms for the $\mathrm{Kum}^3$-deformation type the general expectation that deformation equivalent hyper-K\"ahler varieties with Hodge isometric rational second cohomology should have isomorphic motives, as Frobenius algebra objects. This statement was recently proven for the homological motives of varieties of $\mathrm{K}3^{[n]}$-type in \cite{MaulikShenYin-OrlovConjecture}.

\subsubsection*{Derived categories and motives}
The crepant resolution $Y_K\to K/G$ and the coarse moduli space map from the quotient stack $[K/G]\to K/G$ are both minimal smooth models of the singular variety $K/G$.  
The following result fits in the framework of McKay correspondence; it should be compared to \eqref{eqn:DerivedEqKummer} and \eqref{eqn:IsomMotKummer} for the classical Kummer construction.

\begin{maintheorem}[Derived categories and motives, Theorem \ref{thm:HyperKummerIsBKR}, Corollary \ref{cor:DerivedEquivalenceHyperKummer}, Corollary \ref{cor:MHRC-additive}]
\label{thm-intro:DerCat&Motives}
 Let $K$ be a projective hyper-K\"ahler variety of $\mathrm{Kum}^3$-type. Let $Y_K$ be the associated hyper-Kummer $\mathrm{K}3^{[3]}$-variety. Let the notation be as in the hyper-Kummer construction in Diagram \eqref{Diag:hyperKummerConstruction}.
 \begin{enumerate}[label=(\roman*)]
     \item We have an isomorphism $Y_K\cong \Hilb^G(K)$, where $\Hilb^G(K)$ is the equivariant Hilbert scheme. 
     \item The Fourier--Mukai functor 
     \begin{equation}
Rp'_* \circ L{q'}^* \colon  \Db(Y_K)\xrightarrow{\simeq} \Db([K/G])
\end{equation}
induces an equivalence of triangulated categories, with inverse functor equivalent to $(Rq'_*\circ L{p'}^*)^G$.
\item There is an isomorphism in the category of rational Chow motives $\CHM_\QQ$:
\begin{equation}
\label{eqn:MHRCY_K}
                    \h(Y_K)\simeq \h_{\orb}([K/G]),
\end{equation}
    where the $\h_{\orb}$ denotes the orbifold Chow motive of the Deligne--Mumford stack $[K/G]$.
 \end{enumerate}
As a consequence, we have isomorphisms for K-theory, Hochschild (co)homology and rational Chow groups:
\begin{align*}
    \mathrm K^*(Y_K)&\simeq \mathrm K^*([K/G]);\\
    \mathrm{HH}^*(Y_K) &\simeq \mathrm{HH}^*([K/G]);\\
    \CH^*(Y_K)_{\QQ}&\simeq \CH^*_{\orb}([K/G])_{\QQ}.
\end{align*}
\end{maintheorem}

\begin{remark}
    See Theorem \ref{thm:MotiveOfYK} for a more concrete version of \eqref{eqn:MHRCY_K} unraveling the definition of the orbifold Chow motive $\h_{\orb}([K/G])$. Ignoring the ring structure, $\h_{\orb}([K/G])$ (resp.~$\mathrm K([K/G])$, $\mathrm{HH}^*([K/G])$, $\CH^*_{\orb}([K/G])_{\QQ}$) can be explicitly expressed in terms of the Chow motives (resp.~$\mathrm{K}$-theory, Hochschild cohomology, Chow rings) of various fixed loci of the $G$-action on $K$ by \cite{Vistoli-EquivariantK} (resp. \cite{Arinkin-CaldararuHablicsek-HKR}, \cite{FuVialTian-MHRC}). Taking into account the multiplicative structure, the motivic hyper-K\"ahler resolution conjecture formulated by Fu--Tian--Vial \cite{FuVialTian-MHRC} predicts that \eqref{eqn:MHRCY_K} can be promoted into an isomorphism of algebra objects in $\CHM_{\QQ}$, up to some sign change; this is proved in an upcoming work of Christopher Nicol in a more general context. 
\end{remark}

\subsection{Companion manifolds and hyper-Kummer K3 surfaces}
It turns out that the richness of the geometry of the hyper-Kummer construction goes beyond the close analogs of the classical Kummer construction presented in Section \ref{subsec:Intro-ResultsHK}: for any $\mathrm{Kum}^3$-type manifold $K$, there is an interesting \textit{configuration} of hyper-K\"ahler manifolds canonically associated with $K$; we call them the \textit{companion hyper-K\"ahler manifolds} of $K$.
Let $G_1\simeq (\ZZ/2\ZZ)^4$ be the subgroup of $G$ consisting of automorphisms acting trivially also on $H^3(K,\ZZ)$.

The companion hyper-K\"ahler manifolds of $K$ are the following (see Section \ref{subsec:Companions}):

\begin{itemize}
    \item The hyper-Kummer sixfold $Y_K$ constructed as crepant resolution of $K/G$, see Diagram \eqref{Diag:hyperKummerConstruction};
    \item For any of the $16$ elements $\sigma\in G\backslash G_1$,
    \begin{itemize}
        \item its fixed locus $K^{\sigma}$ has a unique 4-dimensional component $W_\sigma$, that is a hyper-K\"ahler manifold of $\mathrm{K}3^{[2]}$-type; 
        \item the $G$-action stabilizes $W_\sigma$ and the quotient $W_\sigma/G$ admits a crepant resolution $M_\sigma$, which is a hyper-K\"ahler manifold of $\mathrm{K}3^{[2]}$-type.
    \end{itemize}
    \item For any of the $120$ unordered pairs of distinct elements $\sigma, \sigma'$ of $G\backslash G_1$,
    \begin{itemize}
        \item their common fixed locus $V_{\sigma, \sigma'}=W_\sigma\cap W_{\sigma'}$ is a K3 surface embedded in $K$;
        \item the $G$-action stabilizes $V_{\sigma, \sigma'}$ and the quotient $V_{\sigma, \sigma'}/G$ admits a crepant resolution $S_{\sigma, \sigma'}$, which is a K3 surface.
    \end{itemize}
\end{itemize}
 In total, given a $\mathrm{Kum}^3$-type manifold $K$, there are the following canonically associated hyper-K\"ahler manifolds: the hyper-Kummer sixfold $Y_K$ of $\mathrm{K}3^{[3]}$-type, the $16$ $\mathrm{K}3^{[2]}$-type submanifolds $W_\sigma$ (all isomorphic to each other), the $16$ $\mathrm{K}3^{[2]}$-type crepant resolutions $M_\sigma$ (all isomorphic to each other), the $120$ K3 surfaces $V_{\sigma, \sigma'}$ embedded in $K$ (all isogenous to each other but falling into $15$ isomorphism classes in general), and the $120$ crepant resolution K3 surfaces $S_{\sigma, \sigma'}$ (which are all isomorphic to each other in the projective case, see Corollary \ref{cor:Relation-YKMKSK} below). These hyper-K\"ahler manifolds fit into the following diagram:

\begin{equation}\label{Diag:Constellation}
\begin{tikzcd}
    K \arrow{r}  & K/G & Y_K \arrow{l} \\
    W_\sigma 
    \arrow[hook]{u}\arrow{r} & W_\sigma/G 
    \arrow[hook]{u} & M_\sigma 
    \arrow{l} 
    \\
    V_{\sigma,\sigma'} 
    \arrow[hook]{u} \arrow{r}& V_{\sigma,\sigma'}/G \arrow[hook]{u} & S_{\sigma,\sigma'}
    \arrow{l} 
\end{tikzcd}
\end{equation}
where the maps from left to middle are quotients by the action of $G$ and the maps from right to middle are crepant resolutions. 
The crucial property is that the whole configuration deforms together with the $\mathrm{Kum}^3$-type manifold $K$ (Remark \ref{rmk:deformation}). This follows from the deformation invariance of the group $G$, which goes back to Hassett--Tschinkel \cite{hassettTschinkel}.

We establish results analogous to Theorem \ref{thm-intro:HodgeLattices} (Hodge lattice) and Theorem \ref{thm-intro:DerCat&Motives} (McKay correspondences for derived categories and motives) for each of the companion hyper-K\"ahler manifolds, see Proposition \ref{prop:cohomologyW&V}, Proposition \ref{prop:cohomologyM&S}, and Theorem \ref{thm:HyperKummerIsBKR-ForM}. We fully characterize the K3 surfaces $V_{\sigma, \sigma'}$ and $S_{\sigma,\sigma'}$ in terms of their Hodge lattice,
see Theorem \ref{thm:characterizeV} and Theorem \ref{thm:characterizeS}.

\paragraph*{Hyper-Kummer K3 surfaces.} For this Introduction, let us only highlight our results on the $120$ K3 surfaces $S_{\sigma, \sigma'}$ associated with $K$. In the projective setting, these K3 surfaces are all isomorphic to each other; thus, we obtain a \textit{hyper-Kummer} K3 surface associated with any variety of $\mathrm{Kum}^3$-type. 

\begin{maintheorem}[Hyper-Kummer K3, Theorem \ref{thm:TranscendentalResolutions}, Corollary \ref{cor:Relation-YKMKSK}]
\label{thm-intro:HyperKummerK3Basic}
Let $K$ be a projective manifold of $\mathrm{Kum}^3$-type.
   \begin{enumerate}[label=(\roman*)]
       \item The $\mathrm{K}3$ surfaces $S_{\sigma,\sigma'}$ associated with $K$ are all abstractly isomorphic to each other; we denote by $S_K$ this $\mathrm{K}3$ surface, called the \emph{hyper-Kummer} $\mathrm{K}3$ surface associated with $K$.
       \item The transcendental Hodge lattice of $S_K$ is Hodge isometric to $\widehat{H}_{\mathrm{tr}}^2(K,\ZZ)(2)$ (see Theorem \ref{thm-intro:HodgeLattices} for the notation).
       \item Each of the companion manifolds $Y_K$ and $M_{\sigma}$ associated with $K$ is a smooth and projective moduli space of Bridgeland stable objects in the derived category of the hyper-Kummer $\mathrm{K}3$ surface $S_K$. Their transcendental Hodge lattices are all Hodge isometric to $\widehat{H}^2_{\tr}(K, \ZZ)(2)$.
   \end{enumerate}
\end{maintheorem}

Projective hyper-Kummer K3 surfaces form countably many 4-dimensional families, which contain the families of Kummer surfaces as divisors. In fact, not surprisingly, the hyper-Kummer K3 surface $S_{K^3(A)}$ associated with the generalized Kummer variety $K^3(A)$ on an abelian surface is nothing but the Kummer surface $\Km(A)$.

\begin{maintheorem}[Characterization and classification, Theorem \ref{thm:characterizeS} and Corollary \ref{cor:Transcendental-hyper-KummerK3}]
\label{thm-intro:HyperKummerK3Classification}
Let $S$ be a projective $\mathrm{K}3$ surface.
\begin{enumerate}[label=(\roman*)]
    \item $S$ is a hyper-Kummer $\mathrm{K}3$ surface if and only if there exists a primitive embedding of lattices $H^2_{\mathrm{tr}}(S,\ZZ)\hookrightarrow \mathrm{U}(2)^{\oplus 3}\oplus \langle -4\rangle$.
    \item If $S$ is a hyper-Kummer $\mathrm{K}3$ surface of minimal Picard rank, equal to $16$, then its transcendental lattice 
    $H^2_{\tr}(S, \ZZ)$ is isometric to one of the following lattices, for some $d>0$:
    \begin{itemize}
    \item   $\mathrm{U}(2)^{\oplus 2}\oplus \langle -4\rangle \oplus \langle -4d\rangle$~, called split type.
    \item 
    $\mathrm{U}(2)^{\oplus 2} \oplus \begin{psmallmatrix}
        -4 & 2 \\
        2 & -4d
    \end{psmallmatrix}$~, called non-split type\,;
    \end{itemize}
\end{enumerate}
\end{maintheorem}

In Section \ref{sec:ExamplesHyperKummerK3}, we present some interesting examples of 4-dimensional families of hyper-Kummer K3 surfaces, namely, the Heisenberg-invariant quartic K3 surfaces in $\PP^3$, the degree-8 K3 surfaces which are diagonal (2,2,2)-complete intersections in $\PP^5$, and a family of genus-125 K3 surfaces with $(\ZZ/2\ZZ)^4$ symplectic action.

\subsection{Construction of locally complete families of $\mathrm{Kum}^3$-type manifolds}
\label{subsec-intro:ConstructionKumType}
Several beautiful constructions of locally complete (hence, 20-dimensional) families of projective hyper-K\"ahler manifolds of $\mathrm{K}3^{[n]}$-type have been found: in dimension 4, Beauville--Donagi's Fano varieties of lines on cubic fourfolds \cite{BD85}, O'Grady's double EPW sextics \cite{OGrady-Duke-DoubleEPW}, Iliev--Ranestad's varieties of sum of powers of cubic forms in six variables \cite{Iliev-Ranestad-VSP}, Debarre--Voisin fourfolds via Grassmannian geometry \cite{DV10}; in dimension 6, Iliev--Kapustka--Kapustka--Ranestad's EPW cubes \cite{IKKR-EPWCube}; in dimension 8, Lehn--Lehn--Sorger--van Straten eightfolds via twisted cubics on cubic fourfolds \cite{LLSvS}; let us also mention that in \cite{BLMNPS}, infinitely many such families in arbitrary dimension are constructed as moduli spaces of stable objects in some K3 category. However, so far, there is no available geometric construction of locally complete (hence, $4$-dimensional) families of projective hyper-K\"ahler manifolds of $\mathrm{Kum}^n$-type\footnote{We are informed of a work in progress of F. Giovenzana--Rojas--Song on a geometric construction of a locally complete family of projective $\mathrm{Kum}^2$-type varieties. In the same spirit of \cite{BLMNPS}, a work in progress of Bayer--Perry--Pertusi--Zhao aims at constructing infinitely many such families in arbitrary dimensions from moduli spaces of stable objects in noncommutative abelian surfaces.}. 

We provide a general recipe for such a construction in dimension 6, essentially by reversing the hyper-Kummer construction. The key is the following reconstruction result.

\begin{maintheorem}[Reconstruction, Theorem \ref{thm:biregularCharacterization}]
\label{thm-intro:reconstruction}
Let $Y$ be a hyper-K\"ahler manifold of $\mathrm{K}3^{[3]}$-type. 
Assume that there exist $16$ prime divisors $\{E_i\}_{i=1}^{16}$ on $Y$ satisfying the following conditions: 
\begin{enumerate}[label=(\roman*)]
\item the classes $[E_i]\in H^2(Y,\ZZ)$ are pairwise orthogonal $(-2)$-classes;
\item the saturation of the sublattice $\langle [E_i]\rangle_{i=1,\dots,16}$ of $H^2(Y,\ZZ)$ is the Kummer lattice, and its orthogonal complement is isometric to $\mathrm{U}(2)^{\oplus 3}\oplus \langle -4\rangle$;
\item there exists a birational morphism $Y\to \overline{Y}$ contracting precisely the $16$ divisors $E_i$.  
\end{enumerate}
Then there exists a hyper-K\"ahler manifold $K$ of $\mathrm{Kum}^3$-type such that $\overline{Y}$ is isomorphic to $K/G$, and $Y$ is birational to the associated hyper-Kummer sixfold $Y_K$.
\end{maintheorem}
If $Y$ is projective, we show in Corollary \ref{cor:reconstructionProcedure} that there exists a $(\mathbb{Z}/2\mathbb{Z})^5$-cover $f\colon \widetilde{Y}\to Y$ ramified over the divisors $E_i$ such that the divisors $f^{-1}(E_i)$ can be contracted via a birational morphism $\widetilde{Y}\to Z$, resulting in a variety $Z$ with a crepant resolution of $\mathrm{Kum}^3$-type.
Moreover, in the projective setting condition (iii) can always be achieved after passing to a different smooth birational model of $Y$ (see Remark \ref{rmk:contractionAfterFlops}).

Now we explain our construction of locally complete families of  $\mathrm{Kum}^3$-type varieties. We proceed in three steps:
\begin{enumerate}[label=(Step \arabic*)]
    \item Construct a 4-dimensional locally complete family of projective hyper-Kummer $\mathrm{K}3^{[3]}$-type manifolds $\mathscr{Y}\to B$, together with 16 prime divisors $\{\mathscr{E}_i\to B\}$ with the (fiberwise) lattice-theoretic properties $(i), (ii)$ in Theorem \ref{thm-intro:reconstruction}.
    \item For any $b\in B$, up to replacing $Y_b$ by a smooth birational hyper-K\"ahler model, one has a birational morphism $Y_b\to \overline{Y_b}$ contracting precisely the 16 divisors $E_{i, b}$, where $Y_b$ and $E_{i,b}$ are fibers of $\mathscr{Y}$ and $\mathscr{E}_i$ respectively. Up to base change, this yields a relative birational morphism $\mathscr{Y}\to \overline{\mathscr{Y}}$ contracting precisely the strict transforms of the 16 divisors $\mathscr{E}_i$.
    \item For any $b\in B$ there is a $(\mathbb{Z}/2\mathbb{Z})^5$-Galois cover $\widetilde{Y}_b\to Y_b$ branched along the 16 divisors $E_{i,b}$ with ramification index $2$.
    We obtain a relative finite ramified cover $\widetilde{\mathscr{Y}}\to \mathscr{Y}$, and the Stein factorization of the composition $\widetilde{\mathscr{Y}}\to \mathscr{Y} \to \overline{\mathscr{Y}}$ gives a birational morphism $\widetilde{\mathscr{Y}} \to \mathscr{K}$ over $B$, which contracts the inverse images of the divisors $\mathscr{E}_i$. After taking a simultaneous $\QQ$-factorial terminalization $\widetilde{\mathscr{K}}\to \mathscr{K}$ over $B$, the resulting family is a locally complete family of $\mathrm{Kum}^3$-type varieties. 
\end{enumerate}

For Step 1, one valid source is to first construct a 4-dimensional family of projective hyper-Kummer K3 surfaces with generic Picard rank 16, and then take the relative moduli space of stable sheaves on the fibers with some suitable primitive Mukai vector of square 4. Finding a Mukai vector so that the moduli spaces are birational to hyper-Kummer varieties of $\mathrm{K}3^{[3]}$-type 
is essentially a lattice theoretic problem; on the other hand, the identification of 16 prime divisors  on these moduli spaces with the desired properties requires a more delicate study of their geometry.

In Section \ref{sec:ExamplesHyperKummerK3}, we present two families as in Step 1, namely, the family of Beauville--Mukai systems of torsion sheaves on  Heisenberg-invariant quartic surfaces and that of Hilbert cubes of a family of $(\mathbb{Z}/2\mathbb{Z})^4$-invariant K3 surfaces in $\mathbb{P}^{125}$.
For both families, we describe the $16$ divisors in the moduli space of sheaves with the required properties.

\subsection{Interactions with hyper-K\"ahler manifolds of O'Grady-6 type}
\label{subsec:Intro-OG6}
At present, three deformation types of 6-dimensional compact hyper-K\"ahler manifolds have been discovered: the $\mathrm{Kum}^3$-type, the $\mathrm{K}3^{[3]}$-type \cite{beauville1983varietes} and the OG6-type \cite{O'G03}.
Mongardi--Rapagnetta--Sacc\`a \cite{MRS18} discovered a geometric relation between the last two types: for special hyper-K\"ahler manifolds of OG6-type called \textit{OG6-resolutions} (Definition \ref{def:OG6-resolution}) there is a rational double cover given by a $\mathrm{K}3^{[3]}$-type manifold, which we call the \textit{MRS double cover} of the OG6-resolution.

\begin{center}
    \begin{tikzpicture}[
    box/.style={
        draw,
        rounded corners,
        minimum width=2cm,
        minimum height=1.5cm,
        align=center,
        fill=blue!10
    },
    arrow/.style={
        -{Stealth[scale=0.8]},
        thick
    }
]

\node[box] (left) {$\mathrm{Kum}^3$-type};
\node[box, right=4cm of left] (middle) {$\mathrm{K}3^{[3]}$-type};
\node[box, right=4cm of middle] (right) {OG6-resolutions};

\draw[arrow] (left) -- node[above, midway] {hyper-Kummer } (middle);
\draw[arrow] (left) -- node[below, midway, sloped] {construction} (middle);
\draw[arrow] (right) -- node[above, midway] {MRS double cover} (middle);
\draw[arrow] (right) -- node[below, midway, sloped] {construction} (middle);

\end{tikzpicture}
\end{center}

This \textit{synergy} between the known deformation types of hyper-K\"ahler sixfolds, pictured above, is particularly rich and powerful, especially for the study of algebraic cycles. 

The construction of Mongardi--Rapagnetta--Sacc\`a yields a $5$-dimensional family of manifolds of $\mathrm{K}3^{[3]}$-type, which we compare with the $5$-dimensional family of hyper-Kummer manifolds of $\mathrm{K}3^{[3]}$-type. Firstly, we give a lattice-theoretic characterization of the MRS double covers among manifolds of $\mathrm{K}3^{[3]}$-type, in the same spirit of the characterization of hyper-Kummer sixfolds in Theorem \ref{thm-intro:BirationalCharacterization}. The role of the Kummer lattice is replaced by another rank-16 negative definite lattice, called the Barnes--Wall lattice; see Section \ref{subsec:Barnes--Wall-Lattice} for generalities on this lattice.
\begin{maintheorem}[Characterization, Theorem \ref{thm:characterizationMRS}]
\label{thm-intro:characterization-MRS-Double}
	Let $Z$ be a hyper-K\"ahler manifold of $\mathrm{K}3^{[3]}$-type. Then $Z$ is birational to the MRS double cover of some $\mathrm{OG}6$-resolution $X$ if and only if there exists a primitive embedding of the Barnes--Wall lattice $\mathrm{BW}_{16}$ into the  $\NS(Z)$.
\end{maintheorem}
Comparing Theorem \ref{thm-intro:characterization-MRS-Double} with Theorem \ref{thm-intro:BirationalCharacterization}, one observes that MRS double covers and hyper-Kummer manifolds form \textit{different} $5$-dimensional families of $\mathrm{K}3^{[3]}$-type manifolds.
Nevertheless, in the projective case the two families are closely related.

\begin{maintheorem}[Comparison hyper-Kummer and MRS, Theorem \ref{thm:comparisonHyperkummerMRS}]
\label{thm-intro:Comparison-HyperKummer-MRS}
    Let $K$ be a variety of $\mathrm{Kum}^3$-type, and let $Y_K$ be the associated hyper-Kummer sixfold of $\mathrm{K}3^{[3]}$-type. Assume that $K$ admits a primitive polarization $h$ with divisibility $\mathrm{div}(h)> 1$. Then, $Y_K$ is birational to an MRS double cover.
    
    Conversely, let $X$ be any projective $\mathrm{OG}6$-resolution, and let $Z_X$ be its MRS double cover. Then $Z_X$ is birational to a hyper-Kummer $\mathrm{K}3^{[3]}$-type variety.
\end{maintheorem}

Notice that the second part of the Theorem implies that \textit{any} projective $\mathrm{OG}6$-resolution admits a rational map of degree $2^6$ from a variety of $\mathrm{Kum}^3$-type.

\subsection{Applications to algebraic cycles}\label{subsec-intro:applicationToCycles}
The hyper-Kummer construction gives a new and effective way to tackle deep conjectures on algebraic cycles for hyper-K\"ahler manifolds of $\mathrm{Kum}^3$-type. It constitutes a crucial ingredient in the proof of the Hodge conjecture for such varieties, obtained in \cite{floccariHCKum3}. In our work \cite{FloccariFu-HodgeConjectureWeilFourfolds}, we instead employed the construction of Mongardi--Rapagnetta--Sacc\`a to obtain a new short proof of the Hodge conjecture for abelian fourfolds of Weil-type with discriminant 1, which was proven by Markman \cite{markman2019monodromy, Markman2025Cycles}. In Sections 
\ref{sec:SuperKummer},
\ref{sec:BeauvilleConjecture}, \ref{sec:HodgeTateConjecture}, we obtain some further applications of these constructions to the study of algebraic cycles.

\paragraph{Beauville's weak splitting  conjecture}
Mumford's argument in \cite{Mumford-RationalEquivalence} shows that the Chow group $\CH^i$ of a hyper-K\"ahler variety is infinite dimensional when $i>1$. On the other hand, based on the results established by Beauville and Voisin \cite{Beauville-Voisin} on cycles on K3 surfaces, Beauville conjectured in \cite{Beauville-SplittingBBFiltration} that  for any hyper-K\"ahler variety, the cycle class map should be injective on the subalgebra of the Chow ring generated by the divisor classes.

\begin{maintheorem}[Theorem \ref{thm:BeauvilleConjecture}]
\label{thm-intro:BeauvilleConj}
Let $K$ be a projective hyper-K\"ahler manifold of $\mathrm{Kum}^3$-type. The restriction of the cycle class map $\mathrm{cl}\colon \mathrm{CH}^{*}(K)_{\QQ}\to H^{2*}(K,\QQ)$ to the subalgebra of $\CH^{*}(K)_{\QQ}$ generated by divisor classes is injective.
\end{maintheorem}

Theorem \ref{thm-intro:BeauvilleConj} is the first time that Beauville's  weak splitting conjecture is proved for an entire deformation type of hyper-K\"ahler manifolds of dimension $>2$; see \cite{Voisin-ChowRingHK}, \cite{Fu-BeauvilleVoisinKummer}, \cite{Riess2016WeakSplittingProperty}, \cite{NegutMaulik} for previous evidence.

\paragraph{Homological motives and Hodge conjecture}

The Kuga--Satake construction associates an abelian variety $\mathrm{KS}(X)$ with any hyper-K\"ahler variety $X$. The Kuga--Satake abelian variety only depends on the Hodge lattice $H^2(X,\ZZ)$. 
A general expectation for the motives of hyper-K\"ahler varieties is the following conjecture. 

\begin{conjecture}
\label{conj-intro:MotivationByKS}
    Let $X$ be a projective hyper-K\"ahler variety, and let $\mathrm{KS}(X)$ be the associated Kuga--Satake abelian variety. Then $X$ is motivated by $\mathrm{KS}(X)$, i.e., the motive of $X$ belongs to the thick tensor subcategory of motives generated by the motive of $\mathrm{KS}(X)$.  
\end{conjecture}

The conjecture depends on the choice of a suitable category of motives. In the category of Andr\'e motives \cite{andre1996Motives}, the conjecture holds for any K3 surface by \cite{Andre1996}, and it has been verified for any hyper-K\"ahler variety of known deformation type in \cite{soldatenkov19, FFZ}. In the setting of homological motives, the conjecture is wide open already for K3 surfaces.

We can verify Conjecture \ref{conj-intro:MotivationByKS} for the homological motive of any of the hyper-K\"ahler varieties involved in the hyper-Kummer and MRS constructions (Corollary \ref{cor:HodgeConjecture}). This follows from the following more general result, which is deduced from \cite{floccariKum3, floccariVaresco, floccariHCKum3, FloccariFu-HodgeConjectureWeilFourfolds}.

\begin{maintheorem}[Homological motives, Theorem \ref{thm:motivationbyKSknownCases}, Corollaries \ref{cor:HCPower} and \ref{cor:TateConjecture}]\label{intro-thm:motivationbyKSknownCases}
    Let $X$ be one of the following hyper-K\"ahler varieties:
    \begin{enumerate}[label=(\roman*)]
    \item any $\mathrm{K}3$ surface or $\mathrm{K}3^{[n]}$-type variety 
    satisfying the following condition: 
    there exists an isometric embedding of rational quadratic spaces $H^2_{\mathrm{tr}}(X,\QQ)\hookrightarrow \mathrm{U}_{\QQ}^{\oplus 3}\oplus \langle -a\rangle_{\QQ}$,
    for some positive integer $a$.
    \item any variety of $\mathrm{Kum}^2$ or $\mathrm{Kum}^3$-type;
    \item any $\mathrm{OG}6$-resolution.
    \end{enumerate}
    Then the Kuga--Satake variety $\mathrm{KS}(X)$ of $X$ is isogenous to a power of an abelian fourfold of Weil type with discriminant $1$, and Conjecture \ref{conj-intro:MotivationByKS} holds for the homological motive of $X$. Moreover, the Hodge conjecture and the Tate conjecture hold for $X$ and any of its powers. 
\end{maintheorem}

The second statement follows from the first one, since the Hodge and Tate conjectures hold for powers of any abelian fourfold of Weil type with discriminant $1$ as is proven in \cite{floccari25}, which relies on the works of O'Grady \cite{O'G21}, Markman \cite{markman2019monodromy}, Voisin \cite{voisinfootnotes} and Varesco \cite{varesco}.

\paragraph{K3 surfaces with abelian Chow motive}
Much less is known about Conjecture \ref{conj-intro:MotivationByKS} in the realm of Chow motives, which is its strongest version, even for K3 surfaces. We are able to verify it for the Chow motive of ``half'' of the countably many $4$-dimensional families of projective hyper-Kummer K3 surfaces, which enjoy the special property that their Hilbert cube $S^{[3]}$ is birational to a hyper-Kummer sixfold, as well as to a MRS double cover. This is in fact the case for all hyper-Kummer K3 surfaces of non-split type in the sense of Theorem \ref{thm-intro:HyperKummerK3Classification}, as we show in Theorem \ref{thm:SuperKummer}.

\begin{maintheorem}[K3 surfaces with abelian Chow motives, Theorem \ref{thm:SuperKummerAbelianMotive} and Corollary \ref{cor:K3withAbelianMotive}]
\label{thm-intro:SuperKummerAbelianMotive}
Let $S$ be a projective $\mathrm{K}3$ surface such that there exists an embedding of quadratic spaces
\[H^2_{\mathrm{tr}}(S,\mathbb{Q})\hookrightarrow\mathrm{U}^{\oplus 2}_{\QQ} \oplus \langle -1\rangle_{\QQ}\oplus \langle -b \rangle_{\QQ}\]
for some positive integer $b \equiv 3 \pmod{4}$.
Then the Chow motive of $S$ lies in the thick tensor subcategory of rational Chow motives generated by the Chow motive of its Kuga--Satake variety, which is a power of an abelian fourfold of Weil type with discriminant~$1$. 
In particular, the rational Chow motive of $S$ is abelian, hence finite dimensional in the sense of Kimura--O'Sullivan.
\end{maintheorem}

This theorem proves the Kimura--O'Sullivan finite dimensionality conjecture \cite{Kimura, OSullivan, AndreMotifs} for infinitely many new $4$-dimensional families of K3 surfaces of generic Picard rank $16$. Such a result was previously available only for a few families of K3 surfaces of generic Picard rank $16$, thanks to \cite{paranjape, Laterveer-2016-K3FiniteDimMotive, Ingalls-Logan-Patashnick}.

\subsection*{Acknowledgments} We are pleased to thank Arend Bayer, Francesco Denisi, Bert van Geemen, Annalisa Grossi, Daniel Huybrechts, Radu Laza, Emanuele Macr\`i, Giacomo Mezzedimi, Giovanni Mongardi, Kieran O'Grady, Alessandra Sarti and Claire Voisin for their interest in this work and helpful discussions.

\section{The hyper-Kummer construction}
\label{sec:hyper-KummerConstruction}

In the paper \cite{floccariKum3}, the first-named author introduced a construction which associates a hyper-K\"ahler manifold $Y_K$ of $\mathrm{K}3^{[3]}$-type with any hyper-K\"ahler manifold $K$ of $\mathrm{Kum}^3$-type. We call this procedure the \textit{hyper-Kummer construction} in the present article, as it is one of our main points that this construction parallels in many ways the classical Kummer construction of K3 surfaces from $2$-dimensional complex tori. 

\subsection{The construction} \label{subsec:hyper-KummerConstruction} 
Let $K$ be a hyper-K\"ahler manifold of $\mathrm{Kum}^3$-type. We denote by $\Aut_0(K)$ the group of automorphisms of $K$ whose induced action on $H^2(K,\ZZ)$ is trivial. It is a general fact about compact hyper-K\"ahler manifolds that this group of automorphisms which act trivially on the second cohomology is deformation invariant, as proved in \cite[Theorem 2.1]{hassettTschinkel}. By \cite{boissiere2011higher}, for any $K$ of $\mathrm{Kum}^3$-type we have an isomorphism
\begin{equation} 
\Aut_0(K)\cong (\ZZ/4\ZZ)^4\rtimes \ZZ/2\ZZ, \end{equation}
where the second factor in the semi-direct product acts on the first via the inverse map. 

\begin{definition}\label{def:subgroupsAut0}
We consider the following normal abelian subgroups of $\Aut_0(K)$:
\begin{enumerate}[label=(\roman*)]
    \item the subgroup $\Gamma \cong (\ZZ/4\ZZ)^4 \triangleleft\, \Aut_0(K)$ of automorphisms acting trivially on $H^2(K, \ZZ)$ and $H^3(K,\ZZ)$;
    \item the subgroup $G\cong (\ZZ/2\ZZ)^5 \triangleleft \,\Aut_0(K)$ of automorphisms acting trivially on $H^2(K, \ZZ)$ and $H^4(K,\ZZ)$;
    \item the subgroup $G_1\cong (\ZZ/2\ZZ)^4 \triangleleft\, \Aut_0(K)$ obtained as the intersection $G_1\coloneqq \Gamma\cap G$, which is the group of  automorphisms of $K$ acting trivially on $H^i(K,\ZZ)$ for all $i\neq 6$.
\end{enumerate}
\end{definition}

The fixed locus of an automorphism in $\Aut_0(K)$ is a union of isolated points, K3 surfaces, and manifolds of $\mathrm{K}3^{[2]}$-type, by \cite{oguiso2020no} and \cite{KMO}. 
By \cite[Lemma 2.4]{floccariKum3}, for each $\sigma\in G\setminus G_1$, the fixed locus $K^\sigma$ contains a unique component $W_{\sigma}\subset K$ of dimension $4$; moreover, these are the only automorphisms in $\Aut_0(K)$ with this property. 
These involutions generate the normal subgroup $G\cong (\ZZ/2\ZZ)^5 $ of $\Aut_0(K)$ defined above.

\begin{remark}
    The subgroup $G$ is in fact defined in \cite[Definition 2.1]{floccariKum3} as the subgroup of $\Aut_0(K)$  generated by automorphisms $g\in \Aut_0(K)$ whose fixed locus contains a $4$-dimensional component. For the equivalence with the definition given here, see \cite[Remark 6.2]{floccariHCKum3}.
\end{remark}

The following theorem is the hyper-Kummer construction.
\begin{theorem}[\cite{floccariKum3}] \label{thm:hyperKummerConstruction}
For any manifold $K$ of $\mathrm{Kum}^3$-type, the quotient $K/G$ admits a crepant resolution $Y_K\to K/G$ with $Y_K$ a hyper-K\"ahler manifold of $\mathrm{K}3^{[3]}$-type, obtained by blowing-up the singular locus of $K/G$. 
More precisely, the total fixed locus $\bigcup_{1\neq g\in G} K^g$ is the union of the $16$ submanifolds $W_\sigma$ for $\sigma \in G\setminus G_1$ (equipped with the reduced scheme structure), and, denoting by $\widetilde{K}\coloneqq \mathrm{Bl}_{\bigcup_{\sigma\in G\backslash G_1} W_\sigma} (K)$, 
we have a commutative diagram
\begin{equation}\label{diag:hyperKummerConstruction}
\begin{tikzcd}
    & \tilde{K} \arrow{rd}{q'} \arrow[swap]{ld}{p'} \\
    K \arrow[swap]{rd}{q} \arrow[dashed]{rr}{r} && Y_K \arrow{ld}{p} \\
     & K/G   
\end{tikzcd}
\end{equation}
where $p,p'$ are blow-up maps and $q,q'$ are the quotient maps with respect to the $G$-action.    
\end{theorem}

\subsection{A constellation of companion hyper-K\"ahler manifolds}
\label{subsec:Companions}
Given a hyper-K\"ahler manifold $K$ of $\mathrm{Kum}^3$-type, the hyper-K\"ahler submanifolds $W_\sigma$ and their intersections, quotients and crepant resolutions thereof, give rise to a very interesting configuration of $\mathrm{K}3^{[2]}$-type hyper-K\"ahler manifolds and K3 surfaces canonically contained in or associated with $K$. We call them the \textit{companion hyper-K\"ahler manifolds} of $K$.

Before giving the details, we first summarize the configuration in a diagram. For any $\sigma\neq \sigma'$ in $G\setminus G_1$, we have canonical symplectic varieties fitting in the following diagram:
\begin{equation}\label{diag:configuration}
\begin{tikzcd}
    K \arrow{r} \arrow[dashed, bend right=-20]{rr} & K/G & Y_K \arrow{l} \\
    W_\sigma 
    \arrow[hook]{u} \arrow[dashed, bend right=-20]{rr} \arrow{r} & W_\sigma/G 
    \arrow[hook]{u} & M_\sigma 
    \arrow{l} 
    \\
    V_{\sigma,\sigma'} 
    \arrow[hook]{u} \arrow{r} \arrow[dashed, bend right=-20]{rr} & V_{\sigma,\sigma'}/G \arrow[hook]{u} & S_{\sigma,\sigma'}
    \arrow{l} 
\end{tikzcd}
\end{equation}
where $Y_K$ is the sixfold obtained via the hyper-Kummer construction (Theorem \ref{thm:hyperKummerConstruction}) which is of $\mathrm{K}3^{[3]}$-type,  $W_\sigma$ and $M_\sigma$ are hyper-K\"ahler manifolds of $\mathrm{K}3^{[2]}$-type, and $V_{\sigma,\sigma'}$ and $S_{\sigma,\sigma'}$ are K3 surfaces; the quotient varieties in the middle column are symplectic orbifolds, whose natural symplectic resolutions are in the right column.
Moreover, when the $\mathrm{Kum}^3$-type manifold $K$ deforms, this whole diagram deforms together with $K$ (see Remark \ref{rmk:deformation}).

We give details of \eqref{diag:configuration} in the next few propositions. Fix a $\mathrm{Kum}^3$-type hyper-K\"ahler manifold $K$.

\begin{proposition}[Canonical hyper-K\"ahler submanifolds {\cite[\S2]{floccariKum3}}]
\label{prop:canonicalSubvarietiesI}
\begin{enumerate}[label=(\roman*)]
\item For each $\sigma\in G\setminus G_1$, the submanifold $W_\sigma \subset K$ is hyper-K\"ahler of $\mathrm{K}3^{[2]}$-type. These submanifolds are all mutually isomorphic.
\item For $\sigma\neq \sigma' $ in $G\setminus G_1$, the intersection $V_{\sigma, \sigma'}\coloneqq W_\sigma \cap W_{\sigma'}$ is a $\mathrm{K}3$ surface. 
We obtain ${{16}\choose{2}}=120$ $\mathrm{K}3$ surfaces contained in $K$, which fall into $15$ distinct isomorphism classes, in general.
\item For pairwise distinct $\sigma, \sigma',\sigma''\in G\backslash G_1$, the intersection $Z_{\sigma, \sigma',\sigma''}\coloneqq W_\sigma \cap W_{\sigma'} \cap W_{\sigma''}$ consists of $4$ distinct points, while $4$ or more distinct components $W_\sigma$ do not intersect. 
\end{enumerate} 
\end{proposition} 

The $G$-action on these canonical submanifolds is described as follows.

\begin{proposition}[$G$-actions {\cite[\S2]{floccariKum3}}]
\label{prop:canonicalSubvarietiesII}
\begin{enumerate}[label=(\roman*)]
\item The subgroup $G\subset \Aut_0(K)$ stabilizes globally\footnote{That is, in the decomposition group of $W_{\sigma}$.} each $W_\sigma$, and, hence, also each of the submanifolds $V_{\sigma, \sigma'}$ and $Z_{\sigma,\sigma',\sigma''}$. 
\item For $\sigma\in G\setminus G_1$, the pointwise stabilizer of $W_\sigma$ is the subgroup $\langle \sigma\rangle\cong \ZZ/2\ZZ$, so that $G$ induces a faithful action of $G/\langle \sigma\rangle\cong (\ZZ/2\ZZ)^4$ on $W_\sigma$. 
\item For $\sigma\neq \sigma'$ in $G\setminus G_1$, the pointwise stabilizer of $V_{\sigma,\sigma'}$ in $G$ is $\langle \sigma,\sigma'\rangle\cong (\ZZ/2\ZZ)^2$, so that $G$ induces a faithful action of $G/\langle \sigma, \sigma'\rangle\cong (\ZZ/2\ZZ)^3$ on $V_{\sigma,\sigma'}$. 
\item For pairwise distinct $\sigma,\sigma',\sigma''$ in $G\setminus G_1$, the pointwise stabilizer of $Z_{\sigma,\sigma',\sigma''}$ in $G$ is $\langle \sigma,\sigma',\sigma''\rangle\cong (\ZZ/2\ZZ)^3$, so that $G$ induces a faithful and transitive action of $G/\langle \sigma,\sigma',\sigma''\rangle\cong (\ZZ/2\ZZ)^2$ on $Z_{\sigma,\sigma',\sigma''}$.
\end{enumerate} 
\end{proposition}

The fixed points are described as follows:

\begin{proposition}[{Fixed loci, see \cite[Proposition 2.12]{floccariKum3}}]
    \label{prop:FixedPoints}
\begin{enumerate}[label=(\roman*)]
	\item For any $\sigma\in G\backslash G_1$, 
	\begin{equation}
		\operatorname{Fix}_K(\sigma)= W_{\sigma} \sqcup \bigsqcup_{\substack{\sigma_1+\sigma_2+\sigma_3=\sigma\\
				\sigma_i \text{ distinct in } G\backslash G_1}}Z_{\sigma_1, \sigma_2, \sigma_3}
	\end{equation}
	consists of a hyper-K\"ahler fourfold of $\mathrm{K}3^{[2]}$-type and $140$ isolated points.
	\item For any $\sigma\in G_1\backslash \{\id\}$, 
	\begin{equation}
		\operatorname{Fix}_K(\sigma)=\bigsqcup_{\substack{\sigma_1+\sigma_2=\sigma \\ \sigma_i \in G\setminus G_1}} V_{\sigma_1, \sigma_2}    
	\end{equation}
	is the disjoint union of $8$ $\mathrm{K}3$ surfaces.
	\item For any $\sigma\in G\backslash G_1$, let $G_\sigma:=G/\langle\sigma\rangle\cong (\ZZ/2\ZZ)^4$ which acts faithfully on $W_{\sigma}$; the projection $G\to G/\langle \sigma\rangle$ restricts to an isomorphism $G_1\xrightarrow{\ \sim\ } G/\langle \sigma\rangle$. Then
	\begin{equation}
		\operatorname{Fix}_{W_{\sigma}}(\sigma')=V_{\sigma,\sigma+\sigma'}\sqcup \bigsqcup_{\substack{\sigma_1+\sigma_2=\sigma'\\
				\sigma_1\neq \sigma\neq \sigma_2 }} Z_{\sigma_1, \sigma_2, \sigma} \quad \forall \sigma'\neq 0\in G_1,
	\end{equation}
	is the disjoint union of a $\mathrm{K}3$ surface and $28$ points.
	\item For any $\sigma\neq \sigma'\in G\backslash G_1$, let $G_{\sigma, \sigma'}:=G/\langle\sigma, \sigma'\rangle\cong (\ZZ/2\ZZ)^3$ which acts faithfully on $V_{\sigma, \sigma'}$; the projection $G\to G_{\sigma, \sigma'}$ identifies $G_{\sigma,\sigma'}$ with $G_1/\langle \sigma+\sigma'\rangle$. Then
	\begin{equation}
		\operatorname{Fix}_{V_{\sigma, \sigma'}}(\sigma'')=Z_{\sigma, \sigma', \sigma''+\sigma}\sqcup Z_{\sigma,\sigma', \sigma''+\sigma'} \quad \forall \ 0\neq \sigma''\neq \sigma+\sigma'\in G_1,
	\end{equation}
	consists of $8$ points. 
\end{enumerate}
\end{proposition}

In fact, not only the quotient $K/G$ admits a hyper-K\"ahler resolution (Theorem \ref{thm:hyperKummerConstruction}), but also all the quotients of $W_{\sigma}$ and $V_{\sigma,\sigma'}$ by $G$ admit natural hyper-K\"ahler resolutions:
\begin{proposition}[Resolutions]
\label{prop:ResolutionQuotients} 
\phantomsection
\begin{enumerate}[label=(\roman*)]
    \item For any $\sigma\in G\setminus G_1$, the quotient $W_\sigma /G$ admits a crepant resolution $M_{\sigma} \to W_\sigma /G$ with $M_\sigma$ a hyper-K\"ahler manifold of $\mathrm{K}3^{[2]}$-type, obtained by a single blow-up of $W_\sigma/G$ along its singular locus. 
    More precisely, the total fixed locus $\bigcup_{1\neq g\in G/\langle\sigma\rangle} (W_\sigma)^g$ is the union of the $15$ $\mathrm{K}3$ surfaces $V_{\sigma,\sigma'}$ for $\sigma'\neq \sigma $ in $G\setminus G_1$, and, denoting by $\widetilde{W}_\sigma\coloneqq \mathrm{Bl}_{\bigcup_{\sigma\neq \sigma' \in G\setminus G_1} V_{\sigma,\sigma'}} (W_\sigma)$, there is a commutative diagram
    \begin{equation}\label{diag:Fujikiresolution}
    \begin{tikzcd}
    & \widetilde{W}_\sigma \arrow[swap]{ld}{{p}'} \arrow{rd}{{q}'} \\
    W_\sigma \arrow[swap]{rd}{{q}} \arrow[dashed]{rr}{{r}_{\sigma}} && M_\sigma \arrow{ld}{{p}}\\
    & W_\sigma / G 
    \end{tikzcd}
    \end{equation} 
    where ${p}'$, ${p}$ are blow-up maps, and ${q}'$, ${q}$ are quotient maps with respect to the action of $G$, or rather, of $G/\langle\sigma\rangle$. All the $M_{\sigma}$'s are isomorphic to each other.    
    \item For any $\sigma \neq \sigma'$ in $G\setminus G_1$, the quotient $V_{\sigma,\sigma'}/G$ admits a crepant resolution $S_{\sigma,\sigma'}\to V_{\sigma,\sigma'}/G$ with $S_{\sigma,\sigma'}$ a $\mathrm{K}3$ surface obtained by blowing-up the singular locus of $V_{\sigma,\sigma'}/G$, which consists of $14$ nodes. More precisely, the fixed locus $\bigcup_{1\neq g\in G/\langle \sigma, \sigma'\rangle} (V_{\sigma,\sigma'})^g$ consists of the $56$ points in the union of $Z_{\sigma,\sigma',\sigma''}$ for $\sigma''\in G\setminus G_1$ distinct from $\sigma $ and $\sigma'$. Denoting by $\widetilde{V}_{\sigma,\sigma'}\coloneqq \mathrm{Bl}_{\bigcup_{\sigma''\in G\setminus G_1, \ \sigma\neq\sigma''\neq \sigma'} Z_{\sigma,\sigma',\sigma''}} (V_{\sigma,\sigma'})$, there is a commutative diagram 
    \begin{equation}\label{diag:resolutionK3}
    \begin{tikzcd}
    & \widetilde{V}_{\sigma, \sigma'} \arrow[swap]{ld}{{p}'} \arrow{rd}{{q}'} \\
    V_{\sigma,\sigma'} \arrow[swap]{rd}{{q}} \arrow[dashed]{rr}{{r}_{\sigma,\sigma'}} && S_{\sigma,\sigma'} \arrow{ld}{{p}}\\
    & V_{\sigma,\sigma'} / G 
    \end{tikzcd}
    \end{equation} 
    where ${p}'$, ${p}$ are blow-up maps, and ${q}'$, ${q}$ are quotient maps with respect to the action of $G$, or rather, of $G/\langle\sigma, \sigma'\rangle$.
\end{enumerate}
\end{proposition}

We shall prove Proposition \ref{prop:FixedPoints} and Proposition \ref{prop:ResolutionQuotients} in Section \ref{subsec:ProofProp}. Statement $(ii)$ is the well-known crepant resolution of the quotient of a K3 surface by a symplectic action of $(\ZZ/2\ZZ)^3$, while case~$(i)$ is a deformation of a construction considered by Fujiki~\cite{fujiki1983}.

Now the diagram \eqref{diag:configuration} being explained, let us point out that a crucial property of this diagram is that it deforms together with the $\mathrm{Kum}^3$-manifold $K$, in the following sense.

\begin{remark}[Deformation of the configuration {\cite[Remark 3.6]{floccariHCKum3}}]\label{rmk:deformation}
    The hyper-Kummer construction works in families. Let $\mathcal{K}\to B$ be a family of $\mathrm{Kum}^3$-manifolds. By \cite[Theorem 2.1]{hassettTschinkel}, the groups $\Aut_0(\mathcal{K}_b)$, when $b$ varies in $B$, form a local system of groups over $B$, and we get sub-local systems of abelian groups $G\supset G_1$ with fibers isomorphic to $(\ZZ/2\ZZ)^5$ and $(\ZZ/2\ZZ)^4$ respectively. We then obtain a family $\mathcal{K}/G\to B$ with fibers $K_b/G$, and the blow-up of the total singular locus yields a family $\mathcal{Y}_{\mathcal{K}}\to B$ of hyper-K\"ahler manifolds of $\mathrm{K}3^{[3]}$-type, whose fiber over $b$ is the hyper-Kummer manifold associated with $\mathcal{K}_b$. 
    Moreover, as the $W_\sigma$'s and $V_{\sigma,\sigma'}$'s are defined in terms of the fixed loci of automorphisms in $G$, they deform with~$K$ (they are trianalytic submanifolds in the terminology of Verbitsky~\cite{VerbitskyTrianalytic}). Indeed, up to a finite \'etale base-change, we obtain $16$ subfamilies $\mathcal{W}_\sigma\to B$ of $\mathcal{K}$ over $B$ of $\mathrm{K}3^{[2]}$-manifolds, as well as $120$ subfamilies $\mathcal{V}_{\sigma,\sigma'}\to B$ of $\mathrm{K}3$ surfaces. Further, we may take fiberwise the quotient by $G$ of these families and simultaneously resolve the singularities, obtaining families $\mathcal{Y}\to B$, $\mathcal{M}_\sigma\to B$, $\mathcal{S}_{\sigma,\sigma'}\to B$ of hyper-K\"ahler manifolds fitting in a diagram as \eqref{diag:configuration} of families over $B$.
\end{remark}

\subsection{The case of a generalized Kummer manifold}
\label{subsec:GeneralizedKummer}
We can use Remark \ref{rmk:deformation} to study the configuration of $\mathrm{K}3^{[n]}$-type manifolds attached to an arbitrary $K$ of $\mathrm{Kum}^3$-type via deformation to the best understood example, namely the generalized Kummer variety~$K^3(A)$ associated with an abelian surface or a $2$-dimensional complex torus $A$. This is in fact how Propositions \ref{prop:canonicalSubvarietiesI} and \ref{prop:canonicalSubvarietiesII} are proved in \cite{floccariKum3}.
We first determine the hyper-K\"ahler varieties associated with $K^3(A)$  up to birational isomorphism.

Recall from \cite{beauville1983varietes} that $K^3(A)$ is defined as the fiber of the composition $A^{[4]}\to A^{(4)}\to A$ over $0\in A$, where the first map is the Hilbert--Chow resolution, and the second one sends~$(a_1,a_2,a_3,a_4)\in A^{(4)} $ to $\sum_i a_i \in A$. We let $A_0^{(4)}\subset A^{(4)}$ be the fiber over $0$ of the latter, so that the restriction of the Hilbert--Chow morphism $\nu\colon K^3(A)\to A_0^{(4)}$ is a crepant resolution. Let us fix some more notation.

\begin{notation}
    Let $A$ be an abelian surface or a $2$-dimensional complex torus. We denote by $A[k]$ the subgroup of points of order $k$ in $A$. Given $\alpha\in A$, we denote by $A_{2,\alpha}$ the subset of those $x\in A$ such that $2x=\alpha$, which is a torsor under $A[2]$. We shall denote by $(\alpha, -1)$ the automorphism of $A$ given by $x\mapsto -x+\alpha$; notice that $A_{2,\alpha}$ is the set of fixed points of~$(\alpha,-1)$.
    Given $\tau\in A[2]$, we let $\Km^{\tau}(A)$ be the minimal resolution of $A/\langle (\tau, -1)\rangle$; then $\Km^{\tau}(A)$ is a K3 surface, isomorphic to the Kummer K3 surface $\Km(A)$ associated with $A$ (which is $\Km^0(A)$ with the above notation). The resolution introduces $16$ $(-2)$-curves $\{C_{\alpha}\}_{\alpha\in A_{2,\tau}}$ on $\Km^{\tau}(A)$.
\end{notation}

By \cite{boissiere2011higher}, we have an identification \begin{equation} 
\Aut_0(K^3(A))= A[4] \rtimes \langle -1\rangle,
\end{equation} 
with the action on $K^3(A)$ given by the restriction of the natural action of $ A[4] \rtimes \langle -1\rangle$ on $A^{[4]}$. Under the above isomorphism, we can explicitly identify the groups introduced in Definition \ref{def:subgroupsAut0}:
\begin{equation}
    \Gamma=A[4], \ \ \ G=A[2]\times \langle -1\rangle, \ \ \ G_1=A[2].
\end{equation}
Notice that any $\sigma\in G\setminus G_1$ can be written uniquely as $(\tau,-1)$ for some $\tau\in A[2]$; we shall write $\sigma_{\tau}$ to denote $(\tau,-1)$, and hence label the elements of $G\setminus G_1$ with those of $A[2]$.

The canonical submanifolds $W_\sigma$ and $V_{\sigma,\sigma'}$ of $K^3(A)$ can be described explicitly as follows (our notation is essentially that of \cite[\S3]{floccariHCKum3}).
\begin{proposition}[{\cite[\S2]{floccariKum3}}]\label{prop:CanonicalSubvarietiesK^3(A)}
    \begin{enumerate}[label=(\roman*)]
    \item For any $\tau\in A[2]$, the fixed locus of the automorphism $\sigma_{\tau} :=(\tau, -1)\in G\setminus G_1$ of $K^3(A)$ is the union of $140$ isolated points and a $4$-dimensional component $W_{\sigma_{\tau}}$ which is the strict transform of 
    \begin{equation}
    \{(a,b,-a+\tau, -b +\tau)\ | \ a,b\in A\}\subset A_0^{(4)}
    \end{equation}
    under the Hilbert--Chow morphism.
    The obvious rational map $A\times A\dashrightarrow W_{\sigma_{\tau}}$ induces an isomorphism 
       $ W_{\sigma_{\tau}}\cong (\Km^{\tau}(A))^{[2]}$.
    \item For $\tau\neq\tau'$ in $A[2]$, the intersection $V_{\sigma_{\tau}, {\sigma_{\tau}}_{'}} = W_{\sigma_{\tau}}\cap W_{{\sigma_{\tau}}_{'}}$ is the strict transform of 
     \begin{equation}
    \{(a,-a+\tau, -a+\tau', a+\tau +\tau')\ | \ a\in A \}\subset A_0^{(4)}
    \end{equation}
    under the Hilbert--Chow morphism. The obvious rational map $A\dashrightarrow V_{\sigma_{\tau}, {\sigma_{\tau}}_{'}}$ induces an isomorphism $ 
    V_{\sigma_{\tau}, {\sigma_{\tau}}_{'}} \cong \Km^{\bar{\tau}}(A/\langle \tau+\tau'\rangle)$,
    where $\bar{\tau}$ denotes the image of $\tau$ in~$(A/\langle \tau+\tau'\rangle)[2]$ (notice that $\tau$ and $\tau'$ have the same image in $A/\langle \tau+\tau'\rangle$).
    \item For pairwise distinct $\tau,\tau',\tau''$ in $A[2]$, the intersection $Z_{\sigma_{\tau}, {\sigma_{\tau}}_{'}, {\sigma_{\tau}}_{''}} = W_{\sigma_{\tau}}\cap W_{{\sigma_{\tau}}_{'}}\cap W_{{\sigma_{\tau}}_{''}}$ consists of the $4$ distinct points of the preimage in $K^3(A)$ of 
    \begin{equation}
        \{(a,-a+\tau, -a+\tau',-a+\tau'') \ | \ a\in A_{2,\tau+\tau'+\tau''}\}\subset A_0^{(4)}.
    \end{equation}
    \end{enumerate}
\end{proposition}
\begin{remark}\label{rmk:actionAut_0}
    Let $g=(\epsilon,\pm 1)\in A[4]\rtimes \langle -1\rangle=\Aut_0(K^3(A))$. Then the action of $g$ sends $W_{\sigma_{\tau}}$ to $W_{{\sigma_{2\epsilon+\tau}}}$, $V_{\sigma_{\tau},{\sigma_{\tau}}_{'}}$ to $V_{\sigma_{2\epsilon+\tau}, {\sigma_{2\epsilon+\tau}}_{'}}$ and $Z_{\sigma_{\tau},{\sigma_{\tau}}_{'},{\sigma_{\tau}}_{''}} $ to $Z_{{\sigma_{2\epsilon+\tau}},{\sigma_{2\epsilon+\tau}}_{'},{\sigma_{2\epsilon+\tau}}_{''}}$.
\end{remark}

The remainder of this section is devoted to describe the hyper-Kummer manifolds $Y_{K^3(A)}$, $M_{\sigma_{\tau}}$ and $S_{\sigma_{\tau},{\sigma_{\tau}}_{'}}$ associated with $K^3(A)$. The following proposition is proved at the end of this Section after some preparations.
\begin{proposition}\label{prop:Y_K^3(A)} \phantomsection \begin{enumerate}[label=(\roman*)]
\item The hyper-Kummer $\mathrm{K}3^{[3]}$-manifold $Y_{K^3(A)}$ is birational to $\Km(A)^{[3]}$.
\item For any $\tau\in A[2]$, the hyper-Kummer $\mathrm{K}3^{[2]}$-manifold $M_{\sigma_{\tau}}$ obtained as crepant resolution of the quotient of $W_{\sigma_{\tau}}\subset K^3(A)$ by $G$ is birational to $(\Km(A))^{[2]}$.
\item For any $\tau \neq \tau'$ in $A[2]$, the hyper-Kummer $\mathrm{K}3$ surface $S_{\sigma_{\tau}, {\sigma_{\tau}}_{'}}$ obtained as crepant resolution of the quotient of $V_{\sigma_{\tau}, {\sigma_{\tau}}_{'}} \subset K^3(A)$ by $G$ is isomorphic to the Kummer $\mathrm{K}3$ surface $\Km(A)$.
    \end{enumerate}
\end{proposition}

For any positive integer $n$, there is a birational morphism $m\colon \Km(A)^{[n]}\to (A/\pm 1)^{(n)}$, given as the composition of the Hilbert--Chow morphism $\Km(A)^{[n]}\to \Km(A)^{(n)}$ and the natural morphism $\Km(A)^{(n)}\to (A/\pm 1)^{(n)}$. Now $(A/\pm 1)^{(n)}$ is singular along the union of $17$ components of codimension~$2$: the big diagonal $\overline{D}\subset (A/\pm 1)^{(n)}$ of points $(\overline{a_1},\overline{a_2}, \dots, \overline{a_n})$ supported on at most $n-1$ distinct points, and the $16$ components $\overline{R}_{\tau}$ of points containing in their support the node of $(A/\pm 1)$ corresponding to $\tau\in A[2]$. 
\begin{remark}\label{rmk:Km(A)^[n]}
    Denote by $\overline{U}\subset (A/\pm 1)^{(n)}$ the subset of the points $(\overline{a_1},\dots,\overline{a_n})$ whose support consists of at least $n-1$ distinct points and contains at most $1$ node (with multiplicity 1). 
    Then $\overline{U}\subset (A/\pm 1)^{(n)}$ is open with complement of codimension $>2$. The singular locus of $\overline{U}$ is the disjoint union of the $17$ components $\overline{D}\cap \overline{U}$ and $\overline{R}_{\tau}\cap \overline{U}$, for $\tau\in A[2]$. 
    Denote by $U\subset \Km(A)^{[n]}$ the open subvariety $m^{-1}(\overline{U})$. The restriction $m_{|_U}\colon U\to \overline{U}$ of $m\colon \Km(A)^{[n]}\to (A/\pm 1)^{(n)}$ is identified with the blow-up of $\overline{U} $ along its singular locus. 
    The exceptional divisor has $17$ components, which give in $\Km(A)^{[n]}$ the Hilbert--Chow divisor $D$ and, for $\tau\in A[2]$, the divisors $R_{\tau}\subset \Km(A)^{[n]}$ of subschemes whose support intersects the exceptional curve $C_{\tau}\subset \Km(A)$, lying over $\overline{D}$ and $\overline{R}_{\tau}$ respectively. 
\end{remark}

Consider the Hilbert--Chow morphism $\nu\colon K^3(A)\to A_0^{(4)}$. It is equivariant for the natural action of $G=A[2]\times \langle -1\rangle$. The proof of Proposition \ref{prop:Y_K^3(A)}.$(i)$ is based on the following isomorphism of symplectic orbifolds.
\begin{lemma}\label{lem:isomorphismOrbifolds}
    There is a natural isomorphism of symplectic orbifolds 
    \begin{equation}
        A^{(4)}_0/ G \xrightarrow{\ \sim \ } (A/\pm 1)^{(3)}.
    \end{equation}
\end{lemma}
\begin{proof}
  Let $a=(a_1,a_2,a_3,a_4)\in A_0^{(4)}$, and consider the set
	\begin{equation} 
 S(a)\coloneqq  \{a_1+a_2, \, a_1+a_3, \, a_1+a_4, \, a_2+a_3, \, a_2+a_4, \, a_3+a_4\} \end{equation}
	of points in $A$. Since $a\in A^{(4)}_0$, we have the equations $a_i+a_j = -a_{i'} -a_{j'}$ whenever $\{i,j,i',j'\} =\{1,2,3,4\}$. Hence, the image of $S(a)$ via the quotient $A\to A/\pm 1$ consists of $3$ points, each repeated twice. We thus obtain a well-defined morphism $\overline{q}\colon A^{(4)}_0\to (A/\pm 1)^{(3)}$ by setting 
 \begin{equation} \label{eq:defBarq}
 \overline{q}(a_1,a_2,a_3,a_4)=(\overline{a_1+a_2}, \overline{a_1+a_3},\overline{a_1+a_4}) \in (A/\pm 1)^{(3)}.
 \end{equation} 

 We now consider the action of the group $G$. It is immediate to check that for any $a\in A_0^{(4)}$ and any $g\in G$, we have $\overline{q}(g(a)) = \overline{q}(a)$. Conversely, let $a,a'\in A_0^{(4)}$ be such that $\overline{q}(a)=\overline{q}(a')$. We claim that then there exists $g\in G$ such that $a=g(a')$. 
 
 Indeed, $\overline{q}(a)=\overline{q}(a')$ if and only if $\pi(S(a))=\pi(S(a'))$. Without loss of generality, we may assume $\overline{a_1+a_j} =\overline{a'_1+a'_j}$, for $j=2,3,4$. 
    \begin{itemize}
        \item[$\bullet$] If we have $a_1+a_j = a_1'+a'_j$ in $A$ for $j=2,3,4,$ we find that $a' = a+\tau$ for some $\tau\in A[2]$. Indeed, we obtain $a_1-a_1' = a'_j-a_j$ for all $j$, and hence $a' = (a_1 - \tau, a_2+\tau, a_3+\tau, a_4+\tau)$ where $\tau\coloneqq a_1-a_1'$. Since $a, a'\in A_0^{(4)}$ by assumption, we must have $2\tau=0$, and hence $a'=a+\tau$ is a translate of $a$ by a point of order $2$.
        \item[$\bullet$] If $a_1+a_j = - a'_1 - a'_j$ in $A$ for $j=2,3,4$, we reduce to the previous case by applying the automorphism $-1$, and find $\tau\in A[2]$ so that $a'=-a+\tau$. 
        \item[$\bullet$] Up to permutations of the $a'_i$ and the automorphism $-1$, the last case to consider is when $a_1+a_j=  a'_1 + a'_j$ for $j=2,3$, but $a_1+a_4= -a_1' - a_4'$. 
    Setting $\tau\coloneqq a_1-a_1'$, we must have $a'=(a_1-\tau, a_2+\tau, a_3+\tau, -a_4 + \tau -2a_1)$. 
    Since $a,a'\in A_0^{(4)}$, we have $0 = \sum_i a'_i = \sum_i a_i - 2a_4 + 2\tau -2a_1$. Therefore, $\theta \coloneqq a_4- \tau +a_1$ is a $2$-torsion point. Substituting for $\tau$, we find $a'=(-a_4+\theta, -a_3+\theta, -a_2+\theta, -a_1+\theta)$, and hence $a'$ is obtained from $a$ applying the automorphism $(\theta, -1)\in G$.
    \end{itemize}
    We have thus shown that the fibers of $\overline{q}\colon A_0^{(4)}\to (A/\pm 1)^{(3)}$ are exactly the orbits under the action of $G$. Therefore, $\overline{q}$ factors as the composition of the quotient morphism $A_0^{(4)}\to A_0^{(4)}/G$ with an isomorphism $A_0^{(4)}/G\xrightarrow{\ \sim \ } (A/\pm 1)^{(3)}$. 
\end{proof}

We now turn our attention to the quotients $W_{\sigma_\tau}/G$ of the $\mathrm{K}3^{[2]}$-type submanifolds $W_{\sigma_{\tau}}$ of $K^3(A)$. We will show that in this case $W_{\sigma_{\tau}}/G$ has a crepant resolution $M_{\sigma_{\tau}}$ birational to $\Km(A)^{[2]}$.

The group $G$ induces a faithful action of $G/\langle \sigma_{\tau} \rangle\cong (\ZZ/2\ZZ)^4$ on $W_{\sigma_{\tau} }$ via symplectic automorphisms. 
If $\epsilon\in A[4]$, conjugation by $\epsilon$ in $A[4]\rtimes \langle -1\rangle$ sends $\sigma_{\tau}$ to $\sigma_{\tau+2\epsilon}$, and the action of $\epsilon $ on $K^3(A)$ induces an isomorphism $W_{\sigma_{\tau} } \xrightarrow{\ \sim \ } W_{\sigma_{\tau+2\epsilon}}$. 
We can therefore find $G$-equivariant isomorphisms $W_{\sigma_{\tau} }\xrightarrow{\ \sim \ } W_{{\sigma_{\tau}}_{'}}$, for any $\tau,\tau'$ in $A[2]$. 
It will thus be sufficient to treat the case of $W_0$ with its $G$-action, in which case we just have $\Km(A)^{[2]}$ with the natural action of $A[2]$ (induced by the action of $A[2]$ on $\Km(A)$).
For any $0\neq \theta \in A[2]$, the graph of $\theta\colon \Km(A)\to \Km(A)$ gives the embedding $V_{\theta} = \Km(A/\langle \theta \rangle) \hookrightarrow \Km(A)^{[2]}$, as component of the fixed locus of $\theta$ acting on $\Km(A)^{[2]}$ (see also \cite{KMO}); in our notation, this is the inclusion $V_{\sigma_0,\sigma_\theta}\hookrightarrow W_0$. 

\begin{proposition} \label{prop:Fujikiresolution}
    We have $\bigcup_{0\neq \theta \in A[2]} (\Km(A)^{[2]})^{\theta} =\bigcup_{0\neq \theta \in A[2]} V_{\theta} $. 
    The blow-up of the singular locus of $\Km(A)^{[2]}/A[2]$ yields a smooth hyper-K\"ahler manifold~$M$ of $\mathrm{K}3^{[2]}$-type, birational to $\Km(A)^{[2]}$. We have a commutative diagram 
    \begin{equation} \begin{tikzcd}
        & \mathrm{Bl}_{\bigcup_{0\neq \theta} V_{\theta} } \Km(A)^{[2]} \arrow{rd}{q'} \arrow[swap]{ld}{p'}  \\
        \Km(A)^{[2]}\arrow[swap]{rd}{q} \arrow[dashed]{rr}{\bar{r}} && M\arrow{ld}{p} \\
        & \Km(A)^{[2]}/A[2]
    \end{tikzcd}
    \end{equation}
    in which $p,p'$ are the blow-up maps and $q,q'$ are the quotient maps for the action of $A[2]$.
\end{proposition}
\begin{proof}
    This proposition is due to Fujiki \cite{fujiki1983}. To give a streamlined proof, we may follow \cite[\S4]{floccariKum3}: one checks explicitly that the total fixed locus is the union of the K3 surfaces $V_{\theta}$, which intersect transversally. The singularities of $\Km(A)^{[2]}/A[2]$ are thus locally isomorphic to products of ordinary nodes on a surface; this implies that $M$ is nonsingular and it fits into a diagram as in the statement. 
   We can now apply \cite[Proposition 2.9]{fujiki1983} to show that the resolution $M$ is a hyper-K\"ahler manifold. It follows from Lemma \ref{lem:isomorphismOrbifolds2} below that $M$ is birational to $\Km(A)^{[2]}$; by \cite[Theorem 4.6]{Huy99}, the hyper-K\"ahler manifold $M$ is of $\mathrm{K}3^{[2]}$-type. 
   \end{proof}

   We shall study an explicit rational map of degree $2^4$
    \begin{equation}
        s\colon \Km(A)^{[2]} \dashrightarrow \Km(A)^{[2]}
    \end{equation}
    which coincides with the quotient map for the $A[2]$-action over an open subset of $\Km(A)^{[2]}$, i.e. such that there exists a birational map $\bar{\psi} \colon M \dashrightarrow \Km(A)^{[2]}$ satisfying $\bar{\psi}\circ \bar{r} = s$.
    It comes from the following.
    \begin{lemma}\label{lem:isomorphismOrbifolds2}
        There exists an isomorphism of orbifolds 
        \begin{equation}
            (A/\pm 1)^{(2)}/A[2] \xrightarrow{ \ \sim \ } (A/\pm 1)^{(2)}
        \end{equation}
    \end{lemma}
    \begin{proof}
        Consider the morphism $\overline{s}\colon (A/\pm 1)^{(2)}\to (A/\pm 1)^{(2)}$ defined by 
        \begin{equation}\label{eq:defMaps}
            \overline{s}(\overline{x},\overline{y}) \coloneqq (\overline{x+y}, \overline{x-y}).
        \end{equation}
        One checks easily that $\overline{s}$ is well-defined. We consider the action of $A[2]$ on $(A/\pm 1)^{(2)}$ via translations; we claim that the fibers of $\overline{s}$ are precisely the $A[2]$-orbits. 
        
        Indeed, $\overline{s}(\overline{x},\overline{y})= \overline{s}(\overline{x}', \overline{y}')$ if and only if we have an equality 
        \begin{equation}
        \{x+y, x-y,-x-y,-x+y\}=\{x'+y', x'-y',-x'-y',-x'+y'\}
        \end{equation}
        of subsets of $A$. This clearly holds if $x'=x+\tau$ and $y'=y+\tau$, for some $\tau\in A[2]$. Conversely, if the above equality holds, up to swapping $x$ and $y$ and changing their sign, we may assume that $x+y=x'+y'$ and $x-y=x'-y'$; then $x-x'=y-y'=y'-y$ is a $2$-torsion point $\tau$, and $(\overline{x}', \overline{y}')=(\overline{x+\tau}, \overline{y+\tau}) $.
        We conclude that $\overline{s}$ is identified with the quotient map with respect to $A[2]$, and, therefore, it induces an isomorphism $(A/\pm 1)^{(2)} /A[2]\xrightarrow{\ \sim \ } (A/\pm 1)^{(2)}$.
    \end{proof}
Next, we consider the quotient by $G$ of the K3 surfaces $V_{\sigma_{\tau},{\sigma_{\tau}}_{'}}\subset K^3(A)$. Recall from Proposition \ref{prop:CanonicalSubvarietiesK^3(A)} $(ii)$ that $V_{\sigma_{\tau},{\sigma_{\tau}}_{'}}\cong \Km^{\bar{\tau}}(A/\langle \tau+\tau'\rangle)$. Moreover, the action of $G$ induces the natural action of $A[2]/\langle \tau+\tau'\rangle$ on $\Km^{\bar{\tau}}(A/\langle \tau+\tau'\rangle)$.

\begin{lemma}\label{lem:resolutionK3}
For any $\tau\neq \tau'\in A[2]$, the quotient $V_{\sigma_{\tau},{\sigma_{\tau}}_{'}}/G$ has $14$ nodes 
    and the blow-up of these nodes gives a $\mathrm{K}3$ surface $S_{\sigma_{\tau},{\sigma_{\tau}}_{'}}$ isomorphic to $\Km(A)$.
\end{lemma}
\begin{proof}
    It is well-known (\cite{NikulinFiniteAutomorphism}) that the quotient of a K3 surface by a symplectic action of $(\ZZ/2\ZZ)^3$ has $14$ nodes, and that blowing-up these nodes yields again a K3 surface. In our specific case, $V_{\sigma_{\tau},{\sigma_{\tau}}_{'}}/G$ is birational to $A/(A[2]\times \langle -1\rangle)\simeq A/\pm 1$ (since $A/A[2]\simeq A$), hence the minimal resolution $S_{\sigma_{\tau},{\sigma_{\tau}}_{'}}$ of $V_{\sigma_{\tau},{\sigma_{\tau}}_{'}}/G$ is isomorphic to the Kummer K3 surface $\Km(A)$.
\end{proof}
\begin{proof}[Proof of Proposition \ref{prop:Y_K^3(A)}]
    Statements $(ii)$ and $(iii)$ follow directly from Proposition~\ref{prop:Fujikiresolution} and Lemma \ref{lem:resolutionK3}, respectively. To prove $(i)$, let $q\colon K^3(A)\dashrightarrow \Km(A)^{[3]}$ be the rational map defined by the commutative diagram 
    \begin{equation}
        \begin{tikzcd}
            K^3(A)\arrow{d}{\nu} \arrow[dashed]{r}{q} & \Km(A)^{[3]} \arrow{d} \\
            A_0^{(4)} \arrow{r}{\overline{q}} & (A/\pm 1)^{(3)}
        \end{tikzcd}
    \end{equation}
    where $\overline{q}\colon A_0^{(4)}\to (A/\pm1 )^{(3)}$ is the map defined in the proof of Lemma \ref{lem:isomorphismOrbifolds}.
    The vertical maps are birational. By Lemma \ref{lem:isomorphismOrbifolds}, we conclude that $\Km(A)^{[3]}$ is birational to $K^3(A)/G$, and, hence, to its resolution $Y_{K^3(A)}$.
\end{proof}

\subsection{Proofs of Proposition \ref{prop:FixedPoints} and Proposition \ref{prop:ResolutionQuotients}}
\label{subsec:ProofProp}
Now we give the promised proofs of Propositions \ref{prop:FixedPoints} and \ref{prop:ResolutionQuotients}.

\begin{proof}[Proof of Proposition \ref{prop:FixedPoints}]
     Let $K$ be a $\mathrm{Kum}^3$-manifold. Choose a deformation $\mathcal{K}\to B$ of $K=\mathcal{K}_{b}$ to a generalized Kummer sixfold $K^3(A)=\mathcal{K}_{b'}$ associated with an abelian surface $A$. By Remark \ref{rmk:deformation}, it suffices to prove the proposition for $K^3(A)$.   
In this case, $G=A[2]\times \{\pm 1\}$ and $G_1=A[2]$. Hence $G\backslash G_1=\{\sigma_\tau ~|~ \tau \in A[2]\}$, where $\sigma_\tau:=(\tau, -1)$.
\begin{itemize}
	\item For any $\tau\in A[2]$, the fixed locus of the involution $\sigma_\tau$ on $K$ is the disjoint union of $W_{\sigma_\tau}$ and
	$140$ points $\{\{\epsilon_1, \epsilon_2, \epsilon_3, \epsilon_4\}~|~ \epsilon_i \text{ distinct},~2\epsilon_i=\tau,~\sum_i\epsilon_i=0\}$. Note that these $140$ points are precisely the disjoint union of all the $Z_{\sigma_{\tau_1}, \sigma_{\tau_2}, \sigma_{\tau_3}}$ with $\tau_1, \tau_2, \tau_3$ distinct and summing to $\tau$:
	\begin{equation}
		\operatorname{Fix}_K(\sigma_\tau)= W_{\sigma_\tau} \sqcup \bigsqcup_{\substack{\tau_1+\tau_2+\tau_3=\tau\\
				\tau_i \text{ distinct}}}Z_{\sigma_{\tau_1}, \sigma_{\tau_2}, \sigma_{\tau_3}}.
	\end{equation}
	
	\item For any $0\neq \tau\in A[2]$, the fixed locus of  the involution $(\tau, 1)$ on $K$ is the disjoint union of eight K3 surfaces: $$\operatorname{Fix}_K(\tau, 1)=\bigsqcup_{\tau_1+\tau_2=\tau} V_{\sigma_{\tau_1}, \sigma_{\tau_2}}.$$
	
	\item For each $\tau\in A[2]$ , the action of $G$ on $W_{\sigma_\tau}$ is not faithful, with kernel generated by the involution $\sigma_\tau=(\tau, -1)$. Set  $G_{\sigma_{\tau}}=G/\langle(\tau, -1)\rangle$, which acts faithfully on $W_{\sigma_{\tau}}$. The fixed loci of elements of $G_{\sigma_{\tau}}$ on $W_{\sigma_{\tau}}$ can be obtained by intersecting the above fixed loci with $W_{\sigma_{\tau}}$. We have that the fixed locus is always the disjoint union of a K3 surface and 28 points:
	
	\begin{equation}
	\operatorname{Fix}_{W_{\sigma_{\tau}}}(\tau', -1)=V_{\sigma_{\tau},\sigma_{\tau'}}\sqcup \bigsqcup_{\substack{\tau_1+\tau_2=\tau+\tau'\\
				\tau_1\neq \tau\neq \tau_2 }}Z_{\sigma_{\tau_1}, \sigma_{\tau_2}, \sigma_{\tau}} \quad \forall \tau'\neq \tau\in A[2];
	\end{equation}
	or equivalently,
	\begin{equation}
		\operatorname{Fix}_{W_{\sigma_{\tau}}}(\tau', 1)=V_{\sigma_{\tau},\sigma_{\tau+\tau'}}\sqcup \bigsqcup_{\substack{\tau_1+\tau_2=\tau'\\
				\tau_1\neq \tau\neq \tau_2 }} Z_{\sigma_{\tau_1}, \sigma_{\tau_2}, \sigma_{\tau}} \quad \forall \tau'\neq 0\in A[2].
	\end{equation}
	
	\item For any $\tau\neq \tau'\in A[2]$, the action of $G$ on $V_{\sigma_{\tau}, \sigma_{\tau'}}$ is not faithful, with kernel generated by the two involutions $(\tau,-1)$ and $(\tau', -1)$. Set $G_{\sigma_{\tau}, \sigma_{\tau'}}=G/\langle(\tau, -1), (\tau', -1)\rangle$, which acts on $V_{\sigma_{\tau}, \sigma_{\tau'}}$ faithfully.  The fixed loci are always 8 points:
	\begin{equation}
		\operatorname{Fix}_{V_{\sigma_{\tau}, \sigma_{\tau'}}}(\tau'', -1)=Z_{\sigma_{\tau}, \sigma_{\tau'}, \sigma_{\tau''}} \sqcup Z_{\sigma_{\tau}, \sigma_{\tau'}, \sigma_{\tau+\tau'+\tau''}} \quad \forall \tau'\neq \tau''\neq \tau\in A[2];
	\end{equation}
	or equivalently,
	\begin{equation}
		\operatorname{Fix}_{V_{\sigma_\tau, \sigma_{\tau'}}}(\tau'', 1)=Z_{\sigma_{\tau}, \sigma_{\tau'}, \sigma_{\tau+\tau''}}\sqcup Z_{\sigma_{\tau}, \sigma_{\tau'}, \sigma_{\tau'+\tau''}}  \quad \forall 0\neq \tau''\neq \tau+\tau'\in A[2].
	\end{equation}
\end{itemize}
\end{proof}

\begin{proof}[Proof of Proposition \ref{prop:ResolutionQuotients}]
 Let $K$ be a $\mathrm{Kum}^3$-manifold. Choose a deformation $\mathcal{K}\to B$ of $K=\mathcal{K}_{b}$ to a generalized Kummer sixfold $K^3(A)=\mathcal{K}_{b'}$.
 By Remark \ref{rmk:deformation}, up to a finite \'etale base-change we obtain $16$ subfamilies $\mathcal{W}_{\sigma}\to B$ of $\mathrm{K}3^{[2]}$-manifolds, and $120$ subfamilies of K3 surfaces $\mathcal{V}_{\sigma,\sigma'}\to B$. 

 To prove $(ii)$, we can then take the quotient $\mathcal{V}_{\sigma,\sigma'}/G \to B$ over $B$. We obtain a locally trivial family of nodal surfaces; possibly up to a further finite \'etale base-change, the nodes are given by $14 $ sections. Blowing-up these sections gives a smooth family~$\mathcal{S}_{\sigma,\sigma'}\to B$ of K3 surfaces. 

 As for $(i)$, similarly, we have a family $\mathcal{W}_\sigma \to B$ and we consider the fiberwise quotient $\mathcal{W}_\sigma /G \to B$, which is a locally trivial family. The singular loci of the fibers yield $15$ locally trivial subfamilies $\mathcal{V}_{\sigma,\sigma'}/G\to B$. The blow-up of $\mathcal{W}_{\sigma}/G$ along the union of the subfamilies $\mathcal{V}_{\sigma,\sigma'}/G$ is again a locally trivial family $\mathcal{M}_{\sigma}\to B$. 
 But, by Proposition \ref{prop:Fujikiresolution}, we know that $\mathcal{M}_{\sigma,b'}$ is a smooth hyper-K\"ahler manifold birational to $\Km(A)^{[2]}$. It follows that $\mathcal{M}_\sigma\to B$ is a smooth family of $\mathrm{K}3^{[2]}$-manifolds. The varieties $M_{\sigma}$'s are pairwise isomorphic because the $W_{\sigma}$'s are $G$-equivariantly isomorphic to each other, by Remark \ref{rmk:actionAut_0}.
    
 The commutative diagram in the statement of Proposition \ref{prop:ResolutionQuotients} can be obtained by noticing that locally analytically the singularities of $W_\sigma /G$ are products of ordinary nodes on surfaces, similarly to the case studied in \cite[\S4]{floccariKum3} (cf. Proof of Proposition \ref{prop:Fujikiresolution}).
\end{proof}

\section{Cohomology lattices of companion manifolds}
\label{sec:CohomologyOfCompanion}

In this section, we analyze the induced maps between the second cohomology groups of the companion hyper-K\"ahler manifolds in Diagram \eqref{diag:configuration}. 
Given a hyper-K\"ahler manifold $K$ of $\mathrm{Kum}^3$-type, let the notation be as in Section \ref{sec:hyper-KummerConstruction}. We denote by $\iota_{\sigma}\colon W_{\sigma}\hookrightarrow K$ and $\iota_{\sigma,\sigma'}\colon V_{\sigma,\sigma'}\hookrightarrow K$ the closed immersions of the associated natural submanifolds of $K$. For any  $\sigma\neq \sigma'$ in $G\setminus G_1$, we consider the diagrams 
\begin{equation}
\begin{tikzcd}
    K \arrow[dashed]{rr}{r} && Y_K &&& K \arrow[dashed]{rr}{r} && Y_K
    \\
    W_{\sigma} \arrow[hook]{u}{\iota_{\sigma}} \arrow[dashed]{rr}{r_{\sigma}} && M_{\sigma} &&& V_{\sigma,\sigma'}  \arrow[hook]{u}{\iota_{\sigma,\sigma'}} \arrow[dashed]{rr}{r_{\sigma,\sigma'}} && S_{\sigma,\sigma'}
\end{tikzcd}\end{equation}
There are restriction maps 
\begin{equation}
    \iota^*_{\sigma}\colon H^2(K,\ZZ)\to H^2(W_\sigma,\ZZ) \ \ \ \text{and} \ \ \ \iota_{\sigma,\sigma'}^*\colon H^2(K,\ZZ)\to H^2(V_{\sigma,\sigma'},\ZZ),
\end{equation}
as well as push-forward maps (once we restrict to the subgroup of $G$-invariants):
\begin{equation}
\begin{split}
r_*&\colon H^2(K,\ZZ)\to H^2(Y_K,\ZZ)\ , \\
(r_{\sigma})_{*}&\colon H^2(W_{\sigma},\ZZ)^G \to H^2(M_{\sigma},\ZZ)\ , \\
(r_{\sigma,\sigma'})_{*} & \colon H^2(V_{\sigma,\sigma'},\ZZ)^G\to H^2(S_{\sigma,\sigma'},\ZZ).
\end{split}
\end{equation}
We define $r_*$ (resp. $(r_{\sigma})_{*}$, $(r_{\sigma,\sigma'})_{*}$) as the composition $(q')_*\circ (p')^*$ in diagram \eqref{diag:hyperKummerConstruction} (resp. \eqref{diag:Fujikiresolution}, \eqref{diag:resolutionK3}). Notice that the domain of $r_*$ is the whole of $H^2(K,\ZZ)$, since $G$ acts trivially on it.

\begin{remark}[Overlattice]
\label{rmk:hatH^2} 
The lattice $H^2(K,\ZZ)\simeq \mathrm{U}^{\oplus 3}\oplus \langle -8\rangle$ admits a unique index-2 overlattice $\widehat{H}^2(K,\ZZ)$, which is isometric to $\mathrm{U}^{\oplus 3}\oplus \langle -2\rangle$; see Remark \ref{rmk:unique-overlattice}. We will regard $\widehat{H}^2(K,\ZZ)$ as a $\ZZ$-Hodge structure with the Hodge structure induced by that on $H^2(K,\ZZ)$. Since $\widehat{H}^2(K,\ZZ)$ is the unique overlattice of $H^2(K,\ZZ)$, any isometry of $H^2(K,\ZZ)$ extends to an isometry of $\widehat{H}^2(K,\ZZ)$.
\end{remark}

\subsection{Statement of results}
Let us state the main results of this section.
Let $K$ be a $\mathrm{Kum}^3$-type hyper-K\"ahler manifold. 
Consider the commutative diagram of the hyper-Kummer construction in Theorem \ref{thm:hyperKummerConstruction}:
\begin{equation}
\begin{tikzcd}
& \widetilde{K} \arrow{ld} \arrow{rd}\\
K \arrow[dashed]{rr}{r} \arrow[swap]{rd} && Y_K \arrow{ld} \\
& K/G
\end{tikzcd}
\end{equation}  
The resolution $Y_K\to K/G$ introduces $16$ exceptional divisors $\{E_\sigma\}_{\sigma\in G\setminus G_1}$. Their cohomology classes $[E_\sigma]\in H^2(Y_K,\ZZ)$ are orthogonal to the image of $r_*\colon H^2(K,\ZZ)\to H^2(Y_K,\ZZ)$, and the second cohomology of $Y_K$ is generated over $\QQ$ by the image of $r_*$ and the 16 classes $\{[E_\sigma]\}_{\sigma\in G\setminus G_1}$. 

\begin{proposition}[Cohomology of $Y_K$]
\label{prop:cohomologyY_K}
    For any hyper-K\"ahler manifold $K$ of $\mathrm{Kum}^3$-type, let $Y_K$ be the associated hyper-Kummer $\mathrm{K}3^{[3]}$-manifold.
    The following statements hold:
    \begin{enumerate}[label=(\roman*)]
        \item The push-forward map $r_*\colon H^2(K,\ZZ)\to H^2(Y_K,\ZZ)$ is injective, divisible by $2^4$, and multiplies the quadratic form by a factor $2^9$;
        \item The map $\tfrac{1}{2^4}r_*\colon H^2(K,\ZZ)(2)\to H^2(Y_K,\ZZ)$ is an embedding of lattices which extends to a primitive embedding of lattices and of Hodge structures
        \begin{equation}
            \frac{1}{2^4}r_*\colon \widehat{H}^2(K,\ZZ)(2)\hookrightarrow  H^2(Y_K,\ZZ),
        \end{equation} 
        where on the left-hand side,  $\widehat{H}^2(K,\ZZ)$ is the unique index-2 overlattice of $H^2(K,\ZZ)$, and $(2)$ means that we multiply the quadratic form by a factor $2$. 
        In particular, $\widehat{H}^2(K,\ZZ)(2)\simeq \mathrm{U}(2)^{\oplus 3}\oplus \langle -4\rangle$.
        \item The classes $[E_\sigma]\in H^2(Y_K,\ZZ)$ are pairwise orthogonal $(-2)$-classes; the saturation of the sublattice generated by the $[E_\sigma]$ is the Kummer lattice, and it is the orthogonal complement of the image of the embedding $\frac{1}{2^4}r_*$.
    \end{enumerate}
\end{proposition}
The proof of Proposition \ref{prop:cohomologyY_K} will be given at the end of Section \ref{subsec:CohomologyYK}. For the companion submanifolds $W_{\sigma}$ and $V_{\sigma,\sigma'}$ we have the following result. Its proof will be given in Section \ref{subsec:CohomologyWV}.

\begin{proposition}[Cohomology of $W_\sigma$ and $V_{\sigma,\sigma'}$]
\label{prop:cohomologyW&V}
Let $K$ be a hyper-K\"ahler manifold of $\mathrm{Kum}^3$-type.
\begin{enumerate}[label=(\roman*)]
\item For any $\sigma\in G\setminus G_1$, the pull-back along the embedding $\iota_{\sigma}\colon W_{\sigma}\to K$ induces a primitive embedding of lattices and Hodge structures
    \begin{equation}
        \iota^*_{\sigma}\colon H^2(K,\ZZ)(2)\hookrightarrow H^2(W_{\sigma},\ZZ).
    \end{equation}
    The orthogonal complement to the image of $\iota^*_{\sigma}$ in $H^2(W_{\sigma},\ZZ)$ is isometric to the rank $16$ negative definite lattice $N_W$ of Definition \ref{def:latticeN_W}.
    \item For any $\sigma\neq \sigma'$ in $G\setminus G_1$, the pull-back along the embedding $\iota_{\sigma,\sigma'}\colon V_{\sigma, \sigma'}\to K$ induces an embedding of lattices and Hodge structures
    \begin{equation}
    \iota^*_{\sigma,\sigma'}\colon H^2(K,\ZZ)(4)\hookrightarrow H^2(V_{\sigma,\sigma'},\ZZ).
    \end{equation}
    The image of $\iota^*_{\sigma,\sigma'}$ is not saturated and has index $2^4$ in its saturation, which is isometric to the lattice $\mathrm{U}(2)^{\oplus 3} \oplus \langle -8\rangle$. The orthogonal complement to $\mathrm{im}(\iota^*_{\sigma,\sigma'})$ is isometric to the rank $15$ negative definite lattice $N_V$ of Definition \ref{def:latticeN_V}.
\end{enumerate}      
\end{proposition}

Finally, we describe the cohomology of the hyper-Kummer $\mathrm{K}3^{[2]}$-manifolds $M_{\sigma}$ and K3 surfaces $S_{\sigma,\sigma'}$ in the following result, whose part $(i)$ will be proved in Section \ref{subsec:CohomologyM} and part $(ii)$ will be proved in Section \ref{subsec:CohomologyS}.

\begin{proposition}[Cohomology of $M_\sigma$ and $S_{\sigma, \sigma'}$]
\label{prop:cohomologyM&S}
Let $K$ be a hyper-K\"ahler manifold of $\mathrm{Kum}^3$-type.
    \begin{enumerate}[label=(\roman*)]
    \item The composition $(r_\sigma)_*\circ\iota^*_{\sigma}\colon H^2(K,\ZZ) \to H^2(M_{\sigma},\ZZ)$ is injective and multiplies the form by a factor $2^7$. The map $\tfrac{1}{2^3}((r_\sigma)_*\circ\iota^*)$ is integral and extends to a primitive embedding of lattices and Hodge structures 
        \begin{equation}
            \frac{1}{2^3}((r_\sigma)_*\circ\iota_{\sigma}^*)\colon \widehat{H}^2(K,\ZZ)(2)\hookrightarrow H^2(M_\sigma,\ZZ),
        \end{equation}
        with orthogonal complement isometric to the rank $16$ negative definite lattice $N_M$ of Definition \ref{def:latticeN_M}.
        \item The composition $(r_{\sigma,\sigma'})_*\circ\iota^*_{\sigma,\sigma'}\colon H^2(K,\ZZ) \to H^2(S_{\sigma,\sigma'},\ZZ)$ is injective and multiplies the form by a factor $2^5$. The map $\tfrac{1}{2^2}(r_{\sigma,\sigma'})_*\circ\iota_{\sigma,\sigma'}^*$ is integral and extends to a primitive embedding of lattices and Hodge structures 
        \begin{equation}
            \frac{1}{2^2}((r_{\sigma,\sigma'})_*\circ\iota_{\sigma,\sigma'}^*)\colon \widehat{H}^2(K,\ZZ)(2)\hookrightarrow H^2(S_{\sigma,\sigma'},\ZZ),
        \end{equation}
        with orthogonal complement isometric to the rank $15$ negative definite lattice $N_S$ of Definition \ref{def:latticeN_S}.
    \end{enumerate}
\end{proposition}

\subsection{
Transcendental lattices and moduli spaces}
The following result follows from the description of the cohomology of hyper-Kummer manifolds.
\begin{theorem}
\label{thm:TranscendentalResolutions}
    Let $K$ be a $\mathrm{Kum}^3$-manifold, and consider the associated hyper-K\"ahler sixfold $Y_K$, fourfolds $M_{\sigma}$ and $\mathrm{K}3$ surfaces $S_{\sigma,\sigma'}$. 
    The transcendental cohomology of any of these hyper-Kummer manifolds is Hodge isometric to $\widehat{H}^2_{\mathrm{tr}}(K,\ZZ)(2)$.
\end{theorem}
\begin{proof}
    Proposition \ref{prop:cohomologyY_K} implies that $H^2_{\mathrm{tr}}(Y_K,\ZZ) $ is Hodge isometric to $\widehat{H}^2_{\mathrm{tr}}(K,\ZZ)(2)$. We have the same for $M_{\sigma}$ and $S_{\sigma,\sigma'}$ by Proposition \ref{prop:cohomologyM&S}.
\end{proof}

When $K$ is projective, we have the following consequence. 
\begin{corollary}
\label{cor:Relation-YKMKSK}
Let $K$ be a projective manifold of $\mathrm{Kum}^3$-type. Then:
\begin{enumerate}[label=(\roman*)]
\item the $\mathrm{K}3$ surfaces $S_{\sigma,\sigma'}$ associated with $K$ are all isomorphic to each other; we denote by $S_K$ this $\mathrm{K}3$ surface;
\item the hyper-Kummer $\mathrm{K}3^{[3]}$-variety $Y_{K}$ is a smooth and projective moduli space of Bridgeland stable objects in the derived category of $S_K$, for some primitive Mukai vector of square $4$;
\item for any $\sigma$, the hyper-Kummer $\mathrm{K}3^{[2]}$-variety $M_{\sigma}$ is a smooth and projective moduli space of Bridgeland stable objects in the derived category of $S_K$, for some primitive Mukai vector of square $2$.
\end{enumerate} 
\end{corollary}
\begin{proof}
    The first statement follows from the general fact that projective K3 surfaces of Picard rank at least $12 $ and Hodge isometric transcendental lattices are isomorphic. The  assertions $(ii)$ and $(iii)$ follow from Theorem \ref{thm:TranscendentalResolutions} via \cite[Proposition 1]{Addington2016} and \cite{BM14a}, or \cite[Proposition 3.3]{mongardi2015}. 
\end{proof}

\begin{definition}[Hyper-Kummer K3 surfaces]
\label{def:HyperKummerK3}
Given a projective hyper-K\"ahler manifold $K$ of $\mathrm{Kum}^3$-type, we denote by $S_K$ the isomorphism class of the K3 surfaces $S_{\sigma, \sigma'}$, called the \textit{hyper-Kummer} K3 surface associated with $K$. 
A K3 surface $S$ is called \textit{hyper-Kummer} if it is isomorphic to $S_K$ for some variety $K$ of $\mathrm{Kum}^3$-type.
\end{definition}

Before embarking in the proofs of the above results, we prove a lemma which generalizes \cite[Lemma 4.8]{floccariKum3}. 
\begin{lemma}\label{lem:scalarFactor}
Let $X$ be a hyper-K\"ahler manifold of dimension $2n$ with the action of a finite group $G$ via symplectic automorphisms; let $d$ be the order of $G$. Assume that the quotient of $X$ by $G$ admits a crepant resolution $Y\to X/G$ by a hyper-K\"ahler manifold $Y$. Let $r\colon X\dashrightarrow Y$ be the corresponding rational map of degree $d$. Then the push-forward map 
\begin{equation}
   r_*\colon H^2(X,\ZZ)^G \to H^2(Y,\ZZ)
\end{equation}
multiplies the form by a positive factor equal to
    \begin{equation} 
    d^2 \cdot \sqrt[n]{\frac{c_X}{c_Y\cdot d}}
    \end{equation}
    where $c_X$ and $c_Y$ are the Fujiki constants of $X$ and $Y$ respectively. 
    \end{lemma}
\begin{proof}
    Denote by $q_X$ and $q_Y$ the Beauville--Bogomolov form of $X$ and $Y$ respectively.
    For any $x\in H^2(X,\ZZ)^G$ we have $r^*r_*(x) = d\cdot x$; therefore, using the projection formula and Fujiki's relation \cite{fujiki1987}, we find 
    \begin{equation}
        \begin{split}
        q_Y(r_*x, r_*x)^n & = \frac{1}{c_Y} \int_Y (r_*(x))^{2n} = \\ & = \frac{1}{c_Y\cdot d} \int_X (r^*r_*(x))^{2n} = \\
       &  =\frac{c_X\cdot d^{2n-1}}{c_Y}\cdot \frac{1}{c_X} \int_X x^{2n} = \\ & = \frac{c_X\cdot d^{2n-1}}{c_Y} q_X(x,x)^n.
        \end{split}
    \end{equation} 
    Taking the $n$-th root we obtain the claimed formula; for $n$ even, the sign is determined since for any K\"ahler class $h$ on $X$ we have $q_X(h,h) > 0$ and $r_*(h)$ belongs to the closure of the K\"ahler cone of $Y$, and so we must have $q_{Y}(r_*(h), r_*(h))\geq 0$.
\end{proof}

\begin{remark}
If $e$ denotes the $\mathrm{g.c.d}$ of $q_X(a,b)$ for all $a,b\in H^2(X,\ZZ)$, then $e\cdot d^2 \sqrt[n]{\tfrac{c_X}{c_Y\cdot d}}$ is an integer.
\end{remark}

\begin{remark}
    Let $f\colon X\dashrightarrow Y$ be a finite rational map of degree $d$ between hyper-K\"ahler manifolds of dimension $2n$. Then there is a well-defined injective push-forward map
    $f_* \colon H^2_{\mathrm{tr}}(X,\ZZ)\to H^2(Y,\ZZ)$
    such that $f^*\circ f_*$ is multiplication by $d$ on $H^2_{\mathrm{tr}}(X,\ZZ)$. It follows with the same proof that $f_*$ multiplies the form by $d^2 \cdot \sqrt[n]{\frac{c_X}{c_Y\cdot d}}$.
\end{remark}

\subsection{The cohomology of $Y_K$} 
\label{subsec:CohomologyYK}
Consider the generalized Kummer variety $K^3(A)$ on an abelian surface $A$.
We have shown in Proposition \ref{prop:Y_K^3(A)} that $Y_{K^3(A)}$ is birational to $\Km(A)^{[3]}$; in fact, to prove this proposition we constructed (Lemma \ref{lem:isomorphismOrbifolds}) an explicit rational map $q\colon K^3(A)\dashrightarrow \Km(A)^{[3]}$ of degree $2^5$ such that $\psi\circ r=q$ for a birational map $\psi\colon Y_{K^3(A)} \dashrightarrow \Km(A)^{[3]}$. We now study the push-forward map $q_*\colon H^2(K^3(A),\ZZ)\to H^2(\Km(A)^{[3]},\ZZ)$. 

Recall from Proposition \ref{prop:CanonicalSubvarietiesK^3(A)} that we have $16$ canonical submanifolds $W_{\sigma_{\tau}}\subset K^3(A)$ of $\mathrm{K}3^{[2]}$-type parametrized by $\tau\in A[2]$, and that
we introduced in Remark~\ref{rmk:Km(A)^[n]} divisors $D$ and $R_{\tau}$, $\tau\in A[2]$, of $\Km(A)^{[3]}$.

\begin{lemma}\label{lem:geometryMapq}
    Let $\widetilde{K}^3(A)\to K^3(A)$ denote the blow-up along $\bigcup_{\tau\in A[2]} W_{\sigma_\tau}$; denote by $F_{\sigma_{\tau}}\subset \widetilde{K}^3(A)$ the component of the exceptional divisor over $W_{\sigma_{\tau}}$. Let $\widetilde{E}$ be the strict transform in $\widetilde{K}^3(A)$ of the Hilbert--Chow divisor $E$ of $K^3(A)$.
    Let \begin{equation}
    \widetilde{q}\colon \widetilde{K}^3(A)\dashrightarrow \Km(A)^{[3]}
    \end{equation} 
    be the rational map induced by $q\colon K^3(A)\dashrightarrow \Km(A)^{[3]}$. Then $\widetilde{q}$ is a morphism outside of a subset of codimension~$\geq 2$ contained in the union of $\widetilde{E}$ and the $16$ divisors $F_{\sigma_{\tau}}$. Moreover:
    \begin{enumerate}[label=(\roman*)]
    \item $\widetilde{q}$ restricts to a generically finite rational map $\widetilde{E}\dashrightarrow D$ of degree $2^5$;
    \item for each $\tau\in A[2]$, the restriction of $\widetilde{q}$ induces a generically finite map $F_{\sigma_{\tau}}\dashrightarrow R_{\tau}$, of degree $2^4$.
    \end{enumerate} 
    \end{lemma}~The various maps and relevant loci involved in the lemma are summarized in the following diagram:

      \begin{equation}
        \begin{tikzcd}[column sep = large]
            & & F_{\sigma_{\tau}} \arrow[dashed]{rrd}{\widetilde{q}} \arrow{lld} \arrow[hook]{d} \\
            W_{\sigma_{\tau}} \arrow[hook]{d} & & \widetilde{K}^3(A) \arrow[dashed]{rrd}{\widetilde{q}}\arrow{lld} & & R_{\tau} \arrow[hook]{d} \arrow{rddd} \\
             K^3(A) && \widetilde{E} \arrow[dashed]{rrd}{\widetilde{q}} \arrow{lld} \arrow[hook]{u} && \Km(A)^{[3]} \arrow{rddd} \\
            E \arrow{rrdd} \arrow[hook]{u} && \overline{W}_{\sigma_{\tau}} \arrow[<-, crossing over]{lluu}  \arrow[hook]{d} && D \arrow{rddd} \arrow[hook]{u} \\
            && A_0^{(4)} \arrow[<-, crossing over]{lluu} \arrow[swap, crossing over]{rrrd}{\overline{q}} &&& \overline{R}_{\tau} \arrow[<-, crossing over]{lllu}{\overline{q}} \arrow[hook]{d} \\
            && \overline{E} \arrow[swap]{rrrd}{\overline{q}} \arrow[hook]{u} &&&  (A/\pm1)^{(3)} \\
            &&&&& \overline{D} \arrow[hook]{u}
        \end{tikzcd}
    \end{equation}
    %

\begin{proof}
    Consider the morphism $\overline{q}\colon A_0^{(4)} \to (A/\pm 1)^{(3)}$ defined in \eqref{eq:defBarq}. 
    Denote by $\overline{E}\subset A_0^{(4)}$ the diagonal locus, and by $\overline{W}_{\sigma_{\tau}}\subset A_0^{(4)}$ the image of $W_{\sigma_{\tau}}\subset K^3(A)$, for each $\tau\in A[2]$. 
    Then, with notation as in Remark \ref{rmk:Km(A)^[n]} and using the explicit definition of $q$ from \eqref{eq:defBarq}, one checks the following:
    \begin{enumerate}[label=(\alph*)]
        \item $\overline{q}$ restricts to a finite morphism $\overline{q}\colon \overline{E}\to \overline{D}$ of degree $2^5$; 
          \item for each $\tau \in A[2]$, the restriction of $\overline{q}$ induces a finite morphism $\overline{W}_{\sigma_{\tau}}\to \overline{R}_{\tau}$, of degree $2^4$.
    \end{enumerate}
    Indeed, $\overline{E}$ consists of the points $(a,a,b,-2a-b)$, $a,b\in A$. Hence, $\overline{E}$ is the quotient of $A\times A$ by the involution $(a,b)\mapsto (a, -2a-b)$.
    Around a general point of $\overline{E}$, the map $\overline{q}\colon \overline{E}\to \overline{D}$ is given by $\{a, a, b, -2a-b\}\mapsto \{\overline{2a}, \overline{a+b}, \overline{-a-b}\}$. Hence, the degree is half the degree of
    \begin{equation}
    \begin{split}
            A\times A&\to (A/\pm 1)\times (A/\pm 1)\\
            (a, b)&\mapsto (\overline{a+b}, \overline{2a})
    \end{split}
    \end{equation}
    which is $2^6$; this proves (a). Similarly, for (b), around a general point, the map $\overline{W}_{\sigma_{\tau}}\to \overline{R}_{\tau}$ is given by $\{a, \tau-a, b, \tau-b\}\mapsto \{\overline{\tau}, \overline{a+b}, \overline{\tau+a-b}\}$.
    Therefore, the degree of the map $\overline{W}_{\sigma_{\tau}}\to \overline{R}_{\tau}$ is independent of $\tau$ and it equals the degree of
    \begin{equation}
    \begin{split}
        (A/\pm 1)^{(2)}&\to  (A/\pm 1)^{(2)}\\
        (a, b)&\mapsto (\overline{a+b}, \overline{a-b})
        \end{split}
    \end{equation}
    which is $2^4$; this proves (b).
    
    We then consider the commutative diagram \begin{equation}\label{diag:square1}
        \begin{tikzcd}
            \widetilde{K}^3(A) \arrow{d}{\widetilde{\nu}} \arrow[dashed]{r}{\widetilde{q}} & \Km(A)^{[3]} \arrow{d}{m} \\
            A_0^{(4)} \arrow{r}{\overline{q}} & (A/\pm 1)^{(3)} 
        \end{tikzcd}
    \end{equation}
    Let $\overline{U} \subset (A/\pm 1)^{(3)}$ be the open subset of points whose support contains at least $2$ distinct points and at most $1$ node with multiplicity 1, and let $U\subset \Km(A)^{[3]}$ be the preimage $m^{-1}(\overline{U})$. 
    Then, as in Remark \ref{rmk:Km(A)^[n]}, the restriction of $m_{|_{U}}\colon {U}\to \overline{U}$ is the blow-up of the union of $\overline{D}\cap \overline{U}$ and the $\overline{R}_{\tau}\cap \overline{U}$, $\tau\in A[2]$; the components of the exceptional divisor of this map are the restriction to $U$ of the divisors $D \subset \Km(A)^{[3]}$ and the $R_{\tau}\subset \Km(A)^{[3]}$. 

    Consider the open subset $\overline{V}\coloneqq \overline{q}^{-1}(\overline{U})$ of $A_0^{(4)}$; let $V\coloneqq \nu^{-1}(\overline{V}) \subset K^3(A)$ be its preimage under the Hilbert--Chow morphism $\nu\colon {K}^3(A)\to A_0^{(4)}$. We consider also the open subset $\widetilde{V}\coloneqq \widetilde{\nu}^{-1}(\overline{V})$ of $\widetilde{K}^3(A)$.
    Then the restriction $\nu_{|_{V}}\colon V\to \overline{V}$ of the Hilbert--Chow morphism is just the blow-up of the diagonal $\overline{E}\cap \overline{V}$. Therefore, the restriction $\widetilde{\nu}_{|_{\widetilde{V}}}\colon \widetilde{V}\to \overline{V}$ is identified with the blow-up of the union of $\overline{E}\cap \overline{V}$ and the $\overline{W}_{\sigma_{\tau}}\cap\overline{V}$, for $\tau\in A[2]$. The components of the exceptional divisor of $\widetilde{\nu}_{|_{\widetilde{V}}}$ are the restrictions to $\widetilde{V}$ of the divisors $\widetilde{E}$ and $F_{\sigma_{\tau}}$ of $\widetilde{K}^3(A)$.

    Therefore, restricting diagram \eqref{diag:square1}, we get the commutative diagram 
    \begin{equation}
    \begin{tikzcd}
        \widetilde{V} \arrow{d}{\widetilde{\nu}} \arrow[dashed]{r}{\widetilde{q}} & {U} \arrow{d} \\
        \overline{V} \arrow{r}{\overline{q}} & \overline{U}
    \end{tikzcd}
    \end{equation}
    in which the vertical maps are blow-ups. As we are blowing-up components of codimension $2$, and our singularities are no worse than nodes on surfaces, it follows that $\widetilde{q}_{|_{\widetilde{V}}}$ is a morphism, and statements (a), (b) above imply $(i)$ and $(ii)$ in the statement.

    Let us spell this out. Let $x\in U$, and $\overline{x}$ its image in $\overline{U}$. If $\overline{x} \in \overline{D}\cap \overline{U}$, then $\overline{x}$ has a neighborhood $\overline{O}_{\overline{x}}$ in $\overline{U}$ isomorphic to a neighbourhood of the origin in $(\CC^2/\pm 1)\times (\CC^2)^2$, with preimage in $\overline{V}$ equal to $2^5$ disjoint copies of $\overline{O}_{\overline{x}}$; then blowing-up the singular locus of $\overline{V}$ and $\overline{U}$, we see that $\widetilde{q}$ is a morphism over a neighborhood of $x \in U$, and that it restricts to a morphism $\widetilde{E}\cap \widetilde{V}\to D\cap U$ of degree $2^5$.

    If $\overline{x}\in \overline{R}_{\tau}$, then we find a neighborhood of $\overline{O}_{\overline{x}}$ of $\overline{x}$ in $\overline{U}$ isomorphic to a neighbourhood of the origin in $(\CC^2/\pm 1)\times (\CC^2)^2$, with preimage under $\overline{q}$ consisting of $2^4$ copies of $\CC^2 \times (\CC^2)^2$; for each such copy, the map $\overline{q}$ is the quotient $\CC^2\to (\CC^2/\pm 1)$ on the first factor and an isomorphism $(\CC^2)^2\to (\CC^2)^2$ on the second one. 
    After blowing-up, we find a neighborhood $O_x\subset U$ isomorphic to $\mathrm{Bl}_0(\CC^2/\pm 1)\times (\CC^2)^2$, with preimage in $\widetilde{V}$ consisting of $2^4$ copies of $\mathrm{Bl}_0(\CC^2)\times (\CC^2)^2$. On each of these copies, $\widetilde{q}$ is the morphism given by the quotient $\mathrm{Bl}_0(\CC^2) \to \mathrm{Bl}_0(\CC^2/\pm 1)$ on the first factor and by an isomorphism on the second one; this morphism is a double cover, ramified over $C\times (\CC^2)^2$, where $C\subset \mathrm{Bl}_0(\CC^2)$ is the exceptional curve. The exceptional curve $C$ is the intersection of $F_{\sigma_\tau}$ with the corresponding component of $\widetilde{q}^{-1}(O_x)$. It follows that $\widetilde{q}$ restricts to a morphism $F_{\sigma_\tau}\cap \widetilde{V}\to R_{\tau}\cap U$ of degree $2^4$.
\end{proof}

\begin{notation} 
 Following \cite{beauville1983varietes}, we have 
    \begin{equation}
        H^2(K^3(A),\ZZ) = H^2(A,\ZZ)\oplus \langle \xi\rangle= \mathrm{U}^{\oplus 3}\oplus \langle -8 \rangle,
    \end{equation}
    where $\xi$ is a primitive class of square $-8$ which equals half the class of the Hilbert--Chow divisor $E\subset K^3(A)$. 
    If $\widetilde{K}^3(A)\to K^3(A)$ denotes the blow-up along $\bigcup_{\tau\in A[2]} W_{\sigma_\tau}$, we have 
    \begin{equation}
H^2(\widetilde{K}^3(A),\ZZ) = H^2(A,\ZZ) \oplus \langle \widetilde{\xi}\rangle \oplus \langle [F_{\sigma_{\tau}}] \rangle_{\tau\in A[2]},
    \end{equation} 
    where the $F_{\sigma_{\tau}}$ are the components of the exceptional divisor of $\widetilde{K}^3(A)\to K^3(A)$, and $\widetilde{\xi}$ is the pull-back of $\xi$. 
    For $\Km(A)^{[3]}$, we have instead
    \begin{equation}
        \begin{split}
            H^2(\Km(A)^{[3]},\ZZ) & = H^2(\Km(A),\ZZ) \oplus \langle \delta\rangle \supset  H^2(A,\ZZ)(2) \oplus \langle [R_{\tau}]\rangle_{\tau\in A[2]} \oplus \langle \delta\rangle, 
        \end{split}
    \end{equation}
    where $\delta$ is half of the class of the Hilbert--Chow divisor $D\subset \Km(A)^{[3]}$, which is a primitive class of square $-4$, and, as in Remark \ref{rmk:Km(A)^[n]}, the divisor $R_{\tau}\subset \Km(A)^{[3]}$ parametrizes subschemes whose support intersects the rational curve $C_{\tau}\subset \Km(A)$. The $[R_{\tau}]$ are pairwise orthogonal $(-2)$-classes, and they generate a sublattice of $H^2(\Km(A)^{[3]},\ZZ)$ with saturation isomorphic to the Kummer lattice.
\end{notation}

With the above notation, we have the following result.
\begin{lemma}\label{lem:cohomologyq_*}
    The map of $\ZZ$-modules
    \begin{equation}
        \begin{tikzcd}
            H^2(\widetilde{K}^3(A),\ZZ) \arrow{rr}{\widetilde{q}_*} && H^2(\Km(A)^{[3]},\ZZ)\\
            H^2(A,\ZZ) \oplus \langle \widetilde{\xi}  \rangle \oplus \langle [F_{\sigma_{\tau}}]\rangle_{\tau\in A[2]} \arrow{rr} \arrow{u}{\simeq } && H^2(A,\ZZ)(2) \oplus \langle [R_{\tau}]\rangle_{\tau\in A[2]}\oplus \langle \delta\rangle \arrow[hook]{u}
        \end{tikzcd}
    \end{equation}
    satisfies:
    \begin{enumerate}[label=(\roman*)]
        \item $\widetilde{q}_*$ restricts to the multiplication by $2^4$ from the summand $H^2(A,\ZZ)$ on the left to the summand $H^2(A,\ZZ)(2)$ on the right (note that they are the same as $\ZZ$-modules);
        \item $\widetilde{q}_*(\widetilde{\xi}) = 2^5 \cdot \delta$;
        \item $\widetilde{q}_*([F_{\sigma_{\tau}}])=2^4\cdot [R_\tau]$, for each $\tau\in A[2]$.
    \end{enumerate}
\end{lemma}
\begin{proof}
    Assertions $(ii)$ and $(iii)$ follow immediately from Lemma \ref{lem:geometryMapq}. Taking the orthogonal complement, it follows that $\widetilde{q}_*$ restricts to a map $H^2(A,\ZZ)\to H^2(A,\ZZ)(2)$, i.e., it gives an endomorphism of the $\ZZ$-module $H^2(A,\ZZ)$. 
    As the definition of the map $q\colon K^3(A)\dashrightarrow \Km(A)^{[3]}$ clearly works for $A$ deforming in any family of $2$-dimensional complex tori, this endomorphism is equivariant with respect to the action of the monodromy group $\mathrm{SL}_4$ of $2$-dimensional complex tori. The representation $H^2(A,\ZZ)$ is the irreducible representation $\bigwedge^2 V$, where $V$ is the standard representation of $\mathrm{SL}_4$. By Schur's lemma, $q_*\in \End_{\mathrm{SL}_4}(H^2(A,\ZZ))$ is multiplication by an integer $k$. Since, as a map of lattices, $q_*\colon H^2(A,\ZZ)\to H^2(A,\ZZ)(2)$ multiplies the form by a factor $2^9$ (see \cite[Lemma~4.8]{floccariKum3} or apply Lemma \ref{lem:scalarFactor}), we conclude that $k=2^4$.    
\end{proof}

\begin{proof}[Proof of Proposition \ref{prop:cohomologyY_K}]
As the hyper-Kummer construction works in families (see Remark \ref{rmk:deformation}) and the statement is topological, it suffices to prove the proposition for the generalized Kummer variety $K^3(A)$ associated with an abelian surface $A$. By Proposition \ref{prop:Y_K^3(A)} there exists a birational map $\psi \colon \Km(A)^{[3]} \dashrightarrow Y_{K^3(A)} $ such that $r=\psi\circ q$.

Now Lemma \ref{lem:cohomologyq_*} implies that $q_*\colon H^2(K^3(A),\ZZ)\to H^2(\Km(A)^{[3]})$ is injective and multiplies the form by $2^9$, and moreover $\tfrac{1}{2^4}q_*$ is integral. By Lemma \ref{lem:cohomologyq_*} $(ii)$ we deduce that $\tfrac{1}{2^4}q_*$ extends to a primitive embedding of lattices and Hodge structures
\begin{equation}
    \frac{1}{2^4} q_*\colon \widehat{H}^2(K^3(A),\ZZ)(2)\hookrightarrow H^2(\Km(A)^{[3]},\ZZ),
\end{equation}
mapping isomorphically onto the saturated sublattice $H^2(A,\ZZ)(2) \oplus \langle \delta\rangle$ of $H^2(\Km(A))^{[3]}$. 

By \cite[Lemma 2.6]{Huy99}, birational maps of hyper-K\"ahler manifolds induce Hodge isometries on the second cohomology. Therefore, we have $r_*=\psi_*\circ q_*$, and the above statements about $q_*$ prove $(i)$ and $(ii)$.

Consider now the blow-up $\widetilde{K}^3(A)\to K^3(A)$ of $K^3(A)$ along $\bigcup_{\tau\in A[2]} W_{\sigma_{\tau}}$, and let $F_{\sigma_{\tau}}$ denote the component of the exceptional divisor over $W_{\sigma_{\tau}}$.
By Theorem \ref{thm:hyperKummerConstruction}, the rational map $r\colon K^3(A)\dashrightarrow Y_{K^3(A)}$ extends to a regular morphism $\widetilde{r}\colon \widetilde{K}^3(A)\to Y_{K^3(A)}$, and we have a commutative diagram (the notation is as above) 
\begin{equation}
\begin{tikzcd}
& & E_{\sigma_{\tau}} \arrow[hook]{r} & Y_{K^3(A)}  \\
F_{\sigma_{\tau}} \arrow{rru} \arrow[dashed]{rrd} \arrow[hook]{r} & \widetilde{K}^3(A) \arrow[swap]{rru}{\widetilde{r}} \arrow[dashed]{rrd}{\widetilde{q}} \\
& & R_{\tau} \arrow[hook]{r} & \Km(A)^{[3]} \arrow[dashed, swap]{uu}{\psi}
\end{tikzcd}
\end{equation}

By construction $\widetilde{r}$ restricts to a finite morphism $F_{\sigma_{\tau}}\to E_{\sigma_\tau}$ of degree $2^4$, for each $\tau\in A[2]$. Therefore, $\widetilde{r}_*\colon H^2(\widetilde{K}^3(A),\ZZ)\to H^2(Y_{K^3(A)},\ZZ)$ satisfies $\widetilde{r}_*([F_{\sigma_\tau}])=2^4\cdot [E_{\sigma_{\tau}}]$.
On the other hand, by Lemma \ref{lem:cohomologyq_*}.$(iii)$, we have $\widetilde{q}_*([F_{\sigma_\tau}])=2^4\cdot [R_{\tau}]$. 
Hence, the birational map $\psi$ satisfies $\psi_*([R_{\tau}])=[E_{\sigma_\tau}]$, for any $\tau\in A[2]$. 
As the $[R_{\tau}]\in H^2(\Km(A)^{[3]},\ZZ)$ are pairwise orthogonal $(-2)$-classes and the saturation of the sublattice they generate is isometric to the Kummer lattice, the same holds for the classes $[E_{\sigma_{\tau}}]\in H^2(Y_{K^3(A)},\ZZ)$, which proves $(iii)$.
\end{proof}

\subsection{The cohomology of the companion hyper-K\"ahler submanifolds}
\label{subsec:CohomologyWV}
We study in this subsection the restriction maps in cohomology induced by the embeddings of $W_{\sigma}$ and $V_{\sigma,\sigma'}$ in~$K$ and prove Proposition \ref{prop:cohomologyW&V}. The key ingredient is already in \cite{floccariHCKum3}, which is partially generalized in \cite[Lemma~4.4]{floccari25}. 
Recall that the submanifolds $W_{\sigma}$ and $V_{\sigma,\sigma'}$ deform together with $K$ in the sense of Remark \ref{rmk:deformation}. Hence, it will be sufficient to study the case of $K=K^3(A)$ for some abelian surface or $2$-dimensional complex torus $A$. In this case, we know from Proposition \ref{prop:CanonicalSubvarietiesK^3(A)} that $W_{\sigma_{\tau}}$ is isomorphic to $\Km(A)^{[2]}$ while $V_{\sigma_{\tau},{\sigma_{\tau}}_{'}}$ is isomorphic to $ \Km(A/\langle \tau+\tau'\rangle)$.
Let us fix some notation. 

\begin{notation}\label{notation:cohomologyLattices}
    Let $\Km(A)$ be the Kummer surface associated with $A$; it contains $16$ disjoint rational curves $C_{\tau}$ parametrized by $\tau\in A[2]$. We have a finite index sublattice 
    \begin{equation}
    H^2(A,\ZZ)(2)\oplus \langle [C_{\tau}]\rangle_{\tau\in A[2]} \subset H^2(\Km(A),\ZZ),
    \end{equation}
    where the $[C_{\tau}]$ are pairwise orthogonal $(-2)$-classes and the saturation of the sublattice they generate is isometric to the Kummer lattice $L_{\Km}$. 

    As for $\Km(A)^{[2]}$, we introduced on it $17$ divisors $D$ and $R_{\tau}$, for $\tau\in A[2]$, in Remark~\ref{rmk:Km(A)^[n]}. The classes $[R_{\tau}]$ in $H^2(\Km(A)^{[2]},\ZZ)$ are pairwise orthogonal $(-2)$-classes, and the saturation of the sublattice generated by these $16$ classes is isometric to the Kummer lattice $L_{\Km}$. The class $[D]$ is divisible by $2$; let $\delta\in H^2(\Km(A)^{[2]},\ZZ)$ be half of the class of $D$. Then $\delta$ is primitive of square $-2$, and there is a finite index sublattice
    \begin{equation}
        H^2(A,\ZZ)(2) \oplus \langle [R_{\tau}]\rangle_{\tau\in A[2]}\oplus \langle \delta\rangle \subset H^2(\Km(A)^{[2]},\ZZ).
    \end{equation}
\end{notation}
We identify $H^2(W_{\sigma_{\tau}},\ZZ)$ and $H^2(V_{\sigma_{\tau},{\sigma_{\tau}}_{'}}\,,\ZZ)$ with $H^2(\Km(A)^{[2]},\ZZ)$ and $H^2(\Km(A/\langle\tau+\tau'\rangle,\ZZ)$, respectively. With the above notation, we have the following result. 
\begin{proposition}[{\cite[Proposition 3.11]{floccariHCKum3}}]\label{prop:restrictionXi}
We have 
\begin{equation}
    \iota^*_{\sigma_{\tau}}(\xi) = 2\delta + \frac{1}{2} \sum_{\theta\in A[2]} [R_{\theta}] \ \in H^2(W_{\sigma_{\tau}},\ZZ);
\end{equation}
\begin{equation}
    \iota^*_{\sigma_{\tau},{\sigma_{\tau}}_{'}}(\xi) = \sum_{\theta \in (A/\langle \tau+\tau'\rangle)[2]} [C_{\theta}] \ \in H^2(V_{\sigma_{\tau},{\sigma_{\tau}}_{'}}\,,\ZZ).
\end{equation}
\end{proposition}

As the submanifolds $W_{\sigma_{\tau}}$ and $V_{\sigma_{\tau},{\sigma_{\tau}}_{'}}$ deform with $K^3(A)$ to any $\mathrm{Kum}^3$-manifolds, their cohomology classes are canonical Hodge classes, which means that they remain Hodge on any deformation of $K$. 

\begin{remark}
    Let $X$ be a compact hyper-K\"ahler manifold of dimension $2n$ and let $\alpha\in H^{4j}(X,\ZZ)$ be a canonical Hodge class. Then $\alpha$ satisfies the generalized Fujiki relation \cite{fujiki1983}: there exists a constant $C(\alpha)$ such that 
    \begin{equation}\label{eq:FujikiRelation}
        \int_{X} \alpha \cdot \gamma^{2n-2j} = C(\alpha)\cdot q_{X}(\gamma,\gamma)^{n-j}
    \end{equation}
    for any $\gamma\in H^2(X,\ZZ)$, where $q_X$ is the Beauville--Bogomolov form. The constant $C(1)$ is called the Fujiki constant of $X$; it has been computed for all known deformation types of hyper-K\"ahler manifolds, see \cite{rapagnetta2008beauville}.
\end{remark}

\begin{lemma}
    \begin{enumerate}[label=(\roman*)]
    \item The pull-back $\iota_{\sigma}^*\colon H^2(K,\ZZ)\to H^2(W_{\sigma},\ZZ)$ multiplies the form by a factor $2$.
    \item The pull-back $\iota_{\sigma,\sigma'}^*\colon H^2(K,\ZZ)\to H^2(V_{\sigma,\sigma'},\ZZ)$ multiplies the form by a factor $4$.
    \end{enumerate}
\end{lemma}
\begin{proof}
    We may assume that $K=K^3(A)$ for some abelian surface $A$. 
    By the generalized Fujiki relation \eqref{eq:FujikiRelation} we have 
    \begin{equation}
        \int_{W_{\sigma_{\tau}}} \iota_{\sigma_{\tau}}^*(\gamma)^4 = C(W_{\sigma_{\tau}})\cdot q_{K}(\gamma,\gamma)^2,
    \end{equation}
    for any $\gamma\in H^2(K,\ZZ)$. 
    On the other hand, as on any manifold of $\mathrm{K}3^{[2]}$-type, we have $\int_{W_{\sigma_{\tau}}}\beta^4=3\cdot q_{W_{\sigma_{\tau}}}(\beta,\beta)^2$. Applying this to the class $\xi\in H^2(K^3(A),\ZZ)$, we use Proposition \ref{prop:restrictionXi} to compute the generalized Fujiki constant $C([W_{\sigma_{\tau}}])=12$, and deduce that $\iota^*_{\sigma_{\tau}}$ multiplies the form by a factor $2$. Similarly, we compute that $C([V_{\sigma_{\tau},{\sigma_{\tau}}_{'}}])=4$, from which we deduce that $\iota^*_{\sigma_{\tau},{\sigma_{\tau}}_{'}}$ multiplies the form by a factor $4$.
\end{proof}

\begin{proof}[Proof of Proposition \ref{prop:cohomologyW&V}]
Taking a deformation of $K$ to a very general manifold $K'$ of $\mathrm{Kum}^3$-type such that $H^2(K',\ZZ)$ is an irreducible Hodge structure, we see that the embeddings $\iota_{\sigma}^*$ and $\iota^*_{\sigma, \sigma'}$ must be either injective or zero. But a symplectic form on $K$ restricts to a symplectic form on $W_{\sigma}$ and $V_{\sigma,\sigma'}$, and therefore $\iota_{\sigma}^*$ and $\iota^*_{\sigma, \sigma'}$ are injective. By the above lemma, they multiply the form by a factor $2$ and $4$, respectively.

Assume now that $K=K^3(A)$ for some very general $2$-dimensional complex torus $A$ such that $H^2(A,\ZZ)$ is an irreducible Hodge structure. Then, as $\iota^*_{\sigma}$ is a morphism of Hodge structures, it must embed $H^2_{\mathrm{tr}}(K^3(A),\ZZ)(2)=H^2(A,\ZZ)(2)$ into $H^2_{\mathrm{tr}}(W_{\sigma_{\tau}},\ZZ) = H^2(A,\ZZ)(2)$ as a sublattice with finite index; but then, it must be surjective. 
As $H^2(K^3(A),\ZZ)=H^2(A,\ZZ)\oplus \langle \xi \rangle$, together with Proposition \ref{prop:restrictionXi} we deduce that the image of $\iota^*_{\sigma_{\tau}}$ is the sublattice $H^2(A,\ZZ)(2)\oplus \ZZ\cdot (2\delta +\frac{1}{2} \sum_{\theta\in A[2]} [R_{\theta}])$ of $H^2(\Km(A)^{[2]},\ZZ)$. This sublattice is saturated, with orthogonal complement the lattice $N_W$ of Definition \ref{def:latticeN_W}. This completes the proof of $(i)$. 

As for $(ii)$, similarly, $\iota^*_{\sigma_{\tau},{\sigma_{\tau}}_{'}}$ must send $H^2(A,\ZZ)\subset H^2(K^3(A),\ZZ)$ into $H^2(A/\langle \tau+\tau'\rangle,\ZZ)(2)\subset H^2(V_{\sigma_{\tau},{\sigma_{\tau}}_{'}}\,,\ZZ)$.
Together with Proposition \ref{prop:restrictionXi}, this implies that $\iota^*_{\sigma_{\tau}, {\sigma_{\tau}}_{'}}$ defines an isometric embedding of $H^2(K^3(A),\ZZ)(4)$ into $H^2(V_{\sigma,{\sigma_{\tau}}_{'}},\ZZ)$, whose image is contained with finite index in the sublattice $H^2(A/\langle \tau+\tau'\rangle,\ZZ)(2)\oplus \langle \tfrac{1}{2}\sum_{\theta\in A/\langle\tau+\tau'\rangle }[C_{\theta]} \rangle$. 
The latter is a primitive sublattice of $H^2(V_{\sigma_{\tau} ,{\sigma_{\tau}}_{'}}\,,\ZZ)$, isometric to $\mathrm{U}(2)^{\oplus 3}\oplus \langle -8\rangle$, in which $\mathrm{im}_{\sigma_{\tau},\sigma_{\tau'}}$ must thus have index $2^4$, with orthogonal complement the lattice $N_V$ from Definition \ref{def:latticeN_V}.
\end{proof}

\begin{remark}\label{rmk:compatibility}
    For $\iota_{\sigma_{\tau}}\colon W_{\sigma_{\tau}}\hookrightarrow K^3(A)$, the restriction map 
    \begin{equation}
    \begin{tikzcd}
        H^2(K^3(A),\ZZ) \arrow{rrr}{\iota^*_{\sigma_{\tau}}} &&& H^2(W_{\sigma_{\tau}}, \ZZ)=H^2 (\Km(A)^{[2]},\ZZ)\\
        H^2(A,\ZZ) \oplus \langle \xi \rangle \arrow{u}{\simeq } \arrow{rrr} &&& H^2(A,\ZZ)(2) \oplus \langle \delta\rangle \oplus \langle [R_{\tau}]\rangle_{\tau\in A[2]} \arrow[hook]{u}
    \end{tikzcd}    
    \end{equation}
    sends $\xi$ to $2\delta + \tfrac{1}{2} \sum_{\tau} [R_{\tau}]$ and on the first summand it restricts to the map $\pi_*\colon H^2(A,\ZZ) \to H^2(A,\ZZ)(2) $ induced by $\pi\colon A\dashrightarrow \Km(A)$ (which is just the identity as map of $\ZZ$-modules). 
    
    It is also possible to prove that for $\iota_{\sigma_{\tau}, {\sigma_\tau}_{'}}\colon V_{\sigma_{\tau}, {\sigma_\tau}_{'} }\hookrightarrow K^3(A)$, the restriction map 
    \begin{equation}
    \begin{tikzcd}
        H^2(K^3(A),\ZZ) \arrow{rrr}{\iota^*_{\sigma_{\tau}, {\sigma_\tau}_{'}}} &&& H^2(V_{\sigma_{\tau}, {\sigma_\tau}_{'}}\,, \ZZ)=H^2 (\Km(A/\langle \tau+\tau'\rangle),\ZZ)\\
        H^2(A,\ZZ) \oplus \langle \xi \rangle \arrow{u}{\simeq } \arrow{rrr} &&& H^2(A/\langle \tau+\tau'\rangle ,\ZZ)(2) \oplus \langle [C_{\theta}]\rangle_{\theta \in (A/\langle \tau+\tau'\rangle)[2]} \arrow[hook]{u}
    \end{tikzcd}    
    \end{equation}
    sends $\xi$ to $\sum_{\theta} [C_{\theta}]$ and it restricts to the map $\pi_*\colon H^2(A,\ZZ) \to H^2(\Km(A/\langle \tau+\tau'\rangle),\ZZ)$ induced by the composition $\pi\colon A\to A/\langle \tau+\tau'\rangle \dashrightarrow \Km(A/\langle \tau+\tau'\rangle)$ on the first summand. 

    The last assertion can be seen as follows. Considering the chain of inclusions $V_{\sigma_{\tau}, {\sigma_\tau}_{'}}\subset W_{\sigma_{\tau}}\subset K^3(A)$ it is sufficient to study the restriction along the embedding $\iota_{\tau+\tau'}\colon \Km(A/\langle \tau+\tau' \rangle) \hookrightarrow \Km(A)^{[2]}$ induced by the graph map $\hat{\iota}_{\tau+\tau'}\colon \Km(A)\hookrightarrow \Km(A)^2$ sending $s$ to $(\tau+\tau')(s)$. Setting $\theta\coloneqq \tau+\tau'$, we are reduced to prove the following claim. 
    
    \textbf{Claim:} \textit{the restriction of the pull-back map $\iota_{\theta}^*\colon H^2(\Km(A)^{[2]},\ZZ)\to H^2(\Km(A/\langle \theta\rangle,\ZZ)$ to the summand $H^2(A,\ZZ)(2)\subset H^2(\Km(A)^{[2]},\ZZ)$ is identified with the restriction to the summand $H^2(A,\ZZ)(2)\subset H^2(\Km(A),\ZZ)$ of $(\bar{h}_{\theta})_*\colon H^2(\Km(A),\ZZ)\to H^2(\Km(A/\langle \theta\rangle),\ZZ)$, where $\bar{h}_{\theta}\colon \Km(A)\dashrightarrow \Km(A/\langle \theta\rangle)$ is the rational map induced by the quotient map $h_{\theta} \colon A\to A/\langle \theta\rangle$}.
    \begin{proof}[Proof of the claim]
    For any smooth and projective surface, there is an embedding $\eta_{X}\colon H^2(X,\ZZ)\hookrightarrow H^2(X^{[2]},\ZZ)$, defined as follows. Consider the cartesian diagram 
    \begin{equation}
        \begin{tikzcd}
            \mathrm{Bl}_{\Delta} (X^2) \arrow{d}{\rho} \arrow{r}{\tilde{\nu}} & X^2\arrow{d} \\
             X^{[2]} \arrow{r}{\nu} & X^{(2)};     
        \end{tikzcd}
    \end{equation}
    then $\eta_X(\alpha)\in H^2(X^{[2]},\ZZ)$ is the unique class such that 
    \begin{equation}\rho^* (\eta_X(\alpha)) = \tilde{\nu}^*(\mathrm{pr}_1^*\alpha+\mathrm{pr}_2^*\alpha).\end{equation}
    
    For the abelian surface $A$, we let $\pi_A\colon A\dashrightarrow \Km(A)$ be the natural map; then $(\pi_A)_*$ induces the embedding of $H^2(A,\ZZ)(2)$ into $H^2(\Km(A),\ZZ)$.
    With the above notation, to prove our claim we need to show that
    \begin{equation}\label{eq:CLAIM}
    \iota^*_{\theta} \circ \eta_{\Km(A)}\circ(\pi_A)_* = (\bar{h}_{\theta})_* \circ (\pi_{A})_* \colon H^2(A,\ZZ)\to H^2(\Km(A/\langle \theta\rangle ),\ZZ).
    \end{equation} 
    We consider the commutative diagram 
    \begin{equation}
        \begin{tikzcd}
            &\mathrm{Bl}_{\Delta}(\Km(A)^2) \arrow[near start]{dd}{\rho} \arrow{rr}{\tilde{\nu}} && \Km(A)^2 \arrow{dd} \\
            \mathrm{Bl}_{\mathrm{Fix}(\theta)}(\Km(A)) \arrow{dd}{\bar{\rho}} \arrow[hook]{ur}{\tilde{\iota}_{\theta}} \arrow[crossing over, near start]{rr}{\tilde{\mu}} && \Km(A)  \arrow[hook]{ur}{\hat{\iota}_{\theta}}\\
            &\Km(A)^{[2]}\arrow{rr}&& \Km(A)^{(2)} \\
            \Km(A/\langle \theta\rangle)\arrow[hook]{ur}{\iota_{\theta}} \arrow{rr} && \Km(A)/\langle \theta \rangle\arrow[hook]{ur}{\bar{\iota}_{\theta}} \arrow[<-, crossing over, near end,swap]{uu}
        \end{tikzcd}
    \end{equation}
    The horizontal maps are birational, while the vertical ones are double covers. We notice that $\rho_*\rho^*$ is multiplication by $\deg \rho =2 $ on the second cohomology of $\Km(A)^{[2]}$. So, for any $\alpha \in H^2(A,\ZZ)$, we have: 
    \begin{equation}\label{eq:computation}
    \begin{split}
    2\cdot ( (\iota_{\theta}^* \circ \eta_{\Km(A)} \circ (\pi_A)_*)(\alpha)) & = ((\iota_{\theta}^* \circ \rho_* ) \circ ( \rho^*\circ \eta_{\Km(A)}\circ (\pi_A)_*))(\alpha) \\
    & = (\bar{\rho}_* \circ \tilde{\iota}_{\theta}^* \circ \tilde{\nu}^*)((\mathrm{pr}_1^*\circ (\pi_A)_*)(\alpha) + (\mathrm{pr}_2^*\circ (\pi_A)_*)(\alpha)))\\
    & = (\bar{\rho}_* \circ \tilde{\mu}^* \circ \hat{\iota}_{\theta}^*) ((\mathrm{pr}_1^*\circ (\pi_A)_*)(\alpha) + (\mathrm{pr}_2^*\circ (\pi_A)_*)(\alpha))\\
    & = (\bar{\rho}_* \circ \tilde{\mu}^* )((\hat{\iota}_{\theta}^* \circ \mathrm{pr}_1^*\circ (\pi_A)_*)(\alpha) + (\hat{\iota}_{\theta}^*\circ \mathrm{pr}_2^*\circ (\pi_A)_*)(\alpha)) \\
    & = (\bar{\rho}_*\circ \bar{\mu}^*)(2(\pi_A)_*(\alpha)),
    \end{split}
    \end{equation}
    where we used, in order: that $\iota^*_{\theta}\circ \rho_*=\bar{\rho}_*\circ \tilde{\iota}_{\theta}^*$ by the projection formula; that by definition $(\rho^*\circ \eta_{\Km(A)})(\beta)= \tilde{\nu}^*(\mathrm{pr}_1^*\beta+\mathrm{pr}_2^*\beta)$ for any $\beta\in H^2(\Km(A),\ZZ)$; that $(\tilde{\iota}_{\theta} \circ \tilde{\nu})^* = (\tilde{\mu}^*\circ \hat{\iota}^*_{\theta})$ by the commutativity of the above diagram; finally, for the last equality, that $\hat{\iota}_{\theta}\colon \Km(A) \hookrightarrow \Km(A)^2$ is the inclusion of the graph of $\theta\colon \Km(A)\to \Km(A)$ and therefore $\mathrm{pr}_1\circ \hat{\iota}_{\theta}=\mathrm{id}_{\Km(A)}$ and $\mathrm{pr}_2\circ \hat{\iota}_{\theta}=\theta$, and that $\theta_*\colon H^2(\Km(A),\ZZ)\to H^2(\Km(A),\ZZ)$ induces the identity on the image of $(\pi_A)_*$. 
    By definition, the map $(\bar{h}_{\theta})_*\colon H^2(\Km(A),\ZZ)\to H^2(\Km(A/\langle \theta\rangle),\ZZ)$ coincides with $\bar{\rho}_*\circ \bar{\mu}^*$. Therefore, as $H^2(A,\ZZ)$ is torsion-free, our claim \eqref{eq:CLAIM} follows from \eqref{eq:computation}. 
    \end{proof}
\end{remark}

\subsection{The cohomology of $M_{\sigma}$}
\label{subsec:CohomologyM}
We now analyze the hyper-Kummer $\mathrm{K}3^{[2]}$-manifolds $M_{\sigma}$ associated with a $\mathrm{Kum}^3$-manifold $K$.
Recall from Proposition \ref{prop:ResolutionQuotients} that $M_{\sigma}$ is the crepant resolution of $W_{\sigma}/G$ by blowing up along its singular locus.

Via a suitable deformation, it suffices to analyze the case of the generalized Kummer variety $K^3(A)$ on a $2$-dimensional complex torus $A$. In this case we can identify any $W_{\sigma_{\tau}}$ with the $G$-action with $\Km(A)^{[2]}$ equipped with the natural action of $A[2]$. Let $r_{W}\colon \Km(A)^{[2]}\dashrightarrow M$ be the natural rational map of degree $2^4$. By Proposition \ref{prop:Fujikiresolution}, the total fixed locus for the action of $A[2]$ is the union of $15$ K3 surfaces $V_{\theta}\subset \Km(A)^{[2]}$. The map $r_{W}$ is resolved by taking the blow-up $\widetilde{\Km}(A)^{[2]}$ of $\Km(A)^{[2]}$ along these K3 surfaces; we denote by $P_{\theta}\subset \widetilde{\Km}(A)^{[2]}$ the component of the exceptional divisor of the blow-up lying over $V_{\theta}$. Then we obtain a morphism $\widetilde{r}_{W}\colon \widetilde{\Km}(A)^{[2]}\to M$, which is the quotient by the action of $A[2]$.
Also recall that in Remark \ref{rmk:Km(A)^[n]} we defined $17$ divisors $D$ and $R_{\theta}$, $\theta\in A[2]$, on $\Km(A)^{[2]}$. 

We introduced in Lemma \ref{lem:isomorphismOrbifolds2} a rational map 
$s\colon \Km(A)^{[2]}\dashrightarrow \Km(A)^{[2]}$, such that there exists a birational map $\psi\colon \Km(A)^{[2]}\dashrightarrow M$ satisfying $r_W=\psi\circ s$. 

\begin{lemma}\label{lem:geometryMaps}
    The rational map $s\colon \Km(A)^{[2]}\dashrightarrow \Km(A)^{[2]}$ is well-defined over the complement of the $15$ $\mathrm{K}3$ surfaces $V_{\theta}\subset \Km(A)^{[2]}$, for $0\neq \theta\in A[2]$. Denote by $\widetilde{D}$, $ \widetilde{R}_{\theta}$ the divisors in $\widetilde{\Km}(A)^{[2]}$ obtained as strict transforms of the divisors $D, R_{\theta}$ of $\Km(A)^{[2]}$, and let 
    \begin{equation}
        \widetilde{s}\colon \widetilde{\Km}(A)^{[2]}\dashrightarrow \Km(A)^{[2]}
    \end{equation}
    be the rational map induced by $s$. Then $\widetilde{s}$ extends to a regular morphism of degree $2^4$ outside some subset of codimension at least $2$ contained in the union of the $P_{\theta}$. Moreover:
    \begin{enumerate}[label=(\roman*)]
        \item $\widetilde{s}$ restricts to a generically finite map $\widetilde{D}\dashrightarrow R_0$ of degree $2^4$;
        \item for each $\tau\in A[2]$, $\widetilde{s}$ restricts to a birational map $\widetilde{R}_{\tau}\dashrightarrow D$;
        \item for any $0\neq \theta\in A[2]$, the map $\widetilde{s}$ restricts to a finite map $P_{\theta}\dashrightarrow R_{\theta}$, of degree $2^3$.
    \end{enumerate}
    \begin{equation}
        \begin{tikzcd}[column sep = large]
            \widetilde{D}  \arrow[dashed, bend right=-20]{rr}  && R_0   \arrow[hook]{rd} \\
             \widetilde{R}_{\tau} \arrow[dashed, bend right=-25]{rr} \arrow[hook]{r}& \widetilde{\Km}(A)^{[2]} \arrow[<-, crossing over]{ul} \arrow[crossing over, dashed, bend right=-20]{rr}{\widetilde{s}} & D \arrow[ hook]{r} & \Km(A)^{[2]} \arrow{dddr}\\
             {P}_{\theta}\arrow{dddr} \arrow[hook]{ur} \arrow[dashed, bend right=-20]{rr} && R_{\theta} \arrow[hook]{ur} \\
            &\overline{D} \arrow[<-, crossing over]{uuul} \arrow[hook]{dr} \arrow[bend right=-20]{rr}  && \overline{R}_0 \arrow[<-, crossing over]{uuul}  \arrow[hook]{rd} \\
             & \overline{R}_{\tau} \arrow[<-, crossing over]{uuul} \arrow[crossing over, bend right=-25]{rr} \arrow[hook]{r} & (A/\pm 1)^{(2)}  \arrow[<-, crossing over]{uuul} \arrow[bend right=-20]{rr}{\overline{s}} & \overline{D} \arrow[<-, crossing over]{uuul} \arrow[hook]{r} & (A/\pm 1)^{(2)}  \\
             & \overline{V}_{\theta}\arrow[hook]{ur} \arrow[bend right=-20]{rr} && \overline{R}_{\theta} \arrow[<-, crossing over]{uuul} \arrow[hook]{ur}
        \end{tikzcd}
    \end{equation}
    \end{lemma}
    \begin{proof}
        The rational map $s$ is induced by the morphism $\overline{s}\colon (A/\pm 1)^{(2)}\rightarrow (A/\pm 1)^{(2)}$ of \eqref{eq:defMaps}, which sends $(\overline{x},\overline{y})$ to $(\overline{x+y},\overline{x-y})$. 
With notation as above, the morphism $\overline{s}\colon (A/\pm1)^{(2)}\to (A/\pm 1)^{(2)}$ satisfies the following:
\begin{enumerate}[label=(\alph*)]
    \item for any $\tau\in A[2]$, it restricts to an isomorphism $\overline{s}\colon \overline{R}_{\tau}\to \overline{D}$;
    \item $\overline{s}$ restricts to a finite morphism $\overline{D}\to \overline{R}_0$ of degree $2^4$; 
    \item for any $0\neq \theta\in A[2]$, the morphism $\overline{s}$ restricts to a finite morphism $\overline{s}\colon \overline{V}_{\theta}\to \overline{R}_{\theta}$ of degree $2^3$.
\end{enumerate}
To see this, notice that $\overline{s}(\overline{\tau}, \overline{x}) = (\overline{x+\tau}, \overline{x-{\tau}})$ belongs to the diagonal $\overline{D}$, for any $\tau\in A[2]$. Clearly, the induced map $\overline{R}_{\tau} \to \overline{D}$ is an isomorphism (notice that they are both isomorphic to $A/\pm 1$). 
For a point $(\overline{x}, \overline{x})\in \overline{D}$, we have $\overline{s}(\overline{x},\overline{x})=(\overline{2x},\overline{0})$, which belongs to $\overline{R}_0$. Thus, $\overline{s}\colon \overline{D}\to \overline{R}_0$ is identified with the quotient $(A/\pm 1) \to (A/\langle A[2], \pm 1\rangle)$, and, hence, it is a finite morphism of degree $2^4$. 
Finally, for each $0\neq \theta\in A[2]$, we have $\overline{s}(\overline{x},\overline{x+\theta})= (\overline{\theta}, \overline{2x+\theta})$, which belongs to $\overline{R}_{\theta}$. Recalling that $\overline{V}_{\theta} \subset (A/\pm 1)^{(2)}$ is isomorphic to $A/\langle \theta, \pm 1\rangle$, the induced morphism $\overline{s}\colon \overline{V}_{\theta}\to \overline{R}_{\theta}$ is identified with the degree $2^3$ quotient map $A/\langle \theta, \pm1\rangle \to A/\langle A[2],\pm1\rangle\cong A/\pm 1$.

Now the resolution $m\colon \Km(A)^{[2]}\to (A/\pm 1)^{(2)}$ is obtained by blowing-up the union of the $17$ surfaces $\overline{D}$ and $\overline{R}_{\tau}$ for $\tau \in A[2]$, while $\widetilde{m}\colon \widetilde{\Km}(A)^{[2]}\to (A/\pm 1)^{(2)}$ is the blow-up of the $32$ surfaces $\overline{D}$, the $\overline{R}_{\tau}$ for $\tau\in A[2]$ and the $\overline{V}_{\theta}$ for $0\neq \theta\in A[2]$. Now assertions $(i)$, $(ii)$ and $(iii)$ follow from (a), (b) and (c) as above, with the same argument used in the conclusion of the proof of Lemma \ref{lem:geometryMapq}.
\end{proof}


\begin{notation}\label{notation:cohomologyGroups}
    As in Notation \ref{notation:cohomologyLattices}, we have the finite index sublattice
    \begin{equation}
       H^2(A,\ZZ)(2) \oplus \langle [R_{\tau}]\rangle_{\tau\in A[2]} \oplus \langle \delta\rangle \subset H^2(\Km(A)^{[2]},\ZZ) ,
    \end{equation}
    where the $[R_{\tau}]$ are pairwise orthogonal $(-2)$-classes and the saturation of the sublattice they generate is a Kummer lattice, while $\delta$ is primitive of square $-2$ and equals half the class of the Hilbert--Chow divisor $D$.
    As in the above lemma, we let $\widetilde{\Km}(A)^{[2]}\to \Km(A)^{[2]}$ be the blow-up of the union of the~$15$ K3 surfaces $V_{\theta}=\Km(A/\langle \theta\rangle)\subset \Km(A)^{[2]}$, for $0\neq \theta\in A[2]$; the K3 surface $V_{\theta}$ is embedded in $\Km(A)^{[2]}$ as the resolution of the graph of $\theta\colon \Km(A)\to \Km(A)$. 
    Let $P_{\theta}$ be the component of the exceptional divisor over $V_{\theta}$, and denote by $\widetilde{D}, \widetilde{R}_{\tau}\subset \widetilde{\Km}(A)^{[2]}$ the strict transform of $D, R_{\tau}$ respectively; let $\widetilde{\delta}\in H^2(\widetilde{\Km}(A)^{[2]},\ZZ)$ denote half the class of $\widetilde{D}$.  
    We then have the finite index sublattice \begin{equation}
        H^2(A,\ZZ)(2)\oplus\langle[\widetilde{R}_{\tau}]\rangle_{\tau\in A[2]} \oplus \langle \widetilde{\delta}\rangle  \oplus   \langle [P_{\theta}] \rangle_{0\neq \theta\in A[2]} \ \subset H^2(\widetilde{\Km}(A)^{[2]},\ZZ).
    \end{equation}
\end{notation}

\begin{remark}\label{rmk:invariantparts}
    The invariant sublattice $H^2(W_{\sigma_{\tau}},\ZZ)^G=H^2(\Km(A)^{[2]},\ZZ)^{A[2]}$ has rank $8$. Indeed, it is given by the saturated sublattice
    \begin{equation}
       H^2(\Km(A)^{[2]},\ZZ)^{A[2]}  = H^2(A,\ZZ)(2)\oplus \langle \frac{1}{2} \sum_{\tau\in A[2]} [R_{\tau}]\rangle \oplus \langle \delta\rangle.
    \end{equation}
    Moreover, with the above notation, we have 
    \begin{equation}
        H^2(\widetilde{\Km}(A)^{[2]},\ZZ)^{A[2]}  = H^2(A,\ZZ)(2)\oplus \langle \frac{1}{2} \sum_{\tau\in A[2]} [\widetilde{R}_{\tau}]\rangle \oplus \langle \widetilde{\delta} \rangle \oplus \langle [P_{\theta}]\rangle_{0\neq \theta \in A[2]}.
    \end{equation}
\end{remark}

\begin{lemma}\label{lem:cohomologys_*}
    The map of $\ZZ$-modules
    \begin{equation}
        \begin{tikzcd}
            H^2(\widetilde{\Km}(A)^{[2]},\ZZ) \arrow{r}{\widetilde{s}_*} & H^2(\Km(A)^{[2]},\ZZ)\\
            H^2(A,\ZZ)(2) \oplus\langle[\widetilde{R}_{\tau}]\rangle_{\tau\in A[2]} \oplus \langle \widetilde{\delta}  \rangle \oplus \langle [P_{\theta}]\rangle_{0\neq \theta\in A[2]} \arrow{r} \arrow[hook]{u} & H^2(A,\ZZ)(2) \oplus \langle [R_{\tau}]\rangle_{\tau\in A[2]} \oplus \langle \delta\rangle \arrow[hook]{u}
        \end{tikzcd}
    \end{equation}
    satisfies:
    \begin{enumerate}[label=(\roman*)]
        \item the restriction of $\widetilde{s}_*$ to $H^2(A,\ZZ)(2)\oplus \langle \frac{1}{2} \sum_{\tau\in A[2]} [R_{\tau}]\rangle \oplus \langle \delta\rangle$ 
        multiplies the form by a factor $2^6$, and $\widetilde{s}_*$ restricts to the multiplication by $2^3$ from the summand $H^2(A,\ZZ)(2)$ on the left to the summand $H^2(A,\ZZ)(2)$ on the right;
        \item for any $\tau\in A[2]$, we have $\widetilde{s}_*([\widetilde{R}_{\tau}])= 2\delta$;
        \item $\widetilde{s}_*(\widetilde{\delta}) = 2^3\cdot  [R_0]$;
        \item $\widetilde{s}_*([P_{\theta}])=2^3\cdot [R_\theta]$, for each $0\neq \theta\in A[2]$.
    \end{enumerate}
\end{lemma}
\begin{proof}
    Statements $(ii)$, $(iii)$ and $(iv)$ follow directly from Lemma \ref{lem:geometryMaps}. Taking orthogonal complements, it follows that $\widetilde{s}_*$ restricts to a map $H^2(A,\ZZ)(2)\to H^2(A,\ZZ)(2)$.
    The restriction of $\widetilde{s}_*$ to $H^2(A,\ZZ)(2)\oplus \langle \frac{1}{2} \sum_{\tau\in A[2]} [R_{\tau}]\rangle \oplus \langle \delta\rangle$ is the map of lattices $s_*\colon H^2(\Km(A)^{[2]},\ZZ)^{A[2]} \to H^2(\Km (A)^{[2]},\ZZ)$ induced by $s\colon \Km(A)^{[2]}\to \Km(A)^{[2]}$; which multiplies the form by a factor $2^6$ by Lemma \ref{lem:scalarFactor}.
    Since the construction of the rational map $s\colon \Km(A)^{[2]}\dashrightarrow \Km(A)^{[2]}$ works in families, the restriction $s_*\colon H^2(A,\ZZ)(2)\to H^2(A,\ZZ)(2) $ is equivariant for the action of the monodromy group of $2$-dimensional complex tori. As in the proof of Lemma \ref{lem:cohomologyq_*}, Schur's lemma implies that the endomorphism of $H^2(A,\ZZ)(2)$ induced by $s_*$ is multiplication by an integer $k$. But it multiplies the form by $2^6$, so $k=2^3$.
\end{proof}

\begin{proof}[Proof of Proposition \ref{prop:cohomologyM&S}.$(i)$]
Assume that $K=K^3(A)$ and identify $W_{\sigma_{\tau}}$ with its $G$-action with $\Km(A)^{[2]}$ with the action of $A[2]$.
By Proposition \ref{prop:cohomologyW&V}.$(i)$, we know that the pull-back along the embedding $\iota^*_{\sigma_{\tau}}\colon W_{\sigma_{\tau}}\hookrightarrow K^3(A)$ induces a primitive embedding 
\begin{equation}
\iota^*_{\sigma_{\tau}}\colon H^2(K,\ZZ)(2)\hookrightarrow H^2(W_{\sigma_{\tau}},\ZZ).
\end{equation} 
With our usual identifications $H^2(K^3(A),\ZZ)=H^2(A,\ZZ)\oplus \langle \xi\rangle$, and $H^2(W_{\sigma_{\tau}},\ZZ)=H^2(\Km(A)^{[2]},\ZZ)$, the image of $\iota^*_{\sigma_{\tau}}$ is the sublattice 
\begin{equation}
    H^2(A,\ZZ)(2)\oplus \langle 2\delta+\frac{1}{2}\sum_{\tau\in A[2]} [R_{\tau}] \rangle \ \subset H^2(\Km(A)^{[2]},\ZZ).
\end{equation}
Recall that there exists a birational map $\psi\colon \Km(A)^{[2]}\dashrightarrow M_{\sigma}$ which fits in a commutative diagram 
\begin{equation}
\begin{tikzcd}
    && M_{\sigma_{\tau}} \\
    W_{\sigma_{\tau}}\cong \Km(A)^{[2]} \arrow[dashed]{rru}{r_W} \arrow[dashed]{rrd}{s} \\
    && \Km(A)^{[2]}\arrow[dashed]{uu}{\psi}
\end{tikzcd}
\end{equation}
Since birational maps of hyper-K\"ahler manifolds induce isomorphisms between the second cohomology groups, we can identify $H^2(M_{\sigma_{\tau}},\ZZ)$ with $H^2(\Km(A)^{[2]},\ZZ)$, so that $(r_W)_{*}\colon H^2(W_{\sigma_{\tau}},\ZZ)^G\to H^2(M_{\sigma_{\tau}},\ZZ)$ is identified with 
\begin{equation} s_*\colon H^2(\Km(A)^{[2]},\ZZ)^{A[2]} \to H^2(\Km(A)^{[2]},\ZZ).
\end{equation}
By Remark \ref{rmk:invariantparts}, we have $H^2(\Km(A)^{[2]},\ZZ)^{A[2]}= H^2(A,\ZZ)(2)\oplus \langle \frac{1}{2}\sum_{\tau\in A[2]} [R_{\tau}] \rangle \oplus \langle \delta\rangle$. By Lemma \ref{lem:scalarFactor}, this map multiplies the form by a factor $2^6$. 
Moreover, by Lemma \ref{lem:cohomologys_*}, we have 
\begin{enumerate}[label=(\roman*)]
    \item $s_*$ induces the multiplication by $2^3$ from $H^2(A,\ZZ)(2)\subset H^2(\Km(A)^{[2]},\ZZ)^{A[2]}$ to the sublattice $H^2(A,\ZZ)(2)\subset H^2(\Km(A)^{[2]},\ZZ)$;
    \item $s_*(\tfrac{1}{2} \sum_{\tau\in A[2]} [R_{\tau}]) =2^4\cdot \delta$;
    \item $s_*(\delta)=2^3\cdot [R_0]$.
\end{enumerate}
In particular, $s_*(2\delta + \tfrac{1}{2} \sum_{\tau\in A[2]} [R_{\tau}]) = 2^4 ([R_0] + \delta)$.
We conclude that $\tfrac{1}{2^3} ((r_W)_*\circ \iota^*_{\sigma_{\tau}})$ is integral and defines an embedding of lattices 
\begin{equation}
    \phi\colon H^2(K^3(A),\ZZ)(2)\hookrightarrow H^2(\Km(A)^{[2]},\ZZ),
\end{equation}
with image $H^2(A,\ZZ)(2) \oplus \langle 2\delta+2[R_0]\rangle $; notice that $\phi(\xi)=2\delta+2[R_0]$ is not primitive. Hence, the embedding $\phi$ extends to a primitive embedding $\hat{\phi}$ of $\widehat{H}^2(K^3(A),\ZZ)(2)$ into $H^2(\Km(A)^{[2]},\ZZ)$, with image $H^2(A,\ZZ)(2)\oplus \langle \delta+[R_0]\rangle$; this is a saturated sublattice with orthogonal complement the lattice $N_M$ of Definition \ref{def:latticeN_M}.

As the construction of the hyper-Kummer $\mathrm{K}3^{[2]}$-manifolds $M_{\sigma}$ can be performed in families of manifolds of $\mathrm{Kum}^3$-type by Remark \ref{rmk:deformation}, it follows that for an arbitrary $\mathrm{Kum}^3$-variety, a canonical $\mathrm{K}3^{[2]}$-type submanifold $W_{\sigma}\subset K$ and $r_{W}\colon W_{\sigma}\dashrightarrow M_{\sigma}$ the rational map to the resolution of $W_{\sigma}/G$, the map $\tfrac{1}{2^3}((r_W)_*\circ \iota^*_{\sigma})$ is integral, and yields a primitive embedding of lattices and Hodge structures
\begin{equation} 
\widehat{H}^2(K,\ZZ)(2)\hookrightarrow H^2(M_{\sigma},\ZZ),
\end{equation} 
with orthogonal complement isometric to $N_M$.
\end{proof}

\subsection{The cohomology of the hyper-Kummer K3 surfaces}
\label{subsec:CohomologyS}
We finally consider the \textit{hyper-Kummer K3 surfaces} $S_{\sigma,\sigma'}$ associated with a manifold $K$ of $\mathrm{Kum}^3$-type; in the projective case, we will show that they are all isomorphic to each other, so that a $\mathrm{Kum}^3$-variety $K$ has a well-defined associated hyper-Kummer K3 surface $S_K$. 

Recall that $S_{\sigma,\sigma'}$ is the minimal resolution of the quotient surface $V_{\sigma,\sigma'}/G$, for the naturally embedded K3 surface $\iota_{\sigma,\sigma'}\colon V_{\sigma,\sigma'} \hookrightarrow K$, for some $\sigma\neq \sigma'$ in $G\setminus G_1$.
We denote by $r_V\colon V_{\sigma,\sigma'}\dashrightarrow S_{\sigma,\sigma'}$ the corresponding rational map, which is generically finite of degree $2^3$ as $G$ acts on $V_{\sigma,\sigma'}$ through $G/\langle \sigma,\sigma'\rangle\cong (\ZZ/2\ZZ)^3$. 

As in the previous cases, it will be sufficient to treat the case of $K^3(A)$. For the K3 surface $V_{\sigma_{\tau},{\sigma_{\tau}}_{'}}\subset K^3(A)$, setting $\theta\coloneqq \sigma_\tau+{\sigma_\tau}_{'}$, we have $V_{\sigma_{\tau},{\sigma_{\tau}}_{'}}\cong \Km(A/\langle \theta \rangle)$, with the $G$-action identified with the natural action of $A[2]/\langle \theta\rangle\cong (\ZZ/2\ZZ)^3$ (see Propositions \ref{prop:canonicalSubvarietiesI} and \ref{prop:canonicalSubvarietiesII}). 
The resolution $S_{\sigma_{\tau},{\sigma_{\tau}}_{'}}$ of $V_{\sigma_{\tau},{\sigma_{\tau}}_{'}}/G$ is instead isomorphic to $\Km(A)$.

\begin{remark} 
Translation by $\epsilon\in A[4]$ induces $G$-equivariant isomorphisms $V_{\sigma_{\tau},{\sigma_{\tau}}_{'}}\xrightarrow{ \sim  } V_{V_{\sigma_{\tau} +2\epsilon ,{\sigma_{\tau}}_{'}+2\epsilon}}$. 
    Applying such automorphisms we may assume that $\tau=0$, in which case $V_{\sigma_0, \sigma_\theta} \subset K^3(A)$ is identified with $\Km(A/\langle \theta\rangle)$ and the group $G$ induces the natural action of $A[2]/\langle \theta\rangle$ on it.
\end{remark}

\begin{notation} \label{notation:cohomologyGroupsK3}
We denote by $C_{\tau}\subset \Km(A)$ the $16$ rational curves introduced by the resolution $\Km(A)\to A/\pm 1$, for $\tau\in A[2]$. Their cohomology classes $[C_{\tau}]\subset H^2(\Km(A),\ZZ)$ generate a sublattice with saturation the Kummer lattice, and 
\begin{equation}
     H^2(A,\ZZ)(2) \oplus \langle [C_{\tau}]\rangle_{\tau\in A[2]} \subset H^2(\Km(A),\ZZ)
\end{equation}
with finite index.
We shall denote by $N_{\alpha}$ the $16$ rational curves on $\Km(A/\langle \theta\rangle)$, which are parametrized by $\alpha\in (A[2]\sqcup A_{2,\theta})/\langle \theta\rangle$.
As for the $G$-invariant part $H^2(V_{\sigma_0,{\sigma_\theta}_{'}}\,,\ZZ)^{G}=H^2(\Km(A/\langle \theta\rangle),\ZZ)^{A[2]/\langle \theta\rangle}$, it is the saturation of 
\begin{equation}
     H^2(A/\langle \theta \rangle,\ZZ)(2) \oplus \left\langle \frac{1}{2} \sum_{\alpha\in A[2]/\langle \theta \rangle} [N_{\alpha}] \right\rangle \oplus \left\langle \frac{1}{2} \sum_{\alpha\in A_{2,\theta }/\langle \theta\rangle} [N_{\beta}]\right\rangle \ \subset H^2(\Km(A/\langle \theta\rangle)  ,\ZZ)^{A[2]/\langle \theta\rangle}.
\end{equation}
\end{notation} 
\begin{lemma}\label{lem:push-forwardt_*}
    Consider the rational map $t\colon \Km(A/\langle \theta\rangle) \dashrightarrow \Km(A)$ induced by the quotient with respect to $A[2]/\langle \theta\rangle$. The push-forward 
    \begin{equation}
    \begin{tikzcd}
            H^2(\Km(A/\langle \theta\rangle),\ZZ)^{A[2]/\langle \theta\rangle} \arrow{r}{{t}_*} & H^2(\Km(A),\ZZ)\\
            H^2(A/\langle \theta\rangle,\ZZ)(2) \oplus \langle \frac{1}{2} \sum_{A[2]/\langle \theta \rangle} [N_{\alpha}] \rangle \oplus \langle \frac{1}{2} \sum_{A_{2,\theta}/\langle \theta\rangle} [N_{\beta}]\rangle \arrow[hook]{u} \arrow{r} & H^2(A,\ZZ)(2) \oplus \langle [C_{\tau}]\rangle_{\tau\in A[2]} \arrow[hook]{u}
        \end{tikzcd}
    \end{equation}
    multiplies the form by a factor $2^3$ and satisfies:
    \begin{enumerate}[label=(\roman*)]
        \item $t_*$ restricts to a map $H^2(A/\langle \theta\rangle,\ZZ)(2) \to H^2(A,\ZZ)(2)$, which equals the push-forward along the quotient map $A/\langle \theta\rangle \to A/A[2]=A$;
        \item we have $t_*(\frac{1}{2}\sum_{A[2]/\langle\theta\rangle} [N_{\alpha}])= 4 \cdot [C_0] $
        \item we have $t_*(\frac{1}{2}\sum_{A_{2,\theta}/\langle\theta\rangle} [N_{\alpha}])= 4\cdot [C_{\theta}] $.
    \end{enumerate}
\end{lemma}
\begin{proof}
    It is readily seen that $t_*$ multiplies the form by $2^3$, e.g. via Lemma \ref{lem:scalarFactor}.
    Assertion $(i)$ is clear as $H^2(A,\ZZ)(2)\subset H^2(\Km(A),\ZZ)$ is the push-forward of $H^2(A,\ZZ)$ under the rational map $A\dashrightarrow \Km(A)$ and $t$ is induced by the quotient $A/\langle \theta\rangle \to A$. It is easy to see that the $8$ curves $N_{\alpha}$ for $\alpha\in A[2]/\langle \theta\rangle$ (resp. $N_{\beta}$ for $\beta\in A_{2,\theta} /\langle \theta\rangle$)  are mapped isomorphically to the same curve $C_0\subset \Km(A)$ (resp. $C_{\theta}\subset \Km(A)$) via $t$.
\end{proof}

\begin{proof}[Proof of Proposition \ref{prop:cohomologyM&S}.$(ii)$]
Via a suitable deformation, we may assume that $K=K^3(A)$ and $S_{\sigma,\sigma'}=S_{\sigma_0,\sigma_{\theta}}$. 
The rational map $r_V\colon V_{\sigma_0,\sigma_{\theta}}\dashrightarrow S_{\sigma_0,\sigma_{\theta}}$ is identified with the map $t$ in the above lemma. 
Moreover, by Proposition \ref{prop:cohomologyW&V}, the restriction $\iota^*_{\sigma_0,\sigma_{\theta}}\colon H^2(K^3(A),\ZZ)\to H^2(\Km(A/\langle \theta\rangle),\ZZ)$ is injective and multiplies the form by a factor $4$, and its image is contained with finite index into the primitive sublattice 
\begin{equation}
    H^2(A/\langle \theta\rangle,\ZZ)(2)\oplus \big\langle\frac{1}{2} \sum_{\alpha\in (A[2]\cup A_{2,\theta})/\langle \theta \rangle} [N_{\alpha}]\big\rangle.
\end{equation}
Moreover, $\iota^*_{\sigma_0,\sigma_{\theta}}(\xi) = \sum_{\alpha} [N_{\alpha}]$ is divisible by $2$, so that $\iota^*_{\sigma_0,\sigma_{\theta}}$ gives an embedding $\widehat{H}^2(K,\ZZ)(4)\hookrightarrow H^2(\Km(A/\langle \theta\rangle),\ZZ)$. 

Hence, by Lemma \ref{lem:push-forwardt_*}, the composition $(t_*\circ \iota^*_{\sigma_0,\sigma_\theta})\colon H^2(K^3(A),\ZZ)\to H^2(\Km(A),\ZZ)$ is injective and multiplies the form by a factor $2^5$, with image contained with finite index in the primitive sublattice $H^2(A,\ZZ)(2)\oplus \langle [C_0]+[C_{\theta}]\rangle$ of $H^2(\Km(A),\ZZ)$. 
This map sends the class $\xi\in H^2(K^3(A),\ZZ)$ to $8([C_0]+[C_{\theta}])\in H^2(\Km(A),\ZZ)$. By Remark \ref{rmk:compatibility} and Lemma \ref{lem:push-forwardt_*}.$(i)$, the restriction of $t_*\circ \iota^*_{\sigma_0,\sigma_{\theta}}$ to $H^2(A,\ZZ)\subset H^2(K^3(A),\ZZ)$ coincides with the map $H^2(A,\ZZ)\to H^2(A,\ZZ)(2)\subset H^2(\Km(A),\ZZ)$ induced by the composition of the maps $A\to A/\langle A[2]\times \langle -1\rangle\rangle\dashrightarrow \Km(A)$. Therefore, it is equivariant for the action of the monodromy group of $2$-dimensional complex tori, and it must be multiplication by a scalar. Hence, the restriction of $t_*\circ \iota^*_{\sigma_0,\sigma_{\theta}}$ to $H^2(A,\ZZ)$ is the multiplication by $4$. 

Observe that we get an embedding $t_*\circ \iota^*_{\sigma_0,\sigma_{\theta}} \colon \widehat{H}^2(K^3(A),\ZZ)(2^5) \hookrightarrow H^2(\Km(A),\ZZ)$ which maps $\tfrac{1}{2}(\xi)$ to $4[C_0]+4[C_{\theta}]$. We conclude that $\tfrac{1}{4}(t_*\circ \iota^*_{\sigma_0,\sigma_{\theta}})$ defines an integral embedding 
\begin{equation}
\frac{1}{4}(t_*\circ \iota^*_{\sigma_0,\sigma_{\theta}}) \colon \widehat{H}^2(K^3(A),\ZZ)(2) \hookrightarrow H^2(\Km(A),\ZZ),
\end{equation} 
which is primitive with image $H^2(A,\ZZ)(2)\oplus \langle [C_0]+[C_{\theta}] \rangle$ and whose orthogonal complement is isometric to the lattice $N_S$ of Definition \ref{def:latticeN_S}.
\end{proof}

\section{Hyper-Kummer $\mathrm{K}3^{[3]}$-type manifolds}
\label{sec:HyperKummerK33}
Based on the computations in Section \ref{sec:CohomologyOfCompanion} and the global Torelli theorem \cite{verbitskyTorelli, markman2011survey, huybrechts2011global}, we investigate in this section hyper-Kummer $\mathrm{K}3^{[3]}$-type manifolds via the Hodge structure and the lattice structure on their second cohomology groups. We will first establish a lattice-theoretic characterization of hyper-Kummer manifolds of $\mathrm{K}3^{[3]}$-type up to bimeromorphism in Theorem \ref{thm:birationalCharacterization}. Then, we classify generic \textit{projective} hyper-Kummer $\mathrm{K}3^{[3]}$-type manifolds in terms of their N\'eron--Severi and transcendental lattices (Theorem \ref{thm:NeronSeveriHyperkummer}). Finally, we show that for two Hodge isometric projective varieties of $\mathrm{Kum}^3$-type, although they are in general not birational, their associated  $\mathrm{K}3^{[3]}$-type hyper-Kummer varieties are birational 
(Proposition \ref{prop:HodgeIsometricKummer}).

\subsection{Birational characterization}
As an analog of Nikulin's characterization of Kummer K3 surfaces, we have the following characterization up to birational isomorphism of the hyper-Kummer manifolds of $\mathrm{K}3^{[3]}$-type associated with manifolds of $\mathrm{Kum}^3$-type, purely in terms of the Hodge lattice on their second cohomology. 

\begin{theorem}\label{thm:birationalCharacterization}
Let $Y$ be a hyper-K\"ahler manifold of $\mathrm{K}3^{[3]}$-type. Then the following two conditions are equivalent:
\begin{enumerate}[label=(\roman*)]
\item there exists a primitive embedding $j\colon L_{\Km}\hookrightarrow H^2(Y,\ZZ)$ of the Kummer lattice $L_{\Km}$ with orthogonal complement isometric to $\mathrm{U}(2)^{\oplus 3}\oplus \langle -4\rangle$ and such that $j$ has image contained in $\NS(Y)$;
\item the manifold $Y$ is bimeromorphic to the hyper-Kummer $\mathrm{K}3^{[3]}$-manifold $Y_K$ associated to some $K$ of $\mathrm{Kum}^3$-type.
\end{enumerate} 
\end{theorem} 
\begin{proof}
As bimeromorphic maps of hyper-K\"ahler manifolds induce Hodge isometries between their second cohomology groups, that $(ii)$ implies $(i)$ is a direct consequence of Proposition \ref{prop:cohomologyY_K}. 

Let $\Lambda_{\mathrm{Kum}^3}= \mathrm{U}^{\oplus 3} \oplus \langle -8\rangle$ and $\Lambda_{\mathrm{K}3^{[3]}}=\mathrm{U}^{\oplus 3} \oplus E_8(-1)^{\oplus 2} \oplus \langle -4\rangle$ be the Beauville--Bogomolov lattices of $\mathrm{Kum}^3$ and $\mathrm{K}3^{[3]}$-manifolds respectively. 
    The period domain of $\mathrm{Kum}^3$-type manifolds is 
    \begin{equation}
    \Omega_{\mathrm{Kum}^3} = \{ x\in \PP(\Lambda_{\mathrm{Kum}^3}\otimes \CC)\ | \ (x,x)= 0, \ (x,\bar{x})>0\}.
    \end{equation}
    Following \cite{Huy99}, a marked manifold of $\mathrm{Kum}^3$-type is a pair  $(K,\eta)$ where $K$ is a $\mathrm{Kum}^3$-type manifold and $\eta\colon H^2(K,\ZZ)\xrightarrow{\ \sim \ } \Lambda_{\mathrm{Kum}^3}$ is an isometry. The period of a marked $\mathrm{Kum}^3$-type manifold $(K,\eta)$ is by definition the point $\mathfrak{p}(K,\eta)\coloneqq \eta(H^{2,0}(K))\in \Omega_{\mathrm{Kum}^3}$. Similarly, we have the period domain~$\Omega_{\mathrm{K}3^{[3]}} $ for manifolds of $\mathrm{K}3^{[3]}$-type, and the period of a marked $\mathrm{K}3^{[3]}$-type manifold $(Y,\theta)$ is the point $\mathfrak{p}(Y,\theta)\coloneqq \theta(H^{2,0}(Y))\in \Omega_{\mathrm{K}3^{[3]}}$. 
    
    The hyper-Kummer construction of Theorem \ref{thm:hyperKummerConstruction} gives an injective map of lattices
    \[
    r_*\colon \Lambda_{\mathrm{Kum}^3}(2^9) \hookrightarrow \Lambda_{\mathrm{K3}^{[3]}},
    \]
    which is canonical in the following sense: if $K$ is any manifold of $\mathrm{Kum}^3$-type and $r_K\colon K\dashrightarrow Y_K$ is the rational map given by Theorem~\ref{thm:hyperKummerConstruction}, then there exist isometries $\eta\colon H^2(K,\ZZ)\xrightarrow{\ \sim \ } \Lambda_{\mathrm{Kum}^3}$ and $\theta\colon H^2(Y_K,\ZZ) \xrightarrow{\ \sim \ } \Lambda_{\mathrm{K}3^{[3]}}$ such that $\theta\circ (r_K)_*\circ \eta^{-1}=r_*$. This follows from the fact that the hyper-Kummer construction can be performed in families, see Remark \ref{rmk:deformation}. 
    By Proposition~\ref{prop:cohomologyY_K}, the saturation of $\mathrm{im}(r_*)$ is isometric to $\mathrm{U}(2)^{\oplus 3} \oplus \langle -4\rangle$, and its orthogonal complement is isometric to the Kummer lattice. Let $i\colon L_{\Km} \hookrightarrow \Lambda_{\mathrm{K}3^{[3]}}$ denote the embedding of this orthogonal complement.
    If we consider the subdomain 
    \[ \Omega_{\mathrm{K}3^{[3]}}^{L_{\Km}^{\bot}} \coloneqq \{x\in \Omega_{\mathrm{K}3^{[3]}} \ |  \ x\in i(L_{\Km})^{\bot} \} \subset \Omega_{\mathrm{K}3^{[3]}},
    \]  
    we see that any hyper-Kummer $\mathrm{K}3^{[3]}$-type manifold admits a marking $\theta$ such that $\mathfrak{p}(Y_K,\theta)$ belongs to $\Omega_{\mathrm{K}3^{[3]}}^{L_{\Km}^{\bot}}$. 
    In fact, any point $x \in \Omega_{\mathrm{K}3^{[3]}}^{L_{\Km}^{\bot}}$ is the period of some marked hyper-Kummer $\mathrm{K}3^{[3]}$-type manifold. Indeed, $x$ uniquely determines $z\in \Omega_{\mathrm{Kum}^3}$ such that $x=r_* (z)$. By the surjectivity of the period map \cite[Theorem 8.1]{Huy99}, there exists a marked $\mathrm{Kum}^3$-manifold $(K,\eta)$ with period $z$; since $(r_K)_*\colon H^2(K,\CC)\to H^2(Y_K,\CC)$ preserves the Hodge decomposition, by the above discussion the associated hyper-Kummer $\mathrm{K}3^{[3]}$-manifold $Y_K$ admits a marking $\theta$ such that $\mathfrak{p}(Y_K,\theta)=x$.
    
    If now $Y$ is as in $(i)$, by Lemma \ref{lem:embeddingL_Km} there exists a marking $\theta\colon H^2(Y,\ZZ)\xrightarrow{\ \sim \ } \Lambda_{\mathrm{K}3^{[3]}}$ which identifies the copy $L_{\Km}\subset \NS(Y)\subset H^2(Y,\ZZ)$ of the Kummer lattice with the image of the primitive embedding $i\colon L_{\Km}\hookrightarrow \Lambda_{\mathrm{K}3^{[3]}}$; in other words, such that the period $\mathfrak{p}(Y,\theta)$ belongs to~$\Omega_{\mathrm{K}3^{[3]}}^{L_{\Km}^{\bot}}$. 
    Therefore, $(Y,\theta)$ has the same period of some marked $(Y_K, \theta')$ which is the hyper-Kummer $\mathrm{K}3^{[3]}$-type manifold associated with some $K$ of $\mathrm{Kum}^3$-type. It follows that $Y$ and $Y_K$ are bimeromorphic, as the birational Torelli theorem holds for manifolds of $\mathrm{K}3^{[3]}$-type by \cite{markman2011survey}.
\end{proof}


\subsection{Classification of generic projective hyper-Kummer $\mathrm{K}3^{[3]}$-type varieties}

Theorem \ref{thm:birationalCharacterization} characterizes hyper-Kummer $\mathrm{K}3^{[3]}$-type manifolds up to bimeromorphic isomorphisms. A generic member is non-projective as the N\'eron--Severi lattice is isometric to the Kummer lattice $L_{\Km}$, which is negative definite. It is interesting to study the countably many codimension-1 subfamilies consisting of hyper-Kummer $\mathrm{K}3^{[3]}$-type manifolds associated with the families of polarized varieties of $\mathrm{Kum}^3$-type. In this section, we classify generic members of such families by computing the transcendental and N\'eron--Severi lattices of hyper-Kummer $\mathrm{K}3^{[3]}$-type manifolds associated with varieties of $\mathrm{Kum}^3$-type of Picard rank 1.

Recall that the \textit{monodromy group} $\mathrm{Mon}^2(X)$ of a compact hyper-K\"ahler manifold $X$ is the subgroup of $\GL(H^2(X,\ZZ))$ generated by the monodromy operators coming from all smooth holomorphic families $\mathcal{X}\to B$ with a special fiber $\mathcal{X}_0 = X$. 
It is known (\cite{markman2011survey}) that $\mathrm{Mon}^2(X)$ is a subgroup of finite index in the orthogonal group $\Oo(H^2(X,\ZZ))$, and that it is deformation invariant.
Denoting by $\Lambda$ the Beauville--Bogomolov lattice $H^2(X,\ZZ)$, the set of irreducible components of the moduli space of projective hyper-K\"ahler varieties deformation equivalent to $X$ is in bijection with the set of $\mathrm{Mon}^2(X)$-orbits of positive classes in $\Lambda$.
The monodromy group has been computed for all known deformation types of hyper-K\"ahler manifolds in \cite{Markman-IntegralConstraintsMonodromy} ($\mathrm{K}3^{[n]}$-type), \cite{Mongardi-Monodromy}$ (\mathrm{Kum}^n$-type),  \cite{MongardiRapagnetta} (OG6-type) and \cite{Onorati-MonodromyOG10} (OG10-type).

Let us now specialize the discussion to hyper-K\"ahler manifolds of $\mathrm{Kum}^3$-type. In this case,  $\Lambda_{\mathrm{Kum}^3}\simeq \mathrm{U}^{\oplus 3}\oplus \langle -8\rangle$, hence $A_{\Lambda_{\mathrm{Kum}^3}}\simeq \mathbb{Z}/8\mathbb{Z}$ with the discriminant form determined by the condition that the square of the generator is $-\frac{1}{8}\in \mathbb{Q}/2\mathbb{Z}$.
It is easy to see that $\Oo(A_{\Lambda_{\mathrm{Kum}^3}})\cong \ZZ/2\ZZ$, and so any isometry $g\in \Oo(\Lambda_{\mathrm{Kum}^3})$ acts on $A_{\Lambda_{\mathrm{Kum}^3}}$ as $\pm \mathrm{id}$. 
Denote by $\chi\colon \Oo(\Lambda_{\mathrm{Kum}^3})\to \{\pm 1 \}$ the induced character.  The monodromy group of $\mathrm{Kum}^3$-manifolds is then given as
\begin{equation}
\mathrm{Mon}^2 (\mathrm{Kum}^3) = \{ g\in \Oo^+(\Lambda_{\mathrm{Kum}^3}) \ | \ \det(g)\cdot \chi(g)=1\},
\end{equation}
where $\Oo^+(\Lambda_{\mathrm{Kum}^3})$ is the group of orientation preserving isometries (\cite{markman2011survey}). 

\begin{lemma}
\label{lemma:Mon-Orbit}
 Let $d$ be a positive integer. Let $h\in \Lambda_{\mathrm{Kum}^3}$ be a positive class of square $2d$ and divisibility~$\gamma$. The $\mathrm{Mon}^2(\mathrm{Kum}^3)$-orbit of $h$ is equal to its $\Oo(\Lambda_{\mathrm{Kum}^3})$-orbit. It consists of all the primitive elements $h'\in \Lambda_{\mathrm{Kum}^3}$ of square $2d$, divisibility $\gamma$ and such that $[(h',-)/\gamma]=\pm [(h,-)/\gamma]$ 
 in $A_{\Lambda_{\mathrm{Kum}^3}}$.
\end{lemma}
\begin{proof}
    Let $\widetilde{\Oo}(\Lambda_{\mathrm{Kum}^3})$ denote the stable orthogonal group, that is, the group of isometries which act trivially on the discriminant. By Eichler's criterion (Theorem \ref{thm:eichler}), the $\widetilde{\Oo}(\Lambda_{\mathrm{Kum}^3})$-orbit of $h$ coincides with the $\widetilde{\SO}^+(\Lambda_{\mathrm{Kum}^3})$-orbit of $h$ (see \cite[Proposition 2.15]{Song}), where $\widetilde{\SO}^+(\Lambda_{\mathrm{Kum}^3})=\widetilde{\Oo}^+(\Lambda_{\mathrm{Kum}^3})\cap \widetilde{\SO}(\Lambda_{\mathrm{Kum}^3})$.
    This orbit consists of the primitive classes $h'$ of square $2d$, divisibility $\gamma$ and such that $[{h'}^\vee/\gamma]=[h^\vee/\gamma]$ in the discriminant group of $\Lambda_{\mathrm{Kum}^3}$. Here and in the sequel, $h^\vee:=(h, -)$ in the dual lattice.
    On the other hand, $\widetilde{\SO}^+(\Lambda_{\mathrm{Kum}^3})$ is strictly contained in the monodromy group, as the latter contains isometries acting as $-1$ on $A_{\Lambda_{\mathrm{Kum}^3}}$ (see \cite[Proof of Proposition 3.4]{Song}). It follows that the $\mathrm{Mon}^2(\mathrm{Kum}^3)$-orbit of $h$ contains all primitive classes $h'$ with square $2d$, divisibility $\gamma$, and such that $[(h')^\vee/\mathrm{div}(h')]=\pm [h^\vee /\mathrm{div}(h)]$ in the discriminant group. But this is precisely the $\Oo(\Lambda_{\mathrm{Kum}^3})$-orbit of $h$, by Lemma \ref{lem:orbitsU^3+<-8>}. 
\end{proof}

Hence, the components of the moduli space of polarized varieties of $\mathrm{Kum}^3$-type are in bijection with the $\Oo(\Lambda_{\mathrm{Kum}^3})$-orbits of primitive positive vectors, which we classify in Lemma \ref{lem:orbitsU^3+<-8>}. They are parametrized by triples $(2d, \gamma, \pm \epsilon)$, where $2d$ and $\gamma$ are the square and divisibility of the polarization $h$, and $\pm \epsilon$ is the residual class $\pm [(h,-)/\gamma]$ in the discriminant group. 

Given a polarized variety $(K,h)$ of $\mathrm{Kum}^3$-type, we do not have a natural polarization on the hyper-Kummer variety $Y_K$. Nevertheless we have a natural big and nef class on $Y_K$, which is the primitive generator of $\langle r_* h \rangle\subset \NS(Y_K)$, where $r\colon K\dashrightarrow Y_K$ is the rational map given by the hyper-Kummer construction. By Proposition \ref{prop:cohomologyY_K}, the class $\overline{h} \coloneqq \tfrac{1}{2^4} r_*(h)$ is integral, and $q_{Y_K}(\overline{h},\overline{h}) = 2q_K(h,h)$.

\begin{proposition}\label{prop:invariants-pushforward-polarization}
    Let $(K,h)$ be a polarized variety of $\mathrm{Kum}^3$-type, and $\overline{h}\in H^2(Y_K,\ZZ)$ the induced big and nef class on the associated hyper-Kummer sixfold $Y_K$. Let $2d$ and $\gamma$ denote the square and divisibility of $h$. Then: 
    \begin{enumerate}[label=(\roman*)]
    \item if $\gamma = 1$ then $\overline{h}\in H^2(Y_K,\ZZ)$ is primitive, of divisibility $1$ and square $4d$; 
    \item if $\gamma > 1$ then $\tfrac{1}{2}\overline{h}$ is integral and primitive, of divisibility $\tfrac{1}{2}\gamma$ and square $d$.
    \end{enumerate} 
\end{proposition}
\begin{proof} 
Let $e_i,f_i$, $i=1,2,3$, $\xi$, be a basis of $\Lambda_{\mathrm{Kum}^3}= \mathrm{U}^{\oplus 3} \oplus \langle -8\rangle$. 
Choose vectors $u_i,v_i$, $i=1,2,3$, $r_1,\dots, r_{16}$ and $\delta$ in $\Lambda_{\mathrm{K}3^{[3]}}$ such that: 
the $u_i,v_i$ generate a primitive sublattice $\mathrm{U}(2)^{\oplus 3}$, orthogonal to the $r_i$ and $\delta$; the sublattice $\langle r_{i} \rangle$ has saturation isometric to the Kummer lattice, and the saturation of the sublattice generated by the $u_i,v_i$ and the $r_i$ is a K3 lattice, with orthogonal generated by $\delta$, primitive of square $-4$.
Then, for suitable isometries $H^2(K,\ZZ)\xrightarrow{\ \sim \ } \Lambda_{\mathrm{Kum}^3}$ and $H^2(Y_K,\ZZ) \xrightarrow{\ \sim \ } \Lambda_{\mathrm{K}3^{[3]}}$, the map $\tfrac{1}{2^4}r_*\colon H^2(K,\ZZ)\to H^2(Y_K,\ZZ)$ is integral, multiplies the form by $2$ and sends $e_i$ to $u_i$, $f_i$ to $v_i$ and $\xi$ to $2\delta$, by Lemma \ref{lem:cohomologyq_*}.

Let now $h\in H^2(K,\ZZ)$ be a primitive class. Then $h=\alpha w + \beta \xi$, for coprime integers $\alpha$ and $\beta$ and a primitive $w\in \xi^{\bot}$; we have $\mathrm{div}(h) = \gcd(\alpha, 8)$. We have $\overline{h} = \alpha \overline{w} + 2\beta \delta$, where $\overline{w}\in \mathrm{U}(2)^{\oplus 3} = \langle u_i,v_i\rangle$ is still primitive. 
Then $\gamma=\mathrm{div}(h)= 1 $ if and only if $\alpha$ is odd; in this case $\overline{h}$ is primitive, has square $4d$ and its divisibility in $H^2(Y_K,\ZZ)$ equals $\gcd(\alpha, 4)=1$. 
We have $\gamma=\mathrm{div}(h)> 1$ if and only if $\alpha =2\alpha'$ is even; then $\overline{h} = 2 (\alpha' \overline{w} + \beta\delta)$ and $\tfrac{1}{2}\overline{h}=\alpha'\overline{w} + \beta\delta$ is primitive, of square $d$, and of divisibility $\gcd(\alpha', 4) = \tfrac{1}{2}\gcd(\alpha, 8) = \tfrac{1}{2}\gamma$.
\end{proof}

The next theorem describes the transcendental and N\'eron--Severi lattices of hyper-Kummer $\mathrm{K}3^{[3]}$-varieties associated with $\mathrm{Kum}^3$-varieties of Picard rank $1$. 
We identify $A_{\Lambda_{\mathrm{Kum}^3}}$ with $\ZZ/8\ZZ$; as for the quadratic form, the generator has square $-1/8$. We denote by $L_{4d}$ the unique rank-$17$ lattice with discriminant form equal to $A_{\mathrm{U}(2)^{\oplus 2} \oplus \langle 4d\rangle}$ and signature $(1, 16)$, which is an overlattice of index $2$ of $\langle 4d\rangle \oplus L_{\mathrm{Km}}$ (see Lemma \ref{lem:NS-hyperKummer}). 

\begin{theorem}\label{thm:NeronSeveriHyperkummer}
    Let $(K,h)$ be a polarized hyper-K\"ahler variety of $\mathrm{Kum}^3$-type of Picard rank $1$ with ample generator $h$ of Beauville--Bogomolov square $2d$ and divisibility $\gamma$. Let $Y_K$ denote the hyper-Kummer variety of $\mathrm{K}3^{[3]}$-type associated with $K$. 
    Then we have the following possibilities:
    \begin{equation}
    \smallskip
    \begin{tabular}{|c|c|c|c|c|c|}
    \hline
    $d$ & $\gamma$ & $\pm[h^\vee/\gamma]$ & $H^2_{\mathrm{tr}}(K,\ZZ)$ & $H^2_{\mathrm{tr}}(Y_K,\ZZ)\simeq \widehat{H}^2_{\mathrm{tr}}(K,\ZZ)(2)$ & $\NS(Y_K)$\\
    \hline\hline
    $d$ & $1$ & $0$ & $\mathrm{U}^{\oplus 2} \oplus \begin{psmallmatrix}
        -8 & 0 \\
        0 & -2d
    \end{psmallmatrix}$ & $\mathrm{U}(2)^{\oplus 2} \oplus \begin{psmallmatrix}
        -4 & 0 \\
        0 & -4d
    \end{psmallmatrix}$ & $L_{4d}$
    \\
    \hline
    $4(d'-1) $ & $2$ & $4$ & $\mathrm{U}^{\oplus 2}\oplus \begin{psmallmatrix}
        -8 & 4 \\
        4 & -2d'
    \end{psmallmatrix}$ & $\mathrm{U}(2)^{\oplus 2} \oplus\begin{psmallmatrix}
        -4 & 0 \\
        0 & -d
    \end{psmallmatrix}$ & $L_{\Km}\oplus \langle d \rangle$\\
    \hline 
     $4(4d'-1) $ & $4$ & $\pm2$ & $\mathrm{U}^{\oplus 2}\oplus \begin{psmallmatrix}
        -8 & 2 \\
        2 & -2d'
    \end{psmallmatrix}$ & $\mathrm{U}(2)^{\oplus 2} \oplus\begin{psmallmatrix}
        -4 & 2 \\
        2 & -4d'
    \end{psmallmatrix}$ & $L_{\Km}\oplus \langle d \rangle$ \\
    \hline 
        $4(16d'-1) $ & $8$ & $\pm1$ & $\mathrm{U}^{\oplus 2}\oplus \begin{psmallmatrix}
        -8 & 1 \\
        1 & -2d'
    \end{psmallmatrix}$ & 
    $\mathrm{U}(2)^{\oplus 2} \oplus\begin{psmallmatrix}
        -4 & 2 \\
        2 & -16d'
    \end{psmallmatrix}$
    & $L_{\Km}\oplus \langle d \rangle$\\
    \hline
      $4(16d'-9) $ & $8$ & $\pm3$ & $\mathrm{U}^{\oplus 2}\oplus \begin{psmallmatrix}
        -8 & 3 \\
        3 & -2d'
    \end{psmallmatrix}$ & 
    $\mathrm{U}(2)^{\oplus 2} \oplus\begin{psmallmatrix}
        -4 & 2 \\
        2 & -(16d'-8)
    \end{psmallmatrix}$ 
    & $L_{\Km}\oplus \langle d \rangle$
    \\
    \hline
    \end{tabular}
    \smallskip
    \end{equation}
\end{theorem}
\begin{proof}
    By Lemma \ref{lemma:Mon-Orbit}, orbits of positive classes are parametrized by triples consisting of the square $2d$, the divisibility $\gamma$, and the subset $\{\pm[h^{\vee}/\gamma]\}$ in the discriminant group.  The possibilities are worked out in Lemma \ref{lem:orbitsU^3+<-8>}. In each case, $h^{\bot}$ is also determined, which gives the description of $H^2_{\mathrm{tr}}(K,\ZZ)$ for each orbit. 

    By Proposition \ref{prop:cohomologyY_K}, the transcendental lattice $H^2_{\mathrm{tr}}(Y_K,\ZZ)$ equals $\widehat{H}^2_{\mathrm{tr}}(K,\ZZ)(2)$, that is, the saturation of $H^2_{\mathrm{tr}}(K,\ZZ)$ inside the unique index $2$ overlattice $\widehat{H}^2(K,\ZZ)\cong \mathrm{U}^{\oplus 3} \oplus \langle -2\rangle$ of $H^2(K, \ZZ) \cong \mathrm{U}^{\oplus 3} \oplus \langle -8 \rangle$. Denoting by $\hat{h}$ the class $h$ as an element of $\widehat{H}^2(K,\ZZ)$, we thus have $H^2_{\mathrm{tr}}(Y_K,\ZZ) = (\hat{h}^{\bot})(2)$. The class $\hat{h}\in \widehat{H}^2(K,\ZZ)$ is either primitive or twice a primitive class, whose divisibility and square are computed in Lemma \ref{lem:divHatH}. We can then compute $\hat{h}^{\bot}$, and therefore $H^2_{\mathrm{tr}}(Y_K,\ZZ)$, using Lemma \ref{lem:orbitsU^3+<-2>}.

    Finally, consider the primitive embedding $\tau\colon \widehat{H}^2(K,\ZZ)(2)\hookrightarrow H^2(Y_K,\ZZ)$
    given by Proposition \ref{prop:cohomologyY_K}, with complement $\eta\colon L_{\Km}\hookrightarrow H^2(Y_K,\ZZ)$ the Kummer lattice. Then $\NS(Y_K)$ is the orthogonal complement of $\tau(\widehat{H}^2_{\mathrm{tr}}(K,\ZZ)(2))$ in $H^2(Y_K,\ZZ)$, that is, the saturation of the sublattice of $H^2(Y_K,\ZZ)$ generated by $\tau(\hat{h})$ and $\eta(L_{\Km})$. Then, by Lemma \ref{lem:NS-hyperKummer}, either $\hat{h}$ is primitive (of square $4d$) in $\widehat{H}^2(K,\ZZ)(2)$ and then $\NS(Y_K)$ is isometric to $L_{4d}$, or $\hat{h} = 2\hat{h}_0$ for some primitive $\hat{h}_0\in \widehat{H}^2(K,\ZZ)(2)$ of square $d$, and then $\NS(Y_K)$ is isometric to $L_{\Km}\oplus \langle \hat{h}_0\rangle$. By Lemma \ref{lem:divHatH}, if $\mathrm{div}(h)=1$ we are in the first case, and if $\mathrm{div}(h)\geq 2$ we are in the second case. 
\end{proof} 

\begin{remark}
By Proposition \ref{prop:cohomologyY_K}, the transcendental lattice $H^2_{\mathrm{tr}}(Y_K,\ZZ)$ is an overlattice of $H_{\mathrm{tr}}^2(K,\ZZ)(2)$. Comparing discriminant groups, we see that for $K$ of Picard rank $1$, the transcendental lattice $H^2_{\mathrm{tr}}(Y_K,\ZZ)$ is an overlattice of $H^2_{\mathrm{tr}}(K,\ZZ)(2)$ of index $2$ unless $\mathrm{div}(h)=8$, in which case $H^2_{\mathrm{tr}}(Y_K,\ZZ)$ is Hodge isometric to $H_{\mathrm{tr}}^2(K,\ZZ)(2)$.  
\end{remark}

\begin{corollary}
\label{cor:Transcendental-hyper-KummerK3}
    Let $S$ be a projective hyper-Kummer $\mathrm{K}3$ surface of minimal Picard rank $16$. Then its transcendental lattice $H^2_{\tr}(S, \ZZ)$ is isometric to one of the following lattices for some integer $d>0$,
    \begin{itemize}
        \item  
    $\mathrm{U}(2)^{\oplus 2} \oplus \begin{psmallmatrix}
        -4 & 2 \\
        2 & -4d
    \end{psmallmatrix}$ \,;
    \item $\mathrm{U}(2)^{\oplus 2}\oplus \langle -4\rangle \oplus \langle -4d\rangle$.
    \end{itemize}
\end{corollary}
\begin{proof}
    This statement is a direct consequence of Theorem \ref{thm:TranscendentalResolutions} and Theorem \ref{thm:NeronSeveriHyperkummer}.
\end{proof}

\subsection{Hodge isometric $\mathrm{Kum}^3$-varieties}

The naive birational Torelli theorem fails for hyper-K\"ahler manifolds of $\mathrm{Kum}^n$-type, as is shown by the following example, due to Namikawa \cite{Namikawa02}. 
\begin{example}
Let $A$ be a complex torus of dimension $2$ such that  $A$ is not isomorphic to $A^{\vee}$. Then by Shioda \cite{Shioda-PeriodAbelianSurfaces}, $A$ and $A^\vee$ have Hodge isometric second cohomology, hence the generalized Kummer varieties $K^n(A)$ and $K^{n}(A^{\vee})$ have Hodge isometric second cohomology. However, $K^n(A)$ and $K^{n}(A^{\vee})$ are in general not bimeromorphic, for example, when $A$ has Picard rank $0$; there are also projective examples, see \cite[Section 5]{Namikawa02} and \cite[Theorem 8.2]{Magni}. The essential reason for this failure is that there does not exist any Hodge isometry $H^2(K^n(A),\ZZ)\xrightarrow{ \ \sim \ } H^2(K^n(A^{\vee}),\ZZ)$ which is induced by a parallel transport operator. In fact the moduli space of marked manifolds of generalized Kummer type has at least $4$ connected components (\cite{MR4453973}). 
\end{example}

Nevertheless, we prove that the hyper-Kummer varieties associated with two projective $\mathrm{Kum}^3$-type varieties with Hodge isometric $H^2(-, \mathbb{Z})$ are birational.

\begin{proposition}
\label{prop:HodgeIsometricKummer}
    Let $K_1,K_2$ be projective hyper-K\"ahler varieties of $\mathrm{Kum}^3$-type. Assume that there exists a Hodge isometry $H^2(K_1,\ZZ)\xrightarrow{ \ \sim \ } H^2(K_2,\ZZ)$. Then the associated hyper-Kummer sixfolds $Y_{K_1}$ and $Y_{K_2}$ are birational.
\end{proposition}
\begin{proof}
 Since the birational Torelli theorem holds for manifolds of $\mathrm{K}3^{[3]}$-type, it will be sufficient to show that there exists a Hodge isometry $H^2(Y_{K_1},\ZZ)\xrightarrow{\ \sim \ } H^2(Y_{K_2},\ZZ)$. We fix a Hodge isometry $f\colon H^2(K_{1},\ZZ)\xrightarrow{\sim} H^2(K_2,\ZZ)$.

Consider the rational maps $r_i\colon K_i\dashrightarrow Y_{K_i}$ coming from the hyper-Kummer construction. As the construction works in families, choosing a deformation of $K_1$ to $K_2$, we can choose markings $\beta_{i}\colon H^2(K_i,\ZZ)\to \Lambda_{\mathrm{Kum}^3}$ and $\alpha_i\colon  H^2(Y_{K_i},\ZZ) \to \Lambda_{\mathrm{K}3^{[3]}}$ such that the maps $\alpha_{i}\circ (r_i)_*\circ (\beta_i)^{-1}$ are identified with the same map $r_*\colon \Lambda_{\mathrm{Kum}^3}\to \Lambda_{\mathrm{K}3^{[3]}}$. 
By Proposition \ref{prop:cohomologyY_K}, the map $r_*$ is injective and multiplies the form by $2^9$; moreover, $\tfrac{1}{2^4} r_*$ is integral and induces a primitive embedding 
\begin{equation} 
j\colon \widehat{\Lambda}_{\mathrm{Kum}^3}(2) \hookrightarrow \Lambda_{\mathrm{K}3^{[3]}},
\end{equation}
with orthogonal complement isometric to the Kummer lattice $L_{\Km}$. Here, $\widehat{\Lambda}_{\mathrm{Kum}^3}$ is the unique overlattice of ${\Lambda}_{\mathrm{Kum}^3}$. 
Let us denote by $\hat{T}\subset \Lambda_{\mathrm{K}3^{[3]}}$ the primitive sublattice which is the image of $j$, which is isometric to $\mathrm{U}(2)^{\oplus 3} \oplus \langle -4\rangle$, and let us denote by $T\subset \hat{T}$ the index $2$ sublattice which is the image of $\Lambda_{\mathrm{Kum}^3}(2)$ via $j$. 
We have $2$ different Hodge structures on $T$ and $\hat{T}$: the one coming from $K_1$ via $\beta_1$ and the second one coming from $K_2$ via $\beta_2$. Let $H_1, H_2$ (resp. $\hat{H}_1$, $\hat{H}_2$) be the transcendental lattices of these $2$ Hodge structures on $T$ (resp. on $\hat{T}$). We can view $f$ as an isometry of $T$ which identifies these $2$ different Hodge structures; moreover, $f$ extends to an isometry of $\hat{T}$. We thus get an isometry $f\in \Oo(\hat{T})$, which restricts to a Hodge isometry $f_{|_{\hat{H}_1}}\colon \hat{H}_1\to \hat{H}_2$. To conclude the proof, we will show that $f_{|_{\hat{H}_1}}$ extends to an isometry $F$ of $\Lambda_{\mathrm{K}3^{[3]}}$; then $(\alpha_2)^{-1}\circ F\circ (\alpha_{1}) \colon H^2(Y_{K_1},\ZZ) \to H^2(Y_{K_2},\ZZ)$ is a Hodge isometry.

    Choose a primitive embedding $i\colon \Lambda_{\mathrm{K}3^{[3]}}\hookrightarrow \widetilde{\Lambda}_{\mathrm{K}3}$, with orthogonal complement spanned by a primitive vector $v$ of square $4$. We are thus in the situation of Lemma \ref{lem:extendIsometryToM}: which implies that, setting $M\coloneqq ( i(\hat{T}) \oplus \langle v\rangle)^{\mathrm{sat}}$, the isometry $f\in \Oo(\hat{T})$ extends to an isometry $\tilde{f}$ of $M$ such that $f(v)=\pm v$. Then consider $\tilde{H}_1\coloneqq (i(\hat{H}_1)\oplus \langle v\rangle)^{\mathrm{sat}}$ and $\tilde{H}_2\coloneqq (i(\hat{H}_2)\oplus \langle v\rangle)^{\mathrm{sat}}$, equipped with the Hodge structure of weight $2$ induced by that on $\hat{H}_i$, for which $v$ is of type $(1,1)$.
    Then $\tilde{f}$ gives a Hodge isometry 
    \begin{equation}
        \tilde{f}\colon \tilde{H}_1\xrightarrow{\ \sim \ } \tilde{H}_2,
    \end{equation}   
    which sends $v$ to $\pm v$. Since $K_1, K_2$ are assumed projective, the orthogonal complements $Z_i$ of $\tilde{H}_i$ in $\widetilde{\Lambda}_{\mathrm{K}3}$ is an indefinite lattice of rank at least $17$ and length of discriminant at most $6$. By Theorem \ref{thm:embeddingUnimodular}, it follows that $Z_i$ is uniquely determined by its signature and discriminant form, and that $\Oo(Z_i)\to \Oo(A_{Z_i})$ is surjective, for $i=1,2$. 
    It follows that $Z_1$ and $Z_2$ are isometric and that any isometry $\tilde{H}_1\xrightarrow{\ \sim \ } \tilde{H}_2$ extends to an isometry of $\widetilde{\Lambda}_{\mathrm{K}3}$. Therefore, $\tilde{f}$ extends to $\tilde{F}\in \Oo(\widetilde{\Lambda}_{\mathrm{K}3})$.
    Since $\tilde{F}(v)=\pm v$, restricting to $v^{\bot}$ we get an isometry $F\in \Oo(\Lambda_{\mathrm{K}3^{[3]}})$ whose restriction to $\hat{H}_1$ equals $f$.
    \end{proof}

\begin{remark}
We cannot drop the assumption that $K$ and $K'$ are projective. 
Indeed, let $K,K'$ be Hodge isometric manifolds of $\mathrm{Kum}^3$-type of Picard rank $0$ which are not birational to each other.
Then $Y_K, Y_K'$ are $\mathrm{K}3^{[3]}$-manifolds of Picard rank $16$, with N\'eron--Severi group generated by the components $E_1,\dots , E_{16}$  (resp., $E'_1,\dots ,E'_{16}$) of the exceptional divisor of $Y_K\to K/G$ (resp., $Y_{K'}\to K'/G$). 

Assume that $Y_K$ and $Y_K'$ are bimeromorphic. Then by Proposition \ref{prop:veryGeneralCase}, they are in fact isomorphic. As the $E_i$ and $E'_i$ are the only effective $-2$-classes in $\NS(Y_K)$ (resp., $\NS(Y_K')$), any isomorphism $\phi\colon Y_K\to Y_{K'}$ would send $E_i$ to $E'_j$, for some $j$. Setting $U\coloneqq Y_K\setminus (\bigcup_{i=1}^{16} E_i )$ and $U'\coloneqq Y_{K'}\setminus (\bigcup_{i=1}^{16} E'_i )$, it follows that $\phi$ restricts to an isomorphism $U\xrightarrow{\ \sim \ } U'$. It thus induces an isomorphism $\tilde{\phi}$ between the universal covers $\tilde{U}$, $\tilde{U}'$ of $U,U'$ respectively. But $\tilde{U}\subset K$ and $\tilde{U'}\subset K'$ are Zariski open, so that $K$ and $K'$ are bimeromorphic, which contradicts the hypothesis.
\end{remark}

\section{Companion K3 surfaces}
\label{sec:CompanionK3}
As is explained in Section \ref{subsec:Companions}, given a hyper-K\"ahler manifold $K$ of $\mathrm{Kum}^3$-type, there are 120 canonically embedded K3 surfaces $V_{\sigma, \sigma'}$ and 120 resolution K3 surfaces $S_{\sigma, \sigma'}$, where the indices run through pairs of distinct elements $\sigma$ and $\sigma'$ in $G\backslash G_1$. We study in this section these companion K3 surfaces.

\subsection{Embedded companion K3 surfaces}
We start with the companion K3 surfaces $V_{\sigma,\sigma'}$ contained in a hyper-K\"ahler manifold of $\mathrm{Kum}^3$-type. It turns out that these K3 surfaces are precisely the K3 surfaces with a symplectic action of $(\ZZ/2\ZZ)^4$, which were studied by Garbagnati and Sarti \cite{GarbagnatiSarti}. 
We consider the negative definite lattice $N_V$ of rank $15$ introduced in Definition \ref{def:latticeN_V}. 

\begin{theorem}
\label{thm:characterizeV}
    Let $V$ be a $\mathrm{K}3$ surface. Then the following are equivalent: 
    \begin{enumerate}[label=(\roman*)]
    \item there exists a primitive embedding $j\colon N_V\hookrightarrow \NS(V)$;
    \item there exists a faithful action of $H\cong (\ZZ/2\ZZ)^4$ on $V$ via symplectic automorphisms;
    \item $V$ is isomorphic to an embedded companion $\mathrm{K}3$ surface $V_{\sigma,\sigma'}\subset K$, for some manifold $K$ of $\mathrm{Kum}^3$-type and some $\sigma\neq \sigma'\in G\backslash G_1$.
    \end{enumerate}
    If $V$ is projective, these conditions are further equivalent to the following condition:
    \begin{enumerate}
        \item[(iv)] there exists a primitive embedding $H^2_{\mathrm{tr}}(V,\ZZ)\hookrightarrow \mathrm{U}(2)^{\oplus 3} \oplus \langle -8\rangle$.
    \end{enumerate}
\end{theorem}
\begin{proof}
    The equivalence between $(i)$ and $(ii)$ is due to Nikulin, see \cite{GarbagnatiSarti}. That $(iii)$ implies $(i)$ is a consequence of Proposition \ref{prop:cohomologyW&V} $(ii)$.
    In fact, we can also see that $(iii)$ implies $(ii)$ directly. Consider the action of the group $\Gamma\cong(\ZZ/4\ZZ)^4$ (see Definition \ref{def:subgroupsAut0}) on $K$: if $\epsilon \in \Gamma$ is such that $2\epsilon = \sigma+\sigma'$, then the corresponding automorphism $\epsilon$ of $K$ induces a non-trivial involution of $V_{\sigma,\sigma'}$. Together with the action of $G/\langle \sigma,\sigma'\rangle$ on $V_{\sigma,\sigma'}$, this involution generates a symplectic action of $(\ZZ/2\ZZ)^4$ on $V_{\sigma,\sigma'}$.

    To see that $(i)$ implies $(iii)$ we proceed as in the proof of Theorem \ref{thm:birationalCharacterization}. As the K3 surfaces $V_{\sigma,\sigma'}\subset K$ deform together with $K$, we obtain for each pair $\sigma\neq \sigma'$ a map of lattices 
    \begin{equation}
    \iota_{\sigma,\sigma'}^*\colon \Lambda_{\mathrm{Kum}^3}(4)\hookrightarrow \Lambda_{\mathrm{K}3};
    \end{equation}
by Proposition \ref{prop:cohomologyW&V}, its image is not saturated but has saturation $\mathrm{U}(2)^{\oplus 3}\oplus \langle -8\rangle$, and the orthogonal complement of $\mathrm{im}(\iota^*_{\sigma,\sigma'})$ is the lattice $N_V$. We fix the primitive embedding $i_{\sigma,\sigma'} \colon N_V\hookrightarrow \Lambda_{\mathrm{K}3}$ as the orthogonal to the image of $\iota^*_{\sigma,\sigma'}$.
This map is canonical in the following sense: for any manifold $K$ of $\mathrm{Kum}^3$-type and any canonical K3 surface $\iota_{\gamma,\gamma', K}\colon V_{\gamma,\gamma'}\hookrightarrow K$, there exist isometries $\eta\colon H^2(K,\ZZ)\xrightarrow{\ \sim \ }\Lambda_{\mathrm{Kum}^3}$ and $\theta\colon H^2(V_{\gamma,\gamma'},\ZZ)\xrightarrow{\ \sim \ } \Lambda_{\mathrm{K}3}$ such that 
\begin{equation}
    \theta \circ \iota^*_{\gamma,\gamma',K} \circ \eta^{-1} = \iota^*_{\sigma,\sigma'},
\end{equation}
for some $\sigma,\sigma'$. 
It follows that, for a suitable marking, the K3 surface $V_{\sigma,\sigma'}$ has period in the subperiod domain 
\begin{equation}
    \Omega_{\Lambda_{\mathrm{K}3}}^{i_{\sigma,\sigma'}(N_V)^{\bot}} \coloneqq \{ x\in \Omega_{\Lambda_{\mathrm{K}3}} \ | \ x \in (i_{\sigma,\sigma'}(N_V))^{\bot}\}
\end{equation}
of the period domain $\Omega_{\Lambda_{\mathrm{K}3}}$.
Conversely, any point in $\Omega_{\Lambda_{\mathrm{K}3}}^{i_{\sigma,\sigma'}(N_V)^{\bot}}$ is the period of some $V_{\gamma,\gamma'} \subset K$ for a suitable marking, since $\iota^*_{\sigma,\sigma'}$ induces an isomorphism of period domains $\Omega_{\mathrm{Kum}^3}\cong \Omega_{\Lambda_{\mathrm{K}3}}^{i_{\sigma,\sigma'}(N_V)^{\bot}}$. 

If now $V$ is a K3 surface as in $(i)$, since $N_V$ embeds uniquely into the K3 lattice by Lemma \ref{lem:embeddingN_V}, we find a marking $\theta$ on $V$ such that $(V,\theta)$ has period in $\Omega_{\Lambda_{\mathrm{K}3}}^{i_{\sigma,\sigma'}(N_V)^{\bot}}$. Hence, $V$ is Hodge isometric, and thus isomorphic, to some companion K3 surface $V_{\sigma,\sigma'}\subset K$, for some $K$ of $\mathrm{Kum}^3$-type.

Since the orthogonal to $N_V$ in $\Lambda_{\mathrm{K}3}$ is isometric to $\mathrm{U}(2)^{\oplus 3} \oplus \langle -8 \rangle$, we have that $(i)$ implies $(iv)$. If $V$ is projective, then there exists a unique primitive embedding of $H^2_{\mathrm{tr}}(V,\ZZ)$ in the K3 lattice (by Theorem \ref{thm:embeddingUnimodular}), which necessarily contains a primitive copy of $N_V$ in its orthogonal complement. Thus $(iv)$ is equivalent to $(i)$ for $V$ projective.
\end{proof}

\subsection{Hyper-Kummer K3 surfaces}
We obtain a similar characterization for the hyper-Kummer K3 surfaces $S_{\sigma,\sigma'}$ associated with a $\mathrm{Kum}^3$-manifold $K$, as the minimal resolution of $V_{\sigma,\sigma'}/G$. We introduce in Definition \ref{def:latticeN_S} another negative definite lattice $N_S$ of rank $15$.

\begin{theorem}\label{thm:characterizeS}
    Let $S$ be a $\mathrm{K}3$ surface. The following are equivalent:
    \begin{enumerate}[label=(\roman*)]
    \item there exists a primitive embedding $N_S \hookrightarrow \NS(S)$;
    \item there exists a $\mathrm{K}3$ surface $V$ with a faithful symplectic action of $H\cong (\ZZ/2\ZZ)^4$ and a subgroup $H'\subset H$ of index $2$ such that $S$ is the minimal resolution of $V/H'$;
    \item $S$ is isomorphic to the $\mathrm{K}3$ surface $S_{\sigma,\sigma'}$ associated with some hyper-K\"ahler manifold $K$ of $\mathrm{Kum}^3$-type.
    \end{enumerate} 
    If $S$ is projective, these conditions are furthermore equivalent to the following condition:
    \begin{enumerate}
        \item[$(iv)$] there exists a primitive embedding of lattices $H^2_{\mathrm{tr}}(S,\ZZ)\hookrightarrow \mathrm{U}(2)^{\oplus 3}\oplus \langle -4\rangle$.
    \end{enumerate}
\end{theorem}
\begin{proof}
    That $(iii)$ implies $(i)$ is a consequence of Proposition~\ref{prop:cohomologyM&S}~$(ii)$.
    Moreover, by Theorem \ref{thm:characterizeV}, condition $(iii)$ implies $(ii)$, since, by construction, the hyper-Kummer K3 surface $S_{\sigma,\sigma'}$ is the minimal resolution of $V_{\sigma,\sigma'}/(\ZZ/2\ZZ)^3$. 

    If $S$ and $V$ are as in $(ii)$, since $\mathrm{K}3$ surfaces with $(\ZZ/2\ZZ)^4$-action form a single deformation family, we may deform $V$ to a Kummer surface $\Km(A)$ so that $H$ is identified with the group $A[2]$ of points of order $2$, and $H'$ with a subgroup of $A[2]$ of index $2$. We are then in the situation of Lemma \ref{lem:push-forwardt_*}, which implies that $N_S$ embeds primitively in the N\'eron--Severi group of $S$. Hence, $(ii)$ implies $(i)$.

    To show that $(i)$ implies $(iii)$ we proceed exactly as in the proof of Theorem \ref{thm:characterizeV}. By Proposition~\ref{prop:cohomologyM&S}.$(ii)$, the hyper-Kummer construction gives $120$ embeddings 
    \begin{equation} j_{\sigma,\sigma'}\colon \Lambda_{\mathrm{Kum}^3} (2) \hookrightarrow \Lambda_{\mathrm{K}3}\end{equation}
    of lattices, whose image is contained in its saturation with index $2$ and has complement the lattice $N_S$; we denote by $i_{\sigma,\sigma'}$ the corresponding primitive embedding of $N_S$ in $\Lambda_{\mathrm{K}3}$. As the construction of the hyper-Kummer K3 surfaces works in families, for any $S_{\gamma,\gamma'}$ associated to a $\mathrm{Kum}^3$-manifold $K$, there exist markings $\eta$ on $K$ and $\theta$ on $S_{\gamma,\gamma'}$ such that 
    \begin{equation}
        \theta\circ (\tfrac{1}{4}(r_{\gamma,\gamma'})_*\circ \iota_{\gamma,\gamma'}^*)\circ \eta^{-1} = j_{\sigma,\sigma'},
    \end{equation}
    for some pair $\sigma,\sigma'$; here, $r_{\gamma,\gamma'}\colon V_{\gamma,\gamma'}\dashrightarrow S_{\gamma,\gamma'}$ is the quotient rational map, while $\iota_{\gamma,\gamma'}\colon V_{\gamma,\gamma'}\hookrightarrow K$ is the inclusion.
    Defining the sub-period domains
    \begin{equation}
        \Omega_{\Lambda_{\mathrm{K}3}}^{i_{\sigma,\sigma'}(N_S)^{\bot}}\coloneqq \{ x \in \Omega_{\Lambda_{\mathrm{K}3}} \ | \ x\in (i_{\sigma,\sigma'}(N_S))^{\bot}\},
    \end{equation}
    it follows that any hyper-Kummer K3 surface has period in some $\Omega_{\Lambda_{\mathrm{K}3}}^{i_{\sigma,\sigma'}(N_S)^{\bot}}$ for a suitable marking. Moreover, $j_{\sigma,\sigma'}$ induces an isomorphism between $\Omega_{\mathrm{Kum}^3}$ and $\Omega_{\Lambda_{\mathrm{K}3}}^{i_{\sigma,\sigma'}(N_S)^{\bot}}$, so that any point of this sub-period domain is the period of a hyper-Kummer K3 surface.

    If $S$ is as in $(i)$, since $N_S$ admits a unique primitive embedding into $\Lambda_{\mathrm{K}3}$ up to isometry (Lemma \ref{lem:embeddingN_S}), there exists a marking on $S$ such that its period lies in $\Omega_{\Lambda_{\mathrm{K}3}}^{i_{\sigma,\sigma'}(N_S)^{\bot}}$. By the above and the Torelli theorem, $S$ is isomorphic to a hyper-Kummer K3 surface $S_{\sigma,\sigma'}$ associated with some $K$ of $\mathrm{Kum}^3$-type.

    The orthogonal of any primitive sublattice $N_S\subset \Lambda_{\mathrm{K}3}$ is isometric to $\mathrm{U}(2)^{\oplus 3} \oplus \langle -4\rangle$, thus $(i)$ implies $(iv)$. If $S$ is projective then there is a unique embedding of $H^2_{\mathrm{tr}}(S,\ZZ)$ in the K3 lattice up to isometry (Theorem \ref{thm:embeddingUnimodular}), and it follows that $(iv)$ is equivalent to $(i)$ in this case. 
\end{proof}

\section{Relation with O'Grady-6 type manifolds}
\label{sec:RelationMRSDouble}
A construction of Mongardi--Rapagnetta--Sacc\`a \cite{MRS18} gives a rational double cover of certain OG6-manifolds by manifolds of $\mathrm{K}3^{[3]}$-type.
In this section, we study the relationship between these rational double covers and the hyper-Kummer $\mathrm{K}3^{[3]}$-type manifolds.

\subsection{Mongardi--Rapagnetta--Sacc\`a double covers}
Let $A$ be an abelian surface, let $v_0$ be a primitive and effective Mukai vector of square $2$, and let $v=2v_0$. Fix a $v$-generic polarization $H$ and denote by $K_{A,H}(v)$ the Albanese fiber of the moduli space $M_{A,H}(v)$ of semistable sheaves on $A$ with Mukai vector $v$ (see \cite{huybrechts2010geometry}). Then $K_{A,H}(v)$ is projective but singular, and it admits a crepant resolution $\widetilde{K}_{A,H}(v)\to K_{A,H}(v)$ which is a hyper-K\"ahler manifold of $\mathrm{OG}6$-type. This was first shown by O'Grady in \cite{O'G03} for $v=(2,0,-2)$; the construction was generalized to more general Mukai vectors in \cite{LS06, PR13}.

The moduli space $K_{A,H}(v)$ is a (singular) irreducible symplectic variety with second Betti number $7$ (\cite{PeregoRapagnetta-2023}). The deformation theory of these irreducible symplectic varieties is the same as if they were smooth, as long as we restrict to locally trivial deformations \cite{bakkerLehn}. 

\begin{definition}
\label{def:OG6-resolution}
    A complex analytic variety $X$ is called a \emph{singular $\mathrm{OG}6$-variety} if it is a locally trivial deformation (in the sense of Flenner--Kosarew \cite{Flenner-Kosarew}) of the singular symplectic variety $K_{A,H}(v)$ with $A, v, H$ as above.
    Any singular $\mathrm{OG}6$-variety $X$ admits a crepant resolution $\widetilde{X}\to X$ by blowing up along the singular locus, where  $\widetilde{X}$ is a hyper-K\"ahler manifold of $\mathrm{OG}6$-type. Any crepant resolution of a singular OG6-variety will be called a \emph{$\mathrm{OG}6$-resolution}.
\end{definition}

The resolution $\widetilde{X}\to X$ can be described as follows. The singular locus in the fibers of a locally trivial family fits into a locally trivial family as well. Choosing a locally trivial deformation of $X$ to O'Grady's example \cite{O'G03}, we have closed subsets $$\Omega\subset \Sigma \subset X$$
such that:
\begin{itemize}
    \item $\Sigma$ is the singular locus of $X$, along which $X$ has generically a transversal $A_1$-singularity; 
    \item $\Omega$ is the singular locus of $\Sigma$, and consists of $256$ points.
\end{itemize}
Following Lehn--Sorger \cite{LS06}, the resolution $\widetilde{X}\to X$ is obtained via the blow-up of the reduced singular locus $\Sigma$. O'Grady's original procedure, which gives the same resolution, consists instead of first blowing up $\Omega$, then blowing up the strict transform of $\Sigma$, and then contracting the total transform of $\Omega$ onto $256$ smooth $3$-dimensional quadrics.

\begin{remark}
    The singular locus $\Sigma$ of a moduli space $K_{A,H}(v)$ as above is isomorphic to $(A\times A^{\vee})/\langle -1\rangle$. It follows that, for an arbitrary singular OG6-variety $X$, the singular locus $\Sigma$ is isomorphic to $B/\langle-1\rangle$ for a $4$-dimensional complex torus $B$ which is not necessarily a product of $2$-dimensional tori (\cite{FloccariFu-HodgeConjectureWeilFourfolds}). 
\end{remark}

Let $X=K_{A,H}(v)$ be a singular $\mathrm{OG}6$-variety obtained as a moduli space of sheaves on an abelian surface. In \cite{MRS18}, the authors constructed a rational double cover of $K_{A,H}(v)$ by a variety of $\mathrm{K}3^{[3]}$-type.
They observe that there must exist a double cover $W\to X$ ramified along $\Sigma$, because the class of the exceptional divisor of $\widetilde{X}\to X $ is divisible by~$2$ in $\Pic(\widetilde{X})$. They next prove that there exists a birational model of $W$ which is a variety~$Z_X$ of $\mathrm{K}3^{[3]}$-type. In fact, the arguments of Mongardi--Rapagnetta--Sacc\`a apply to all singular $\mathrm{OG}6$-varieties.

\begin{theorem}[\cite{MRS18}] \label{thm:MRSconstruction}
Let $X$ be a singular $\mathrm{OG}6$-variety, and let $\widetilde{\Sigma}$ denote the exceptional divisor of the associated $\mathrm{OG}6$-resolution $\widetilde{X}\to X$.
There exists a rational double cover 
\begin{equation}s\colon Z_X\dashrightarrow \widetilde{X} \end{equation}
ramified along $\widetilde{\Sigma}$, with $Z_X$ a hyper-K\"ahler manifold of $\mathrm{K}3^{[3]}$-type.
\end{theorem}
\begin{proof}
We follow \cite{MRS18}. Let $X$ be a singular OG6-manifold, with associated OG6-resolution $\widetilde{X}\to X$. Since $X$ may be deformed via a locally trivial deformation to the OG6-moduli space $K_{A,H}(v)$ on an abelian surface $A$, the class of the exceptional divisor $\widetilde{\Sigma}$ of $\widetilde{X}\to X$ is divisible by $2$ in $\Pic(\widetilde{X})$, by \cite[Theorem 3.3.1]{rapagnetta2007topological}. 
Arguing as in \cite[\S4]{MRS18}, there exists a commutative diagram 
\begin{equation}
\begin{tikzcd}
     \widetilde{W} \arrow{r}{\widetilde{\phi}} \arrow{d} & \widetilde{X} \arrow{d} \\
        W \arrow{r}{\phi} & X
\end{tikzcd}
\end{equation}
of complex analytic spaces, in which $\widetilde{\phi}$ and $\phi$ are double covers ramified over $\widetilde{\Sigma}$ and $\Sigma$ respectively. We consider the subvarieties $\Gamma\coloneqq \phi^{-1}(\Omega)$ and $\Delta\coloneqq \phi^{-1}(\Sigma)$ of $W$; the restriction of $\phi$ induces an isomorphism $\Delta\cong \Sigma$. Hence $\Delta$ is isomorphic to $B/\langle -1\rangle$ for a $4$-dimensional complex torus $B$, and $\Gamma$ consists of $256$ points; moreover, $\Gamma$ is the singular locus of $W$.

As in \cite[Proof of Proposition 5.3]{MRS18}, we consider the blow-up~$\overline{W}$ of $W$ along $\Gamma$. Then $\overline{W}$ is non-singular and contains $256$ exceptional divisors $I_i$, isomorphic to the incidence variety $I\subset \mathbb{P}^3\times \mathbb{P}^{3,\vee}$ for all $i$. Each of the two projections gives $I$ the structure of $\mathbb{P}^2$-bundle over $\mathbb{P}^3$. Choose a projection $p_i\colon I_i\to \mathbb{P}^3$ for each $i$. Then we may apply Nakano contraction theorem to obtain a complex K\"ahler manifold $Z_X$ with a morphism $ \pi\colon \overline{W}\to Z_X$ which restricts to the $\mathbb{P}^2$-bundle $p_i$ over each $I_i$. 
Moreover, the construction may be performed in locally trivial families (up to finite base-change), so that the manifolds $Z_{X}$ and $Z_{X'}$ are deformation equivalent, for any singular OG6-varieties $X, X'$. By \cite{MRS18}, if $X=K_{A,H}(v)$ is an OG6-moduli space on an abelian surface~$A$, then $Z_X$ is a hyper-K\"ahler variety of $\mathrm{K}3^{[3]}$-type; therefore, for any singular OG6-variety $X$, the manifold $Z_X$ constructed above is of $\mathrm{K}3^{[3]}$-type. By construction, $\phi$ gives a rational double cover $s\colon Z_X\dashrightarrow \widetilde{X}$.
\end{proof}

\begin{remark}
   The isomorphism class of the $\mathrm{K}3^{[3]}$-variety $Z_X$ is not canonically determined by $X$, but only its birational class is. Indeed, for $i=1,2,\dots, 256$ we need to choose one of the two projections $I_i\to \mathbb{P}^3$, and different choices may lead to birational but not isomorphic manifolds $Z_X$, related by Mukai flops. 
\end{remark}

The main goal of this section is the study of the $\mathrm{K}3^{[3]}$-manifolds arising from the above theorem.
\begin{definition}
    Let $Z$ be a hyper-K\"ahler manifold of $\mathrm{K}3^{[3]}$-type. We say that $Z$ is an \emph{MRS-double cover} if $Z$ is bimeromorphic to the manifold $Z_X$ arising as a rational double cover of a singular $\mathrm{OG}6$-variety $X$ via the construction of Mongardi--Rapagnetta--Sacc\`a. 
\end{definition}

\begin{remark}\label{rmk:explicitExampleMRS}
	The double cover of Theorem \ref{thm:MRSconstruction} was first observed by Rapagnetta \cite{rapagnetta2007topological} in the following special case.
	Let $J$ be the Jacobian of a very general curve of genus $2$, and let $\Theta$ be a symmetric theta divisor. The linear system $|2\Theta|$ induces an embedding of $J/\pm 1$ into $\PP^3$ as a quartic surface with $16$ nodes. Let $\Km(J)$ be the Kummer K3 surface associated with $J$, and denote by $H$ the pull-back to $\Km(J)$ of the hyperplane divisor of $J/\pm 1\subset \PP^3$. 
	Consider the singular OG6-variety $K_{J,\Theta}(0,2\Theta, 2)$; a general point of this moduli space corresponds to a line bundle supported on a curve in the linear system $|2\Theta|$. The MRS-double cover $Z_{K_{J,H}(0,2\Theta,2)}$ is then identified (up to birationality) with the moduli space $M_{\Km(J)}(0, H, 1)$ (which is also birational to $\Km(J)^{[3]}$).
	A general point of this moduli space has the form $i_*(L)$ for a line bundle $L$ supported on a curve $i\colon C\hookrightarrow \Km(J)$ in the linear system $|H|$ and which does not intersect the $16$ exceptional curves $C_i$ for the blow-up $\Km(J)\to J/\pm 1$; the rational map $M_{\Km(J)}(0,H,1)\dashrightarrow K_{J,\Theta}(0, 2\Theta, 2)$ is induced by the pull-back of sheaves along the rational map $J\dashrightarrow \Km(J)$. 
\end{remark}

\subsection{Characterization of MRS-double covers}
The goal of this section is to characterize the $\mathrm{K}3^{[3]}$-manifolds which are birational to MRS-double covers in terms of their second cohomology. See Appendix \ref{subsec:Barnes--Wall-Lattice} for details on the Barnes--Wall lattice.

\begin{theorem}\label{thm:characterizationMRS}
	Let $Z$ be a manifold of $\mathrm{K}3^{[3]}$-type. Then $Z$ is the MRS-double cover of some singular $\mathrm{OG}6$-variety $X$ if and only if there exists a primitive embedding of the Barnes--Wall lattice $\mathrm{BW}_{16}$ into the N\'eron--Severi lattice of $Z$.
\end{theorem}

We will first show the following relation between the transcendental lattices of a singular OG6-variety $X$ and the associated MRS-double cover $Z_X$. 

\begin{proposition}\label{prop:cohomologyMRS}
	Let $s\colon Z_X\dashrightarrow X$ be the MRS-double cover of a singular $\mathrm{OG}6$-variety $X$. Then the pull-back $s^*\colon H^2(X,\ZZ)\to H^2(Z_X,\ZZ)$ multiplies the form by a factor $2$ and defines a primitive embedding of lattices and Hodge structures 
	\begin{equation}
		s^*\colon H^2(X,\ZZ)(2)\hookrightarrow H^2(Z_X,\ZZ),
	\end{equation}
	with orthogonal complement isometric to the Barnes--Wall lattice $\mathrm{BW}_{16}$.
	In particular, we have a Hodge isometry $s^*\colon H^2_{\mathrm{tr}}(X,\ZZ)(2)\xrightarrow{\ \sim \ } H^2_{\mathrm{tr}}(Z_X,\ZZ)$.
\end{proposition}
\begin{proof}
    The second cohomology of a singular OG6-variety $X$ carries a pure Hodge structure of weight $2$; moreover, as a lattice, we have $H^2(X,\ZZ) = \mathrm{U}^{\oplus 3}\oplus \langle -2\rangle$.
    The pull-back map $s^*$ is well-defined and non-zero, since the pull-back of the symplectic form on the regular locus of $X$ induces a symplectic form of $Z_X$. Moreover, considering a locally trivial deformation of $X$ to a very general singular OG6-manifold for which the Hodge structure $H^2(X,\ZZ)$ is irreducible, it follows that $s^*$ is injective. 
    Applying Lemma \ref{lem:scalarFactor}, we find that $s^*$ multiplies the form by a factor $2$.
    We thus get an embedding $s^*\colon H^2(X,\ZZ)(2)\hookrightarrow H^2(Z_X,\ZZ)$ (the Fujiki constant of a singular OG6-variety is the same as that of its crepant resolution, computed in \cite{rapagnetta2007topological}). 
    
    Let $\phi\in\mathrm{Bir}(Z_X)$ be the birational covering involution associated with $s$. Then $\phi$ induces an isometry of $H^2(Z_X,\ZZ)$, which we still denote by $\phi$. Let $T\subset H^2(Z_X,\ZZ)$ be the invariant sublattice for $\phi$, and let $S$ be the orthogonal complement (on which $\phi$ acts as $-1$). The image of $s^*$ is contained in $T$ with finite index; hence, $T$ is a rank $7$ lattice of signature $(3,4)$, containing with finite-index $s^*(H^2(X,\ZZ)(2))\simeq \mathrm{U}(2)^{\oplus 3}\oplus \langle -4\rangle$. The anti-invariant lattice $S$ is negative definite of rank $16$. As a consequence of Markman's results \cite[Proposition 1.8, Theorem 9.17]{markman2011survey}, the anti-invariant lattice $S$ contains no $(-2)$-class (see \cite[Lemma 2.9]{10.1093/imrn/rnae112}). 
    Moreover, $\phi$ is a monodromy operator and it has to act as $\pm 1$ on the discriminant group, by \cite[Lemma 4.2]{Markman-IntegralConstraintsMonodromy}. 
    Let $i\colon H^2(Z_X,\ZZ)\hookrightarrow \widetilde{\Lambda}_{\mathrm{K}3}$ be a primitive embedding, with orthogonal complement $\langle v\rangle$ for some primitive $v$ of square $4$. Then $\phi$ extends to an isometry $\tilde{\phi}$ of $\widetilde{\Lambda}_{\mathrm{K}3}$, such that $\tilde{\phi}(v)=\pm v$ according to whether the isometry $\phi$ acts on $A_{H^2(Z_X,\ZZ)}$ as $1$ or $-1$.
    
    We claim that $\phi\in \Oo(H^2(Z_X,\ZZ))$ acts trivially on the discriminant group.
    Indeed, assume this is not the case. Let $\widetilde{T}$ and $\widetilde{S}$ be the invariant and anti-invariant lattices for $\tilde{\phi}\in \Oo(\widetilde{\Lambda}_{\mathrm{K}3})$. Then, since $\phi$ acts as $-1$ on the discriminant group, we have $\widetilde{T}=T$, while $\widetilde{S}$ is the saturation of $S\oplus \langle v \rangle$. Now, since $\widetilde{\Lambda}_{\mathrm{K}3}$ is unimodular, the invariant and anti-invariant sublattices for the involution $\tilde{\phi}$ are $2$-elementary (see \cite[Proposition 3.6.4]{Nikulin1980}). Hence, we have $A_T = (\ZZ/2\ZZ)^k$ for some $k$; since $T$ has odd rank, $k$ must be odd (\cite[Theorem 3.6.2]{Nikulin1980}).
    But then $T$ would have to be an overlattice of $\mathrm{U}(2)^{\oplus 3}\oplus \langle -4\rangle$ of index $\sqrt{\frac{2^8}{2^k}}$, which is not an integer as $k$ is odd. 

    Therefore, $\phi$ acts trivially on $A_{H^2(Z_X,\ZZ)}$, and $\phi$ extends to $\tilde{\phi}\in \widetilde{\Lambda}_{\mathrm{K}3}$ such that $\tilde{\phi}(v)=v$. The anti-invariant sublattice for $\tilde{\phi}$ equals the anti-invariant sublattice $S$ of $\phi$, which is negative definite of rank $16$. Hence, $S$ is $2$-elementary, with discriminant of length at most $8$, and $\phi_{|_S}$ acts trivially on $A_S$ (since this action has to match the one on the discriminant of the invariant sublattice). 

    Since $S$ contains no $-2$-classes, it can be primitively embedded in the Leech lattice $\Lambda_{24}$, which is the unique negative definite unimodular lattice of rank $24$ which contains no $-2$-classes. This follows from \cite[Proposition 2.2]{huybrechts2005equivalences} and its proof, which can be applied since the action of $\tilde{\phi}$ is symplectic with respect to the Hodge structure on $\widetilde{\Lambda}_{\mathrm{K}3}$ induced by the inclusion of $H^2(Z_X,\ZZ)$.
    Hence, we have a primitive embedding $S\hookrightarrow \Lambda_{24}$. Since $\phi_{|_S}$ is the identity on the discriminant, it extends to an isometry $\Phi$ of the Leech lattice, which acts as the identity on $S^{\bot}$ (and as $-1$ on $S$). The possible anti-invariant sublattices for an involution of the Leech lattice are classified in \cite{HaradaLang1990}; the only one of rank $16$ is the Barnes--Wall lattice $\mathrm{BW}_{16}$.

    This shows that the orthogonal complement to $\mathrm{im}(s^*)$ is isometric to $\mathrm{BW}_{16}$. By Lemma \ref{lem:embeddingBWintoK33}, there is a unique isometry class of primitive embeddings of $\mathrm{BW}_{16}$ in $\Lambda_{\mathrm{K}3^{[3]}}$, with orthogonal complement isometric to $\mathrm{U}(2)^{\oplus 3}\oplus \langle -4\rangle$. Hence, the saturation of the image of 
    \begin{equation}
        s^*\colon H^2(X,\ZZ)(2)\hookrightarrow H^2(Z_X,\ZZ)
    \end{equation}
    is isometric to $\mathrm{U}(2)^{\oplus 3}\oplus \langle -4\rangle$, and it follows that $s^*$ is a primitive embedding.
\end{proof}

\begin{proof}[{Proof of Theorem \ref{thm:characterizationMRS}}]
Since birational maps of hyper-K\"ahler manifolds induce Hodge isometries between their second cohomology groups, for any MRS-double cover $Z$ there exists a primitive embedding $\mathrm{BW}_{16}\hookrightarrow \NS(Z)$, by Proposition \ref{prop:cohomologyMRS}.

Again by Proposition \ref{prop:cohomologyMRS}, the construction of Mongardi--Rapagnetta--Sacc\`a gives a primitive embedding of lattices
\begin{equation} \label{eq:pull-backMRS}
    s^*\colon \Lambda_{\mathrm{OG}6^{\mathrm{sing}}}(2)\hookrightarrow \Lambda_{\mathrm{K}3^{[3]}},
\end{equation}
where $\Lambda_{\mathrm{OG}6^{\mathrm{sing}}}$ is the lattice underlying the second cohomology of singular OG6-varieties, which is isometric to $\mathrm{U}^{\oplus 3}\oplus \langle -2\rangle$. The map $s^*$ is canonical in the following sense: let $X$ be any singular OG6-variety, and let $s_X\colon Z_X\dashrightarrow X$ be the rational double cover of Theorem \ref{thm:MRSconstruction}; then, there exist markings $\eta\colon H^2(X,\ZZ)\xrightarrow{\ \sim \ } \Lambda_{\mathrm{OG}6^{\mathrm{sing}}}$ and $\theta\colon H^2(Z_X,\ZZ)\xrightarrow{\ \sim \ } \Lambda_{\mathrm{K}3^{[3]}}$ such that $s_X^*\colon H^2(X,\ZZ)\to H^2(Z_X,\ZZ)$ equals the composition $\theta^{-1}\circ s^*\circ \eta$. 
Moreover, the orthogonal complement to $\mathrm{im}(s^*)$ is isometric to the Barnes--Wall lattice; let $i\colon \mathrm{BW}_{16}\hookrightarrow \Lambda_{\mathrm{K}3^{[3]}}$ denote this embedding. 

Consider now the subspace 
\begin{equation}
\Omega_{\mathrm{K}3^{[3]}}^{\mathrm{BW}_{16}^{\bot}} \coloneqq \{ x\in \Omega_{\mathrm{K}3^{[3]}} \ | \ x\in i(\mathrm{BW}_{16})^{\bot} \}
\end{equation}
of the period domain $\Omega_{\mathrm{K}3^{[3]}}$ of manifolds of $\mathrm{K}3^{[3]}$-type. By the above, any MRS-double cover $Z$ admits a marking $\theta$ such that the period $\mathfrak{p}(Z,\theta)$ lies in $\Omega_{\mathrm{K}3^{[3]}}^{\mathrm{BW}_{16}^{\bot}}$.
Conversely, any $x\in \Omega_{\mathrm{K}3^{[3]}}^{\mathrm{BW}_{16}^{\bot}}$ is the period of a suitably marked MRS-double cover. Indeed, $x$ belongs to the orthogonal complement to $i(\mathrm{BW}_{16})$, which is identified with the image of $\Lambda_{\mathrm{OG}6^{\mathrm{sing}}}$ under the map $s^*$ of \eqref{eq:pull-backMRS}. Let $y\in \Lambda_{\mathrm{OG}6^{\mathrm{sing}}}$ be the preimage of $x$ under this map. Then $y$ is a point in the period domain $\Omega_{\mathrm{OG}6^{\mathrm{sing}}}$, and, therefore, there exists a marked singular OG6-variety $(X,\eta)$ whose period is $y$, by the surjectivity of the period map. By construction, the period map will send the associated MRS-double cover $Z_X$ to the point $x$, for a suitable marking on $Z_X$. 

Recall that there exists a unique primitive embedding of $\mathrm{BW}_{16}$ into $\Lambda_{\mathrm{K}3^{[3]}}$ up to isometry (Lemma \ref{lem:embeddingBWintoK33}). 
If now $Z$ is any manifold of $\mathrm{K}3^{[3]}$-type such that there exists an embedding $j\colon \mathrm{BW}_{16}\hookrightarrow \NS(Z)$, we can find a marking $\theta\colon H^2(Z,\ZZ)\xrightarrow{\ \sim \ } \Lambda_{\mathrm{K}3^{[3]}}$ such that $\theta\circ j$ coincides with the primitive embedding $i\colon \mathrm{BW}_{16}\hookrightarrow \Lambda_{\mathrm{K}3^{[3]}}$ that we have fixed. It follows that $\mathfrak{p}(Z,\theta)$ lies in $\Omega_{\mathrm{K}3^{[3]}}^{\mathrm{BW}_{16}^{\bot}}$; hence, $Z$ has the same period as some marked MRS-double cover $Z_X$ of a singular OG6-variety $X$. As the birational Torelli theorem holds for manifolds of $\mathrm{K}3^{[3]}$-type, we conclude that $Z$ and $Z_X$ are bimeromorphic, and, therefore, $Z$ is a MRS-double cover.
\end{proof}

\subsection{Transcendental and N\'eron--Severi lattices}
We shall now compute the N\'eron--Severi and transcendental lattices of a MRS-double cover of a singular OG6-variety of Picard rank $1$. 
The connected components of the moduli space of projective singular OG6-varieties are in bijection with the orbits of primitive positive classes in $\Lambda_{\mathrm{OG}6^{\mathrm{sing}}}\simeq \mathrm{U}^{\oplus 3}\oplus \langle -2\rangle$ under the monodromy group $\mathrm{Mon}^2(\mathrm{OG}6^{\mathrm{sing}})$ of parallel transport operators in all locally trivial families of singular $\mathrm{OG}6$-varieties. 
This monodromy group $\mathrm{Mon}^2(\mathrm{OG}6^{\mathrm{sing}})$ was computed in \cite[Proposition 4.2]{MongardiRapagnetta}; it is equal to the group of orientation-preserving isometries:
\[
\mathrm{Mon}^2(\mathrm{OG}6^{\mathrm{sing}}) = \Oo^+(\Lambda_{\mathrm{OG}6^\mathrm{sing}}).
\]

\begin{lemma}
    Let $h\in \Lambda_{\mathrm{OG}6^{\mathrm{sing}}}$ be a positive class. Then the $\mathrm{Mon}^2(\mathrm{OG}6^{\mathrm{sing}})$-orbit of $h$ is determined by its square and divisibility.
\end{lemma}
\begin{proof}
    Since $\Lambda_{\mathrm{OG}6^{\mathrm{sing}}}\simeq \mathrm{U}^{\oplus 3}\oplus \langle -2\rangle$, any isometry of $\Lambda_{\mathrm{OG}6^{\mathrm{sing}}}$ acts trivially on the discriminant group. By Eichler's criterion Theorem \ref{thm:eichler}, classes of fixed square and divisibility form a single ${\mathrm{O}}(\Lambda_{\mathrm{OG}6^{\mathrm{sing}}})$-orbit, which in fact equals the $\SO^+(\Lambda_{\mathrm{OG6}^{\mathrm{sing}}})$-orbit of $h$. Hence, the $\mathrm{Mon}^2(\mathrm{OG}6^{\mathrm{sing}})$-orbit of $h$ coincides with its orbit under the orthogonal group.
\end{proof}

We can then compute the transcendental and N\'eron--Severi lattices of MRS-double covers of projective singular OG6-varieties of Picard rank $1$.
\begin{theorem}\label{thm:NeronSeveriMRS}
    Let $(X,h)$ be a polarized singular $\mathrm{OG}6$-variety of Picard rank $1$, with $h$ of degree $2m$ and divisibility $\beta$. Let $Z_X$ denote the MRS-double cover of $X$. Then we have the following possibilities:
    \begin{equation}
    \smallskip
    \begin{tabular}{|c|c|c|c|c|}
    \hline
    $m$ & $\beta$ & $H^2_{\mathrm{tr}}(X,\ZZ)$ & $H^2_{\mathrm{tr}}(Z_X,\ZZ)$ & $\NS(Z_X)$\\
    \hline\hline
    $m$ & $1$ & $\mathrm{U}^{\oplus 2} \oplus \begin{psmallmatrix}
        -2 & 0 \\
        0 & -2m
    \end{psmallmatrix}$ & $\mathrm{U}(2)^{\oplus 2} \oplus \begin{psmallmatrix}
        -4 & 0 \\
        0 & -4m
    \end{psmallmatrix}$ & $L_{\Km}\oplus \langle 4m\rangle$
    \\
    \hline
    $4m'-1 $ & $2$ & $\mathrm{U}^{\oplus 2}\oplus \begin{psmallmatrix}
        -2 & 1 \\
        1 & -2m'
    \end{psmallmatrix}$ & $\mathrm{U}(2)^{\oplus 2} \oplus\begin{psmallmatrix}
        -4 & 2 \\
        2 & -4m'
    \end{psmallmatrix}$ & $L_{\Km}\oplus \langle 4m \rangle$\\
    \hline
    \end{tabular}
    \smallskip
    \end{equation}
\end{theorem}
\begin{proof}
    The various possibilities for $m,\beta$ are listed in Lemma \ref{lem:orbitsU^3+<-2>}, which moreover describes $h^{\bot}$ and hence the transcendental lattice $H^2_{\mathrm{tr}}(X,\ZZ)$.
    Let $\iota\colon \mathrm{BW}_{16}\hookrightarrow H^2(Z_X,\ZZ)$ be the primitive embedding of the Barnes--Wall lattice, which is the orthogonal complement to the image of the pull-back map $s^*\colon H^2(X,\ZZ)(2)\hookrightarrow H^2(Z_X,\ZZ)$ along the MRS-double cover $Z_X\dashrightarrow X$. Then, $\NS(Z_X)$ is the saturation in $H^2(Z_X,\ZZ)$ of the sublattice generated by $s^*(h)\oplus \iota(\mathrm{BW}_{16})$; by Remark \ref{rmk:BW+h}, $\NS(Z_X)$ is isometric to $L_{\Km} \oplus \langle 4m\rangle$. 
\end{proof}

Via Hodge theory, we can associate a well-defined K3 surface with any singular OG6-variety which is projective.

\begin{definition}
    Let $X$ be a projective singular $\mathrm{OG}6$-variety. We define the associated $\mathrm{K}3$ surface as the unique $\mathrm{K}3$ surface $S_X$ such that there exists a Hodge isometry $H^2_{\mathrm{tr}}(S_X,\ZZ)\xrightarrow{\ \sim \ } H^2_{\mathrm{tr}}(X,\ZZ)(2)$.
\end{definition}

The K3 surface $S_X$ exists by the Torelli theorem, since the lattice $H^2_{\mathrm{tr}}(X,\ZZ)(2)$ admits a primitive embedding in the K3 lattice. Moreover, as $X$ is projective and $b_2(X)=7$, this embedding is unique up to isometry, and $S_X$ is determined up to isomorphism. 
\begin{remark}\label{rmk:K3-OG6}
    By the surjectivity of the period map for singular $\mathrm{OG}6$-varieties, a projective K3 surface $S$ is isomorphic to the K3 surface $S_X$ associated with some singular $\mathrm{OG}6$-variety $X$ if and only if there exists a primitive embedding of $H^2_{\mathrm{tr}}(S,\ZZ)$ into $\Lambda_{\mathrm{OG}6^{\mathrm{sing}}}(2)=\mathrm{U}(2)^{\oplus 3}\oplus \langle -4\rangle$, which is equivalent to saying that $S$ is a hyper-Kummer K3 surface, by Theorem \ref{thm:characterizeS}.
\end{remark}

\begin{proposition}
    Let $X$ be a projective singular $\mathrm{OG}6$-variety, $S_X$ the associated $\mathrm{K}3$ surface, and $Z_X$ an MRS-double cover of $X$. Then $Z_X$ is birational to a smooth and projective moduli space of stable sheaves on the $\mathrm{K}3$ surface $S_X$.
\end{proposition}
\begin{proof}
This follows immediately from \cite[Proposition 4]{Addington2016}, since the transcendental lattice of $Z_X$ is Hodge isometric to that of $S_X$ by Proposition \ref{prop:cohomologyMRS}.
\end{proof}

\subsection{Comparison of hyper-Kummer sixfolds and MRS-double covers}
The hyper-Kummer and the MRS constructions give two $5$-dimensional families of hyper-K\"ahler manifolds of $\mathrm{K}3^{[3]}$-type, of generic Picard rank $16$. 
Remarkably, these hyper-K\"ahler manifolds share the same transcendental lattices. However, the N\'eron--Severi group of a very general hyper-Kummer $\mathrm{K}3^{[3]}$-manifold is isometric to the Kummer lattice $L_{\Km}$, while that of a very general MRS-double cover is isometric to the Barnes--Wall lattice $\mathrm{BW}_{16}$, so the two families are distinct. Nevertheless, the moduli space $M_{\Km(J)}(0,H,1)$ considered by Rapagnetta (Remark \ref{rmk:explicitExampleMRS}) is both an MRS-double cover and a hyper-Kummer $\mathrm{K}3^{[3]}$-variety, by \cite{floccariKum3}. We will now compare the families of projective hyper-Kummer $\mathrm{K}3^{[3]}$-varieties and MRS-double cover of projective singular OG6-varieties.
First, we observe the following. 
\begin{corollary}
    Let $S$ be a projective $\mathrm{K}3$ surface. The following are equivalent:
    \begin{enumerate}[label=(\roman*)]
    \item $S$ is isomorphic to the $\mathrm{K}3$ surface $S_X$ associated with a projective singular OG6-variety $X$;
    \item $S$ is isomorphic to the hyper-Kummer $\mathrm{K}3$ surface $S_K$ associated with a $\mathrm{Kum}^3$-variety $K$.   
    \end{enumerate}
\end{corollary}
\begin{proof}
    In fact, both statements are equivalent to the existence of a primitive embedding of $H^2_{\mathrm{tr}}(S,\ZZ)$ into $\mathrm{U}(2)^{\oplus 3}\oplus \langle -4\rangle$, by Theorem \ref{thm:characterizeS} and Remark \ref{rmk:K3-OG6}.
\end{proof}

Therefore, in the projective case, hyper-Kummer $\mathrm{K}3^{[3]}$-type manifolds and MRS-double covers can be realized as moduli spaces on hyper-Kummer K3 surfaces. 

\begin{theorem}\label{thm:comparisonHyperkummerMRS}
    Let $K$ be a variety of $\mathrm{Kum}^3$-type, and let $Y_K$ be the associated hyper-Kummer sixfold of $\mathrm{K}3^{[3]}$-type. Assume that $K$ admits a primitive polarization $h$ with divisibility $\mathrm{div}(h)> 1$. Then, $Y_K$ is birational to an MRS-double cover.
    
    Conversely, let $X$ be a projective singular OG6-variety, and let $Z_X$ be its MRS-double cover. Then $Z_X$ is birational to a hyper-Kummer $\mathrm{K}3^{[3]}$-variety.
\end{theorem}
\begin{proof}
    If $h$ is a primitive polarization of divisibility $>1$ and square $2d$, then $\NS(Y_K)$ contains a primitive sublattice isometric to $L_{\Km}\oplus \langle d\rangle$, by Theorem \ref{thm:NeronSeveriHyperkummer}. Hence, by Remark \ref{rmk:BW+h}, there exists a primitive embedding $\mathrm{BW}_{16}\hookrightarrow \NS(Y_K)$, so $Y_K$ is birational to a MRS-double cover by Theorem \ref{thm:characterizationMRS}. 

    If instead $Z_X$ is the MRS-double cover of some projective singular OG6-variety $X$, then $\NS(Z_X)$ contains a sublattice $\mathrm{BW}_{16}\oplus \langle 4d\rangle$, with saturation of index $2$ isometric to $L_{\Km}\oplus \langle 4d\rangle$. By Lemma \ref{lem:MRSisHyperKummer}, there is a primitive embedding of $L_{\Km}$ in $(\mathrm{BW}_{16}\oplus \langle 4d\rangle)^{\mathrm{sat}}$, such that the induced embedding of $L_{\Km}$ into $H^2(Z_X,\ZZ)$ has orthogonal complement isometric to $\mathrm{U}(2)^{\oplus 3}\oplus \langle -4\rangle$. Hence, $Z_X$ is birational to some hyper-Kummer $\mathrm{K}3^{[3]}$-variety by Theorem \ref{thm:birationalCharacterization}.
\end{proof}

\begin{remark}\label{rmk:notMRS}
    If $K$ is a variety of $\mathrm{Kum}^3$-type of Picard rank $1$ equipped with a polarization $h$ of divisibility $1$, then the associated hyper-Kummer variety $Y_K$ is not a MRS-double cover. Indeed, by Theorem \ref{thm:NeronSeveriHyperkummer}, in this case $\NS(Y_K)$ is the lattice $L_{4d}$, which embeds primitively in the K3 lattice with orthogonal complement isometric to $\mathrm{U}(2)^{\oplus 2} \oplus \langle -4d\rangle$. If there was a primitive embedding of $\mathrm{BW}_{16}$ in $L_{4d}$, then we would get a primitive embedding $\mathrm{BW}_{16}\hookrightarrow \Lambda_{\mathrm{K}3}$, which is impossible.
\end{remark}

\section{Constructing $\mathrm{Kum}^3$-type varieties from hyper-Kummer sixfolds}
\label{sec:Reconstrution}

For the time being, no geometric description of a very general (i.e.~of Picard rank $1$) projective hyper-K\"ahler variety of generalized Kummer type is known. 
In dimension $6$, we propose a general construction of such examples as well as of locally complete families of $\mathrm{Kum}^3$-type varieties from finite rational covers of hyper-K\"ahler sixfolds of $\mathrm{K}3^{[3]}$-type. We summarize our strategy as follows. 

\textbf{Step 1:} describe locally complete families $\mathcal{S}\to U$ of hyper-Kummer K3 surfaces together with the following additional data: a relative primitive Mukai vector $\mathsf{v}$ of square $4$ which is algebraic on each fiber of the family such that the relative moduli space $\mathcal{M}\to U$ of (complexes of) sheaves with Mukai vector $\mathsf{v}$ stable with respect to some relative generic stability condition is a smooth family of $\mathrm{K}3^{[3]}$-type varieties; up to shrinking the base of the family and taking an \'etale base-change, a collection of $16$ relative divisors $\mathcal{E}_i$ on the relative moduli space $\mathcal{M}$ such that, for each $u\in U$, the divisors $\mathcal{E}_{i, u}$ on the $\mathrm{K}3^{[3]}$-variety $\mathcal{M}_u$ satisfy lattice-theoretic conditions (i) and (ii) of Theorem \ref{thm:biregularCharacterization} below, ensuring that each $\mathcal{M}_u$ is a hyper-Kummer sixfold. 

\textbf{Step 2:} describe, up to shrinking $U$ and a finite base-change, a relative birational transformation $\mathcal{M}\dashrightarrow \mathcal{M}'$ where $\mathcal{M}'$ is a smooth family of $\mathrm{K}3^{[3]}$-varieties which admits a birational morphism $q\colon \mathcal{M}\to \overline{\mathcal{M}}'$ over $U$ contracting the $16$ divisors $\mathcal{E}_i'$ obtained from the $\mathcal{E}_i$ (this is always possible for each fiber by Remark \ref{rmk:contractionAfterFlops}).

\textbf{Step 3:} take a finite cover $f\colon \widetilde{\mathcal{M}}'\to \mathcal{M}'$ relatively over $U$, where for any $u \in U$,  $f_u$ is the $(\ZZ/2\ZZ)^5$-cover of $\mathcal{M}_u'$ branched over each divisor $\mathcal{E}_{i,u}'$ with index $2$, and then take the relative Stein factorization of $q\circ f \colon \widetilde{\mathcal{M}}'\to \overline{\mathcal{M}}'$ to obtain a family $\mathcal{Z}\to U$ of primitive symplectic varieties. By Corollary \ref{cor:reconstructionProcedure}, a simultaneous $\QQ$-factorial terminalization of $\mathcal{Z}\to U$ will give a locally complete family of $\mathrm{Kum}^3$-varieties.

In Section \ref{sec:ExamplesHyperKummerK3}, we will describe  two families as in Step 1. However, the birational transformation needed in Step 2 seems hard to describe, due to the high Picard rank of the hyper-Kummer varieties of $\mathrm{K}3^{[3]}$-type.

\subsection{$\mathrm{Kum}^3$-type varieties as rational covers}
We now explain and prove the results mentioned in the above strategy.


\begin{theorem}\label{thm:biregularCharacterization}
Let $Y$ be a hyper-K\"ahler manifold of $\mathrm{K}3^{[3]}$-type. Assume that there exist $16$ prime divisors $E_i$, $i=1,\dots, 16$, on $Y$ such that: 
\begin{enumerate}[label=(\roman*)]
\item the classes $[E_i]\in H^2(Y,\ZZ)$ are pairwise orthogonal $(-2)$-classes;
\item the saturation of the sublattice $\langle [E_i]\rangle_{i=1,\dots,16}$ of $H^2(Y,\ZZ)$ is the Kummer lattice, and its orthogonal complement is isometric to $\mathrm{U}(2)^{\oplus 3}\oplus \langle -4\rangle$;
\item there exists a birational morphism $q\colon Y\to \overline{Y}$ whose exceptional locus is the union of the $16$ divisors $E_i$. 
\end{enumerate}
Then there exists a manifold $K$ of $\mathrm{Kum}^3$-type such that $\overline{Y}$ is isomorphic to $K/G$. In particular, $Y$ is birational to the hyper-Kummer $\mathrm{K}3^{[3]}$-manifold $Y_K$. 
\end{theorem}

Notice that any hyper-Kummer $\mathrm{K}3^{[3]}$-manifold contains $16$ divisors $E_i$ with the above properties, by Proposition \ref{prop:cohomologyY_K}.

\begin{corollary}\label{cor:reconstructionProcedure}
    Let $Y$ be as in Theorem \ref{thm:biregularCharacterization}; assume further that $Y$ is projective. Then there exists a finite morphism $f\colon \widetilde{Y}\to Y$ which is Galois with Galois group $(\ZZ/2\ZZ)^5$, ramified with index $2$ along each divisor  $E_i$. Moreover, taking the Stein factorization of the composition of $f$ with the birational contraction $q\colon Y\to \overline{Y}$ yields a commutative diagram 
    \begin{equation}
    \begin{tikzcd}
        \tilde{Y} \arrow{dr}{\phi} \arrow{r}{f} \arrow[swap]{d}{q'} & Y \arrow{d}{q} \\
        Z \arrow[swap]{r}{f'} & \overline{Y}
    \end{tikzcd}
    \end{equation}
    in which $q'$ is the birational contraction of the $16$ divisors $f^{-1}(E_i)$, $f'$ is a finite morphism which is \'etale over the smooth locus of $\overline{Y}$, and $Z$ is a primitive symplectic variety whose $\QQ$-factorial terminalization is a smooth hyper-K\"ahler variety of $\mathrm{Kum}^3$-type.    
\end{corollary}

\begin{proof}
    By assumption, the classes of the divisors $[E_i]$ generate a sublattice of $H^2(Y,\ZZ)$ with saturation the Kummer lattice, and so $\langle [E_i]\rangle^{\mathrm{sat}}/\langle [E_i]\rangle \cong (\ZZ/2\ZZ)^5$. By the theory of abelian covers developed in \cite{Pardini-AbelianCover} and in particular \cite[Theorem 2.1]{Alexeev-Pardini-AbelianCovers}, it follows that there exists a finite morphism $f\colon \tilde{Y}\to Y$ which is an abelian cover with Galois group $(\ZZ/2\ZZ)^5$, ramified with index $2$ over each $E_i$. By construction, $f'$ is finite and $q'$ has connected fibers.
    Since $f$ is \'etale over $Y\setminus \bigcup_i E_i$, it follows that $q'$ is injective over $\widetilde{Y}\setminus \bigcup_i f^{-1}(E_i)$. On the other hand, $q'(f^{-1}(E_i))$ must be a fourfold in $Z$. Therefore $q'$ is birational and contracts exactly the $16$ divisors $f^{-1}(E_i)$. Moreover, the image of $Y\setminus \bigcup_i E_i$ in $\overline{Y}$ has codimension $2$ and equals the regular locus of $\overline{Y}$. Therefore $f'$ is a quasi-\'etale morphism, which is \'etale over the regular locus $\overline{Y}^{\mathrm{reg}}$ of $\overline{Y}$. 

    Applying Theorem \ref{thm:biregularCharacterization}, we deduce that there exists a finite quasi-\'etale morphism $K\to \overline{Y}$, \'etale over the regular locus of $\overline{Y}$, for some $\mathrm{Kum}^3$-manifold $K$. If $U\subset K$ is the preimage of the regular locus of $\overline{Y}$, then $U$ is simply connected, since it is the complement of a codimension $2$ closed subset in the smooth simply connected variety $K$. Hence, $U$ is the universal cover of $\overline{Y}^{\mathrm{reg}}$, which thus has fundamental group $(\ZZ/2\ZZ)^5$. 
    But then $(f')^{-1}(\overline{Y}^{\mathrm{reg}})\subset Z$ must also be the universal cover of $\overline{Y}^{\mathrm{reg}}$, so $Z$ and $K$ are birational.
    We next observe that $\overline{Y}$ is a symplectic variety in the sense of \cite{beauvillesymplectic}, and in fact a primitive symplectic variety (\cite[Definition 3.1]{bakkerLehn}) with $\QQ$-factorial singularities (\cite[Example 3.4]{BakkerLehn2021GlobalTorelliSingularSymplectic}).
    Since $Z$ is a quasi-\'etale cover of $\overline{Y}$, it is also a symplectic variety, and in fact a primitive symplectic variety since $Z$ is birational to the smooth hyper-K\"ahler variety $K$. 
    By \cite{BCHM} (see also \cite{bakkerLehn}), there exists a $\QQ$-factorial terminalization $\widetilde{Z}\to Z$. Since $K$ and $\widetilde{Z}$ are birational $\QQ$-factorial and terminal primitive symplectic varieties, they must be isomorphic in codimension $1$. Applying \cite[Theorem 6.16]{bakkerLehn}, $\widetilde{Z}$ and $K$ are deformation equivalent, i.e., $\widetilde{Z}$ is a smooth manifold of $\mathrm{Kum}^3$-type, and $\widetilde{Z}\to Z$ is a symplectic resolution. 
    %
\end{proof}
\begin{remark}
    A local analysis of $Z\to \overline{Y}$ should show that $Z$ is already $\QQ$-factorial and terminal, and hence $Z$ is a smooth hyper-K\"ahler manifold of $\mathrm{Kum}^3$-type by Corollary \ref{cor:reconstructionProcedure}.
\end{remark}

\begin{remark}\label{rmk:contractionAfterFlops}
    Let $Y$ be a projective variety of $\mathrm{K}3^{[3]}$-type satisfying assumptions (i) and (ii) in Theorem \ref{thm:biregularCharacterization}. We can then conclude that there exists a finite cover $\widetilde{Y}\to Y$ with Galois group $(\ZZ/2\ZZ)^5$ ramified with index $2$ over each $E_i$, such that $\widetilde{Y}$ is birational to a $\mathrm{Kum}^3$-variety $K$. Moreover, $Y$ is birational to the hyper-Kummer sixfold $Y_K$ associated with $K$.
    
    This follows from Corollary \ref{cor:reconstructionProcedure}, since there exists a $\mathrm{K}3^{[3]}$-variety $Y'$ birational to $Y$ and a birational morphism $Y'\to \overline{Y}'$ which contracts the strict transforms of the divisors $E_i$. 
    In other words, we can achieve (iii) after replacing $Y$ with another smooth birational model.
    Indeed, each $E_i$ is a prime exceptional divisor (\cite{markmanPrime}) and can be contracted after some flops by a result of Druel \cite[Proposition 1.4]{druel}. We can moreover find a smooth birational model of $Y$ on which the strict transforms of the $E_i$ can be simultaneously contracted. This follows from the fact that the intersection matrix of the sublattice of $H^2(Y,\ZZ)$ spanned by the classes $[E_i]$ is negative definite. By \cite[Lemma 4.15]{denisi} and its proof, there exists a big divisor $L$ in the birational K\"ahler cone of $Y$, such that the set of  prime divisors $E$ such that $[E]\in [L]^{\bot}\subset H^2(Y,\ZZ)$ is the set $\{E_1,\dots, E_{16}\}$. Moreover (by \cite[Theorem 1.2]{matsushita2014}) $L$ becomes a big and nef divisor $L'$ on some smooth birational model $Y'$ of $Y$; for $m$ large enough, the linear system $|mL'|$ induces a birational morphism $Y'\to \overline{Y}'$ which contracts the strict transforms of the divisors $E_i$.
\end{remark}



To prove Theorem \ref{thm:biregularCharacterization}, we will need to recall some results on the birational geometry of hyper-K\"ahler manifolds. 
Let $X$ be a hyper-K\"ahler manifold. We denote by $\mathcal{K}(X)$ the K\"ahler cone of $X$, which is open in $H^{1,1}(X)\cap H^2(X,\RR)$. 
The positive cone $\mathrm{Pos}(X)$ is the connected component of 
\begin{equation}
\{\alpha\in H^{1,1}(X)\cap H^2(X,\RR), \ q_X(\alpha,\alpha)>0 \}
\end{equation}
containing the K\"ahler cone. 
To study the birational geometry of hyper-K\"ahler manifolds, one introduces the birational K\"ahler cone
\begin{equation} 
\mathcal{BK}(X)\coloneqq \bigcup_{f\colon X\dashrightarrow X' } f^*(\mathcal{K}(X')), 
\end{equation}
where the union runs over all bimeromorphic maps to a hyper-K\"ahler manifold $X'$. The closure of $\mathcal{BK}(X)$ intersected with $\NS(X)\otimes \RR$ coincides with the closure of the cone of \textit{movable} divisors, that is, the cone generated by divisors whose stable base locus has no fixed part (\cite[Theorem 7]{hassett-Tschinkel2009}).
A prime exceptional divisor on a hyper-K\"ahler manifold $X$ is a prime divisor with negative Beauville--Bogomolov form. By \cite[Theorem 1.5]{markman2011survey}, the closure of the birational K\"ahler cone also coincides with the closure of the fundamental exceptional chamber 
\begin{equation} 
\{\alpha\in \mathrm{Pos}(X) \ | \ q_X(\alpha, [E])> 0 \text{ for any prime exceptional divisor } E\subset X\}.
\end{equation} 
The fundamental exceptional chamber is equivalently characterized in terms of stably prime exceptional divisors, which, if $X$ is of $\mathrm{K}3^{[n]}$-type, admit a numerical characterization, see \cite[Proposition 1.8]{markman2011survey} and \cite[Theorem 1.11]{markmanPrime}.


For a hyper-K\"ahler manifold $X$ of $\mathrm{K}3^{[3]}$-type, the K\"ahler cone can be described explicitly in terms of the second cohomology and the Beauville--Bogomolov pairing $q_X$, by \cite{BM14a,BHT, mongardino}. 
For a class $D\in H^2(X,\ZZ)$, we let $\mathrm{div}(D)$ be its divisibility, i.e., the largest integer $d$ such that $q(D, \cdot)/d$ belongs to $H^2(X,\ZZ)^{\vee}$.

\begin{definition}[Classification of wall divisors for $\mathrm{K}3^{[3]}$-type manifolds]
\label{def:numericalWallDivisors}
    Let $X$ be a manifold of $\mathrm{K}3^{[3]}$-type.
	The set $\mathsf{NW}_X$ of numerical wall divisor classes is the set of classes $D\in \NS(X)$ with the following numerical invariants:
	\begin{equation}\label{eq:tableNumericalWalls}
	\begin{tabular}{|c|c|}
		\hline
		$ q_X(D,D)$ & $\mathrm{div}(D)$\\
		\hline
		 $-12$ & $2$ \\
		 \hline
		 $-36$ & $4$ \\
		 \hline
		 $-2$ & $1$ \\
		 \hline
		 $-4$ & $4$ \\ 
		 \hline
		 $-4$ & $2$\\
		 \hline
	\end{tabular}
	\end{equation}
	The set $W_X$ of wall divisors consists of elements $D$ in $\mathsf{NW}_X$ such that there exists a K\"ahler class $\omega$ on $X$ with $q_X(D, \omega)>0$.
\end{definition}
The notion of wall divisors is equivalent to that of MBM classes \cite{AmerikVerbitsky2015RationalCurves}. We have the following result.
\begin{theorem}[{\cite[Theorem 2.14]{mongardino}}]\label{thm:wall-divisors}
	Let $X$ be a manifold of $\mathrm{K}3^{[3]}$-type. Then 
	\begin{equation}
	\mathcal{K}(X) = \{ \omega\in \mathrm{Pos}(X) \ | \ q_X(\omega, D)>0 \text{ for every } D\in W_X \}.
	\end{equation}
\end{theorem}

The stably prime exceptional divisors are the numerical divisors with square and divisibility as in the last three entries of table \eqref{eq:tableNumericalWalls}, by \cite[Theorem 1.11]{markmanPrime}. 
Thus, the K\"ahler cone is an open chamber in the complement of the wall divisors in the closure of the birational K\"ahler cone. 

We apply this result to the very general hyper-Kummer $\mathrm{K}3^{[3]}$-manifold. 
\begin{proposition}\label{prop:veryGeneralCase}
    Let $K$ be a manifold of $\mathrm{Kum}^3$-type of Picard rank $0$. If $Y'$ is a hyper-K\"ahler manifold bimeromorphic to the hyper-Kummer manifold $Y_K$, then $Y'$ and $Y_K$ are isomorphic.
\end{proposition}
\begin{proof}
    The proposition is equivalent to the equality $\mathcal{K}(Y_K)=\mathcal{BK}(Y_K)$ between the K\"ahler cone and the birational K\"ahler cone. 
    For $K$ as in the statement, the N\'eron--Severi group $\NS(Y_K)$ is isometric to the Kummer lattice $L_{\Km}$, by Proposition \ref{prop:cohomologyY_K}, and the transcendental lattice of $Y_K$ is isometric to $\mathrm{U}(2)^{\oplus 3}\oplus \langle -4\rangle$.  
    Moreover, any class in $H^2(Y_K,\ZZ)$ contained in the sublattice $L_{\Km}$ has divisibility $1$ (indeed, this embedding of $L_{\Km}$ into $\Lambda_{\mathrm{K}3^{[3]}}$ factors through a copy of the K3 lattice).
    
    Therefore, the set of numerical wall divisors consists of the $32$ classes with square~$-2$ in $L_{\Km}$, which are given by $\pm[E_i]$, $i=1,\dots,16$, where the $E_i$ are the components of the exceptional divisor of $Y_K\to K/G$.
    By Theorem \ref{thm:wall-divisors}, the K\"ahler cone of $Y_K$ is thus given by 
    \begin{equation}
        \mathcal{K}(Y_K) =\{\omega\in \mathrm{Pos}(Y_K)\ | \ q_{Y_K}(\omega, [E_i])>0, \ \text{for} \ i=1,\dots,16\}.
    \end{equation}
    On the other hand, each $E_i$ is a prime exceptional divisor. Hence, for any $\omega\in \overline{\mathcal{BK}(Y_K)}$ we have $q_{Y_K}(\omega, [E_i])\geq 0$, so that $\overline{\mathcal{BK}(Y_K)}\subset \overline{\mathcal{K}({Y_K})}$ and therefore ${\mathcal{BK}(Y_K)}=\mathcal{K}(Y_K)$.
\end{proof}

\begin{proof}[{Proof of Theorem \ref{thm:biregularCharacterization}}]
Let $Y$ be as in the statement of the theorem. By assumption, there exists a morphism $\pi_0\colon Y\to \overline{Y}$ contracting the $16$ divisors $E_i$.
Then $\overline{Y}$ is a singular symplectic variety (\cite{beauvillesymplectic}). Let $\mathrm{Def}(Y)$ be the local deformation space of $Y$, and let $\mathrm{Def}^{\mathrm{lt}}(\overline{Y})$ denote the local deformation space parametrizing locally trivial deformations of $\overline{Y}$. We consider the closed subspace $\mathrm{Def}(Y, \langle [E_i] \rangle_{i=1}^{16})$ of deformations along which the cohomology classes of all the divisors $E_i$ remain Hodge. 
We denote by $\mathcal{Y}$ and $\overline{\mathcal{Y}}$ the universal families over $\mathrm{Def}(Y)$ and $\mathrm{Def}^{\mathrm{lt}}(\overline{Y})$, respectively. 
By \cite[Theorem 1.1]{bakkerLehn} (see also \cite[Proposition 2.3]{LehnPacienza}), there exists a commutative diagram
\begin{equation}
\begin{tikzcd}
    \mathcal{Y}_{|_{\mathrm{Def}(Y, \langle[E_i]\rangle)}} \arrow{d} \arrow{r}{\pi} & \overline{\mathcal{Y}} \arrow{d} 
    \\
    \mathrm{Def}(Y, \langle[E_i]\rangle) \arrow{r}{\pi_*} & \mathrm{Def}^{\mathrm{lt}}(\overline{Y})
\end{tikzcd}
\end{equation}
in which $\pi_*$ is an isomorphism and $\pi$ is a morphism extending $\pi_0\colon Y\to \overline{Y}$, such that $\pi_t\colon \mathcal{Y}_t \to \overline{\mathcal{Y}}_t$ is the birational contraction of $16$ prime divisors with cohomology classes the parallel transport of the classes $[E_i]$, for any $t\in \mathrm{Def}(Y,\langle [E_i]\rangle)$. 

Consider now a very general point $t\in \mathrm{Def}(Y,\langle [E_i]\rangle)$. Then $\mathcal{Y}_t$ is a manifold of $\mathrm{K}3^{[3]}$-type whose N\'eron--Severi group is generated by the parallel transport of the classes $[E_i]$, and, hence, $\NS(\mathcal{Y}_t)$ is isometric to the Kummer lattice $L_{\Km}$, while $H^2_{\mathrm{tr}}(\mathcal{Y}_t,\ZZ)$ is isometric to $\mathrm{U}(2)^{\oplus 3} \oplus \langle -4\rangle$. 
By Theorem \ref{thm:birationalCharacterization} and Proposition \ref{prop:veryGeneralCase}, it follows that $\mathcal{Y}_t$ is isomorphic to the hyper-Kummer $\mathrm{K}3^{[3]}$-manifold associated with some manifold $K_t$ of $\mathrm{Kum}^3$-type of Picard rank $0$. It follows that the contraction $\overline{\mathcal{Y}_t}$ must be isomorphic to $K_t/G$. 
But $\mathrm{Def}^{\mathrm{lt}}(K/G)$ is naturally isomorphic to $\mathrm{Def}(K)$. Therefore, $\overline{\mathcal{Y}}_{t'}$ is isomorphic to $K_{t'}/G$ for some manifold $K_{t'}$ of $\mathrm{Kum}^3$-type, for any $t'\in \mathrm{Def}^{\mathrm{lt}}(\overline{Y})$. 
\end{proof}

\subsection{The hyper-Kummer manifold of $K^3(T)$}
As an example, we describe the hyper-Kummer manifold associated with the generalized Kummer variety of a general complex torus. 
Let $T$ be a very general $2$-dimensional complex torus, of Picard rank $0$, and consider the (non-projective) Kummer surface $\Km(T)$. 
Then $\NS(\Km(T))$ is a Kummer lattice, which is the saturation of the sublattice spanned by the $16$ exceptional curves $R_{\tau}$, with $\tau\in T[2]$. 
Consider the associated generalized Kummer sixfold $K^3(T)$. The associated hyper-Kummer sixfold $Y_{K^3(T)}$ is birational to $\Km(T)^{[3]}$, by Proposition \ref{prop:Y_K^3(A)}; we will now describe the birational transformation $\Km(T)^{[3]}\dashrightarrow Y_{K^3(T)}$, as a composition of explicit flops.

We have $\NS(\Km(T)^{[3]})\cong \NS(\Km(T))\oplus \ZZ\cdot \delta$, where $\delta$ is half the class of the Hilbert--Chow divisor, of square $-4$. As this lattice is negative definite, it is fairly easy to understand its K\"ahler cone via Theorem \ref{thm:wall-divisors}. By \cite[Theorem 9.17]{markman2011survey}, a stably prime exceptional divisor $E$ on a manifold of $\mathrm{K}3^{[3]}$-type satisfies one of the following: 
\begin{itemize}
    \item[$\bullet$] $(E,E)=-2$ and $\mathrm{div}(E) = 1$;
    \item[$\bullet$] $(E,E)=-4$ and $\mathrm{div}(E) = 2$;
    \item[$\bullet$] $(E,E) = -4$ and $\mathrm{div}(E) = 4$.
\end{itemize}
The prime exceptional divisors are exactly $\delta$ and the $16$ divisors $r_i$ parametrizing length-3 subschemes on $\Km(T)$ with support intersecting the exceptional curve $R_i\subset \Km(T)$.
This determines the closure of the birational K\"ahler cone of $\Km(T)$:
\begin{equation}
\overline{\mathcal{BK}(\Km(T)^{[3]})} = \{ h\in \mathrm{Pos}(\Km(T)^{[3]})\ | \ (h, \delta) \geq 0, \ (h,r_{\tau})\geq 0 \ \text{for any}\ \tau\in T[2]\}.
\end{equation}

The further subdivision into chambers is given by the wall divisors, which are determined numerically as in Definition \ref{def:numericalWallDivisors}. The only wall divisors are:
\begin{itemize}
    \item[$\bullet$] for each $\tau\in T[2]$, a divisor $l_\tau \coloneqq 2r_\tau - \delta$, of square $-12$ and divisibility $2$;
    \item[$\bullet$] for each $\tau\in T[2]$, a divisor $p_{\tau}\coloneqq 4r_\tau - \delta$, of square $-36$ and divisibility $4$. 
\end{itemize}
According to \cite[Corollary 4.7]{amerik2022MBM}, each $l_{\tau}$ is dual to an extremal curve in a Lagrangian $\PP^3$, denoted $L_{\tau}$; in our case $L_{\tau}$ is the $\PP^3$ parametrizing length-3 subschemes of $\Km(T)$ fully supported on the rational curve $R_{\tau}$. Each $p_{\tau}$ is instead dual to an extremal curve which moves to cover a $4$-fold $P_{\tau}$ birational to a $\PP^2$-bundle over a K3 surface; in our case, $P_{\tau}$ is the subvariety parametrizing subschemes with at least $2$ points in their support lying on $R_{\tau}$. 

This determines a wall and chamber decomposition of the birational K\"ahler cone of $\Km(T)^{[3]}$. The hyper-Kummer sixfold $Y_{K^3(T)}$ is the birational model which is ``farthest'' from $\Km(T)^{[3]}$: to reach it, we need to cross all the walls. 
\begin{proposition}
    The hyper-Kummer sixfold $Y_{K^3(T)}$ is obtained from $\Km(T)^{[3]}$ via the flops along the $16$ Lagrangian $L_{\tau}\cong \PP^3$ and the $16$ fourfolds $P_{\tau}$ (which are birational to $\PP^2$-bundles over $\Km(T)$).
\end{proposition}
\begin{proof}
Since we have the Hilbert--Chow contraction $\Km(T)^{[3]}\to \Km(T)^{(3)}$, the K\"ahler cone of $\Km(T)^{[3]}$ must be given by 
\begin{equation}
    \mathcal{K}(\Km(T)^{[3]}) = \{ h\in \mathcal{BK}(\Km(T)^{[3]})\ | \ (h,p_{\tau}) > (h, l_{\tau})>0, \ \text{for any} \ \tau\in T[2]\}. 
\end{equation}
On the other hand, the hyper-Kummer sixfold $Y_{K^3(T)}$ comes with the contraction $Y_{K^3(T)}\to K^3(T)/G$, whose exceptional locus must be the union of the divisors $r_i$, since there are no other divisors on $\Km(T)^{[3]}$ with square $-2$. Hence, its K\"ahler cone must be given by 
\begin{equation}
    \mathcal{K}(Y_{K^3(T)}) = \{ h\in \mathcal{BK}(\Km(T)^{[3]})\ | \ (h,l_{\tau}) < (h, p_{\tau}) < 0, \ \text{for any} \ \tau\in T[2]\}. 
\end{equation}
Hence, to reach the K\"ahler cone of $Y_{K^3(T)}$ from that of $\Km(T)^{[3]}$ we need to cross the $16$ walls $l_{\tau}^{\bot}$ as well as the $16$ walls $p_{\tau}^{\bot}$.
\end{proof}

    Hence, $\Km(T)^{[3]}$ and $Y_{K^3(T)}$ are non-isomorphic crepant resolutions of $(T/\pm 1)^{(3)}$. When $T$ specializes to an abelian surface $A$, the computation of the birational K\"ahler cone becomes much harder; in general, more flops or other birational transformations will be needed in order to contract the $16$ divisors $r_{\tau}$ on $\Km(A)^{[3]}$.


\section{Examples of hyper-Kummer K3 surfaces}
\label{sec:ExamplesHyperKummerK3}

In this section we describe some locally complete families of hyper-Kummer K3 surfaces and hyper-Kummer sixfolds. By locally complete we mean that the corresponding period map has open image in a $4$-dimensional period domain parametrizing projective hyper-Kummer K3 surfaces or hyper-Kummer sixfolds, respectively. The families which we present are families of K3 surfaces with a symplectic action of $(\ZZ/2\ZZ)^4$, studied by Garbagnati and Sarti \cite{GarbagnatiSarti}. 

\subsection{Heisenberg invariant quartic surfaces}
\label{subsec:heisenberg-invariant}
Consider the action of the group $(\ZZ/2\ZZ)^4$ on $\mathbb{P}^3$ generated by the following four involutions sending a point $[x:y:z:w]\in \mathbb{P}^3$ to
\begin{equation}
    [z:w:x:y], \ \ [y:x:w:z], \ \ [x:y:-z:-w], \ \ [x:-y:z:-w],
\end{equation}
respectively. 
A \textit{Heisenberg invariant quartic} is a quartic surface $S\subset \mathbb{P}^3$ which is stabilized by the $(\ZZ/2\ZZ)^4$-action. Thus any such surface is equipped with an action of $(\ZZ/2\ZZ)^4$, which in fact acts via symplectic automorphisms. These K3 surfaces are related to beautiful geometry, studied in various papers \cite{barth-nieto, GarbagnatiSarti, eklund}.
It is known that Heisenberg invariant quartics form a family of dimension~$4$, parametrized by a linear space $\mathbb{P}^4$. 
The very general such quartic is a smooth K3 surface of Picard rank $16$, with transcendental lattice isometric to 
\begin{equation}\label{eq:transcendetnalHeisenbergQuartics}
    H^2_{\mathrm{tr}}(S,\ZZ)=\mathrm{U}(2)^{\oplus 2}\oplus \langle -4\rangle \oplus \langle -8 \rangle.
\end{equation}

This family contains the $3$-dimensional family of quartics with $16$ nodes, but not the family of (smooth) Kummer surfaces of principally polarized abelian surfaces. 
    Instead, the family of smooth Heisenberg invariant quartics contains the $3$-dimensional family of Kummer surfaces of $(1,3)$-polarized abelian surfaces. If $A$ is a $(1,3)$-polarized abelian surface, we denote by $C_{\tau},\tau\in A[2]$, the $16$ exceptional curves on $\Km(A)$ and by $L$ the nef class of square $12$ coming from the polarization on $A$. In this notation, 
    \begin{equation}
        H\coloneqq L -\frac{1}{2}\sum_{\tau\in A[2]} C_{\tau}
    \end{equation}
    is very ample and $|H|$ embeds $\Km(A)$ in $\PP^3$ as a Heisenberg invariant quartic. The anti-invariant lattice with respect to the $(\ZZ/2\ZZ)^4$-action is the saturation of the sublattice generated by the classes $C_{\tau}-C_{\tau'}$, for $\tau\neq \tau'\in A[2]$; it is isometric to the negative definite lattice $N_V$ of rank $15$ and discriminant $(\ZZ/2\ZZ)^6\times \ZZ/8\ZZ$ of Definition \ref{def:latticeN_V}.
    The N\'eron--Severi group of a very general Heisenberg invariant quartic $S$ is then equal to the saturation of $N_V\oplus \langle H\rangle$ in $\NS(\Km(A))$, as the orthogonal complement $(\tfrac{1}{2}\sum_{\tau\in A[2]} C_{\tau})^{\bot}$ in $\NS(\Km(A))$. It is an overlattice of $N_V\oplus \langle H\rangle$ of index $2$.


Let $U\subset \mathbb{P}^4$ be the open subset parametrizing smooth Heisenberg invariant quartics; we thus have a $4$-dimensional family $\mathcal{S}\to U$ of K3 surfaces.
\begin{proposition}
    For any $u\in U$, the $\mathrm{K}3$ surface $\mathcal{S}_u$ is hyper-Kummer.
\end{proposition}
\begin{proof}
    According to Theorem \ref{thm:characterizeS}, a projective K3 surface $S$ is hyper-Kummer if and only if its transcendental lattice admits a primitive embedding into $\mathrm{U}(2)^{\oplus 3}\oplus \langle -4\rangle$. By \eqref{eq:transcendetnalHeisenbergQuartics}, this condition holds for any smooth Heisenberg invariant quartic, which is thus hyper-Kummer.
\end{proof}

As a consequence, there exist moduli spaces of sheaves on $\mathcal{S}_u$ which are hyper-Kummer sixfolds. Using our lattice characterization, we can find one such up to birationality. 
\begin{proposition}\label{prop:heisenberg-invariant-isMRS}
    Let $S\subset \PP^3$ be a smooth Heisenberg invariant quartic; denote by $H$ a hyperplane divisor class. The moduli space $M_S(0,H,0)$ is birational to a hyper-Kummer sixfold, as well as to an MRS-double cover.
\end{proposition}
\begin{proof}
    According to Corollary \ref{thm:comparisonHyperkummerMRS}, it is sufficient to show that $M_S(0,H,0)$ is birational to some MRS-double cover. By Theorem \ref{thm:characterizationMRS}, this is the case if and only if there exists a primitive embedding of the Barnes--Wall lattice $\mathrm{BW}_{16}$ in $\NS(M_S(0,H,0))$.
    It is enough to check this for a very general Heisenberg invariant quartic $S$.
    By the usual description of the second cohomology of moduli spaces of sheaves on K3 surfaces via the Mukai homomorphism, we have 
    \begin{equation}
        \NS(M_S(0,H,0)) = (0,H,0)^{\bot}\subset \NS(S)\oplus \mathrm{U},
    \end{equation}
    which equals $H^{\bot}\oplus \mathrm{U}$, where $H^{\bot}$ is the orthogonal to $H$ in $\NS(S)$. Since we assume that $S$ is very general, we have $H^{\bot}=N_V$, and, hence, $\NS(M_S(0,H,0)) = N_V\oplus \mathrm{U}$. This is a rank $17$ lattice of signature $(1,16)$, uniquely determined by its signature and discriminant form. It admits a primitive embedding into the Mukai lattice $\widetilde{\Lambda}_{\mathrm{K}3}$ which is unique up to isometry and with orthogonal complement isometric to $\mathrm{U}(2)^{\oplus 3} \oplus \langle -8\rangle$ (see Lemma \ref{lem:embeddingN_V}). We conclude that $\NS(M_S(0,H,0))$ is isometric to $L_{\Km}\oplus \langle 8\rangle$; by Lemma \ref{rmk:BW+h}, there exists a primitive embedding $\mathrm{BW}_{16}\hookrightarrow \NS(M_S(0,H,0))$, as desired.
\end{proof}

We next describe $16$ divisors as in Theorem \ref{thm:biregularCharacterization}. For any $d$, the moduli space $M_{S}(0,H,d-2)$ (for a generic stability condition) parametrizes sheaves of pure dimension $1$ on $S$ supported on the curves in the linear system $|H|$ with Euler characteristic $d-2$. It is a Beauville--Mukai system of curves of genus $3$, the compactification of the relative Picard variety $\mathrm{Pic}^{d}(\mathcal{C}/|H|)$ for the curves in the linear system $|H|$. 
There is a Lagrangian fibration $\pi\colon M_{S}(0,H,d-2) \to |H|$, sending a torsion sheaf to its support. 
Each $\mathrm{Pic}^{d}(\mathcal{C}/|H|)$ is a torsor under the group scheme $\mathrm{Pic}^0(\mathcal{C}/|H|)$ of relative Jacobians.

As already mentioned, the family of Heisenberg invariant quartics contains the $3$-dimensional family of Kummer surfaces on $(1,3)$-polarized abelian surfaces. Recall that 
$\NS(\Km(A)) = \langle L, C_{\tau}\rangle^{\mathrm{sat}}$, where $C_{\alpha}$ are the 16 embedded smooth rational $(-2)$-curves and $L$ is the nef class of square $12$ coming from the $(1,3)$ polarization on $A$, orthogonal to each $C_{\alpha}$, and $\Km(A)$ is embedded in $\mathbb{P}^3$ via the polarization 
\begin{equation}
    H\coloneqq L-\frac{1}{2} \sum_{\alpha\in A[2]} C_{\alpha}.
\end{equation} 
 In fact, this is the subfamily of Heisenberg invariant quartic surfaces containing a line: the embedding $\Km(A)\hookrightarrow \mathbb{P}^3$ sends the exceptional curves $C_{\alpha}$ to disjoint lines. 
The N\'eron--Severi group of $\Km(A)$ is the saturation of the sublattice spanned by the degree-$4$ polarization $H$ and the 16 exceptional curves $C_{\alpha}$. 
However, a very general Heisenberg invariant quartic $S$ does not contain any line, and neither the cohomology classes of $L$ or any $C_{\alpha}$ remain Hodge on this deformation. The classes which remain Hodge on the entire family are instead $H$ and the anti-invariant sublattice $N_V$, generated (up to finite index) by the differences $[C_{\alpha}]-[C_{\beta}]\in H^2(\Km(A),\ZZ)$. Of these 240 differences only 15 are linearly independent. For instance, the $15$ classes $T_{\alpha}\coloneqq [C_0]-[C_{\alpha}]$ for $\alpha\neq 0$ generate a sublattice of finite index in $N_V$. Each $T_{\alpha}$ is orthogonal to $H$, has square $-4$, and $(T_{\alpha}, T_{\beta})=-2$ for $\alpha\neq \beta$.

Hence, for a very general Heisenberg invariant K3 surface $S$, we have 15 Hodge classes $T_{i}$ orthogonal to the polarization $H$. Then $\mathcal{O}_S(T_{i})$ restricts to a line bundle of degree $0$ on any smooth curve $C$ in $|H|$. Each $T_i$ therefore corresponds to a rational section of $M_S(0,H,-2)\to |H|$. Via tensor product, each $\mathcal{O}_S(T_i)$ induces a birational self-map of $M_{S}(0,H,d-2)$ for any $d$, preserving the Lagrangian fibration. 

For $M_S(0,H,0)$, i.e., $d=2$, there is a relative theta divisor $\Theta$ defined as the closure of the image of the natural map $\mathrm{Sym}^2(\mathcal{C}/|H|) \to \mathrm{Pic}^2(\mathcal{C}/|H|)$, which, over any $C\in |H|$, sends $P+Q\in C^{(2)}$ to $\mathcal{O}_C(P+Q)$. We have 
\begin{equation}
    \Theta = \{F\in M_{S}(0,H,0)\ | \ H^0(S,F)\neq 0\}.
\end{equation}
For $i=1,\dots, 15$, translating by $T_i$, we also get divisors 
\begin{equation} 
\Theta_i \coloneqq \{F\in M_S(0,H,0)\ | \ H^0(S, F\otimes \mathcal{O}_S(-T_i))\neq 0\},
\end{equation}
obtained from $\Theta$ using the birational transformations given by the $T_i$. We put $\Theta_0 \coloneqq \Theta$, so that we have $16$ prime divisors $\Theta_0,\dots, \Theta_{15}$.

\begin{proposition}
\label{prop:16DivisorsOnModuliOnHeisenburg}
    The $16$ divisors $\Theta_i$, $i=0,\dots, 15$ are prime exceptional, with square $-2$ and pairwise orthogonal for the Beauville--Bogomolov form. Moreover, the saturation of $\langle [\Theta_i] \rangle_{i=0}^{15}$ in $\NS(M_S(0,H,0))$ is a Kummer lattice, with orthogonal complement in $H^2(M_S(0,H,0),\ZZ)$ isometric to $\mathrm{U}(2)^{\oplus 3} \oplus \langle -4\rangle$. 
\end{proposition}
\begin{proof}
As usual, we identify $H^2(M_S(0,H,0),\ZZ)$ with $(0,H,0)^{\bot}\subset \widetilde{H}^2(S,\ZZ)$ via the Mukai homomorphism \cite{Muk84, O'G97}. Then the class $\pi^*(\mathcal{O}(1))$ inducing the Lagrangian fibration is $(0,0, -1)$, while $[\Theta] = (1, 0, 1)$.
Therefore $\Theta$ is a prime exceptional divisor of square $-2$. The birational map $F\mapsto F\otimes \mathcal{O}_S(T_i)$ sends a Mukai vector $(a,b,c)$ to $(a, aT_i + b, c+(b, T_i) - 2a)$, and therefore $[\Theta_i] = (1, T_i, -1)$ in $H^2(M_S(0,H,0))$. The $16$ classes $(1,0,1)$ and $(1,T_i,-1)$ are pairwise orthogonal with square $-2$. Each $\Theta_i$ is thus prime exceptional of square $-2$.

To determine the saturation of the sublattice spanned by $\Theta_0,\dots, \Theta_{15},$ we specialize $S$ to a Kummer $\mathrm{K}3$ surface $\Km(A)$ of a $(1,3)$-polarized abelian surface, so that $T_i$ specializes to the difference $C_0-C_i$ of exceptional curves in $\Km(A)$. 
It is not hard to check that the saturation in $\widetilde{H}^2(\Km(A),\ZZ)$ of the sublattice spanned by $(1,0,1)$ and the $15$ classes $(1, C_0-C_1, -1), \dots, (1,C_0-C_{15}, -1)$ is a Kummer lattice, whose orthogonal complement is isometric to $\mathrm{U}(2)^{\oplus 3} \oplus \langle -4\rangle$.
\end{proof}

\begin{corollary}
    For a general Heisenberg invariant quartic $S\subset \mathbb{P}^3$, there exists a finite Galois cover $\tilde{M}\to M_S(0,H,0)$ with Galois group $(\mathbb{Z}/2\mathbb{Z})^5$, ramified over the divisors $\Theta_0, \Theta_1,\dots, \Theta_{15}$, such that $\tilde{M}$ is birational to a variety $K$ of $\mathrm{Kum}^3$-type. If $S$ is very general, then $K$ has Picard rank $1$, with polarization of degree $16$ and divisibility $2$.
\end{corollary}
\begin{proof}
    By Proposition \ref{prop:16DivisorsOnModuliOnHeisenburg} and Remark \ref{rmk:contractionAfterFlops}, there exists a finite cover $\widetilde{M}\to M_{S}(0,H,0)$ with the listed properties as long as the $16$ divisors $\Theta_i$ remain prime, which holds for general $S$. 
    For very general $S$, the variety $\widetilde{M}$ is birational to a $\mathrm{Kum}^3$-variety $K$ of Picard rank $1$, and $M_S(0,H,0)$ is birational to the hyper-Kummer variety $Y_K$. By Proposition \ref{prop:heisenberg-invariant-isMRS}, $Y_K$ is also birational to a MRS-double cover; by Remark \ref{rmk:notMRS}, the polarization on the $\mathrm{Kum}^3$-type variety has divisibility at least $2$. Since we know that the transcendental lattice of $M_S(0,H,0)$ is Hodge isometric to $\mathrm{U}(2)^{\oplus 2}\oplus \langle -4\rangle \oplus \langle -8 \rangle$, by Theorem \ref{thm:NeronSeveriHyperkummer} the only possibility is that the polarization on $K$ has divisibility $2$ and square $16$. 
\end{proof}

\subsection{A family of hyper-Kummer K3 surface in $\mathbb{P}^{125}$}
\label{subsec:genus125K3}
Let $J$ be a very general principally polarized abelian surface; hence, $J$ is the Jacobian of a curve $C$ of genus $2$. Fix a symmetric principal polarization $\Theta$. The divisor $\Theta$ is the image of the Abel-Jacobi embedding $C\hookrightarrow J$, and it is unique up to translation by a point in $J[2]$. Moreover, $\Theta$ is the only effective divisor in its rational equivalence class.
As is well-known, the complete linear system $|2\Theta|$ is $3$-dimensional, and the induced map gives an embedding of $J/\pm 1$ into $\mathbb{P}^3$ as a quartic surface with $16$ nodes. Let $\Km(J)$ denote the K3 surface obtained by blowing up the nodes of $J/\pm 1$, and let $R_\tau$, $\tau\in J[2]$, denote the exceptional curve over the node corresponding to $\tau$. The set $\{R_\tau\}_{\tau\in J[2]}$ consists of $16$ pairwise orthogonal $-2$-classes, and the saturation of $\langle R_{\tau}\rangle_{\tau\in J[2]}$ is the Kummer lattice. Let $H$ be the divisor on $\Km(J)$ coming from the hyperplane section of $J/\pm 1\subset \mathbb{P}^3$; then $H$ is primitive of square $4$ and it is orthogonal to the $R_\tau$. By construction, the linear system $|H|$ induces the morphism $\phi\colon \Km(J)\to \mathbb{P}^3$ which is the blow-up of the nodes of $J/\pm 1$.
The N\'eron--Severi group of $\Km(J)$ has rank $17$, and it is generated over $\QQ$ by $H$ and the $R_{\tau}$; in fact $\NS(\Km(J))$ is an overlattice of index $2$ of $H\oplus \langle R_{\tau}\rangle^{\mathrm{sat}}$. 
The translations by order $2$ points on $J$ induce an action of $J[2]\cong (\ZZ/2\ZZ)^4$ on $\Km(J)$ via symplectic automorphisms.

The divisor 
\begin{equation} 
L\coloneqq 8H - \frac{1}{2}\sum_{\tau\in J[2]} R_{\tau}
\end{equation}
has square $248$ and the linear system $|L|$ embeds $\Km(J)$ in $\PP^{125}$. We can identify $|L|$ with the linear system of odd sections of $|16\Theta|$ on $J$, i.e., the $-1$-eigenspace $H^{0}(J,\mathcal{O}(\Theta)^{\otimes 16})^{-1}$, see \cite{bauer1994}. Since $(L, R_{\tau})=1$, the embedding $\Km(J)\hookrightarrow \PP^{125}$ sends the exceptional curves $R_{\tau}$ to lines.

The $(\ZZ/2\ZZ)^4$-action on $J$ induces an action of $(\ZZ/2\ZZ)^4$ on $\PP^{125} = |L|^{\vee}$. 
We consider a family of K3 surfaces in $\PP^{125}$ stabilized by this action, namely, we look at those surfaces in $\PP^{125}$ invariant under the $(\ZZ/2\ZZ)^4$-action which lie in a connected component of the Hilbert scheme containing $\Km(J)\hookrightarrow \PP^{125}$.
Taking the GIT quotient with respect to the normalizer of $(\ZZ/2\ZZ)^{4}$ in $\mathrm{PGL}_{126}$ we obtain a compactification $V$ of a component of the moduli space of K3 surfaces with a symplectic action of $(\ZZ/2\ZZ)^4$ and degree $248$, which is a $4$-dimensional variety (by \cite[Theorem 7.1]{GarbagnatiSarti}, this moduli space has $3$ components, each of dimension $4$). Here, by degree we mean the degree of the unique $(\ZZ/2\ZZ)^4$-invariant polarization on the generic element of the family. 

Let $U\subset V$ be the locus parametrizing smooth surfaces and the family $\mathcal{S}\to U$ of K3 surfaces.

\begin{proposition}
    For any $u\in U$, the $\mathrm{K}3$ surface $\mathcal{S}_u$ is hyper-Kummer. 
\end{proposition}
\begin{proof}
We determine the transcendental lattice of the very general element $S$ of the family. By construction, $S$ admits a polarization of degree $248$, and a deformation to $\Km(J)\hookrightarrow \PP^{125}$ under which this polarization specializes to the class $L$. Moreover, we can choose this deformation to be equipped with a fiberwise symplectic action of $(\ZZ/2\ZZ)^4$. Hence, the anti-invariant part for the action of this group in cohomology remains of Hodge type along the family. It follows that the image of $\NS(S)$ in $\NS(\Km(J))$ under specialization is a rank $16$ lattice which is the saturation of the sublattice spanned by $L$ and the classes $R_i - R_j$, for all $i$ and $j$. The saturation of the sublattice spanned by the $R_i-R_j$ is the anti-invariant sublattice $N_V$; its complement in the K3 lattice is isometric to $\mathrm{U}(2)^{\oplus 3} \oplus \langle -8\rangle$. As a class in this complement, $L$ has square $248$ and divisibility $8$. We can then explicitly compute 
\begin{equation}\label{eq:transcendentalHeisenberg125}
    H^2_{\mathrm{tr}}(S,\ZZ) = \mathrm{U}(2)^{\oplus 2} \oplus \begin{psmallmatrix}
        -8 & 2 \\ 2 & -16
    \end{psmallmatrix} \cong \mathrm{U}(2)^{\oplus 2} \oplus \begin{psmallmatrix}
    -4 & 2 \\ 2 & -32
\end{psmallmatrix}
\end{equation}
The second isometry follows as the discriminant forms of $\begin{psmallmatrix}
    -8 & 2 \\ 2 & -16
\end{psmallmatrix}$ and $\begin{psmallmatrix}
    -4 & 2 \\ 2 & -32
\end{psmallmatrix}$ are isomorphic (as is not too hard to check), and the lattices $\mathrm{U}(2)^{\oplus 2} \oplus \begin{psmallmatrix}
    -8 & 2 \\ 2 & -16
\end{psmallmatrix}$ and $\mathrm{U}(2)^{\oplus 2} \oplus \begin{psmallmatrix}
    -4 & 2 \\ 2 & -32
\end{psmallmatrix}$, are uniquely determined by their signature and discriminant forms. 
Therefore, $H^2_{\mathrm{tr}}(S,\ZZ)$ embeds primitively in $\mathrm{U}(2)^{\oplus 3} \oplus \langle -4\rangle$, and $S$ is a hyper-Kummer K3 surface by Theorem \ref{thm:characterizeS}.
\end{proof}

Let $S$ be a very general K3 surface in the family $\mathcal{S}\to U$ described above. Considering a specialization to $\Km(J)$ in this family, the image of $\NS(S)$ in $\NS(\Km(J))$ is also described as the saturation of the sublattice spanned by the $16$ classes $H-R_\tau$, $\tau\in J[2]$. For each $\tau$, the class $H-R_{\tau}$ is effective of square $2$, and $(H-R_{\tau}, H-R_{\tau'} )=4$. Moreover, the orthogonal $(H-R_{\tau})^{\bot}$ in $\NS(S)$ contains no $-2$-classes. Therefore, the N\'eron--Severi group of the very general K3 surface $S$ in the family is generated (up to saturation) by $16$ ample classes $T_i$, $i=1,\dots, 16$, which have square $2$ and such that $(T_i,T_j)=4$.
The $(\ZZ/2\ZZ)^4$-action is transitive on the $T_i$. Each $T_i$ induces a double cover $\phi_{|T_i|}\colon S\to \mathbb{P}^2$, ramified over a smooth sextic curve whose isomorphism class does not depend on $i$. The degree $248$ polarization $L$ is recovered as $L=\tfrac{1}{2} \sum_{i=1}^{16} T_i$.

Taking Hilbert schemes of points on the K3 surfaces in $\mathcal{S}$ we obtain a smooth family $\mathcal{Y}\to U$ with fibers isomorphic to the sixfolds $(\mathcal{S}_u)^{[3]}$ of $\mathrm{K}3^{[3]}$-type. 
For very general $S$ in our family, $\NS(S^{[3]})$ has rank $17$, generated up to finite index by the classes $t_i$ induced by the $T_i\in \NS(S)$ and the Hilbert--Chow class $\delta$ (which is primitive of square $-4$). 
Under a specialization to $(\Km(J))^{[3]}$ we can identify $\NS(S^{[3]})$ with the saturation of the lattice spanned by the classes $h-r_i$ and $\delta$, where $h,r_i$ denote the classes in $\NS(\Km(J)^{[3]})$ corresponding to $H,R_i$ from $\NS(\Km(J))$ and $\delta$ is the Hilbert--Chow class.

\begin{proposition}
    Let $S$ be a very general $\mathrm{K}3$ in the family $\mathcal{S}\to U$. Then $S^{[3]}$ contains $16$ prime exceptional divisors $E_1,\dots , E_{16}$ with cohomology classes $[E_i]= t_i -\delta$. Moreover, the $[E_i]$ are pairwise orthogonal with Beauville-Bogomolov square $-2$ and the saturation of $\langle [E_i]\rangle$ in $H^2(S^{[3]},\ZZ)$ is a Kummer lattice, with orthogonal complement isometric to $\mathrm{U}(2)^{\oplus 3} \oplus \langle -4\rangle$.
\end{proposition}
\begin{proof}
On $\Km(J)^{[3]}$ we have $16$ prime divisors with cohomology classes $h-r_i-\delta$, as can be seen as follows. The isotropic class $h-\delta$ induces a rational Lagrangian fibration $f\colon \Km(J)^{[3]}\dashrightarrow |H|$, which sends a general subscheme $\xi$ to the unique curve in $|H|$ which contains it. 
Recall that $|H|$ is the hyperplane linear system for $J/\pm 1\subset \mathbb{P}^3$; hence, it contains $16$ planes $P_i\coloneqq |H-R_i| + R_i$, for $i=1,\dots, 16$, parametrizing hyperplanes through the $i$-th node. Curves in $P_i\subset |H|$ are the union of the exceptional curve $R_i$ with a curve of genus $2$, isomorphic to $\Theta\subset J$.
Now $\overline{f^{-1}(P_i)}$ is the closure of the locus of those $\xi\in \Km(J)^{[3]}$ whose support lies on a unique curve $C\in |H|$ which, moreover, belongs to the plane $P_i$. Certainly $\overline{f^{-1}(P_i)}$ contains the divisor $r_i$, but it contains another irreducible component, whose general points have support that does not intersect $R_i$. Hence, $\overline{f^{-1}(P_i)}= r_i + e_i$ for a prime divisor $e_i$; in fact, for a general curve $C\in P_i$, the fibre $\overline{f^{-1}(C)}$ is the union of $2$ irreducible components, both birational to a $\mathbb{P}^1$-bundle over $J$. 
Since $\overline{f^{-1}(P_i)}$ is linearly equivalent to $h-\delta$, we have $e_i = h-\delta-r_i$ in $H^2(\Km(J)^{[3]},\ZZ)$. We easily compute that the $e_i$ are pairwise orthogonal $(-2)$-classes, that the saturation of $\langle e_i\rangle_{i=1,\dots, 16}$ in $H^2(\Km(J)^{[3]},\ZZ)$ is a Kummer lattice, with orthogonal complement isometric to $\mathrm{U}(2)^{\oplus 3}\oplus \langle -4\rangle$. 
The divisors $e_i$ also admit the following description: the $2$-dimensional linear system $P_i$ induces a rational double cover $\phi_i\colon \Km(J)\dashrightarrow \PP^2$, which is the projection of $J/\pm 1$ from the $i$-th node. Then, by its definition, $e_i\subset \Km(J)^{[3]}$ equals the closure of the locus of subschemes whose support maps to $3$ collinear points in $\PP^2$ via $\phi_i$.

Let now $S$ be very general in the family $\mathcal{S}\to U$. Its Picard group is generated by $16$ ample classes $T_i$ of degree $2$, which induce double covers $\phi_i\colon S\to \PP^2$. We define $16$ divisors $E_i\subset S^{[3]}$ as the closure of the locus of those subschemes whose support maps to $3$ collinear points via $\phi_i$. We specialize $S$ to $\Km(J)$ in such a way that $|T_i|$ specializes to the linear system $P_i$; taking Hilbert cubes, the $E_i$ specialize to the divisors $e_i\subset \Km(J)^{[3]}$ defined above. Hence, the $E_i\subset S^{[3]}$ are prime divisors, and their cohomology class $[E_i] = t_i-\delta \in \NS(S^{[3]})$ have the clalimed intersection properties.
\end{proof}


\begin{corollary}\label{cor:HyperKummerSixfoldsFromGenus125K3}
    For general $S$ in the family $\mathcal{S}\to U$, there exists a finite cover $\tilde{Y}\to S^{[3]}$ of degree $2^5$, ramified over the divisors $E_1, \dots, E_{16}$, with $\tilde{Y}$ birational to a variety $K$ of $\mathrm{Kum}^3$-type. If $S$ is very general, then $K$ has Picard rank $1$, with polarization of degree $248$, divisibility $8$ and residual invariant $\pm 1$.
\end{corollary}
\begin{proof}
    For $S$ very general, we described $16$ prime exceptional divisors $E_i\subset S^{[3]}$ in Proposition \ref{prop:hyper-KummerFamilyII}. By \cite[Theorem 1.11]{markmanPrime}, the $E_i$ are stably prime exceptional, which implies that they specialize to prime exceptional divisors $E'_i\subset (S')^{[3]}$ for a general K3 $S'$ in the family. 
    By Remark \ref{rmk:contractionAfterFlops} and Corollary \ref{cor:reconstructionProcedure}, there exists a cover $\tilde{Y}\to S^{[3]}$ ramified over the $E_i$, with $\tilde{Y}$ birational to a variety $K$ of $\mathrm{Kum}^3$-type. 
    
    If $S$ is very general, then $K$ has Picard rank $1$, $\NS(K)= \ZZ\cdot h$. Moreover, we can compute that $\NS(S^{[3]})$ is the saturation of the orthogonal direct sum of $\langle E_i\rangle_{i=1}^{16}$ and the class $4l - 31 \delta$, where $l=\tfrac{1}{2}\sum_{i=1}^{16} t_i$ is the class coming from the $(\ZZ/2\ZZ)^4$-invariant polarization $L\in \NS(S)$, of degree $248$. The class $4l-31\delta$ must therefore be the primitive generator of the image of $\NS(K)$ under the rational map $K\dashrightarrow \tilde{Y} \rightarrow S^{[3]}$. As a class of $H^2(S^{[3]},\ZZ)$, $4l-31\delta$ has divisibility $4$ and square $124$. By Proposition \ref{prop:invariants-pushforward-polarization}, we conclude that $h\in H^2(K,\ZZ)$ has divisibility $ 8$ and square $248$. Comparing the transcendental lattice of $Y_K$ \eqref{eq:transcendentalHeisenberg125} with the table in Theorem \ref{thm:NeronSeveriHyperkummer}, we further deduce that the residual invariant $\pm [(h,-)/8]$ must be $\pm 1$.
\end{proof}

\subsection{A family of degree $8$ hyper-Kummer K3 surfaces}
\label{subsec:DiagonalOctics}
Consider the family of complete intersections in $\PP^5$ of quadrics of the form 
\begin{equation}\label{eq:3quadrics}
    \begin{cases}
        \sum_{i=0}^5 \alpha_i x_i^2 & =0\\
        \sum_{i=0}^5 \beta_i x_i^2 & =0\\
        \sum_{i=0}^5 \gamma_i x_i^2 & =0
    \end{cases}
\end{equation}
for complex parameters $\alpha_i, \beta_i,\gamma_i$. This family is studied in \cite{GarbagnatiSarti}. They show that, modulo the action of projective automorphisms, one obtains a $4$-dimensional family of K3 surfaces of generic Picard rank $16$, equipped with a symplectic action of $(\ZZ/2\ZZ)^4$ given by changing the signs of an even number of coordinates of $\PP^5$. 
Such K3 surfaces have been further studied by Laterveer \cite{Laterveer-2016-K3FiniteDimMotive}, who shows that their quotients by the $(\ZZ/2\ZZ)^4$-action are the nodal K3 surfaces obtained as double covers of $\PP^2$ branched along $6$ lines studied in \cite{paranjape}.

\begin{proposition}\label{prop:hyper-KummerFamilyII}
    Any $\mathrm{K}3$ surface $S$ which is the intersection of $3$ diagonal quadrics as in \eqref{eq:3quadrics} is a hyper-Kummer $\mathrm{K}3$ surface.
\end{proposition}

To prove Proposition \ref{prop:hyper-KummerFamilyII}, we will calculate the transcendental lattice of the very general element of the family. This family contains the $3$-dimensional family of Jacobian Kummer surfaces, which are embedded in $\PP^5$ via the polarization $F=2H-\tfrac{1}{2} \sum R_{\tau}$ (with the same notation of \S\ref{subsec:genus125K3}).

\begin{lemma}
\label{lemma:transcedenetalOctic}
    Let $S\subset \PP^5$ be a very general complete intersection as in \eqref{eq:3quadrics}. The transcendental lattice of $S$ is isometric to the lattice $\mathrm{U}(2)^{\oplus 2}\oplus \langle -4\rangle \oplus \langle -4\rangle$.
\end{lemma}
\begin{proof}
    Since the $(\ZZ/2\ZZ)^4$-action is preserved in the family, the anti-invariant lattice (which is isometric to the lattice $N_V$ of Definition \ref{def:latticeN_V}) for this action must remain contained in the N\'eron--Severi group of all K3 surfaces in the family. In $\NS(\Km(J))$, this anti-invariant part is the orthogonal to $\tfrac{1}{2}\sum_{\tau\in J[2]} R_{\tau}$ inside the Kummer lattice generated (after saturation) by the $R_{\tau}$. It follows that, for very general $S$ as in the statement, $\NS(S)$ is isometric to the saturation inside $\NS(\Km(J))$ of the sublattice spanned by $F$ and the classes $R_{\tau} - R_{\tau'}$ for all $\tau,\tau'$ in $J[2]$. Therefore, $H^2_{\mathrm{tr}}(S,\ZZ)$ is isometric to the orthogonal complement $F^{\bot}$ inside $H^2(\Km(J),\ZZ)^{J[2]}\cong \mathrm{U}(2)^{\oplus 3} \oplus \langle -8\rangle$. 
    As a class of this lattice, $F$ has square $8$ and divisibility $4$. By \cite[Theorem 7.1]{GarbagnatiSarti} (or by an explicit computation similar to those done in the proofs of Lemma \ref{lem:orbitsU^3+<-2>} and Lemma \ref{lem:orbitsU^3+<-8>}), we obtain 
    \begin{equation}
    \label{eqn:TranscendentalOctic}
        H^2_{\mathrm{tr}}(S,\ZZ) = \mathrm{U}(2)^{\oplus 2} \oplus\begin{psmallmatrix}
            -8 & 4 \\ 4 & -4
        \end{psmallmatrix}.
    \end{equation}
    This lattice is in fact isometric to $\mathrm{U}(2)^{\oplus 2}\oplus \langle -4\rangle\oplus \langle -4\rangle$: indeed, if $\{a,b\}$ is a basis of the last summand in \eqref{eqn:TranscendentalOctic}, we just need to replace this basis by the basis $\{a+b, b\}$.
\end{proof}
\begin{proof}[{Proof of Proposition \ref{prop:hyper-KummerFamilyII}}]
By the description of the transcendental lattice of the very general element in the family, the proposition follows immediately via Theorem \ref{thm:characterizeS}.
\end{proof}

In contrast with the previous two families we could not find a natural moduli spaces on these K3 surfaces that is a hyper-Kummer sixfold.


\section{Super-Kummer K3 surfaces and abelian motives}
\label{sec:SuperKummer}

As is shown in Corollary \ref{cor:Relation-YKMKSK}, a projective hyper-Kummer sixfold is birational to a smooth projective moduli space of stable sheaves on a hyper-Kummer K3 surface. As Hilbert schemes of points are the best understood moduli spaces of stable sheaves, we are naturally led to investigate hyper-Kummer sixfolds that are birational to Hilbert cubes of some K3 surface; in other words, we study the K3 surfaces whose Hilbert cube is birational to a hyper-Kummer sixfold, called \textit{super-Kummer} K3 surfaces. To study such K3 surfaces, both the hyper-Kummer construction and the MRS-double construction play essential roles, exemplifying the power of the \textit{synergy} mentioned in the Introduction Section \ref{subsec:Intro-OG6}.

\subsection{Super-Kummer K3 surfaces}
\begin{definition}
\label{def:SuperKummerK3}
A K3 surface $S$ is called \emph{super-Kummer} if its Hilbert cube $S^{[3]}$ is bimeromorphic to the hyper-Kummer sixfold $Y_K$ associated with some hyper-K\"ahler manifold $K$ of $\mathrm{Kum}^3$-type.
\end{definition}

We first clarify the relations between Kummer, super-Kummer and hyper-Kummer K3 surfaces:
\begin{lemma}
\label{lemma:SuperKummerIsHyperKummer}
    A (classical) Kummer $\mathrm{K}3$ surface is super-Kummer. A projective super-Kummer $\mathrm{K}3$ surface is a hyper-Kummer $\mathrm{K}3$ surface.
\end{lemma}
\begin{proof}
    The first assertion follows from Proposition \ref{prop:Y_K^3(A)}.
    For the second assertion, given a projective super-Kummer K3 surface $S$, by definition, its Hilbert cube $S^{[3]}$ is bimeromorphic to some hyper-Kummer sixfold $Y_K$ associated with some $\mathrm{Kum}^3$-type manifold $K$. Therefore, we have Hodge isometries \[H^2_{\tr}(S,\ZZ)\simeq H^2_{\tr}(S^{[3]},\ZZ)\simeq  H^2_{\tr}(Y_K, \ZZ)\simeq H^2_{\tr}(S_K, \ZZ),\] where $S_K$ is the hyper-Kummer K3 surface associated with $Y_K$ and the last Hodge isometry follows from Theorem \ref{thm:TranscendentalResolutions}. Since $S$ is projective with Picard number at least $12$, the Hodge isometry $H^2_{\tr}(S, \ZZ)\simeq H^2_{\tr}(S_K, \ZZ)$ extends to a Hodge isometry $H^2(S, \ZZ)\simeq H^2(S_K, \ZZ)$. The global Torelli theorem for K3 surfaces implies that $S\simeq S_K$, hence $S$ is hyper-Kummer.
\end{proof}
In short, we have the following inclusions in the \textit{projective} setting:
\[\{\text{Kummer K3 surfaces} \}\subset \{\text{super-Kummer K3 surfaces} \}\subset \{\text{hyper-Kummer K3 surfaces}\}.\]
We will see in the sequel that both inclusions are strict.

\begin{remark}[Non-projective situation]
    It is easy to see that a hyper-Kummer K3 surface $S_{\sigma,\sigma'}$ associated with a very general non-projective manifold of $\mathrm{Kum}^3$-type of Picard rank $0$ is not super-Kummer. Indeed, in this case $\NS(S_{\sigma,\sigma'}) = N_S$ (see Proposition \ref{prop:cohomologyM&S}), and if $S_{\sigma,\sigma'}$ were a super-Kummer K3 surface then the lattice $\NS(S^{[3]}) \simeq N_S \oplus \langle -4\rangle$ would have to be isometric to the Kummer lattice $L_{\Km}$, which is not the case. It is not clear to us whether a non-projective super-Kummer K3 surface is necessarily isomorphic to a hyper-Kummer K3 surface $S_{\sigma,\sigma'}$ associated with some manifold of $\mathrm{Kum}^3$-type.
    Nevertheless, note that the transcendental lattice of a very general super-Kummer K3 surface is isometric to that of a very general hyper-Kummer K3, namely, to $\mathrm{U}(2)^{\oplus 3}\oplus \langle -4\rangle$.
\end{remark}

In the rest of this section, we focus on projective super-Kummer K3 surfaces. We shall characterize projective super-Kummer K3 surfaces of minimal Picard rank $16$ among hyper-Kummer K3 surfaces. For a projective hyper-Kummer surface $S$ of Picard rank $16$, by Corollary \ref{cor:Transcendental-hyper-KummerK3} we have either
\begin{equation}
    H^2_{\mathrm{tr}}(S,\ZZ)=\mathrm{U}(2)^{\oplus 2}\oplus \langle -4\rangle \oplus \langle -4m\rangle
\end{equation}
for some $m>0$, or
\begin{equation}
H^2_{\mathrm{tr}}(S,\ZZ)=\mathrm{U}(2)^{\oplus 2} \oplus \begin{psmallmatrix}
        -4 & 2 \\
        2 & -4m
    \end{psmallmatrix}
\end{equation}
for some $m>0$. We say that $S$ is of \textit{split type} in the first case, and that $S$ is of \textit{non-split type} in the second case.

\begin{theorem}
\label{thm:SuperKummer}
    Let $S$ be a projective hyper-Kummer $\mathrm{K}3$ surface of minimal Picard rank $16$. Then $S$ is super-Kummer if and only if it is of non-split type.
    In this case, $S^{[3]}$ is also birational to a MRS-double cover.
\end{theorem}
\begin{proof}
    Recall that $H^2_{\mathrm{tr}}(S^{[3]},\ZZ) $ is isometric to the transcendental lattice of $S$, while $\NS(S^{[3]})=\NS(S)\oplus \langle -4 \rangle$. Denoting by $T\coloneqq H^2_{\mathrm{tr}}(S,\ZZ)\oplus \langle 4\rangle$, the lattice $\NS(S^{[3]})$ can also be characterized as the unique even lattice of signature $(1,16)$ and discriminant form isometric to $A_T(-1)$. 
    
    Assume first that $S$ is of split type, so $H^2_{\mathrm{tr}}(S,\ZZ) = \mathrm{U}(2)^{\oplus 2}\oplus \langle -4\rangle \oplus \langle -4m\rangle$, for some $m>0$. Then $A_T = (\ZZ/2\ZZ)^4 \oplus \ZZ/4\ZZ \oplus \ZZ/4m\ZZ\oplus \ZZ/4\ZZ$.
    According to Theorem \ref{thm:NeronSeveriHyperkummer}, the N\'eron--Severi group of a hyper-Kummer sixfold of Picard rank $17$ is either the lattice $L_{4d}$ or $L_{\Km}\oplus \langle 4d\rangle$, for some integer $d>0$; we have $A_{L_{4d}} = (\ZZ/2\ZZ)^4\oplus \ZZ/4d\ZZ$ and $A_{L_{\Km}}\oplus \langle 4d\rangle = (\ZZ/2\ZZ)^6 \oplus \ZZ/4d\ZZ$. None of these groups could be isomorphic to $A_T$, and it follows that $S^{[3]}$ cannot be birational to a hyper-Kummer sixfold, i.e., $S$ is not super-Kummer.

    Assume now that $S$ is of non-split type, so that $H^2_{\mathrm{tr}}(S,\ZZ)=\mathrm{U}(2)^{\oplus 2}\oplus \begin{psmallmatrix}
        -4 & 2 \\ 2 & -4m
    \end{psmallmatrix}$ for some $m>0$. Then $T=H^2_{\mathrm{tr}}(S,\ZZ)\oplus \langle 4\rangle$ is isometric to the lattice $\mathrm{U}(2)^{\oplus 3} \oplus \langle -4(4m-1)\rangle$, by the elementary Lemma \ref{lem:an_elementary_calculation} below. 
    But then $\NS(S^{[3]})$ is the unique lattice of signature $(1,16)$ and discriminant form $A_T(-1)$, and therefore $\NS(S^{[3]})=L_{\Km}\oplus \langle 4(4m-1)\rangle$. By Remark \ref{rmk:BW+h}, there exists a primitive embedding $\mathrm{BW}_{16}\hookrightarrow \NS(S^{[3]})$. By Theorem \ref{thm:characterizationMRS}, $S^{[3]}$ is birational to some MRS-double cover, and, hence, $S^{[3]}$ is birational to some hyper-Kummer variety of $\mathrm{K}3^{[3]}$-type, by Theorem \ref{thm:comparisonHyperkummerMRS}. 
\end{proof}

\begin{lemma}\label{lem:an_elementary_calculation}
Let $m>0$ be an integer. The lattice $L\coloneqq \begin{psmallmatrix}
    -4 & 2\\ 2 & -4m
\end{psmallmatrix}  \oplus \langle 4\rangle$
is isometric to $\mathrm{U}(2) \oplus \langle -4(4m-1)\rangle$.
\end{lemma}
\begin{proof}
 By assumption, $L$ has a basis $\alpha,\beta, \gamma$ in which the intersection matrix is \begin{equation}\begin{pmatrix}
 -4 & 2 & 0 \\
 2 & -4m & 0\\
 0 & 0 & 4
 \end{pmatrix}.
 \end{equation}
Consider the vectors 
\begin{equation}
    a \coloneqq \alpha + \gamma, \ \ \ b\coloneqq \beta + m (\alpha+ \gamma), \ \ \ c = 4m\alpha + 2 \beta + (4m-1) \gamma.
\end{equation}
It is easy to check that $a,b,c$ gives a basis of $L$, in which the intersection matrix becomes
\begin{equation}\begin{pmatrix}
 0 & 2 & 0 \\
 2 & 0 & 0\\
 0 & 0 & -4(4m-1)
 \end{pmatrix},
 \end{equation}
 which shows that $L$ is isometric to $\mathrm{U}(2)\oplus \langle -4(4m-1)\rangle$.
\end{proof}


\begin{remark}\label{rmk:transcendental_hyperKummerK3}
    Let $S$ be a projective K3 surface of Picard rank $16$ which is super-Kummer. Hence, $S$ is non-split and we have $H^2_{\mathrm{tr}}(S,\ZZ) = \mathrm{U}(2)^{\oplus 2}\oplus \begin{psmallmatrix}
        -4 & 2 \\ 2 & -4m
    \end{psmallmatrix}$
    for some $m>0$. Looking at the table given in Theorem \ref{thm:NeronSeveriMRS}, we deduce that $S^{[3]}$ is birational to the MRS-double cover associated with a projective singular OG6-variety of Picard rank $1$ with polarization of square $8m-2$ and divisibility $2$.
    By Theorem \ref{thm:NeronSeveriHyperkummer}, the Hilbert cube $S^{[3]}$ is birational to the hyper-Kummer sixfold associated with a $\mathrm{Kum}^3$-variety of Picard rank $1$, with polarization of divisibility $\geq 4$, but the square and divisibility are not fully determined (each of the last $3$ cases in the table of Theorem \ref{thm:NeronSeveriHyperkummer} is possible, in principle).
    \end{remark}

\subsection{K3 surfaces with abelian motives}
\label{subsec:AbelianMotives}
In this section, we show that projective super-Kummer K3 surfaces have abelian Chow motives, obtaining countably many 4-dimensional families of K3 surfaces of generic Picard rank $16$ with abelian Chow motives. More precisely, we establish the following theorem:

\begin{theorem} \label{thm:SuperKummerAbelianMotive}
Let $S$ be a projective super-Kummer $\mathrm{K}3$ surface. Then the Chow motive of $S$ lies in the thick tensor subcategory of  the category of rational Chow motives generated by the Chow motive of an abelian fourfold of Weil type with discriminant $1$. 
\end{theorem}

We can characterize super-Kummer K3 surfaces as those projective K3 surfaces $S$ such that there exists a primitive embedding $H^2_{\mathrm{tr}}(S,\ZZ)\hookrightarrow \mathrm{U}(2)^{\oplus 2} \oplus \begin{psmallmatrix}
    -4 & 2 \\ 2 & -4m
\end{psmallmatrix}$,
for some integer $m>0$. Passing to rational coefficients, we obtain the following result from Theorem \ref{thm:SuperKummerAbelianMotive}. 

\begin{corollary}\label{cor:K3withAbelianMotive}
Let $S$ be a projective $\mathrm{K}3$ surface. If there exists an embedding of rational quadratic spaces $H^2_{\mathrm{tr}}(S,\mathbb{Q})\hookrightarrow\mathrm{U}^{\oplus 2}_{\QQ} \oplus \langle -1\rangle_{\QQ}\oplus \langle -b \rangle_{\QQ}$ for some positive integer $b \equiv 3 \pmod{4}$, then the Chow motive of $S$ lies in the thick tensor subcategory of  the category of rational Chow motives generated by the Chow motive of an abelian fourfold of Weil type with discriminant 1.  
\end{corollary}
\begin{proof}
    Write $b=4m-1$ for some integer $m>0$. Observe that we have an isometry of rational quadratic spaces:
    \begin{equation}
    \mathrm{U}^{\oplus 2}_{\QQ} \oplus \langle -1\rangle_{\QQ}\oplus \langle -(4m-1) \rangle_{\QQ} \simeq \left(\mathrm{U}(2)^{\oplus 2} \oplus \begin{psmallmatrix}
        -4 & 2 \\
        2 & -4m
    \end{psmallmatrix}
    \right)\otimes_{\ZZ} \mathbb{Q}.
\end{equation}
Hence, there exist a super-Kummer K3 surface $S'$ and a Hodge isometry $H^2_{\mathrm{tr}}(S,\QQ)\xrightarrow{\ \sim \ } H^2_{\mathrm{tr}}(S',\QQ)$, by Theorem \ref{thm:SuperKummer}. By \cite{huybrechtsMotives, FuVial}, the rational Chow motives of $S$ and $S'$ are isomorphic; the conclusion follows from Theorem \ref{thm:SuperKummerAbelianMotive}.
\end{proof}

\begin{remark}
    Among the three families of examples in Section \ref{sec:ExamplesHyperKummerK3},
    \begin{itemize}
        \item by \eqref{eq:transcendetnalHeisenbergQuartics}, a generic Heisenberg-invariant quartic is not super-Kummer. It would be interesting to show that these K3 surfaces have abelian Chow motives. 
        \item by Corollary \ref{cor:HyperKummerSixfoldsFromGenus125K3}, the genus-125 K3 surfaces  with symplectic action of $(\ZZ/2\ZZ)^4$ considered in \ref{subsec:genus125K3} are super-Kummer, and therefore have abelian Chow motives;
        \item by Lemma \ref{lemma:transcedenetalOctic}, a generic diagonal $(2,2,2)$-complete intersection K3 surface considered in \ref{subsec:DiagonalOctics} is not super-Kummer. Nevertheless, Laterveer showed in \cite{Laterveer-2016-K3FiniteDimMotive} that such a K3 surface has abelian Chow motives. 
    \end{itemize}
\end{remark}

The first step towards the proof of Theorem \ref{thm:SuperKummerAbelianMotive} is to relate geometrically the super-Kummer K3 surface to a Weil-type abelian fourfold (which is canonical up to isogeny) via a singular OG6-variety. 

    More precisely, by Theorem \ref{thm:SuperKummer}, for any projective super-Kummer K3 surface $S$, there exists a birational isomorphism 
    \begin{equation}
        \phi\colon Z\dashrightarrow S^{[3]}
    \end{equation}
    between the Hilbert cube $S^{[3]}$ and a hyper-Kummer $\mathrm{K}3^{[3]}$-type variety $Z$ that is an MRS-double cover of a singular OG6-variety $X$. 
    By \cite[Theorem 3.4]{FloccariFu-HodgeConjectureWeilFourfolds}, the singular locus of $X$ is of the form 
    \begin{equation}
        \Sing(X)\simeq B/{\pm 1},
    \end{equation}
     where $B$ is an abelian fourfold of Weil type with discriminant 1. We have the Mongardi--Rapagnetta--Sacc\`a rational double cover $\varphi\colon Z\dashrightarrow X$. The closure of the preimage of $\Sing(X)$ is naturally identified with $\widetilde{B}/{\pm 1}$, where $\widetilde{B}=\Bl_{B[2]}(B)$ is the blow-up of $B$ along its $256$ points of order $2$. The involution $-1$ on $B$ naturally lifts to $\Bl_{B[2]}(B)$. We have the following cartesian square: 
     \begin{equation}
         \begin{tikzcd}
            \widetilde{B}/{\pm 1} \arrow[r, hookrightarrow] \arrow[d] & Z \arrow[d, "\varphi"]\\
            B/{\pm 1} \arrow[r, hookrightarrow]& X
        \end{tikzcd}
     \end{equation}
     By \cite[Proposition 5.3]{MRS18}, the rational map $\varphi\colon Z\dashrightarrow X$ is in fact a morphism, described as the composition of the contraction $Z\to \overline{Z}$ of the $256$ exceptional $\mathbb{P}^3$ contained in $\widetilde{B}/\pm 1 \subset Z$ and finite morphism $\overline{Z}\to X$ of degree $2$ ramified along $\mathrm{Sing}(X)$. 
     Denote by $f\colon \widetilde{B} \to Z$ the composition $\widetilde{B}\to \widetilde{B}/{\pm 1}\hookrightarrow Z$. 
     
     \begin{remark} Although not used in the proof, the key observation which leads to the idea of our proof is that $f^*\colon H^{2,0}(Z)\to H^{2,0}(\widetilde{B})$ is injective. This simply follows from the fact that $\dim (B) = 4 > \tfrac{1}{2} \dim (Z)$, so that the restriction of the symplectic form to $\widetilde{B}/\pm 1$ is non-zero; as $H^{2,0}(\widetilde{B}/\pm 1)\simeq H^{2,0}(\widetilde{B})$, we see that $f^*$ is injective. Moreover, since $Z$ is a moduli space of stable sheaves on the super-Kummer K3 surface $S$, a correspondence induced by a suitable Chern class of the universal sheaf $\mathcal{E}\in \Db(Z\times S)$ gives an isomorphism $H^{2,0}(Z)\simeq H^{2,0}(S)$. 
     The Bloch--Beilinson conjecture predicts that there is an algebraic correspondence 
     between $\widetilde{B}$ and $S$ inducing a surjective map on 0-cycles $\CH_0(\widetilde{B})\to \CH_0(S)$, which would imply that $S$ is Chow-motivated by $\widetilde{B}$, hence by $B$.
     \end{remark}
     
By Riess \cite{Rie14}, the birational map $\phi\colon Z\dashrightarrow S^{[3]}$ induces an isomorphism $\phi_*$ of Chow rings.
\begin{lemma}\label{lemma:betaNonzero}
    The cohomology class of the algebraic cycle 
    \begin{equation} \beta\coloneqq \phi_*(f_*[\widetilde{B}])\in \CH_4(S^{[3]})
    \end{equation} 
    is non-trivial in $H^4(S^{[3]}, \QQ)$.
\end{lemma}
\begin{proof}
    Since $B/\pm 1\subset X$ deforms to all locally trivial deformations of $X$, by Fujiki relation there exists a constant $c\in \QQ$ such that for any $\alpha\in H^2(X,\ZZ)$ we have
    \begin{equation}
        \int_X [B/\pm 1]\cdot \alpha^4 =c q_X(\alpha)^2,
    \end{equation}
    where $q_X$ is the Beauville--Bogomolov--Fujiki quadratic form on $H^2(X,\ZZ)$. Taking an ample class $\alpha$, we see that $c>0$. By the projection formula, for any ample class $\alpha$ on $X$ we have 
    \begin{equation}
        \int_Z f_*[\widetilde{B}]\cdot\varphi^*(\alpha)^4= 2 \int_Z [\widetilde{B}/\pm 1] \cdot \varphi^*(\alpha)^4=2 \int_X [B/\pm 1] \cdot\alpha^4= 2 c q_X(\alpha)^2 >0 .
    \end{equation}
    In particular, $f_*[\widetilde{B}]\neq 0$ in $H^4(Z, \QQ)$. 
    Since the birational map $\phi$ induces an isomorphism of rational cohomology algebras, $\phi_*(f_*[\widetilde{B}])\neq 0$ in $H^4(S^{[3]}, \QQ)$.   
\end{proof}

\subsection{Correspondences}
\label{subsec:Correspondences}
Let the notation be as above. Using Nakajima-type operators on Hilbert schemes of $S$, we define several natural correspondences between $S$ and $\widetilde{B}$.

Consider the following commutative diagram:
    \begin{equation}
    \label{diag:CorrespondenceSBS}
        \begin{tikzcd}
            &&F\arrow[d, hookrightarrow, swap, "i_2"] \arrow[r, "h"]& S \arrow[d, hookrightarrow, swap, "j_2"] \arrow[dr, hookrightarrow, "\delta"]&&\\
            \widetilde{B} \arrow[r, "f"] & Z \arrow[r, dashed, "\phi", "\simeq"'] & S^{[3]} \arrow[r, "\tau"]& S^{(3)} & S^3 \arrow[l, swap, "\pi"]\arrow[r, "p_i"]& S\\
            && E \arrow[u, hookrightarrow, "i_1"] \arrow[r, "g"]& S^2 \arrow[u, hookrightarrow, "j_1"] \arrow[ur, hookrightarrow, "\iota"] \arrow[urr, "p_i"']& &
        \end{tikzcd}
    \end{equation}
    in which:
    \begin{itemize} 
    \item $f$ and $\phi$ are as above; 
    \item $\tau$ is the Hilbert--Chow morphism and $\pi$ is the natural quotient map by the symmetric group;
    \item $i_1\colon E\hookrightarrow S^{[3]}$ is the inclusion of the Hilbert--Chow divisor, and $g\colon E \to S^2$ sends a non-reduced length-3 subscheme $\xi$ to $(x,y)$ such that the 0-cycle underlying $\xi$ is $x+2y$;
    \item the morphism $\iota\colon S^2\to S^3$ sends $(x,y)$ to $(x,y,y)$, and $\delta\colon S\to S^3$ is the inclusion of the small diagonal;
    \item $i_2\colon F\hookrightarrow S^{[3]}$ is the inclusion of the locus of length-3 subschemes supported at a single point, and $h$ maps such a subscheme to its support;
    \item $p_i$ stands for the projection onto the  $i$-th factor.
    \end{itemize} Both squares in the commutative diagram are cartesian.

    We define the following six types of self-correspondences from $S$ to itself factorizing through $\widetilde{B}$. We let $\h^2_{\mathrm{tr}}(S)$ be the transcendental part of the motive of $S$. The notation is as in Diagram \eqref{diag:CorrespondenceSBS}.
    \begin{enumerate}[label=(\roman*)]
        \item For any $D, D'\in \CH^1(S)$, define $\Gamma_{D, D'}\colon \h^2_{\tr}(S)\to \h(S)$ as the following composition:
        \begin{equation}
            \begin{tikzcd}[column sep =small]
                & \h^2_{\tr}(S) \arrow[r, "p_1^*"] & \h(S^3) \arrow[r, "\pi_*"] & \h(S^{(3)})  \arrow[r, "\tau^*"] & \h(S^{[3]}) \arrow[r, "\phi_*^{-1}", "\simeq"'] &\h(Z) \arrow[r, "f^*"] &\h(\widetilde{B})\arrow[d, "f_*"] \\
              \h(S) &\h(S^3)(4) \arrow[l, "p_{1,*}"']&& \h(S^3)(2) \arrow[ll, "\cdot 1\times D\times D'"'] & \h(S^{(3)})(2) \arrow[l, "\pi^*"'] & \h(S^{[3]})(2) \arrow[l, "\tau_*"']& \h(Z)(2) \arrow[l, "\simeq", "\phi_*"']
            \end{tikzcd}
        \end{equation}
        Here $1\times D\times D':=p_2^*(D)\cdot p_3^*(D')$ and $\phi_*$ is Riess' isomorphism of ring objects in the category of Chow motives \cite{Rie14}. The composition of the first three maps $\tau^*\circ \pi_*\circ p_1^*$ is a positive multiple of the correspondence induced by the incidence cycle $[S^{[1,3]}]\in \CH_6(S^{[3]}\times S)$, which is the Nakajima's operator $\mathfrak{p}_{-1}(1)\circ\mathfrak{p}_{-1}(1)$ . Similarly, the composition of the last four maps  is a positive multiple of $\mathfrak{p}_{1}(D)\circ \mathfrak{p}_{1}(D')$.
        \item Define $\Gamma_1\colon \h(S) \to \h(S)$ as the following composition:
        \begin{equation}
            \begin{tikzcd}
                \h(S) \arrow[r, "p_1^*"] & \h(S^2)  \arrow[r, "{[E]}_*"]& \h(S^{[3]})(1) \arrow[r, "\phi_*^{-1}", "\simeq"'] & \h(Z)(1) \arrow[dr, "f^*"]& \\
                &&&&\h(\widetilde{B})(1)\arrow[dl, "f_*"] \\
                 \h(S) &\h(S^2)(2)  \arrow[l, "p_{1,*}"']  & \h(S^{[3]})(3) \arrow[l, "{[E]}^*"'] & \h(Z)(3) \arrow[l, "\simeq", "\phi_*"']&
            \end{tikzcd}
        \end{equation}
        \item Define $\Gamma_2\colon \h(S) \to \h(S)$ as in (ii) but using $p_2$ instead of $p_1$, that is, as the following composition:
        \begin{equation}
            \begin{tikzcd}
                \h(S) \arrow[r, "p_2^*"] & \h(S^2)  \arrow[r, "{[E]}_*"]& \h(S^{[3]})(1) \arrow[r, "\phi_*^{-1}", "\simeq"'] & \h(Z)(1) \arrow[dr, "f^*"]& \\
                &&&&\h(\widetilde{B})(1)\arrow[dl, "f_*"] \\
                 \h(S) &\h(S^2)(2)  \arrow[l, "p_{2,*}"']  & \h(S^{[3]})(3) \arrow[l, "{[E]}^*"'] & \h(Z)(3) \arrow[l, "\simeq", "\phi_*"']&
            \end{tikzcd}
        \end{equation}
    \item For any $D\in \CH^1(S)$,  define $\Theta_{1,D} \colon \h^2_{\tr}(S) \to \h(S)$ as the following composition:
        \begin{equation}
            \begin{tikzcd}
                \h^2_{\tr}(S) \arrow[r, "p_1^*"] & \h(S^3) \arrow[r, "\pi_*"] & \h(S^{(3)})  \arrow[r, "\tau^*"] & \h(S^{[3]}) \arrow[r, "\phi_*^{-1}", "\simeq"'] &\h(Z) \arrow[dr, "f^*"] &\\
                &&&&&\h(\widetilde{B})\arrow[dl, "f_*"] \\
              \h(S) & \h(S^2)(2) \arrow[l, "p_{1,*}"'] & \h(S^2)(1) \arrow[l, "\cdot p_2^*(D)"'] & \h(S^{[3]})(2) \arrow[l, "{[E]}^*"']& \h(Z)(2) \arrow[l, "\simeq", "\phi_*"']&
            \end{tikzcd}
        \end{equation}
    \item For any $D\in \CH^1(S)$,  define $\Theta_{2,D} \colon \h^2_{\tr}(S) \to \h(S)$ as the following composition:
        \begin{equation}
            \begin{tikzcd}
                \h^2_{\tr}(S) \arrow[r, "p_1^*"] & \h(S^3) \arrow[r, "\pi_*"] & \h(S^{(3)})  \arrow[r, "\tau^*"] & \h(S^{[3]}) \arrow[r, "\phi_*^{-1}", "\simeq"'] &\h(Z) \arrow[dr, "f^*"] &\\
                &&&&&\h(\widetilde{B})\arrow[dl, "f_*"] \\
              \h(S) & \h(S^2)(2) \arrow[l, "p_{2,*}"'] & \h(S^2)(1) \arrow[l, "\cdot p_1^*(D)"'] & \h(S^{[3]})(2) \arrow[l, "{[E]}^*"']& \h(Z)(2) \arrow[l, "\simeq", "\phi_*"']&
            \end{tikzcd}
        \end{equation}
    \item Define, $\Lambda\colon \h(S)\to \h(S)$ as the following composition:
        \begin{equation}
            \begin{tikzcd}
                \h(S) \arrow[r, "p_1^*"] & \h(S^3) \arrow[r, "\pi_*"] & \h(S^{(3)})  \arrow[r, "\tau^*"] & \h(S^{[3]}) \arrow[r, "\phi_*^{-1}", "\simeq"'] &\h(Z) \arrow[dr, "f^*"] &\\
                &&&&&\h(\widetilde{B})\arrow[dl, "f_*"] \\
               &  & \h(S) & \h(S^{[3]})(2) \arrow[l, "{[F]}^*"']& \h(Z)(2) \arrow[l, "\simeq", "\phi_*"']&
            \end{tikzcd}
        \end{equation}
    \end{enumerate}

\begin{lemma}
\label{lemma:ScalarCorrespondences}
The six types of correspondences defined above are all multiples of the identity map. More precisely,
\begin{align}
      \label{eqn:GammaDD'}  \Gamma_{D, D'}=& \left(2\int_{S^3} (\pt\times D\times D')\cdot \pi^*(\tau_*\beta)\right) \id_{\h^2_{\tr}(S)}.\\
       \label{eqn:Gamma1} \Gamma_1=&-\left(\int_{S^2}(\pt\times 1)\cdot \iota^*(\pi^*(\tau_*\beta))\right) \id_{\h(S)}.\\
       \label{eqn:Gamma2} \Gamma_2=&-\left(\int_{S^2}(1\times \pt)\cdot \iota^*(\pi^*(\tau_*\beta))\right) \id_{\h(S)}.\\
       \label{eqn:Theta1} \Theta_{1, D}=& 2\left(\int_{S^2} (\pt\times D)\cdot ({[E]}^*\beta)\right)\id_{\h^2_{\tr}(S)}.\\
       \label{eqn:Theta2}\Theta_{2, D}=&  4\left(\int_{S^2} (D\times \pt)\cdot ({[E]}^*\beta)\right)\id_{\h^2_{\tr}(S)}.\\
       \label{eqn:Lambda} \Lambda=& 6\left(\int_S\pt\cdot ({[F]}^*\beta)\right) \id_{\h(S)}.
\end{align}
where $\beta=\phi_*(f_*[\widetilde{B}])\in \CH^2(S^{[3]})$.
\end{lemma}
\begin{proof}
    Note first the following computation, which is involved in each of the correspondences in the statement:
    \begin{equation}
    \label{eqn:CommonPart}
        \phi_*\circ f_*\circ f^*\circ \phi^{-1}_*=\phi_*\circ (\cdot f_*[\widetilde{B}])\circ \phi^{-1}_*=\cdot \phi_*(f_*[\widetilde{B}])=\cdot \beta ,
    \end{equation}
     where the first equality uses the projection formula, the second uses the fact that $\phi_*$ is an isomorphism of ring objects.
     
    To prove \eqref{eqn:GammaDD'}, we calculate
    \begin{equation}
    \begin{split}
        \pi^*\circ \tau_*\circ\phi_*\circ f_*\circ f^*\circ \phi^{-1}_*\circ \tau^*\circ\pi_*\circ p_1^* 
        =&\pi^*\circ  \tau_*\circ (\cdot \beta)\circ \tau^*\circ \pi_*\circ p_1^*\\
        =& \pi^*\circ (\cdot \tau_*\beta)\circ \pi_*\circ p_1^*\\
        =& (\cdot \pi^*\tau_*\beta)\circ \pi^*\circ \pi_*\circ p_1^*
        \end{split}
    \end{equation}
    where the first equality follows \eqref{eqn:CommonPart}, the second equality uses the projection formula, and the last uses that $\pi^*$ is a ring homomorphism. Therefore, we get the following equality of morphisms of Chow motives $\h^2_{\tr}(S)\to \h(S)$:
    \begin{equation} 
    \begin{split}
        \Gamma_{D,D'}=&p_{1,*}\circ (\cdot (1\times D\times D')\cdot \pi^*(\tau_*\beta))\circ 2(p_1^*+p_2^*+p_3^*)\\
        =&p_{1,*}\circ (\cdot (1\times D\times D')\cdot \pi^*(\tau_*\beta))\circ 2p_1^*\\
        =&\cdot 2 p_{1,*}((1\times D\times D')\cdot \pi^*(\tau_*\beta))
    \end{split}
    \end{equation}
    where the first equality uses that $\pi^*\circ \pi_*=\sum_{\sigma\in \mathfrak{S}_3}\sigma$, the second equality uses the fact that the map $\cdot D$ (as well as $\cdot D'$) is 0 on $\h^2_{\tr}(S)$, the third equality follows from the projection formula. Finally, since $p_{1,*}((1\times D\times D')\cdot \pi^*(\tau_*\beta))\in \CH^0(S)$ is a multiple of the fundamental class of $S$, and this multiple is nothing but the intersection number $2\int_{S^3}(1\times D\times D')\cdot \pi^*(\tau_*\beta)$. Hence, the correspondence $\Gamma_{D, D'}$ preserves $\h^2_{\tr}(S)$ and formula \eqref{eqn:GammaDD'} is proved.

    To show \eqref{eqn:Gamma1} and \eqref{eqn:Gamma2}, we have for $i=1, 2$,
    \begin{equation}
    \begin{split}
       \Gamma_i =&p_{i,*}\circ [E]^*\circ (\cdot \beta) \circ [E]_*\circ p_i^*\\
        =&p_{i,*}\circ g_*\circ i_1^*\circ (\cdot \beta) \circ i_{1,*}\circ g^*\circ p_i^*\\
        =&p_{i,*}\circ g_*\circ (\cdot i_1^*\beta) \circ i_1^* \circ i_{1,*}\circ g^*\circ p_i^*\\
        =&p_{i,*}\circ g_*\circ (\cdot i_1^*\beta) \circ (\cdot c_1(\mathcal{O}_E(E)))\circ g^*\circ p_i^*\\
        =&\cdot p_{i,*}(g_*(i_1^*\beta \cdot c_1(\mathcal{O}_E(E))))\\
        =&\cdot p_{i,*}(j_1^*(-\tau_*\beta))\\
         =&\cdot p_{i,*}(\iota^*\pi^*(-\tau_*\beta))
    \end{split}
    \end{equation}
    where the first equality is by \eqref{eqn:CommonPart}, the third equality uses that $i_1^*$ is a ring homomorphism, the fifth  uses the projection formula, the sixth uses the excess intersection formula. Finally, $p_{i,*}(\iota^*\pi^*(-\tau_*\beta))$ is a multiple of the fundamental class of $S$, and this multiple is the intersection number given in \eqref{eqn:Gamma1} and \eqref{eqn:Gamma2}.

    To show \eqref{eqn:Theta1}, we have for any $D\in \CH^1(S)$
    \begin{equation} 
    \begin{split}
        \Theta_{1, D}=& p_{1, *}\circ(\cdot p_2^*D)\circ {[E]}^*\circ\phi_*\circ f_*\circ f^*\circ \phi^{-1}_*\circ \tau^*\circ\pi_*\circ p_1^*\\
        =& p_{1, *}\circ (\cdot p_2^*D)\circ {[E]}^*\circ (\cdot \beta)\circ \tau^*\circ\pi_*\circ p_1^*\\
        =& p_{1, *}\circ (\cdot p_2^*D)\circ g_*\circ i_1^*\circ (\cdot \beta)\circ \tau^*\circ\pi_*\circ p_1^*\\
        =& p_{1, *}\circ (\cdot p_2^*D)\circ g_*\circ (\cdot i_1^*\beta)\circ i_1^* \circ \tau^*\circ\pi_*\circ p_1^*\\
        =& p_{1, *}\circ (\cdot p_2^*D)\circ g_*\circ (\cdot i_1^*\beta)\circ g^*\circ \iota^*\circ \pi^*\circ\pi_*\circ p_1^*\\
        =& p_{1, *}\circ (\cdot p_2^*D)\circ  (\cdot g_*(i_1^*\beta))\circ \iota^*\circ \pi^*\circ\pi_*\circ p_1^*\\
        =& p_{1, *}\circ (\cdot p_2^*D\cdot {[E]}^*(\beta))\circ \iota^*\circ 2(p_1^*+p_2^*+p_3^*)\\
        =& 2p_{1, *}\circ (\cdot p_2^*D\cdot {[E]}^*(\beta))\circ (p_1^*+2p_2^*)
    \end{split}
    \end{equation}
    where the first equality is the definition, the second one uses \eqref{eqn:CommonPart}, the fourth uses that $i_1^*$ is a ring homomorphism, the sixth uses the projection formula, the seventh uses $\pi^*\circ \pi_*=\sum_{\sigma\in \mathfrak{S}_3}\sigma$.
    
    Now note that $\cdot D$ is the zero map on $\h^2_{\tr}(S)$, hence $(\cdot p_2^*D)\circ p_2^*=0$ on $\h^2_{\tr}(S)$. Therefore, by the projection formula,
    \begin{equation}
        \Theta_{1, D}= 2p_{1, *}\circ (\cdot p_2^*D\cdot {[E]}^*(\beta))\circ p_1^*=\cdot 2p_{1, *}(p_2^*D\cdot {[E]}^*(\beta)).
    \end{equation}
    As $2p_{1, *}(p_2^*D\cdot {[E]}^*(\beta))$ is a multiple of the fundamental class of $S$, the correspondence $\Theta_{1,D}$ preserves $\h^2_{\tr}(S)$ and the multiple is the intersection number
    $2\int_{S^2} (\pt\times D)\cdot {[E]}^*(\beta)$ This proves \eqref{eqn:Theta1}. The proof for \eqref{eqn:Theta2} is almost the same as for \eqref{eqn:Theta1}, and we omit it here.

    To show \eqref{eqn:Lambda}, 
    \begin{equation}
        \begin{split}
        \Lambda=& [F]^*\circ \phi_*\circ f_*\circ f^*\circ \phi^{-1}_*\circ \tau^*\circ\pi_*\circ p_1^*\\
        =& [F]^*\circ (\cdot \beta)\circ \tau^*\circ\pi_*\circ p_1^*\\
        =& h_*\circ i_2^*\circ (\cdot \beta)\circ \tau^*\circ\pi_*\circ p_1^*\\ 
        =& h_*\circ (\cdot i_2^*\beta)\circ i_2^*\circ \tau^*\circ\pi_*\circ p_1^*\\ 
        =& h_*\circ (\cdot i_2^*\beta)\circ h^*\circ\delta^*\circ\pi^*\circ\pi_*\circ p_1^*\\  
        =& (\cdot h_*(i_2^*\beta))\circ \delta^*\circ 2(p_1^*+p_2^*+p_3^*)\\  
        =& (\cdot h_*(i_2^*\beta))\circ 6 \id\\
        =& \cdot 6 {[F]}^*(\beta),
        \end{split}
    \end{equation}
    where the first equality is the definition, the second is by \eqref{eqn:CommonPart}, the fourth uses that $i_2^*$ is a ring homomorphism, the sixth uses the projection formula.
    Finally, note that $6 {[F]}^*(\beta)$ is a multiple of the fundamental class of $S$, and that this multiple is the intersection number $6\int_S\pt\cdot {[F]}^*(\beta)$. The formula \eqref{eqn:Lambda} is proved.    
\end{proof}

\subsection{Proof of Theorem \ref{thm:SuperKummerAbelianMotive}}

    It suffices to prove the theorem for the very general super-Kummer K3 surfaces. Indeed, if any very general super-Kummer K3 surface $S$ is motivated by its corresponding abelian fourfold $B$, then by the standard Hilbert-scheme argument, the correspondences between $S$ and $B$ can be spread out to relative correspondences between any family of super-Kummer K3 surfaces and the corresponding family of  abelian fourfolds, then by applying the specialization map, we can conclude that any super-Kummer K3 surface is motivated by its corresponding abelian fourfold. Therefore without loss of generality , let $S$ be a very general super-Kummer K3 surface of Picard number 16, let $Z$ be a MRS-double cover birational to $S^{[3]}$ and consider the corresponding diagram \eqref{diag:CorrespondenceSBS}.
    
    By Lemma \ref{lemma:ScalarCorrespondences}, each of the six types of correspondences $\Gamma_{D,D'}, \Gamma_1, \Gamma_2, \Theta_{1,D}, \Theta_{2,D}, \Lambda$ defined in Section \ref{subsec:Correspondences} induces an endomorphism of $\h^2_{\tr}(S)$ which is a multiple of the identity and factors through the motive $\h(\widetilde{B})$. If at least one of these multiples is non-zero, 
    then there exists a split injection of $\h^2_{\tr}(S)$ into $\h(\widetilde{B})$, hence $\h(S)=\h^2_{\tr}(S)\oplus \1\oplus\1(-1)^{\oplus \rho}\oplus \1(-2)$ lies in the thick tensor subcategory of rational Chow motives generated by $\h(\widetilde{B})$. As $\h(\widetilde{B})\simeq \h(B)\oplus \1(-1)^{\oplus 256} \oplus \1(-2)^{\oplus 256}\oplus \1(-3)^{\oplus 256}$, we see that $S$ is Chow motivated by $B$. 

    Assume, for contradiction, that all the multiples appearing in Lemma \ref{lemma:ScalarCorrespondences} were zero. By \cite{deCataldoMigliorini2002}, we have the decomposition of rational Chow motives:
\begin{equation}
   (\tau_*, {[E]}^*, {[F]}^*) \colon  \h(S^{[3]}) \xrightarrow{\simeq } \h(S^{(3)}) \oplus \h(S\times S)(-1) \oplus \h(S)(-2).
\end{equation}
Taking realization, we obtain the following isomorphism of rational Hodge structures:
\begin{equation}
\label{eqn:DCM-Decomposition-Cohomology}
    (\tau_*, {[E]}^*, {[F]}^*) \colon H^4(S^{[3]}, \QQ)\xrightarrow{\simeq } H^4(S^{(3)}, \QQ) \oplus H^2(S\times S, \QQ)(-1) \oplus H^0(S, \QQ)(-2)
\end{equation}
Consider the cohomology class of $\beta\coloneqq \phi_*(f_*[\widetilde{B}])$ in $H^4(S^{[3]},\QQ)$, and denote its three components by $\beta_1\coloneqq \tau_*(\beta)\in  H^4(S^{(3)},\QQ)$, $\beta_2\coloneqq {[E]}^*(\beta)\in H^2(S\times S,\QQ)$ and $\beta_3\coloneqq {[F]}^*(\beta)\in H^0(S,\QQ)$, respectively.  
    
    The assumption by absurd that all the multiples appearing in Lemma \ref{lemma:ScalarCorrespondences} were zero can be formulated as saying that the image of the cohomology class of the algebraic cycle $\beta$ under \eqref{eqn:DCM-Decomposition-Cohomology} satisfies:
    \begin{enumerate}[label=(\roman*)]
        \item $\int_{S^3} (\pt\times D\times D')\cdot \pi^*(\beta_1)=0$ for any $D, D'\in \CH^1(S)$;
        \item $\int_{S^2}(\pt\times 1)\cdot \iota^*(\pi^*(\beta_1))=0$;
        \item $\int_{S^2}(1\times \pt)\cdot \iota^*(\pi^*(\beta_1))=0$;
        \item $\int_{S^2} (\pt\times D)\cdot (\beta_2)=0$ for any $D\in \CH^1(S)$;
        \item $\int_{S^2} (D\times \pt)\cdot (\beta_2)=0$ for any $D\in \CH^1(S)$;
        \item $\int_S\pt\cdot \beta_3=0$.
    \end{enumerate}
    We claim that these equations imply that $\beta_1$, $\beta_2$ and $\beta_3$ are all zero. 
    
    Clearly, (vi) implies $\beta_3=0$.
    Since $\beta$ is a Hodge class, $\beta_2\in H^2(S\times S,\QQ)$ is also a Hodge class, i.e., in the N\'eron--Severi group $\NS(S\times S)\simeq \NS(S)\oplus \NS(S)$. By the nondegeneracy of the intersection pairing on $\NS(S)$, (iv) and (v) together imply $\beta_2=0$.
    The cohomology class $\pi^*(\beta_1)$ is a Hodge class in the Hodge structure $H^4(S^3, \QQ)^{\mathfrak{S}_3}$, which can be decomposed as follows (all spaces are with $\QQ$-coefficients):
    \begin{equation}\label{eq:decompositionH^4} H^4(S^3, \QQ)^{\mathfrak{S}_3} \simeq H^4(S)\oplus \Sym^2H^2(S)\simeq  H^4(S)\oplus \Sym^2\NS(S)\oplus \Sym^2T(S)\oplus \NS(S)\otimes T(S),\end{equation} 
    where $T(S)\coloneqq H^2_{\tr}(S,\QQ)$ is the transcendental cohomology of $S$ and $\NS(S)$ denotes the rational N\'eron--Severi space of $S$. 
    We denote the four components of $\pi^*\beta_1$ for the decomposition \eqref{eq:decompositionH^4} by 
    $\gamma_1\in H^4(S)$, $\gamma_2\in \Sym^2\NS(S)$, $\gamma_3\in \Sym^2T(S)$ and $\gamma_4\in \NS(S)\otimes T(S)$. It remains to show that these components are all zero, for then $\beta_1=\frac{1}{6}\pi_*(\pi^*(\beta_1))=0$.
    
    Since $\NS(S)\otimes T(S)$ does not contain any non-zero Hodge class, $\gamma_4=0$. In the intersection number in (i), only $\gamma_2$ can contribute. By taking an orthogonal basis of $\NS(S)$, a direct computation shows that the vanishing condition (i) implies that $\gamma_2=0$. 
    In the intersection number in (iii), it is straightforward to check that the contribution from $\gamma_3$ is zero, so only $\gamma_1$ can contribute. Since $\gamma_1$ is a multiple of the class $\pt\times 1\times 1+1\times \pt\times 1+1\times 1\times \pt$, we see that $\iota^*(\gamma_1)$ is a multiple of $\pt\times 1+ 2\cdot 1\times \pt$. The vanishing condition in (iii) then implies that this multiple must be zero, and therefore $\gamma_1=0$.
    
    Finally, consider the Hodge class $\gamma_3 \in \Sym^2T(S)$. Since, by assumption, $S$ is a very general super-Kummer K3 surface, its Mumford--Tate group is isomorphic to that of some very general variety of $\mathrm{Kum}^3$-type of Picard rank $1$, and so the Mumford--Tate group of $S$ is the full group $\mathrm{GO}(T(S))$ of orthogonal similitudes of $T(S)$. As a result, the space of Hodge classes in $\Sym^2T(S)$ is $1$-dimensional, generated by the identity (seen as a Hodge endomorphism of $T(S)$). Therefore, $\gamma_3$ is a multiple of the identity of $T(S)$. Then the vanishing condition in (ii) implies that the trace of $\gamma_3$ is zero, so this multiple must be zero and $\gamma_3=0$.
    
    We have therefore proved the vanishing of $\beta_1$ (by the vanishing of $\gamma_i$, $i=1,2,3,4$), $\beta_2$ and $\beta_3$. Hence, $\beta=0$ by the decomposition \eqref{eqn:DCM-Decomposition-Cohomology}. But this contradicts Lemma \ref{lemma:betaNonzero}. The proof of Theorem \ref{thm:SuperKummerAbelianMotive} is complete.

\section{Derived categories}
\label{sec:DerivedCategories}
In this section, we establish relations between a $\mathrm{Kum}^3$-type manifold $K$ and the associated hyper-Kummer sixfold $Y_K$ at the level of derived categories. 

The construction of invariant Hilbert schemes plays a key role in what follows; we recall its definition following \cite{BridgelandKingReid}. For references, see the original work of Ito and Nakamura \cite{ItoNakamura96, ItoNakamura2000, Nakamura2001}, as well as Brion's account \cite{Brion-InvariantHilbert} in a broader context. Let $X$ be a smooth variety equipped with a faithful action of a finite group $G$. A $G$-\textit{cluster} is a $G$-invariant 0-dimensional subscheme $Z\subset X$ such that the global sections $\Gamma(\mathcal{O}_Z)$ with the natural $G$-action is a regular $G$-representation. In particular, a $G$-cluster is of length $|G|$ and its reduced support consists of a single $G$-orbit. A free $G$-orbit gives rise to a $G$-cluster, and $G$-clusters can be viewed as generalizations of such free $G$-orbits. The notion of families of $G$-clusters and isomorphisms between them is the natural one. 

It can be shown (cf.~\cite{Brion-InvariantHilbert}) that there exists a fine moduli space $G\operatorname{-Hilb}(X)$ of $G$-clusters called the $G$-Hilbert scheme of $X$, together with a universal family of $G$-clusters on $G\operatorname{-Hilb}(X)\times X$. In this article, we always take the main irreducible component $\Hilb^G(X)$ of $G\operatorname{-Hilb}(X)$, which is the closure of the locus of points parametrizing the free $G$-orbits, and we denote by $\mathcal{U}\subset \Hilb^G(X)\times X$ the universal family. We have a natural Hilbert--Chow morphism 
\begin{equation} \Hilb^G(X)\to X/G,\end{equation} 
which is sometimes a good candidate for a crepant resolution of singularities; see \cite{ItoNakamura96, ItoNakamura2000, Nakamura2001, BridgelandKingReid}.

The construction of invariant Hilbert schemes works more generally in the complex analytic setting, by using Douady spaces instead of Hilbert schemes, and the Hilbert--Chow morphism is replaced by the Douady--Barlet morphism. However, we will keep using the algebraic terminology even when treating analytic objects in the sequel. 

For later use, we record the simplest example of invariant Hilbert schemes. 

\begin{example}
\label{ex:Involution}
    Let $S$ be a smooth complex surface equipped with a biholomorphic involution $\iota$, inducing a faithful $\ZZ/2\ZZ$-action on $S$, such that there are only isolated fixed points $O\subset S$. Then the quotient surface $S/\iota$ has only isolated singular points, which are $A_1$ singularities. We still denote by $O$ the singular locus. Analytically locally at each point of $O$, the surface is isomorphic to the germ $(\CC^2, 0)$ with the involution $\iota=-1$.

    The singularity of $S/\iota$ is resolved by a single blow-up at the singular points, obtaining a crepant resolution $\Bl_O(S/\iota)$. Alternatively, one can blow up $S$ at $O$, over which we have a lifting $\tilde{\iota}$ of the involution with purely codimension-1 fixed locus, and the smooth quotient recovers the crepant resolution. We have the following commutative diagram
    \begin{equation} 
\begin{tikzcd}
    &\widetilde{S}=\Bl_O(S) \arrow[swap]{dl}{p'} \arrow{dr}{q'}&\\
   S  \arrow{dr}[swap]{q}&& \Bl_{O}(S/\iota)=\widetilde{S}/\tilde{\iota} \arrow{dl}{p}\\
    &S/\iota&
\end{tikzcd}
\end{equation}
which, as can be easily checked, is \textit{cartesian up to nilpotents}, i.e., $\widetilde{S} = (S\times_{S/\iota} \widetilde{S}/\widetilde{\iota})_{\mathrm{red}}$. In this situation, we have 
\begin{equation} 
\Hilb^{\ZZ/2\ZZ}(S)\cong\Bl_{O}(S/\iota),
\end{equation}
and, via this isomorphism, $\widetilde{S}\subset S\times \Bl_{O}(S/\iota)$ is identified with the universal $\ZZ/2\ZZ$-cluster.
\end{example}

This applies in particular to the classical Kummer construction. Let $A$ be an abelian surface and $\Km(A)$ the associated Kummer K3 surface; we then have the following   diagram which is cartesian up to nilpotents:
\begin{equation} 
\begin{tikzcd}
    &\widetilde{A} \arrow[swap]{dl}{p'} \arrow{dr}{q'}&\\
    A  \arrow{dr}[swap]{q}&& \Km(A) \arrow{dl}{p}\\
    &A/{-1}&
\end{tikzcd}
\end{equation}
where $p'\colon \widetilde{A}\to A$ is the blow-up of $A$ along the set $A[2]$ of points of order $2$ and $p\colon \Km(A)\to A/{-1}$ is the blow-up of $A/{-1}$ along its singular locus. Then we are in the situation of Example \ref{ex:Involution}, with $\ZZ/2\ZZ$ acting on $A$ via $\pm 1$. Hence, there is a natural isomorphism
\begin{equation}
    \Km(A)\cong \Hilb^{\ZZ/2\ZZ}(A),
\end{equation}
and, via this identification, the universal cluster $\mathcal{U}\subset  \Hilb^{\ZZ/2\ZZ}(A)\times A$ is identified with $\widetilde{A}\subset \Km(A)\times A$. Moreover, as an easy example of \cite{BridgelandKingReid}, the Fourier--Mukai functor 
\begin{equation}
\Phi=Rp'_* \circ {q'}^* \colon  \Db(\Km(A))\xrightarrow{\simeq} \Db([A/{(\ZZ/2\ZZ)}])
\end{equation}
induces an equivalence of triangulated categories, with quasi-inverse $\Psi=(Rq'_*\circ L{p'}^*)^{\ZZ/2\ZZ}$.

We establish the analogous result for the hyper-Kummer construction.
\begin{theorem}
\label{thm:HyperKummerIsBKR}
 Let $K$ be a projective hyper-K\"ahler variety of $\mathrm{Kum}^3$-type endowed with the canonical action of $G\cong (\ZZ/2\ZZ)^{5}$ (see Definition \ref{def:subgroupsAut0}). Let $Y_K$ be the associated hyper-Kummer $\mathrm{K}3^{[3]}$-type variety. Let the notation be as in Theorem~\ref{thm:hyperKummerConstruction}: 
 \begin{equation}
 \label{diag:HyperKummerConstruction}
     \begin{tikzcd}
		& \tilde{K} \arrow[swap]{dl}{p'} \arrow{dr}{q'} \\
		K \arrow{dr}[swap]{q} \arrow[dashed]{rr}{r} && Y_K \arrow{dl}{p} \\
		& K/G
	\end{tikzcd}
 \end{equation}
 Then:
 \begin{enumerate}
     \item[(i)] Diagram \eqref{diag:HyperKummerConstruction} is cartesian up to nilpotents.
     \item[(ii)] We have an isomorphism $Y_K\cong \Hilb^G(K)$.
     \item[(iii)] If we identify $Y_K$ with $\Hilb^G(K)$ via the previous isomorphism, then the universal $G$-cluster $\mathcal{U}$ is identified with $\widetilde{K}$.
 \end{enumerate}
\end{theorem}

\begin{corollary}
\label{cor:DerivedEquivalenceHyperKummer}
    In the above situation of hyper-Kummer construction, the Fourier--Mukai functor 
     \begin{equation}
\Phi=Rp'_* \circ {q'}^* \colon  \Db(Y_K)\xrightarrow{\simeq} \Db([K/G])
\end{equation}
induces an equivalence of triangulated categories, and the quasi-inverse is given by  $\Psi=(Rq'_*\circ L{p'}^*)^G$.
\end{corollary}
\begin{proof}
    A direct application of Bridgeland--King--Reid \cite[Theorem 1.1]{BridgelandKingReid} yields that $\Phi$ is an equivalence. Note that any symplectic resolution is semi-small, hence the numerical condition in \textit{loc.~cit.}~is satisfied, see \cite[Corollary 1.3]{BridgelandKingReid}.
    The quasi-inverse of $\Phi$, which is its left adjoint, can be computed directly using standard adjunctions.
\end{proof}

As a preparation for the proof of Theorem \ref{thm:HyperKummerIsBKR}, we need the following proposition on products of invariant Hilbert schemes. The result is possibly known to experts, but we provide full details here as we cannot find a reference.
\begin{proposition}
\label{prop:ProductGHilbert}
    Let $X_1$, $X_2$ be two smooth varieties endowed respectively with faithful actions of finite groups $G_1$ and $G_2$. Then the following holds:
    \begin{enumerate}
        \item[(i)] We have an isomorphism $\Hilb^{G_1}(X_1)\times \Hilb^{G_2}(X_2)\cong \Hilb^{G_1\times G_2}(X_1\times X_2)$.
        \item[(ii)] Let $\mathcal{U}_i\subset \Hilb^{G_i}(X_i)\times X_i$ be the universal $G_i$-clusters, for $i=1, 2$. Then under the previous identification, the universal $(G_1\times G_2)$-cluster in  $X_1\times X_2\times \Hilb^{G_1\times G_2}(X_1\times X_2)$ is identified with $\mathcal{U}_1\times \mathcal{U}_2$.
    \end{enumerate}
\end{proposition}
\begin{proof}
    As both sides of $(i)$ are fine moduli spaces, it is enough to show that they are isomorphic as functors of points. For simplicity of notation, we establish such an isomorphism on the level of geometric points; the same argument works for $S$-points for any scheme $S$ by reasoning relatively over $S$.
    
    First of all, we have the following natural morphism
    \begin{equation}
    \begin{split}
            \Hilb^{G_1}(X_1)\times \Hilb^{G_2}(X_2)&\to \Hilb^{G_1\times G_2}(X_1\times X_2)\\
                [Z_1]\times [Z_2] & \mapsto [Z_1\times Z_2],
    \end{split}
    \end{equation}
    which is well-defined, since the product of a $G_1$-cluster in $X_1$ and a $G_2$-cluster in $X_2$ is a $(G_1\times G_2)$-cluster in $X_1\times X_2$.

    We define a morphism in the other direction. Let $p_i\colon X_1\times X_2\to X_i$ be the $i$-th projection. For a given $(G_1\times G_2)$-cluster $Z\subset X_1\times X_2$, we would like to find clusters $Z_1\subset X_1$ and $Z_2\subset X_2$, such that $Z=Z_1\times Z_2$. Since the morphism $Z\to X_1$ is $G_2$-equivariant (where $X_1$ is equipped with the trivial $G_2$-action), the sheaf $p_{1,*}\mathcal{O}_Z$ inherits a $G_2$-action. We obtain a natural morphism
    \begin{equation} p_1^*\left((p_{1,*}\mathcal{O}_Z)^{G_2}\right) \hookrightarrow p_1^*p_{1,*}\mathcal{O}_Z\to \mathcal{O}_Z,\end{equation} 
    where the second map comes from the adjunction $(p_1^*, p_{1,*})$. Similarly, there is a canonical morphism $p_2^*\left((p_{2,*}\mathcal{O}_Z)^{G_1}\right)\to \mathcal{O}_Z$, and we get a morphism \begin{equation} f\colon (p_{1,*}\mathcal{O}_Z)^{G_2}\boxtimes (p_{2,*}\mathcal{O}_Z)^{G_1}= p_1^*\left((p_{1,*}\mathcal{O}_Z)^{G_2}\right) \otimes p_2^*\left((p_{2,*}\mathcal{O}_Z)^{G_1}\right)\to \mathcal{O}_Z.\end{equation}
    This morphism fits into the following commutative diagram as the bottom row:
    \begin{equation}
        \begin{tikzcd}
        \mathcal{O}_{X_1}\boxtimes \mathcal{O}_{X_{2}} \arrow{r}{\cong} \arrow[d] & \mathcal{O}_{X_1\times X_2} \arrow[twoheadrightarrow]{d}\\
            (p_{1,*}\mathcal{O}_Z)^{G_2}\boxtimes (p_{2,*}\mathcal{O}_Z)^{G_1} \arrow{r}{f} & \mathcal{O}_Z
        \end{tikzcd}
    \end{equation}
    The surjectivity of the right vertical map implies that $f$ is surjective. However, by comparing the length (both are $|G_1|\cdot|G_2|$), we see that $f$ is an isomorphism and the left vertical map is surjective, hence so are both tensor factors. Therefore, $(p_{1,*}\mathcal{O}_Z)^{G_2}$ is the structure sheaf of a subscheme $Z_1\subset X_1$, $(p_{2,*}\mathcal{O}_Z)^{G_1}$ gives rise to $Z_2\subset X_2$, and $f$ induces an identification $Z_1\times Z_2=Z$.
    By construction, $\Gamma(\mathcal{O}_{Z_1})=\Gamma(\mathcal{O}_Z)^{G_2}$ is isomorphic to the regular $G_1$-representation. Hence, $Z_1\subset X_1$ is a $G_1$-cluster, and similarly for $Z_2\subset X_2$.
   
    This defines a morphism 
    \begin{equation}
     \begin{split}
           \Hilb^{G_1\times G_2}(X_1\times X_2)&\to \Hilb^{G_1}(X_1)\times \Hilb^{G_2}(X_2) \\
                [Z] &\mapsto ([(p_{1,*}\mathcal{O}_Z)^{G_2}], [(p_{2,*}\mathcal{O}_Z)^{G_1}]).
    \end{split}
    \end{equation}
    One checks easily that it is inverse to the morphism constructed at the beginning of the proof. Moreover, the proof also shows the identification of the universal clusters (it is given by $f$).    
\end{proof}

\begin{proof}[Proof of Theorem \ref{thm:HyperKummerIsBKR}]
    We first prove $(i), (ii), (iii)$ analytically locally over $K/G$ and conclude by gluing canonical isomorphisms.
    
    By \cite[Proposition~4.2]{floccariKum3}, for any point $z\in K/G$, there exists an integer $j\in \{0, 1, 2, 3\}$ and an analytic neighborhood $U$ of $z$ in $K/G$, such that 
    \begin{itemize}
        \item  $(U, z)\cong (\CC^2/{-1}, 0)^j\times (\CC^2, 0)^{3-j}$ as germs;
        \item $q^{-1}(z)$ consists of a single $G$-orbit of cardinality $2^{5-j}$, with every point having stabilizer isomorphic to $(\ZZ/2\ZZ)^j$;
        \item As germs, $(q^{-1}(U), q^{-1}(z))\cong \coprod^{2^{5-j}} (\CC^2, 0)^j\times(\CC^2, 0)^{3-j}$, with the action of the stabilizer $(\ZZ/2\ZZ)^j$ given by the $(-1)$-involutions on the first $j$ factors.
    \end{itemize}
    Since $p$ is the blow-up of $K/G$ along its reduced singular locus, and $p'$ is the blow-up of $K$ along the union of the codimension-2 fixed loci, we have the following isomorphisms of germs (around the corresponding preimages of $z$):
    \begin{itemize}
        \item $p^{-1}(U)\cong \coprod^{2^{5-j}} (\Bl_0(\CC^2/{-1}))^j\times (\CC^2)^{3-j}$;
        \item $(q\circ p')^{-1}(U)\cong \coprod^{2^{5-j}} (\Bl_0\CC^2)^j\times (\CC^2)^{3-j}$, with the action of the stabilizer $(\ZZ/2\ZZ)^j$ given by the lifting of the $(-1)$-involutions on the first $j$ factors.
    \end{itemize}
    Here, we used \cite[Lemma 4.5]{floccariKum3}.

    Now it follows that, analytically locally over $K/G$, Diagram~\eqref{diag:HyperKummerConstruction} 
    is cartesian up to nilpotents. Moreover, by Proposition \ref{prop:ProductGHilbert}, which holds also in the analytic category, applied to (a power of) Example \ref{ex:Involution}, we have
    \begin{equation} p^{-1}(U)\cong \Hilb^{G}(q^{-1}(U)),\end{equation}
    and via this isomorphism, $(q\circ p')^{-1}(U)$ is identified with the universal $G$-cluster.
    Therefore, Diagram \eqref{diag:HyperKummerConstruction} is cartesian up to nilpotents, and by the universal property of invariant Hilbert schemes, we obtain a natural classifying morphism
    \begin{equation} Y_K\to \Hilb^G(K),\end{equation}
    which is birational and a local isomorphism by the local analysis above; hence, it is an isomorphism. 
    
    Finally, the local analysis shows that analytically locally over $K/G$, the reduced fiber product is identified with the universal $G$-cluster. Since this is a local property, the reduced fiber product $K\times_{K/G} Y_K$, which is $\widetilde{K}$ by $(i)$, is identified with the universal $G$-cluster.  
\end{proof}

By the same argument, we have the following results:
\begin{theorem}
\label{thm:HyperKummerIsBKR-ForM}
 Let the notation be as in Theorem~\ref{thm:HyperKummerIsBKR}. 
 For any $\sigma\in G\backslash G_1$, we consider the faithful action of $G_{\sigma}:=G/\langle\sigma\rangle$ on $W_{\sigma}$. Let $M_{\sigma}$ be the crepant resolution given as the blow-up of $W_{\sigma}/G_{\sigma}$ along its singular locus. 
 \begin{equation}
     \begin{tikzcd}
		& \widetilde{W_{\sigma}} \arrow[swap]{dl}{p'_{\sigma}} \arrow{dr}{q'_{\sigma}} \\
		W_{\sigma} \arrow{dr}[swap]{q_{\sigma}} \arrow[dashed]{rr}{r_{\sigma}} && M_{\sigma} \arrow{dl}{p_{\sigma}} \\
		& W_{\sigma}/G_{\sigma}
	\end{tikzcd}
 \end{equation}
 Then:
 \begin{enumerate}
     \item[(i)] the above diagram is cartesian up to nilpotents.
     \item[(ii)] We have an isomorphism $M_{\sigma}\cong \Hilb^{G_{\sigma}}(W_{\sigma})$.
     \item[(iii)] If we identify $M_{\sigma}$ with $\Hilb^{G_{\sigma}}(W_{\sigma})$ via the previous identification, then the universal $G_{\sigma}$-cluster is identified with $\widetilde{W_{\sigma}}$.
 \end{enumerate}
\end{theorem}

Applying \cite{BridgelandKingReid}, we get a derived equivalence:

\begin{corollary}
    Let the notation be as in Theorem \ref{thm:HyperKummerIsBKR-ForM}. For any $\sigma\in G\backslash G_1$, we have an equivalence of triangulated categories
    \begin{equation}
         \Phi_{\sigma}=Rp'_{\sigma,*} \circ {q'_{\sigma}}^* \colon \Db(M_\sigma)\cong \Db([W_\sigma/G_\sigma]),
    \end{equation}
whose quasi-inverse is given by  $\Psi_{\sigma}=(Rq'_{\sigma, *}\circ L{p'}_{\sigma}^*)^{G_{\sigma}}$.
\end{corollary}

\section{Motives and hyper-K\"ahler resolution conjecture}
\label{sec:Motives}

Given a projective hyper-K\"ahler variety $K$ of $\mathrm{Kum}^3$-type, let $Y_K$ be the associated hyper-Kummer variety, which is a projective variety of $\mathrm{K}3^{[3]}$-type. In this section, we compute the Chow motive of $Y_K$ by relating it to the Chow motives of $K$ and of the various companion hyper-K\"ahler varieties of lower dimensions naturally associated with the canonical action of $G$ on $K$, 
see Section \ref{sec:hyper-KummerConstruction} for the details.

Recall that $\Aut_0(K)\cong (\ZZ/4\ZZ)^4\rtimes \ZZ/2$ has normal subgroups $G\cong (\ZZ/2\ZZ)^5$ and $G_1\cong (\ZZ/2\ZZ)^4$ (see Definition \ref{def:subgroupsAut0}). For any $\sigma\in G\backslash G_1$, we define $G_{\sigma}\coloneqq G/\langle\sigma\rangle$, similarly, $G_{\sigma, \sigma'}\coloneqq G/\langle\sigma, \sigma'\rangle$.

\begin{theorem}
\label{thm:MotiveOfYK}
    Let $K$ be a projective hyper-K\"ahler variety of $\mathrm{Kum}^3$-type endowed with the canonical $G\cong (\ZZ/2\ZZ)^{5}$-action. Let $Y_K$ be the associated hyper-Kummer $\mathrm{K}3^{[3]}$-type variety. Then there is an isomorphism in the category $\CHM(K/G)_\QQ$ of relative rational Chow motives over $K/G$:
    \begin{align}
    \begin{split}
    \label{eqn:MotivicDecomposition-Relative}
        (Y_K\to K/G)&\cong \1_{K/G} \oplus \bigoplus_{\sigma\in G\backslash G_1}(W_{\sigma}/G\to K/G)(-1) \\&\oplus \bigoplus_{\sigma\neq \sigma'\in G\backslash G_1}(V_{\sigma, \sigma'}/G\to K/G)(-2)\oplus \bigoplus_{\substack{\sigma, \sigma', \sigma''\in G\backslash G_1\\ \text{ distinct}}} (Z_{\sigma, \sigma', \sigma''}/G\to K/G)(-3),
    \end{split}
    \end{align}        
    where the indices run through \emph{unordered} tuples of elements in $G\backslash G_1$. In particular, we have an isomorphism in the category of (absolute) rational Chow motives $\CHM_\QQ$:
    \begin{align}
    \label{eqn:MotivicDecomposition}
        \h(Y_K)\cong \h(K/G)\oplus \h(W/G)(-1)^{\oplus 16} \oplus \h(V/G)(-2)^{\oplus 120}\oplus \1(-3)^{\oplus 560}.
    \end{align}
\end{theorem}

The rational Chow motive $\mathfrak{h}(W_{\sigma}/G)$ (resp. $\mathfrak{h}(V_{\sigma,\sigma'}/G)$) does not depend on $\sigma$ (resp. on $\sigma,\sigma'$), by Remark \ref{rmk:isomorphismsW_sigma} (resp. Remark \ref{rmk:isogeniesV_sigma}); we denote it by $\mathfrak{h}(W/G)$ (resp. $V_{\sigma,\sigma'}/G$).

\begin{remark}
    Similarly, let $M$ denote $M_\sigma$ (the isomorphism class does not depend on $\sigma$). We can show that there is an isomorphism in $\CHM_\QQ$:
    \begin{equation}
                \h(M)\cong \h(W/G)\oplus \h(V/G)(-1)^{\oplus 15}\oplus \1(-2)^{\oplus 105}.
    \end{equation}
\end{remark}

\begin{remark} \label{rmk:isomorphismsW_sigma}
Recall from Remark \ref{rmk:actionAut_0} that the action of $\Aut_0(K)$ is transitive on the companion hyper-K\"ahler fourfolds $W_{\sigma}\subset K$, which are therefore all isomorphic to each other. Moreover, the induced isomorphism $W_{\sigma} \xrightarrow{ \ \sim \ } W_{\sigma'}$ is equivariant with respect to the $G$-action. Hence, the orbifolds $[W_{\sigma}/G]$ (respectively, the quotients $W_{\sigma}/G$), for ${\sigma\in G\setminus G_1}$, are all isomorphic to each other. Since the crepant resolution $M_{\sigma}$ is obtained by blowing-up the singular locus of $W_{\sigma}/G$, the $16$ hyper-K\"ahler fourfolds $M_{\sigma}$ associated with $K$ are also pairwise isomorphic.
\end{remark}

\begin{remark}\label{rmk:isogeniesV_sigma}
As for the companion K3 surfaces, while the $120$ hyper-Kummer K3 surfaces $S_{\sigma,\sigma'}$ are all isomorphic to each other (Corollary \ref{cor:Relation-YKMKSK}), the $120$ K3 surfaces $V_{\sigma,\sigma'}\subset K$ fall into $15$ isomorphism classes, in general. Nevertheless, these $120$ K3 surfaces are all isogenous to each other: in fact, $H^2_{\mathrm{tr}}(V_{\sigma,\sigma'},\QQ)$ is Hodge isometric to $H^2_{\mathrm{tr}}(K,\QQ)$, for any $\sigma\neq \sigma'$ in $G\setminus G_1$, by Proposition \ref{prop:cohomologyW&V}(ii). In particular, the $V_{\sigma, \sigma'} $ have isomorphic rational Chow motives by \cite{huybrechtsMotives}.
\end{remark} 

The proof of Theorem \ref{thm:MotiveOfYK} is based on the following Lemma \ref{lemma:AnalyzingFibers} below.
Note that the resolutions $Y_K\to K/G$ and $M_{\sigma}\to W_{\sigma}/G_{\sigma}$ are symplectic resolutions, hence are semismall. We can apply the results of de Cataldo--Migliorini \cite{dCM04} to compute the motive of $Y_K$ and $M_{\sigma}$. To this end, we need to analyse the strata of the resolutions as well as their fibers. 

The semismall resolution $p\colon Y_K\to K/G$ is stratified as follows:
\begin{equation}
    K/G= (K/G)^o \sqcup\bigsqcup_{\sigma\in G\backslash G_1}(W_{\sigma}/G)^o\sqcup\bigsqcup_{\sigma\neq \sigma'\in G\backslash G_1} (V_{\sigma, \sigma'}/G)^o\sqcup \bigsqcup_{\substack{\sigma, \sigma', \sigma''\in G\backslash G_1\\ \text{ distinct}}} (Z_{\sigma, \sigma', \sigma''}/G)
\end{equation}
Here $(-)^o$ denotes the smooth locus. Note that each $Z_{\sigma, \sigma', \sigma''}/G$ is a single point (Proposition \ref{prop:canonicalSubvarietiesII}). 

\begin{lemma}
\label{lemma:AnalyzingFibers}
The fibers of $p\colon Y_K\to K/G$ are described as follows:
    \begin{itemize}
    \item $p$ is an isomorphism over $(K/G)^o$;
    \item $p$ is a $\PP^1$-bundle over $(W_{\sigma}/G)^o$, for each $\sigma\in G\backslash G_1$;
    \item $p$ is a $\PP^1\times \PP^1$-bundle over $(V_{\sigma,\sigma'}/G)^o$;
    \item The fiber of $p$ over each point $Z_{\sigma, \sigma', \sigma''}/G$ is $\PP^1\times \PP^1\times \PP^1$.
    \end{itemize}
In particular, all the strata of $p$ are relevant with irreducible fibers, hence with trivial local systems. 
\end{lemma}
\begin{proof} 
This is essentially \cite[Proposition 4.2]{floccariKum3}.
The first assertion is clear. 
The fibers of $p$ over points $(W_{\sigma}/G)^o$ are $\PP^1$ since $K/G$ has transversal $A_1$-singularity along $(W_{\sigma}/G)^o$.
Around a point $P\in (V_{\sigma, \sigma'}/G)^o$, $K/G$ is isomorphic to the germ of $(\CC^2/{-1})^2\times \CC^2$ at the origin in such a way that the germ at $P$ of $(V_{\sigma, \sigma'}/G)^o$ is identified with the germ $(0,0)\times \CC^2$. Hence the resolution $p$, which is the blow-up along the singular locus, is analytically locally isomorphic to the following germ
\begin{equation}
    (\Bl_0(\CC^2/{-1}))^2 \times \CC^2 \to (\CC^2/{-1})^2\times \CC^2,
\end{equation}
and it is clear that the fiber of $p$ over a point $(0,0)\times \CC^2$ is $\PP^1\times \PP^1$.
The proof for the fiber over $Z_{\sigma, \sigma', \sigma''}/G$ is similar, by identifying the blow-up with $(\Bl_0(\CC^2/{-1}))^3\to (\CC^2/{-1})^3$ as germs.
\end{proof}

\begin{proof}[{Proof of Theorem \ref{thm:MotiveOfYK}}]
By Lemma \ref{lemma:AnalyzingFibers}, we can apply \cite[Corollary 2.3.9]{dCM04} to obtain the formula \eqref{eqn:MotivicDecomposition-Relative} for relative motives. Pushing forward to a point, and noting that, by Remarks \ref{rmk:isomorphismsW_sigma} and \ref{rmk:isogeniesV_sigma}, the isomorphism class of $\h(W_{\sigma}/G)$ (resp. $\h(V_{\sigma, \sigma'}/G)$ ) does not depend on $\sigma$ (resp. $\sigma$ and $\sigma'$), we obtain the isomorphism \eqref{eqn:MotivicDecomposition} of Chow motives.
\end{proof}

\subsection{Relation to orbifold motives}
We can formulate Theorem \ref{thm:MotiveOfYK} in a more concise form by relating it to orbifold motives. For the definition of orbifold Chow motives, we refer to \cite{FuVialTian-MHRC} (based on constructions in \cite{Abramovich-Graber-Vistoli-2002}, \cite{Abramovich-Graber-Vistoli-2008}, \cite{Jarvis-Kaufmann-Kimura}).
\begin{corollary}
\label{cor:MHRC-additive}
    Let the notation be as in Theorem \ref{thm:MotiveOfYK}. We have isomorphisms in the category of rational Chow motives $\CHM_\QQ$:
    \begin{align}
                \h(Y_K)&\cong \h_{\orb}([K/G]);\\
                \h(M_{\sigma})&\cong \h_{\orb}([W_\sigma/G_{\sigma}]), \quad \forall \sigma\in G\backslash G_1,
    \end{align}
    where $\h_{\orb}$ denotes the orbifold Chow motive of a Deligne--Mumford stack.
\end{corollary}

\begin{proof}
By definition of orbifold Chow motive, noting that $G$ is an abelian group and $G$ preserves (globally) each $W_{\sigma}$, $V_{\sigma', \sigma''}$, and $Z_{\sigma, \sigma', \sigma''}$, plugging in the information on fixed loci, we have that
\begin{align}
\label{eqn:OrbMotiveOfK/G}
\begin{split}
     \h_{\orb}([K/G])&\cong \h(K/G)\oplus
        \bigoplus_{\sigma\in G_1\backslash\{\id\}}\bigoplus_{\substack{\sigma_1+\sigma_2=\sigma\\ \sigma_i\in G\backslash G_1}} \h(V_{\sigma_1, \sigma_2}/G)(-2)\\
        &\oplus\bigoplus_{\sigma\in G\backslash G_1}\left(\h(W_{\sigma}/G)(-1)\oplus \bigoplus_{\substack{\sigma_1+\sigma_2+\sigma_3=\sigma\\
    \sigma_i \text{ distinct in } G\backslash G_1}}\h(Z_{\sigma_1, \sigma_2, \sigma_3}/G)(-3)\right),\\
    &\cong \h(K/G)\oplus \bigoplus_{\sigma_1\neq  \sigma_2\in G\backslash G_1}\h(V_{\sigma_1, \sigma_2}/G)(-2) \oplus \bigoplus_{\sigma\in G\backslash G_1}\h(W_{\sigma}/G)(-1)\oplus \1(-3)^{560} \\
    &\cong \h(K/G)\oplus  \h(W/G)(-1)^{\oplus 16} \oplus \h(V/G)(-2)^{\oplus  120}\oplus \1(-3)^{\oplus 560}.
\end{split}
\end{align}

Similarly, for any $\sigma\in G\backslash G_1$,
\begin{align}
\begin{split}
      \h_{\orb}([W_{\sigma}/G_{\sigma}])&\cong \h(W_{\sigma}/G_{\sigma})\oplus
        \bigoplus_{\sigma'\in G_1\backslash\{\id\}}\left( \h(V_{\sigma, \sigma+\sigma'}/G)(-1)\oplus \bigoplus_{\substack{\sigma_1+\sigma_2=\sigma'\\\sigma_i\neq \sigma}}\h(Z_{\sigma_1, \sigma_2, \sigma}/G)(-2)\right)\\
        &\cong \h(W/G)\oplus \h(V/G)(-1)^{\oplus 15}\oplus \1(-2)^{\oplus 105}.  
\end{split}
\end{align}
 The result follows from Theorem \ref{thm:MotiveOfYK} by comparing \eqref{eqn:OrbMotiveOfK/G} and \eqref{eqn:MotivicDecomposition}, similarly for $[W_\sigma/G_\sigma]$.
\end{proof}

\section{Beauville's weak splitting conjecture}
\label{sec:BeauvilleConjecture}
As a striking consequence of the hyper-Kummer construction, we can prove Beauville's weak splitting conjecture \cite{Beauville-SplittingBBFiltration} for all $\mathrm{Kum}^3$-type hyper-K\"ahler varieties.

Let $X$ be a projective hyper-K\"ahler variety. Consider the subalgebra $\mathrm{R}^{\bullet}(X) \subset \mathrm{CH}^{\bullet}(X)_{\QQ}$ generated by divisor classes. 
The following weak splitting conjecture was proposed by Beauville in \cite{Beauville-SplittingBBFiltration}:
\begin{conjecture}\label{conj:beauville}
    For any hyper-K\"ahler variety $X$, the cycle class map $\mathrm{cl}\colon \mathrm{CH}^{\bullet}(X)_{\QQ}\to H^{2\bullet}(X,\QQ)$ induces an inclusion $\mathrm{R}^{\bullet}(X) \hookrightarrow H^{2\bullet}(X,\QQ)$.
\end{conjecture}

For K3 surfaces, the conjecture is a reformulation of results established by Beauville and Voisin in their influential paper \cite{Beauville-Voisin}. 
There are now many cases in which the conjecture is verified. For example, it was proven by the second author for the generalized Kummer variety $K^n(A)$ on an abelian surface $A$ (\cite{Fu-BeauvilleVoisinKummer}). 
For any hyper-K\"ahler variety $X$ of one of the known deformation types, the conjecture has been proven by Riess \cite{Riess2016WeakSplittingProperty} under the assumption that $X$ admits a rational Lagrangian fibration. We will use her result together with the hyper-Kummer construction to prove the following result. 
\begin{theorem}
\label{thm:BeauvilleConjecture}
    Let $K$ be a projective hyper-K\"ahler sixfold of $\mathrm{Kum}^3$-type. Then Beauville's weak splitting conjecture holds for $K$.
\end{theorem}
\begin{proof}
    By the aforementioned work of Riess \cite{Riess2016WeakSplittingProperty}, Beauville's conjecture holds for the hyper-Kummer $\mathrm{K}3^{[3]}$-variety $Y_K$ associated to $K$. Indeed, $Y_K$ has Picard rank at least $17$, and therefore there exists an isotropic vector in $\NS(Y_K)$ by Meyer's theorem \cite{meyers}. By \cite{BM14a}, this implies the existence of a Lagrangian fibration $Y'\to \PP^3$ for a birational hyper-K\"ahler model $Y'$ of $Y_K$, and so Conjecture \ref{conj:beauville} holds for $Y'$ by \cite{Riess2016WeakSplittingProperty}, hence also holds for $Y_K$ by \cite{Rie14}.

    Recall that $Y_K$ is the resolution of the quotient $K/G$. The rational map $r\colon K\dashrightarrow Y_K$ is resolved by a diagram 
    \begin{equation}
    \begin{tikzcd}
        & \widetilde{K} \arrow{dr}{q} \arrow[swap]{ld}{p} \\
        K \arrow[dashed]{rr}{r} && Y_K
    \end{tikzcd}
    \end{equation}
    where $p$ is a blow-up and $q$ is the quotient by the induced action of $G$ on $\widetilde{K}$. The pull-back map 
    \begin{equation} p^*\colon \mathrm{CH}^{\bullet}(K)_{\QQ}\to \mathrm{CH}^{\bullet}(\widetilde{K})_{\QQ}\end{equation} 
    is an embedding of $\QQ$-algebras. 
    Moreover, the $G$-invariant subalgebra $(\mathrm{CH}^{\bullet}(\widetilde{K})_{\QQ})^G$ is identified with $\mathrm{CH}^{\bullet}(Y_K)_{\QQ}$. These identifications are compatible with the identification of cohomology groups $H^{\bullet}(\widetilde{K},\QQ)^G \cong H^{\bullet}(Y_K,\QQ)$, in the sense that we have a commutative diagram
    \[ 
    \begin{tikzcd}
    (\mathrm{CH}^{\bullet}(K)_{\QQ})^G \arrow[hook]{r}{p^*} \arrow{d}{\mathrm{cl}}& 
    (\mathrm{CH}^{\bullet}(\widetilde{K})_{\QQ})^G \arrow{r}{\sim} \arrow{d}{\mathrm{cl}} & \mathrm{CH}^{\bullet}(Y_K)_{\QQ} \arrow{d}{\mathrm{cl}}
    \\
    H^{2\bullet}(K,\QQ)^G \arrow[hook]{r}{p^*} & H^{2\bullet}(\widetilde{K},\QQ)^G \arrow{r}{\sim } & H^{2\bullet}(Y_K,\QQ).
    \end{tikzcd}
    \]
    Moreover, 
    the top row restricts to \begin{equation}  \begin{tikzcd}
    \mathrm{R}^{\bullet}(K)^G \arrow[hook]{r}{p^*} & \mathrm{R}^{\bullet}(\widetilde{K})^G \arrow{r}{\sim} & \mathrm{R}^{\bullet}(Y_K).
    \end{tikzcd}
    \end{equation}

    The group $G$ acts trivially on the second cohomology of $K$ and so on $\mathrm{CH}^1(K)_{\QQ}$, since $\Pic^0(K)=0$; therefore, $G$ acts trivially on the subalgebra of $\CH^{\bullet}(K)_{\QQ}$ generated by $\CH^1(K)_{\QQ}$, which is precisely $R^{\bullet}(K)$. Therefore, $R^{\bullet}(K) = R^{\bullet}(K)^G$ injects into $R^{\bullet}(Y_K)$. Since Beauville's weak splitting conjecture holds for $Y_K$, we deduce that it holds as well for $K$.    
\end{proof}

\section{Homological motives, Hodge and Tate conjectures}
\label{sec:HodgeTateConjecture}
The Kuga--Satake construction (\cite{KugaSatake, deligne1971conjecture})  produces an abelian variety from the Hodge structure on the second cohomology of any hyper-K\"ahler variety.  
It is generally expected that the motive of any hyper-K\"ahler variety is controlled by the motive of the associated Kuga--Satake abelian variety. We make this precise via the following conjecture.

\begin{conjecture}
\label{conj:MotivationByKS}
    Let $X$ be a projective hyper-K\"ahler variety, and let $\mathrm{KS}(X)$ be the associated Kuga--Satake abelian variety. Then $X$ is motivated by $\mathrm{KS}(X)$, i.e., the motive of $X$ belongs to the thick tensor subcategory of motives generated by the motive of $\mathrm{KS}(X)$.  
\end{conjecture}

The conjecture depends on the choice of a category of motives.  If we use Andr\'e motives \cite{andre1996Motives}, Conjecture \ref{conj:MotivationByKS} is known for K3 surfaces by Andr\'e \cite{Andre1996}, and for all hyper-K\"ahler varieties of known deformation types by \cite{soldatenkov19} and \cite{FFZ}. 
The strongest version of Conjecture \ref{conj:MotivationByKS} is in the category of Chow motives; Theorem \ref{thm:SuperKummerAbelianMotive} confirms the conjecture for the rational Chow motive of super-Kummer K3 surfaces.
In this section, we instead investigate Conjecture \ref{conj:MotivationByKS} in the category of homological motives. Although weaker than its Chow-theoretic version, the conjecture is already wide open in this homological setting. In the sequel, $\Mot$ is the category of rational homological motives with respect to singular cohomology, and $\mathsf h$ is the contravariant functor sending a smooth projective variety to its homological motive. 

We will show that Conjecture \ref{conj:MotivationByKS} holds for the homological motives of all of the varieties involved in the hyper-Kummer construction and the Mongardi--Rapagnetta--Sacc\`a double cover construction. In fact, we have the following result, which is based on several previous works. 

\begin{theorem}[\cite{floccariKum3, floccariVaresco, floccariHCKum3, FloccariFu-HodgeConjectureWeilFourfolds}] \label{thm:motivationbyKSknownCases}
    Consider the following hyper-K\"ahler varieties:
    \begin{enumerate}[label=(\roman*)]
    \item any $\mathrm{K}3$ surface or $\mathrm{K}3^{[n]}$-type variety 
    satisfying the following condition: 
    there exists an isometric embedding of rational quadratic spaces $H^2_{\mathrm{tr}}(X,\QQ)\hookrightarrow \mathrm{U}_{\QQ}^{\oplus 3}\oplus \langle -a\rangle_{\QQ}$,
    for some positive integer $a$.
    \item any variety of $\mathrm{Kum}^2$-type; 
    \item any variety of $\mathrm{Kum}^3$-type;
    \item any $\mathrm{OG}6$-resolution.
    \end{enumerate}
    Then the Kuga--Satake variety $\mathrm{KS}(X)$ of $X$ is isogenous to a power of an abelian fourfold of Weil type with discriminant $1$, and Conjecture \ref{conj:MotivationByKS} holds for the homological motive of $X$, that is, $\mathsf h(X)$ belongs to the thick tensor subcategory of homological motives generated by the motive of $\mathrm{KS}(X)$.  
\end{theorem}
\begin{proof}
    (i). For a K3 surface $S$ satisfying the condition in (i), the conclusion follows from \cite[Theorem 5.11]{floccari25}. If $X$ is a $\mathrm{K}3^{[n]}$-type variety as in (i), then $X$ is birational to a smooth and projective moduli space on a K3 surface $S$ as above, by \cite{Addington2016}. The Kuga--Satake variety of $X$ is isogenous to that of $S$. 
    By \cite{Rie14, Bue18}, $X$ is motivated by $S$, and therefore Conjecture \ref{conj:MotivationByKS} holds for $X$ as well.

    (ii). This is a consequence of \cite{floccariVaresco}.
    Let $K$ be any variety of $\mathrm{Kum}^n$-type and let $A$ be its intermediate Jacobian. By \cite{O'G21, markman2019monodromy, voisinfootnotes}, $A$ is an abelian fourfold of Weil type with discriminant $1$, the Kuga--Satake variety of $K$ is isogenous to $A^4$, and the Kuga--Satake--Hodge conjecture holds for $K$. Moreover, there exists an algebraic cycle $\gamma$ inducing the canonical isomorphism of Hodge structures between $H^1(A,\QQ)$ and $H^3(K,\QQ)$.

    Assume now that $K$ is of $\mathrm{Kum}^2$-type. 
    The standard conjectures hold for $K$ by \cite{foster}.
    Hence, the homological motive of $K$ decomposes into K\"unneth components $\mathsf{h} (K) = \bigoplus_{i=0}^{8} \mathsf{h}^i(K)$. Further, by work of Jannsen \cite{Jan92} and Andr\'e \cite{andre1996Motives}, the tensor subcategory of homological motives generated by $\mathsf{h}(K)$ is abelian and semisimple. See \cite{Ara06} and \cite[\S2]{floccariVaresco} for more details and consequences of this result. Since the Kuga--Satake--Hodge conjecture holds for $K$, the degree-$2$ component $\mathsf{h}^2(K)$ lies in the tensor subcategory $\langle \mathsf{h}(\mathrm{KS}(K))\rangle_{\mathrm{Mot}}^{\otimes}$ generated by the motive of the Kuga--Satake abelian variety. Moreover, thanks to the existence of the cycle $\gamma$, the K\"unneth component $\mathsf{h}^3(K)$ also belongs to this category.
    We split the motive of $K$ into the sum of the even part $\mathsf{h}^+(K)\coloneqq \bigoplus_i \mathsf{h}^{2i}(K)$ and the odd part $\mathsf{h}^-(K)\coloneqq\bigoplus_{i} \mathsf{h}^{2i+1}(K)$. By \cite[Lemma 4.2]{floccariVaresco}, there exists a submotive $\mathsf{a}_2(K)$ of $\mathsf{h}^+(K)$ with realization the subalgebra generated by $H^2(K,\QQ)$; moreover, $\mathsf{a}_2(K)$ belongs to $\langle \mathsf{h}^2(K)\rangle^{\otimes}_{\mathsf{Mot}}$. By \cite{hassettTschinkel}, we have $\mathsf{h}^+(K) =\mathsf{a}_2(K) \oplus \mathsf{Q}^{\oplus 80}(-2)$ in $\mathsf{Mot}$, and therefore $\mathsf{h}^+(K)$ belongs to $\langle \mathsf{h}(\mathrm{KS}(K))\rangle_{\mathsf{Mot}}^{\otimes}$.
Finally, also $\mathsf{h}^-(K)$ lies in $\langle \mathsf{h}(\mathrm{KS}(K))\rangle_{\mathsf{Mot}}^{\otimes}$, being the image of the cup-product map $\mathsf{a}_2(K) \otimes \mathsf{h}^3(K) \to \mathsf{h}(K)$ which is a morphism of motives in the abelian and semisimple category of motives generated by the motives of $A$ and $K$. 

(iii). For a variety of $\mathrm{Kum}^3$-type, the conclusion is a corollary of the results from \cite{floccariHCKum3}, which crucially relies on the hyper-Kummer construction. 
Recall that there is a natural action of the group $\Aut_0(K) \cong (\ZZ/4\ZZ)^4\ltimes \ZZ/2\ZZ$ of automorphisms acting trivially on the second cohomology, with the subgroups $\Gamma\cong (\ZZ/4\ZZ)^4$ and $G\cong (\ZZ/2\ZZ)^5$ of automorphisms acting trivially also on the third and fourth cohomology, respectively. 
By \cite[Proof of Theorem 1.1]{floccariHCKum3}, there is a decomposition 
\begin{equation} \label{eq:motive_decomposition}
    \mathsf{h}(K) = \mathsf{h}(K)^G \oplus \mathsf{h}^-(K) \oplus \mathsf{Q}^{\oplus 240}(-3)
\end{equation}
of the motive of $K$, where $\mathsf{h}(K)^G$ is the $G$-invariant part of the motive of $K$, the summand $\mathsf{h}^-(K)$ is the odd part of the motive, with realization the odd cohomology, and the remaining summand is of Hodge--Tate type. To define $\mathsf{h}^{-}(K)$ as a motive, we take the complement of $(\mathsf{h}(K)^\Gamma)^G$ in $\mathsf{h}(K)^{\Gamma}$; then $\mathsf{h}^{-}$ is a well-defined direct summand of $\mathsf{h}(K)$, with realization the odd cohomology of $K$.
Now $\mathsf{h}(K)^G$ is a direct summand of the motive $\mathsf{h}(Y_K)$ of the hyper-Kummer variety associated with $K$; since the standard conjectures hold for $Y_K$ (\cite{CM13}), they hold for the motive $\mathsf{h}(K)^G$. The standard conjectures hold for $\mathsf{h}^-(K)$ by \cite{foster}. We conclude that the standard conjectures hold for $K$. 
Hence, the subcategory of homological motives generated by $\mathsf{h}(K)$ is abelian and semisimple. 

The Kuga--Satake variety of $Y_K$ is isogenous to a power of $\mathrm{KS}(K)$. Moreover, $Y_K$ satisfies the condition in (i). Thus, since $\mathsf{h}(K)^G$ is a direct summand of the motive of the hyper-Kummer variety $Y_K$, it lies in $\langle \mathsf{h}(\mathrm{KS}(K))\rangle_{\mathsf{Mot}}^{\otimes}$ by $(i)$. In particular, the K\"unneth component $\mathsf{h}^2(K)$ lies in $\langle \mathsf{h}(\mathrm{KS}(K))\rangle_{\mathsf{Mot}}^{\otimes}$. 
By \cite{floccariVaresco}, there exists a submotive $\mathsf{a}_2(K)$ of $\mathsf{h}(K)$ with realization the subalgebra generated by $H^2(K,\QQ)$ and which lies in $\langle \mathsf{h}^2(K)\rangle_{\mathsf{Mot}}^{\otimes}$. 
Since the odd cohomology of $K$ is generated by $H^3(K,\QQ)$ as a module over the subalgebra generated by $H^2(K,\QQ)$, as before $\mathsf{h}^-(K)$ is the image of the cup-product morphism $\mathsf{a}_2(K)\otimes \mathsf{h}^3(K)\to \mathsf{h}(K)$; hence, $\mathsf{h}^-(K)\in \langle \mathsf{h}(\mathrm{KS}(K))\rangle_{\mathsf{Mot}}^{\otimes}$. Therefore, by the decomposition \eqref{eq:motive_decomposition}, we conclude that $\mathsf{h}(K)$ belongs to the subcategory $\langle \mathsf{h}(\mathrm{KS}(K))\rangle^{\otimes}_{\mathsf{Mot}}$ of homological motives.

(iv). This is proved in our paper \cite[Theorem 5.7]{FloccariFu-HodgeConjectureWeilFourfolds}.
\end{proof}

\begin{corollary}\label{cor:HodgeConjecture}
\begin{enumerate}[label=(\roman*)]
    \item Let $K$ be a hyper-K\"ahler variety of $\mathrm{Kum}^3$-type. Then Conjecture \ref{conj:MotivationByKS} holds for the homological motive of $K$ and any of the companion hyper-K\"ahler manifolds $W_{\sigma}$, $V_{\sigma,\sigma'}$, $Y_K$, $M_{\sigma}$ and $S_K$.
    \item Let $X$ be a singular $\mathrm{OG}6$-variety. Then Conjecture \ref{conj:MotivationByKS} holds for the homological motive of any MRS double cover $Z_X$ of $X$, and any crepant resolution of $X$.
\end{enumerate}
\end{corollary}
\begin{proof}
    The conclusion for a variety of $\mathrm{Kum}^3$-type and any OG6-resolution follows directly from Theorem \ref{thm:motivationbyKSknownCases}. By Proposition \ref{prop:cohomologyY_K} and Proposition \ref{prop:cohomologyM&S}, the transcendental lattices of $Y_K$, $M_{\sigma}$ and $S_{K}$ are all Hodge isometric to $\widehat{H}_{\mathrm{tr}}^2(K,\ZZ)(2)$; with $\QQ$-coefficients, they embed in $\mathrm{U}_{\QQ}^{\oplus 3} \oplus \langle -1\rangle_{\QQ}$. Similarly, by Proposition \ref{prop:cohomologyW&V},  $H^2_{\mathrm{tr}}(W_{\sigma},\ZZ)$ is Hodge isometric to $H^2_{\mathrm{tr}}(K,\ZZ)(2)$, while the K3 surfaces $V_{\sigma,\sigma'}$ have $H^2_{\mathrm{tr}}(V_{\sigma,\sigma'},\QQ)$ Hodge isometric to $H^2_{\mathrm{tr}}(K,\QQ)$. Hence, $H^2_{\mathrm{tr}}(V_{\sigma,\sigma'},\QQ)$ admits an isometric embedding in $\mathrm{U}_{\QQ}^{\oplus 3}\oplus \langle -2\rangle_{\QQ}$.
    By Proposition \ref{prop:cohomologyMRS}, an MRS double cover $Z_X$ of a singular OG6-variety $X$ has transcendental lattice $H^2_{\mathrm{tr}}(Z_X,\ZZ)$ Hodge isometric to $H^2_{\mathrm{tr}}(X,\ZZ)(2)$. Thus $H^{2}_{\mathrm{tr}}(X,\QQ)$ admits an isometric embedding in $\Lambda_{\mathrm{OG}6^{\mathrm{sing}}}(2)\otimes \QQ \cong \mathrm{U}_{\QQ}^{\oplus 3} \oplus \langle -1\rangle_{\QQ}$. Hence, the corollary follows from Theorem \ref{thm:motivationbyKSknownCases}.
\end{proof}

\subsection{Hodge conjecture}
Theorem \ref{thm:motivationbyKSknownCases} has very strong consequences for the Hodge conjecture of the varieties considered in this paper. Following Varesco \cite[Definition 1.2]{Varesco-HodgeSimilarities}, we say that two hyper-K\"ahler varieties $X$ and $X'$ are \textit{Hodge similar} if there is a Hodge similarity between $H^2_{\mathrm{tr}}(X,\QQ)$  and $H^2_{\mathrm{tr}}(X',\QQ)$, that is, an isomorphism of Hodge structures which rescales the form by some scalar factor. 
 
\begin{corollary}
\label{cor:HCPower}
Let $X_1,\cdots, X_n$ be any collection of hyper-K\"ahler varieties as in Theorem \ref{thm:motivationbyKSknownCases} such that the $X_i$ are Hodge similar to each other. Then the Hodge conjecture holds for all products of powers of $X_1, \cdots, X_n$ as well as of their Kuga--Satake varieties.
In particular, the Hodge conjecture holds for all powers of any hyper-K\"ahler variety $X$ as in Theorem~\ref{thm:motivationbyKSknownCases}.
\end{corollary}
\begin{proof}
By \cite[Proposition 0.2]{Varesco-HodgeSimilarities}, all the Kuga--Satake varieties $\mathrm{KS}(X_i)$ have the same isogeny factors. In fact, each of the Kuga--Satake varieties $\mathrm{KS}(X_i)$ is isogenous to some power of the same abelian fourfold $A $, which is of Weil type with discriminant $1$. 
Therefore, the tensor subcategories $\langle \mathsf h(\mathrm{KS}(X_i)) \rangle_{\mathsf{Mot}}^{\otimes}$ of $\Mot$ all coincide with $\langle \mathsf h(A) \rangle_{\mathsf{Mot}}^{\otimes}$.
The Hodge conjecture has been established for any abelian fourfold of Weil type with discriminant $1$ \cite{markman2019monodromy} and all its powers \cite{floccari25} (see also \cite{FloccariFu-HodgeConjectureWeilFourfolds} for an alternative proof); equivalently, the Hodge realization functor $\mathsf R\colon \langle \mathsf h(A) \rangle_{\mathsf{Mot}}^{\otimes} \to \mathsf{HS}_{\QQ}$ is full. Let $Y$ be a product of powers of the $X_i$ and their Kuga--Satake varieties as in the statement. Theorem \ref{thm:motivationbyKSknownCases} implies that $\mathsf h(Y)\in \langle \mathsf h(A) \rangle_{\mathsf{Mot}}^{\otimes}$.
Hence, the restriction of the functor $\mathsf R$ to the subcategory $\langle \mathsf h(Y)\rangle_{\mathsf{Mot}}^{\otimes}$ is full; in other words, the Hodge conjecture holds for (all powers of) $Y$.
\end{proof}

\begin{remark}
Corollary \ref{cor:HCPower} implies the Shafarevich conjecture for varieties in Theorem \ref{thm:motivationbyKSknownCases}:  any Hodge similitude $\phi\colon H_{\mathrm{tr}}^2(X,\QQ)\to H_{\mathrm{tr}}^2(Y,\QQ)$ between hyper-K\"ahler varieties $X,Y$ as in Theorem \ref{thm:motivationbyKSknownCases} is algebraic. 
Some cases of Shafarevich conjecture were proven in \cite{Buskin, huybrechtsMotives, markmanRational, Varesco-HodgeSimilarities, floccari25, OGrady2026HKManifolds}.
\end{remark}


As another interesting consequence of Corollary \ref{cor:HCPower}, we observe the following Torelli-type result:

\begin{corollary}
\label{cor:HomologicalMotiveIso}
Let $X$ and $X'$ be deformation equivalent projective hyper-K\"ahler manifolds such that $H^2(X, \QQ)$ and $H^2(X', \QQ)$  are Hodge isometric. 
Assume that $X$ and $X'$ are both as in Theorem \ref{thm:motivationbyKSknownCases}.
Then their homological motives are isomorphic as Frobenius algebra objects in $\Mot$:
    \begin{equation}
        (\mathsf h (X), \cup)\simeq (\mathsf h (X'), \cup).
    \end{equation}   
\end{corollary}
\begin{proof}
We can apply \cite{solda19} or \cite[Theorem 4.7]{floccari2020} to obtain an isomorphism of rational cohomology algebras $f\colon H^*(X, \QQ)\xrightarrow{\simeq} H^*(X', \QQ)$ such that it moreover respects the class of a point, hence is Frobenius. Then the Hodge conjecture for $X\times X'$, established in Corollary \ref{cor:HCPower}, allows us to conclude that $f$ and $f^{-1}$ are both algebraic.
\end{proof}

\begin{remark}
    By Taelman \cite[Theorems 4.10 and 4.11]{Taelman-DerivedEqHK}, if two hyper-K\"ahler varieties $X$, $X'$ of known deformation type are derived equivalent, then $H^2(X,\QQ)$ and $H^2(X', \QQ)$  are Hodge isometric. Therefore, Corollary \ref{cor:HomologicalMotiveIso} applies to $X$ and $X'$ as in Theorem \ref{thm:motivationbyKSknownCases} that are Fourier--Mukai partners, hence provides new instances of the so-called \textit{multiplicative Orlov conjecture} proposed by Fu--Vial \cite[Conjecture 4.7]{FuVial} for homological motives; see \cite{MaulikShenYin-OrlovConjecture} for another recent result.
\end{remark}
\begin{remark}
    In Corollary \ref{cor:HomologicalMotiveIso}, if $X$ and $X'$ are defined over a subfield $F\subset \CC$, then there exists a finite extension $F'/F$ such that the isomorphism of Frobenius algebra objects between $\mathsf h(X)$ and $\mathsf h(X')$ holds in the category of homological motives over $F'$. In particular, there is a $\Gal(\overline{F}/F)$-equivariant isomorphism of graded $\QQ_{\ell}$-Frobenius algebras between $H^{*}_{\et}(X_{\overline{F}}, \QQ_{\ell})$ and $H^{*}_{\et}(X'_{\overline{F}}, \QQ_{\ell})$.
\end{remark}

\subsection{Tate conjecture}
Given a smooth projective variety $X$ defined over a finitely generated subfield $F\subset \CC$, for any prime number $\ell$ and any integer $i$, we have the (continuous) Galois representation on the \'etale cohomology:
\begin{equation}
    \rho_{\ell} \colon \Gal(\overline{F}/F) \longrightarrow \GL\left(H^{2i}_{\et}(X_{\overline{F}}, \QQ_{\ell}(i))\right).
\end{equation}
The Tate conjecture predicts that the representation $\rho_{\ell}$ is semisimple\footnote{We include here the prediction on the semisimplicity, which is sometimes referred to as the Grothendieck--Serre conjecture in the literature; see for example \cite[Conjecture 12.5]{Jannsen-MixedMotivesKTheoryLNM}.} and the cycle class map to the space of Galois-invariant cohomology classes
\begin{equation}
    \mathrm{cl}_{\ell}\colon \CH^i(X)_{\QQ_\ell} \longrightarrow  H^{2i}_{\et}(X_{\overline{F}}, \QQ_{\ell}(i))^{\Gal(\overline{F}/F)}
\end{equation}
is surjective.

A powerful strategy towards the Tate conjecture (in characteristic zero) is to reduce it to the Hodge conjecture via the following Mumford--Tate conjecture (see e.g. \cite{moonen2017}):
\begin{conjecture}[Mumford--Tate conjecture]
  Let $X$ be a smooth projective variety defined over a finitely generated subfield $F\subset \CC$. 
  Then under the identification $\GL(H^*(X^{an}, \QQ)\otimes \QQ_{\ell})\simeq \GL(H^{*}_{\et}(X_{\overline{F}}, \QQ_{\ell}))$ provided by the Artin comparison isomorphism, we have 
  $$\MT(X)\otimes \QQ_{\ell} = G_{\ell}(X)^{\circ},$$
  where $\MT(X)\subset \GL(H^*(X^{an},\QQ))$ is the Mumford--Tate group of $X$ and $G_{\ell}(X)^{\circ}$ is the identity component of the Zariski closure of the $\ell$-adic Galois representation.
\end{conjecture}

\begin{corollary}
\label{cor:TateConjecture}
    Let $F$ be a finitely generated subfield of $\CC$ and let $\ell$ be a prime number.
    Let $X_1, \cdots, X_n$ be smooth projective $F$-varieties such that their associated complex varieties are as in Theorem \ref{thm:motivationbyKSknownCases} and are Hodge similar to each other. Then the Tate conjecture holds for all products of powers of $X_i$'s. In particular, the Tate conjecture holds for all powers of any $F$-variety $X$ with $X_\CC$ as in Theorem \ref{thm:motivationbyKSknownCases}.
\end{corollary}
\begin{proof}
The proof is similar to \cite[Corollary 5.12]{FloccariFu-HodgeConjectureWeilFourfolds}. We include it for completeness.
For a variety $X$ satisfying the Mumford--Tate conjecture, the Tate conjecture for (powers of) $X$ is equivalent to the Hodge conjecture for (powers of) $X_\CC$; see \cite[Proposition 2.3.2]{moonen17}. By the cases of the Hodge conjecture established in Corollary \ref{cor:HCPower}, it suffices to show the Mumford--Tate conjecture for all products of powers of $X_i$.
As all the varieties in Theorem \ref{thm:motivationbyKSknownCases} are motivated by abelian varieties, we can invoke Commelin's result \cite[Theorem 10.3]{commelin2019} to reduce to showing the Mumford--Tate conjecture for each individual variety in the list of Theorem \ref{thm:motivationbyKSknownCases}. The Mumford--Tate conjecture for K3 surfaces is proved in \cite{Andre1996}, while in \cite[Theorem 1.18]{FFZ}, we proved the Mumford--Tate conjecture for all hyper-K\"ahler varieties of known deformation types. 
\end{proof}

\appendix
\section{Lattices}
\label{sec:Appendic-Lattices}

We collect in this appendix some results about lattices following Nikulin \cite{Nikulin1980}, and carry out some computations which are relevant to the main body of the paper.

A lattice $L$ is a free $\ZZ$-module of finite rank, denoted by $\rk(L)$, equipped with a non-degenerate integer-valued symmetric bilinear form, which we often denote by $(a,b)_L$ or just $(a,b)$.
The associated quadratic form on $L\otimes_{\ZZ} \RR$ can be diagonalized with $t_+$ entries equal to~$1 $ and $t_-$ entries equal to $-1$, and the signature of the lattice $L$ is $(t_+, t_{-})$; we have $\mathrm{rk}(L) = t_++t_-$.
If $L$ is a lattice and $k$ is a non-zero integer, we let $L(k)$ be the lattice obtained from $L$ by multiplying the bilinear form by a factor $k$. A lattice is called \emph{even} if $(v,v)\in 2\ZZ$ for any $v\in L$, and \emph{odd} otherwise; all the lattices considered in this paper are even. We denote by $\Oo(L)$ the group of orthogonal transformations of the lattice $L$; for any $k\in \ZZ$, the orthogonal group $\Oo(L(k))$ is canonically identified with $\Oo(L)$.

The dual lattice of $L$ is $L^{\vee}\coloneqq \Hom_{\ZZ}(L,\ZZ)$. The bilinear form gives an embedding with finite cokernel $L\to L^{\vee}$, sending $x$ to $(x,-)$; moreover, $L^{\vee}$ is naturally equipped with a $\QQ$-valued bilinear form. Thus $A_L\coloneqq L^{\vee}/L$ is a finite abelian group, called the \emph{discriminant group} of $L$; a lattice is called \emph{unimodular} if $A_L$ is trivial.
For any even lattice $L$ there is a non-degenerate quadratic form $q_L$ on~$A_L$ with values in $\QQ/2\ZZ$, i.e. $q_L\colon A_L\to \QQ/2\ZZ$ is such that $q_L(na)=n^2q_L(a)$ for $a\in A_L$, $n\in \ZZ$. Denoting by $\bar{w}\in A_L$ the class of $w\in L^{\vee}$, we have $q_L(\bar{w}) = (w,w)_{L^{\vee}}$ modulo $2\ZZ$.  We have an associated symmetric bilinear form $b_L(a,b)=\frac{1}{2}(q_L(a+b)-q_L(a)-q_L(b))$ on~$A_L$ with values in~$\QQ/\ZZ$. 
We denote by $\mathrm{length}(A_L)$ the minimal number of generators of the abelian group $A_L$, called the \textit{length} of $A_L$. There is a natural action of $\Oo(L)$ on $L^{\vee}$, which induces a homomorphism $\Oo(L)\to \Oo(A_L,q_L)$, where $\Oo(A_L,q_L)$ is the group of isometries of the discriminant group. By definition, the \textit{stable orthogonal group} $\tilde{\Oo}(L) $ is the kernel of this homomorphism.

The signature and the discriminant form are the basic invariants of a lattice. An \emph{overlattice} of an even lattice $L$ is a lattice $M$ with an isometric embedding $L\subset M$ with finite index. 
The following result is proved by Nikulin \cite{Nikulin1980}.
    \begin{theorem}[{\cite[Proposition 1.4.1]{Nikulin1980}}] \label{thm:overlattices}
    There is a bijection between even overlattices of $L$ and isotropic subgroups of the discriminant group $A_L$.
    \end{theorem}

The bijection sends an even overlattice $L\subset M$ to the isotropic subgroup $H= M/L$ of $A_L$. 
From the chain of embeddings $L\subset M\subset M^{\vee} \subset L^\vee$ one sees that $A_M = (M^{\vee}/L)/H$, and in fact $M^{\vee}/L =H^{\bot}\subset A_L$. It follows that $|A_M|=\tfrac{|A_L|}{[M:L]^2}$, where $[M:L]$ is the index of $L$ in $M$.

A sublattice $L\subset M$ is \textit{saturated} if $M/L$ is torsion-free, i.e. $mw \in L$ implies $w\in L$, for any $w\in M$ and non-zero $m\in\ZZ$. 
An embedding $L\hookrightarrow M$ of lattices is called \textit{primitive} if the image of $L$ is saturated; equivalently, if its cokernel is torsion free. 
Two primitive embeddings $j, j'\colon L\to M$ are said to be equivalent if there exist isometries $f\in \Oo(L)$ and $g\in \Oo(M)$ such that $j'=g\circ j \circ f$.

Let $L,M$ be even lattices and let $j\colon L\hookrightarrow M$ be a primitive embedding with orthogonal complement $K$. The embedding $j$ corresponds to the datum of the gluing subgroup $H:=M/{(L\oplus K)}\subset A_{L\oplus K}$ which determines the overlattice $M$ of $L\oplus K$. Since $L$ and $K$ are primitive in $M$, the projections of $H$ to $A_L$ and $A_K$ are injective, i.e., $H$ is the graph of an anti-isometry $\gamma\colon H_L \to H_K$ for subgroups $H_L$, $H_K$ of $A_L$, $A_K$ respectively.
Nikulin \cite[Proposition 1.15.1]{Nikulin1980} proves the following result, which allows one to classify primitive embeddings of a given lattice $L$ into another lattice $M$.
\begin{theorem}\label{thm:embedding}
    The primitive embeddings of an even lattice $L$ with invariants $(t_+, t_-, q_L)$ into an even lattice $M$ with invariants $(m_+, m_-, q_M)$ are determined by the data $(H_L, H_M, \gamma, K, \gamma_K)$ where: 
    \begin{enumerate}[label=\arabic*)]
    \item $H_L$ is a subgroup of $A_L$, $H_M$ is a subgroup of $A_M$; 
    \item $\gamma\colon H_L\to H_M$ is an isomorphism such that $q_M(\gamma(x))=q_L(x)$ for all $x\in H_L$;
    \item $K$ is an even lattice of signature $(m_+-t_+, m_--t_-)$;
    \item $\gamma_K\colon A_K \to \delta(-1)$ is an isomorphism compatible with the quadratic forms, where $\delta$ is obtained as follows: let $\Gamma\subset A_L\oplus A_M(-1)$ be the graph of $\gamma$; then $\Gamma$ is an isotropic subgroup of $A_L\oplus A_M(-1)$ and $\delta\coloneqq \Gamma^{\bot}/\Gamma$, with the induced quadratic form.
    \end{enumerate}
    Two such sets $(H_L, H_M, \gamma, K, \gamma_K)$ and $(H'_L, H'_M, \gamma', K', \gamma'_{K'})$ determine equivalent primitive embeddings if and only if the following conditions hold:
    \begin{enumerate}
    \item[a)]~there exists $\phi\in \Oo(L)$ such that $\bar{\phi}\in \Oo(A_L,q_L)$ induces an isomorphism $\bar{\phi} \colon H_L\xrightarrow{\ \sim \ } H_L'$;
    \item[b)]~there exist $\xi\in\Oo(A_M,q_M)$ and an isometry $\psi\colon K\to K'$ such that $\gamma'\circ \bar{\phi} = \xi \circ \gamma$ and $(\bar{\phi}^{-1},\xi) \circ \gamma_K =  \gamma_K' \circ \bar{\psi}$, where $(\bar{\phi}^{-1}, {\xi}) \colon \delta\xrightarrow{\ \sim \ } \delta'$ is the isometry induced by $(\bar{\phi}^{-1}, {\xi})\in \Oo(A_L \oplus A_M(-1))$.
    \end{enumerate}
\end{theorem}

    We now recall some facts about isometries of lattices.
    The following result is again a consequence of the work of Nikulin. 

    \begin{theorem}[{\cite[Thm. 1.14.2 and Rmk. 1.14.5]{Nikulin1980}}, {\cite[Chapter 14, 1.5]{huyK3}}]\label{thm:orthogonalGroups}
    Let $L$ be an even indefinite lattice. Assume that either:
    \begin{enumerate} [label=(\roman*)]
        \item $\rk(L) \geq \mathrm{length}(A_L)+2$, or
        \item $A_L=(\ZZ/2\ZZ)^3 \oplus A'$ for some abelian group $A'$ and, for any odd prime $p$, the length of $A_L \otimes_{\ZZ} \ZZ_p$ is at most $\mathrm{rk}(L)-2$. 
    \end{enumerate}
    Then the lattice $L$ is uniquely determined by its signature and discriminant form, and the natural map $\Oo(L)\to \Oo(A_L,q_L)$ is surjective.
    \end{theorem}

    \begin{remark}\label{rmk:relaxing}
        By \cite[Thm. 1.14.2]{Nikulin1980}, the hypothesis in the above theorem may be further relaxed as follows: for any odd prime $p$, the length of $A_L\otimes_{\ZZ} \ZZ_p$ is at most $\mathrm{rk}(L)-2$ and, if the length of $A_L\otimes_{\ZZ} \ZZ_2=\mathrm{rk}(L)$, then $A_L$ contains a copy $A_{\mathrm{U}(2)}$, where $\mathrm{U}(2)$ is a hyperbolic plane with the form multiplied by $2$.    \end{remark}

    Combining this theorem with Theorem \ref{thm:embedding}, we obtain the following criterion for the uniqueness of primitive embeddings of a lattice into a unimodular lattice.

    \begin{theorem}\label{thm:embeddingUnimodular}
        Let $j\colon L\hookrightarrow M$ be a primitive embedding of an even lattice $L$ into an even unimodular lattice $M$, and assume the complement of $j(L)$ satisfies the assumption of Theorem \ref{thm:orthogonalGroups}. Then all primitive embeddings of $L$ in $M$ are equivalent to $j$.
    \end{theorem}

   Let $L$ be a lattice and $h\in L$. The \emph{divisibility} $\mathrm{div}(h)$ is the largest integer such that $(h,-)/\mathrm{div}(h)$ belongs to $L^{\vee}$, or, equivalently, the positive generator of the ideal $(h, L)$ of $\ZZ$. The following statement is originally due to Eichler.
   
    \begin{theorem}[{\cite[Lemma 3.5]{gritsenko2010moduli}}]\label{thm:eichler}
        Let $L$ be a lattice containing at least two orthogonal hyperbolic planes. Let $h\in L$ be a primitive class, of square $2d$ and divisibility $\gamma$. The $\tilde{\Oo}(L)$-orbit of $h$ consists precisely of those primitive vectors $h'$ of square $2d$, divisibility $\gamma$ and such that $[(h,-)/\gamma]=[(h', -)/\gamma]$ in $A_L$.
    \end{theorem}

    Finally, let $M$ be an overlattice of $L$, corresponding to an isotropic subgroup $H$ of $A_L$. Then an isometry $f$ of $L$ extends to an isometry of $M$ if and only if the action $\bar{f}$ on $A_L$ satisfies $\bar{f}(H)=H$. 
    If $L\hookrightarrow M$ is a primitive embedding with orthogonal complement $K$, let $H\subset A_L \oplus A_K$ be the corresponding gluing subgroup; recall that $H$ is the graph of an anti-isometry $\gamma\colon H_L\to H_K$, for subgroups $H_L, H_K$ of $A_L,A_K$ respectively. Then $f\in \Oo(L)$ can be extended (non-uniquely) to an isometry $\tilde{f}\in \Oo(M)$ if and only if the induced action $\bar{f}$ of $f$ on $A_L$ stabilizes $H_L$ and there exists an isometry $g\in \Oo(K)$ whose action on $A_K$ stabilizes $H_K$ and satisfies $\bar{g}_{|_{H_K}} = \gamma\circ \bar{f}_{|_{H_L}}\circ \gamma^{-1}$.

\subsection{Some examples} We will now discuss some lattices which play a role in the paper. The hyperbolic plane $\mathrm{U}$ is the rank $2$ lattice of signature $(1,1)$ generated by isotropic classes $e,f$ such that $(e,f)=1$. It has the property that any embedding $\mathrm{U}\hookrightarrow L$ into an even lattice $L$ is automatically saturated and $L$ splits as the orthogonal direct sum $L=\mathrm{U}\oplus \mathrm{U}^{\bot}$. 
We shall illustrate the use of Theorem \ref{thm:embedding} to understand the orbits of primitive vectors in lattices of the form $\mathrm{U}(2)^{\oplus k }\oplus \mathrm{U}$. 

\begin{lemma}\label{lem:orbitsPrimitiveVectors}
    Let $M=\mathrm{U}(2)^{\oplus k }\oplus \mathrm{U}$, for some $k\geq 1$. Let $d\neq 0$ be an integer. If $d$ is odd, there is a single $\Oo(M)$-orbit of primitive vectors of square $2d$, and any such necessarily has divisibility $1$. If $d$ is even, there are exactly two $\Oo(M)$-orbits of primitive vectors of square $2d$: the orbit of vectors with divisibility $1$ and that of vectors with divisibility $2$. For a primitive class $h$ of square $2d$, we have $h^{\bot}=\mathrm{U}(2)^{\oplus k} \oplus \langle -2d\rangle$ if $\mathrm{div}(h)=1$ and $h^{\bot}=\mathrm{U}(2)^{\oplus k-1}\oplus \mathrm{U}\oplus \langle -2d\rangle$ if $\mathrm{div}(h)=2$.
\end{lemma}
\begin{proof}
   The discriminant group $A_{\langle 2d\rangle}$ is isomorphic to $\ZZ/2d\ZZ$, with form $k\mapsto k^2/2d \in \QQ/2\ZZ$. 
    On the other hand, $A_M=A_{\mathrm{U}(2)^{\oplus k}}$ is the cokernel of $\mathrm{U}(2)^{\oplus k}\hookrightarrow \tfrac{1}{2}\mathrm{U}(2)^{\oplus k}$; hence, $A_M=(\ZZ/2\ZZ)^{2k}$. Any element in $A_M$ has square $1$ or $0$, and there are $3$ orbits under $\Oo(A_M, q_M)$: the orbit of $0$, that of elements where the quadratic form takes value $1$ and that of non-zero elements where the form takes value $0$. By Theorem \ref{thm:orthogonalGroups}, the homomorphism $\Oo(M)\to \Oo(A_M,q_M)$ is surjective.

    We describe the equivalence classes of primitive embeddings of the lattice $L\coloneqq \langle 2d\rangle$ into $M$ via Theorem \ref{thm:embedding}. 
    They are determined by the sets $(H_L,H_M,\gamma, K,\gamma_K)$. Since $A_{M}\cong (\ZZ/2\ZZ)^{2k}$, we must have $H_M = 0 $ or $H_M\cong \ZZ/2\ZZ$. But $A_L$ contains a unique subgroup isomorphic to $\ZZ/2\ZZ$, and the quadratic form takes value $d/2\in \QQ/2\ZZ$ on its generator. If $d$ is odd, this subgroup is not isometric to any subgroup of $A_M$ and therefore $H_L=H_M=0$. We then get a unique equivalence class of primitive embeddings of $\langle 2d\rangle$ into $M$; the image must be generated by a primitive class $h$ of divisibility $1$, and we have $h^{\bot}=\mathrm{U}(2)^{\oplus k}\oplus \langle -2d\rangle$. 

    If $d$ is even, in the same way we get a single orbit of primitive vectors $h$ of square $2d$ and divisibility $1$, with $h^{\bot}=\mathrm{U}(2)^{\oplus k}\oplus \langle -2d\rangle$. But in this case $d/2 \in \QQ/2\ZZ$ is $0$ or $1$ according to whether $d$ is divisible by $4$ or not. Hence, there exist other primitive embeddings with $H_L\cong \ZZ/2\ZZ$ and choosing a subgroup $H_M$ generated by an element on which the form takes value $0$ or $1$, respectively. As there is a single orbit of such vectors under $\Oo(A_M, q_M)$, the discriminant form of $K$ is uniquely determined and is isomorphic to $A_{\mathrm{U}(2)^{\oplus k-1}\oplus \langle -2d\rangle}$, which has length $2k-1$. Since $\rk(K)=2k+1$, the lattice $K$ is uniquely determined by its signature and discriminant form, by Theorem \ref{thm:orthogonalGroups}. Therefore, $K=\mathrm{U}(2)^{\oplus k-1}\oplus \mathrm{U}\oplus \langle -2d\rangle$. Moreover, $\Oo(K)\to \Oo(A_K)$ is surjective, and therefore all primitive embeddings $L\hookrightarrow M$ with $H_L\cong \ZZ/2\ZZ$ are equivalent to each other. The images of the generator of $L$ via these embeddings give the unique $\Oo(M)$-orbit of primitive vectors of square $2d$ (with $d$ even) and divisibility $2$.
\end{proof}

The K3 lattice is the lattice $\Lambda_{\mathrm{K}3}$ arising as the second cohomology of K3 surfaces. We have 
\begin{equation}
    \Lambda_{\mathrm{K3}}=\mathrm{U}^{\oplus 3} \oplus E_8(-1)^{2},
\end{equation}
where $E_8$ is the unique positive definite even unimodular lattice of rank $8$. The Mukai lattice is 
\begin{equation}
\widetilde{\Lambda}_{\mathrm{K}3}\coloneqq \Lambda_{\mathrm{K}3}\oplus \mathrm{U},
\end{equation}
corresponding to the full cohomology of a K3 surface equipped with the Mukai pairing. 

For an integer $n\geq 2$, we denote by $\Lambda_{\mathrm{K}3^{[n]}}$ the lattice which is $H^2(X,\ZZ)$ for a hyper-K\"ahler manifold $X$ of $\mathrm{K}3^{[n]}$-type, equipped with the Beauville--Bogomolov form. We have 
\begin{equation}
    \Lambda_{\mathrm{K}3^{[n]}} = \Lambda_{\mathrm{K}3}\oplus \langle -2n + 2 \rangle.
\end{equation}
Note that $\Lambda_{\mathrm{K}3^{[n]}}$ has discriminant group $A_{\Lambda_{\mathrm{K}3^{[n]}}}\cong \ZZ/(2n-2)\ZZ$, generated by the class $\delta^{\vee}:=\frac{1}{2n-2}(\delta, -)$ where $\delta$ is a generator of the second summand. The form takes value $-\tfrac{1}{2n-2}\in \QQ/2\ZZ$ on $\delta^{\vee}$.

Let us also recall that for any $n>1$, the lattice $\Lambda_{\mathrm{Kum}^n}$ which is the second cohomology of hyper-K\"ahler manifolds of $\mathrm{Kum}^n$-type is 
\begin{equation}
\Lambda_{\mathrm{Kum}^n}=\mathrm{U}^{\oplus 3}\oplus \langle -2n-2\rangle.
\end{equation}
We have $A_{\Lambda_{\mathrm{Kum}^n}}=\ZZ/(2n+2)\ZZ$, generated by the class $\xi^{\vee}\coloneqq \tfrac{1}{2n+2}(\xi, -) $ where $\xi$ is a generator of the second summand; the form takes value $-\tfrac{1}{2n+2}$ on $\xi^{\vee}$.

\begin{remark}\label{rmk:unique-overlattice}
    Since we are concerned with sixfolds of generalized Kummer type, we record the following fact. By Theorem \ref{thm:overlattices}, the lattice $\Lambda_{\mathrm{Kum}^3}$ has a unique overlattice $\widehat{\Lambda}_{\mathrm{Kum}^3} = \mathrm{U}^{\oplus 3}\oplus \langle -2\rangle$, in which it is contained with index $2$. By the uniqueness, any isometry of $\Lambda_{\mathrm{Kum}^3}$ extends to an isometry of $\widehat{\Lambda}_{\mathrm{Kum}^3}$.
\end{remark}

\subsection{The Kummer lattice}
The Kummer lattice appears in~\cite{Nikulin} and \cite{PSS}. We recall some of its properties following \cite[\S14.3]{huyK3} and \cite[\S2]{GarbagnatiSarti}. Consider the negative definite lattice $R=\bigoplus_{i=1}^{16} \ZZ\cdot r_i$ where each~$r_i$ has square $-2$ and they are orthogonal to each other. We may identify the set of the $r_i$'s with the points of a $4$-dimensional vector space over~$\mathbb{F}_2$ (this essentially amounts to fixing an action of $(\ZZ/2\ZZ)^4$ on $L_{\Km}$ via isometries, transitive on the $r_i$).
The Kummer lattice $L_{\Km}$ is an overlattice of $R$, of index $2^5$. In fact, in $L_{\Km}$ we have the additional classes $\frac{1}{2} \sum_{i=1}^{16} r_i$, and $\frac{1}{2}\sum_{i\in W} r_i$, where $W\subset \FF_2^{\oplus 4}$ is an affine hyperplane (there are $30$ such classes).
Geometrically, this lattice arises as follows. 
\begin{example}
    Let $A$ be a $2$-dimensional complex torus, and let $\Km(A)$ denote the associated Kummer K3 surface, which is the minimal resolution of $A/\pm 1$. Let $R_1,\dots, R_{16}$ be the $16$ disjoint smooth rational curves which are the exceptional divisors of $\Km(A)\to A/\pm 1$. Their cohomology classes $r_i\in H^2(\Km(A),\ZZ)$ are pairwise orthogonal $(-2)$-classes, and the saturation of the sublattice generated by the $r_i$ is the Kummer lattice $L_{\Km}$.
The orthogonal complement of $L_{\Km}$ in $\Lambda_{\mathrm{K}3}$ is isometric to the lattice $\mathrm{U}(2)^{\oplus 3}$, which is the image of the push-forward $\pi_*\colon H^2(A,\ZZ)\to H^2(\Km(A),\ZZ)$ induced by the degree-2 rational map $\pi\colon A\dashrightarrow \Km(A)$. Therefore $A_{L_{\Km}}$ is isometric to $A_{\mathrm{U}(2)^{\oplus 3}}\simeq (\mathbb{Z}/2\mathbb{Z})^{6}$. 
\end{example}

\begin{remark}\label{rmk:A.10}
    Notice that $\mathrm{U}(2)^{\oplus 3}$ satisfies the assumptions of Theorem \ref{thm:orthogonalGroups}. Hence, there exists a unique equivalence class of primitive embeddings $L_{\Km}\hookrightarrow \Lambda_{\mathrm{K}3}$, by Theorem \ref{thm:embeddingUnimodular}. Similarly, there exists a unique equivalence class of primitive embeddings $L_{\Km}\hookrightarrow \widetilde{\Lambda}_{\mathrm{K}3}$; the orthogonal complement is $\mathrm{U}(2)^{\oplus 3}\oplus \mathrm{U}$ in this case.
\end{remark}

This leads to the characterization of Kummer K3 surfaces as those K3 surfaces containing a primitive copy of $L_{\Km}$ in their N\'eron--Severi lattice.
However, there are $2$ equivalence classes of primitive embeddings of $L_{\Km}$ in  $\Lambda_{{\mathrm{K}3}^{[n]}}$ when $n$ is odd.

\begin{lemma} \label{lem:embeddingL_Km}
If $n$ is even, there exists a unique equivalence class of primitive embeddings $j\colon L_{\Km}\hookrightarrow \Lambda_{\mathrm{K}3^{[n]}}$, and we have $\mathrm{im}(j)^{\bot}=\mathrm{U}(2)^{\oplus 3}\oplus \langle -2n+2\rangle$.
If $n$ is odd, there exist $2$ equivalence classes of primitive embeddings $j\colon L_{\Km}\hookrightarrow \Lambda_{\mathrm{K}3^{[n]}}$; for one we have $\mathrm{im}(j)^{\bot}=\mathrm{U}(2)^{\oplus 3}\oplus \langle -2n+2\rangle$ and for the other we have $\mathrm{im}(j)^{\bot} = \mathrm{U}(2)^{\oplus 2}\oplus\mathrm{U} \oplus \langle -2n+2\rangle$.
\end{lemma}
\begin{proof} 
Let $L\coloneqq L_{\Km}$ and $M\coloneqq \Lambda_{\mathrm{K}3^{[n]}}$. We determine the possibilities for the sets $(H_L, H_M, \gamma, K, \gamma_K)$ of  Theorem \ref{thm:embedding}. By the description of $q_L$ and $q_M$, it is clear that either $H_L$ is trivial or $H_L\cong \ZZ/2\ZZ$. Then $H_M$ is either trivial or the unique order $2$ subgroup of $A_M$; the form $q_M$ takes value $-(n-1)/2$ on the generator. If $H_L=H_M=0$, we get an equivalence class of primitive embeddings of $L_{\Km}$ in $\Lambda_{\mathrm{K}3^{[n]}}$; the orthogonal complement is $\mathrm{U}(2)^{\oplus 3}\oplus \langle -2(n-1)\rangle$. If $n-1$ is odd in fact we must have $H_L=H_M=0$, and this is the only possibility. If $n-1$ is even, we may also have that $H_L\cong \ZZ/2\ZZ$ is non-trivial, and $q_L$ takes value $0$ or $1$ on the generator of $H_L$ according to whether $n-1$ is divisible by $4$ or not. All such subgroups are conjugated to each other via isometries of $L$ (\cite[Remark 2.3(5)]{GarbagnatiSarti}). Embedding $\Lambda_{\mathrm{K}3^{[n]}}$ in $\widetilde{\Lambda}_{\mathrm{K}3}$ with complement generated by a primitive vector $v$ of square $2(n-1)$ and applying Lemma \ref{lem:orbitsPrimitiveVectors}, the orthogonal complement of such a primitive embedding $L_{\Km}\hookrightarrow \Lambda_{\mathrm{K}3^{[n]}}$ must be isometric to $K=\mathrm{U}(2)^{\oplus 2}\oplus \mathrm{U}\oplus \langle -2(n-1)\rangle$. Since $\Oo(K)\to \Oo(A_K, q_K)$ is surjective by Theorem~\ref{thm:orthogonalGroups}, these embeddings form a unique equivalence class.
\end{proof}

\subsection{The Barnes--Wall lattice}
\label{subsec:Barnes--Wall-Lattice}
There is another rank $16$ negative definite lattice which plays a role in our paper, the Barnes--Wall lattice $\mathrm{BW}_{16}$. 

\begin{definition}\label{def:BW16}
	The Barnes--Wall lattice $\mathrm{BW}_{16}$ is the unique negative definite lattice of rank $16$ with discriminant form isomorphic to $A_{\mathrm{U}(2)^{\oplus 4}}$ and which contains no $(-2)$-classes.
\end{definition}

The above properties indeed determine uniquely the Barnes--Wall lattice, see \cite{scharlau1994genus}.

\begin{remark}\label{rmk:propertiesBW}
    The discriminant group $A_{\mathrm{BW}_{16}} = A_{\mathrm{U(2)^4}}$ is identified with a quadratic space of dimension $8$ over $\FF_2$ which is the sum of $4$ planes with form $\begin{psmallmatrix}
        0 & 1\\ 1 & 0
    \end{psmallmatrix}$, and $\Oo(A_{\mathrm{BW}_{16}}, q_{\mathrm{BW}_{16}})$ is identified with the orthogonal group $\Oo^+_8(2)$.
    By \cite[Chapter 4, 10]{{ConwaySloane1998}}, the natural homomorphism $\Oo(\mathrm{BW}_{16})\to \Oo(A_{\mathrm{BW}_{16}}, q_{A_{\mathrm{BW}_{16}}})$ is surjective. It follows that there are $3$ orbits for the action of $\Oo(\mathrm{BW}_{16})$ on the discriminant group $A_{\mathrm{BW}_{16}} = A_{\mathrm{U(2)^{\oplus 4}}}$: the orbit of $0$, the orbit of non-trivial elements on which the form takes value $0$ and that of elements on which the form takes value $1$.
\end{remark}

The Barnes--Wall lattice does not embed primitively in the K3 lattice, as the complement would have to be a lattice of rank $6$ and discriminant of length $8$, which is impossible. However, $\mathrm{BW}_{16}$ can be embedded primitively in the Mukai lattice. 

\begin{lemma}\label{lem:embeddingBW16intoMukai}
	There exists a unique equivalence class of primitive embeddings $\mathrm{BW}_{16}\hookrightarrow \widetilde{\Lambda}_{\mathrm{K}3}$. The orthogonal complement is isometric to $\mathrm{U}(2)^{\oplus 4}$.
\end{lemma}
\begin{proof}
	For the existence of such a primitive embedding we may appeal to \cite[Corollary 1.12.3]{Nikulin1980}.
	The uniqueness follows from Theorem \ref{thm:embeddingUnimodular}.
\end{proof}

\begin{lemma}\label{lem:embeddingBWintoK33}
The Barnes--Wall lattice does not embed in $\Lambda_{\mathrm{K}3^{[n]}}$ if $n$ is even. For any odd $n\geq 3$, there exists a unique equivalence class of primitive embeddings $\mathrm{BW}_{16}\hookrightarrow \Lambda_{\mathrm{K}3^{[n]}}$. The orthogonal complement is isometric to $\mathrm{U}(2)^{\oplus 3}\oplus \langle -2(n-1)\rangle$.
\end{lemma}
\begin{proof}
    Let $j\colon \mathrm{BW}_{16}\hookrightarrow \widetilde{\Lambda}_{\mathrm{K}3}$ be a primitive embedding. Primitive embeddings of $\mathrm{BW}_{16}$ in $\Lambda_{\mathrm{K}3^{[n]}}$ correspond to the choice of a primitive vector $v$ in $\mathrm{im}(j)^{\bot}$ of square $2(n-1)$, which gives $j'\colon \mathrm{BW}_{16}\hookrightarrow v^{\bot}=\Lambda_{\mathrm{K}3^{[n]}}$.
    Since $\mathrm{im}(j)^{\bot}=\mathrm{U}(2)^{\oplus 4}$, no such vector exists if $n-1$ is odd; moreover, by Theorem \ref{thm:eichler} applied to $\mathrm{U}^{\oplus 4}$, primitive vectors of fixed square in $\mathrm{im}(j)^{\bot}$ form a single $\Oo(\mathrm{U}(2)^{\oplus 4})$-orbit. Since $\Oo(\mathrm{BW}_{16})\to \Oo(A_{\mathrm{BW}_{16}}, q_{A_{\mathrm{BW}_{16}}})$ is surjective, any isometry of $\mathrm{im}(j)^{\bot}$ extends to an isometry of $\widetilde{\Lambda}_{\mathrm{K}3}$. Hence, for $n$ odd, all primitive embeddings $j'\colon \mathrm{BW}_{16} \hookrightarrow \Lambda_{\mathrm{K}3^{[n]}}$ are equivalent to each other; the orthogonal complement $\mathrm{im}(j')^{\bot} \subset \Lambda_{\mathrm{K}3^{[n]}}$ is necessarily isometric to $\mathrm{U}(2)^{\oplus 3}\oplus \langle -2(n-1)\rangle$. 
    \end{proof}
    

\begin{remark}\label{rmk:BW+h}
    Let $j\colon \mathrm{BW}_{16}\hookrightarrow \widetilde{\Lambda}_{\mathrm{K}3}$ be a primitive embedding and let $v\in \mathrm{im}(j)^{\bot}$ be a primitive class of square $4d$, for some $d>0$. Then the saturation $L$ of the sublattice generated by $\mathrm{im}(j)$ and $v$ is an overlattice of index $2$ of $\mathrm{BW}_{16}\oplus \langle v\rangle$, isometric to $L_{\Km}\oplus \langle 4d\rangle$. In fact, $L$ is the unique lattice of signature $(1,16)$ and discriminant form isomorphic to $(A_{\mathrm{U}(2)^{\oplus 3}\oplus \langle -4d\rangle})(-1)$, but also $L_{\Km}\oplus \langle 4d\rangle$ has these properties. In particular, for any $d>0$, there exists a primitive embedding of $\mathrm{BW}_{16}$ in $L_{\Km}\oplus \langle 4d\rangle$.
\end{remark}

\subsection{Some computations}

We include here some computations which we use in the paper. We start with the following easy observation. 
\begin{remark} \label{rmk:divisibility}
    Let $h\in \mathrm{U}^{\oplus 3}\oplus \langle -2d\rangle$ be a primitive positive class. Let $\xi$ be a generator of the last summand. We can write $h=\alpha h' + \beta \xi$, for a primitive class $h'\in \mathrm{U}^{\oplus 3}$ and coprime integers $\alpha, \beta$. Then    
    \[\mathrm{div}(h) = \mathrm{g.c.d.}(\alpha, 2d).\]
    Indeed, let $t\coloneqq \mathrm{g.c.d}(\alpha, 2d)$, and write $\alpha= t\alpha'$, $2d = t d'$. Then, for a class $l=a  l' + b \xi$ with $l'$ in $\mathrm{U}^{\oplus 3}$, we have $(h,l)= t(a\alpha' (h', l') - d' b \beta)$; hence, $t$ divides $\mathrm{div}_L(h)$. Moreover, $\alpha'$ and $d'\beta$ are coprime, and, therefore, there exist integers $a_0,b_0$ such that $\alpha' \cdot a_0 - d'\beta \cdot b_0=1$. Choosing $g \in \mathrm{U}^{\oplus 3}$ such that $(h', g) =1$ we find $(h, a_0 g + b_0 \xi)= t$.
\end{remark}

\begin{lemma}\label{lem:orbitsU^3+<-2>}
    Let $L\simeq \mathrm{U}^{\oplus 3}\oplus \langle -2\rangle$. The primitive classes of fixed square and divisibility form a single $\widetilde{\Oo}(L)$-orbit. Let $h\in L$ be primitive of square $2d>0$. If $\mathrm{div}(h)=1$, we have $h^{\bot}\simeq \mathrm{U}^{\oplus 2}\oplus \langle -2d\rangle \oplus \langle -2\rangle$. If $\mathrm{div}(h)=2$, then $d=4d'-1$ for some integer $d'$ and $h^{\bot} \simeq \mathrm{U}^{\oplus 2}\oplus \begin{psmallmatrix}
        -2 & 1 \\ 1 & -2d'
    \end{psmallmatrix}$.
\end{lemma}
\begin{proof}
    Since $A_L\cong \ZZ/2\ZZ$, the divisibility of $h$ is either $1$ or $2$. By Eichler's criterion, vectors of fixed square and divisibility form a single orbit under $\tilde{\Oo}(L)$. 
    Let $e_i,f_i$, $i=1,2,3$, $\xi$ be a basis of $L$, where the $e_i,f_i$ generate the three pairwise orthogonal copies of $\mathrm{U}$ and $\xi$ is a primitive vector of square $-2$ orthogonal to each $e_i$ and $f_i$.\\
    If $\mathrm{div}(h) = 1$, by Eichler's criterion, there exists an isometry $\phi$ of $L$ such that $\phi(h)=e_1+d f_1$, and we immediately calculate that $h^{\bot}$ is isometric to $\mathrm{U}^{\oplus 2}\oplus \langle -2d\rangle \oplus \langle -2\rangle$.\\
    If $\mathrm{div}(h)=2$ then $h = 2u + b \xi$, for some primitive class $u\in \mathrm{U}^{\oplus 3}$ and some odd integer $b$. Writing $2m$ for the square of $u$, we thus have $2d = 8m-2b^2 = 2(4m - b^2)$, and therefore $d$ must be congruent to $3$ modulo $4$.
    Writing $d=4d'-1$, by Eichler's criterion, there exists an isometry $\phi$ of $L$ such that $\phi(h)=2(e_1+d'f_1) - \xi$. Then $\phi(h)^{\bot}$ has a basis given by $e_1-d'f_1, f_1-\xi$ and $e_2, f_2, e_3,f_3$. Hence, $h^{\bot}$ is isometric to $\mathrm{U}^{\oplus 2}\oplus \begin{psmallmatrix}
        -2 & 1 \\ 1 & -2d'
    \end{psmallmatrix}$ in this case.
\end{proof}

\begin{lemma}\label{lem:orbitsU^3+<-8>}
    Let $L=\mathrm{U}^{\oplus 3}\oplus \langle -8\rangle$ be the $\mathrm{Kum}^3$-lattice. The $\Oo(L)$-orbit of primitive positive vectors $h$ in $L$ are determined by the square $2d$, the divisibility $\gamma$ and the subset $\{\pm [(h,-)/\gamma]\}\subset A_L$. We have the following possibilities: 
    \begin{enumerate}[label=(\roman*)]
    \item for any $d>0$, there is a single orbit of primitive vectors $h$ of square $2d$ and divisibility $1$, and $h^{\bot}=\mathrm{U}^{\oplus 2}\oplus \langle -8\rangle \oplus \langle -2d\rangle$;
    \item for any $d'>1$ and $d=4(d'-1)$, there is a single orbit of primitive vectors $h$ of square $2d$ and divisibility $2$. We have $h^{\bot}=\mathrm{U}^{\oplus 2} \oplus \begin{psmallmatrix}
        -8 & 4 \\ 4 & -2d'
    \end{psmallmatrix}$;
    \item for any $d'>0$ and $d=4(4d'-1)$, there is a unique orbit of primitive vectors $h$ of square $2d$ and divisibility $4$. We have $h^{\bot} = \mathrm{U}^{\oplus 2}\oplus \begin{psmallmatrix}
        -8 & 2 \\ 2 & -2d'
    \end{psmallmatrix}$;
    \item for any $d'>0$ and $d=4(16d'-1)$ there is a unique orbit of primitive vectors $h$ of square $2d$, divisibility $8$ and such that $[(h,-)/8]\in A_L$ is $\pm 1$. We have $h^{\bot}=\mathrm{U}^{\oplus 2} \oplus \begin{psmallmatrix}
        -8 & 1 \\ 1 & -2d'
    \end{psmallmatrix}$;
    \item[(iv\,$'$)] for any $d'>0$ and $d=4(16d'-9)$ there is a unique orbit of primitive vectors $h$ of square $2d$, divisibility $8$ and such that $[(h,-)/8]\in A_L$ is $\pm 3$. We have $h^{\bot}=\mathrm{U}^{\oplus 2} \oplus \begin{psmallmatrix}
        -8 & 3 \\ 3 & -2d'
    \end{psmallmatrix}$.   \end{enumerate}   
\end{lemma}
\begin{proof}
    By Theorem \ref{thm:eichler}, primitive vectors $h$ of square $2d$, divisibility $\gamma$ and fixed class $[(h,-)/\gamma]\in A_L$ form a single $\widetilde{\Oo}(L)$-orbit. 
    Notice that $\Oo(A_L) = \{ \pm 1 \}$ and that $\Oo(L)\to \Oo(A_L)$ is surjective. Hence, the $\Oo(L)$-orbit of a primitive class $h$ of square $2d$ and divisibility $\gamma$ consists precisely of those primitive $h'$ of square $2d$, divisibility $\gamma$ and such that $[(h',-)/\gamma] = \pm [(h,-)/\gamma]$ in $A_L$. It is a single $\widetilde{\Oo}(L)$-orbit if $\gamma=1$ or $2$ and the union of $2$ distinct $\widetilde{\Oo}(L)$-orbits if $\gamma=4$ or $8$.

    Let $e_i,f_i$, $i=1,2,3$ and $\xi$ be a basis of $\mathrm{U}^{\oplus 3}\oplus \langle -8\rangle$, where each $\langle e_i, f_i\rangle$ is a hyperbolic plane with form $\begin{psmallmatrix}
        0 & 1\\
        1 & 0
    \end{psmallmatrix}$ and $\xi$ is a class of square $-8$ orthogonal to each $e_i,f_i$.

    $(i)$: $\mathrm{div}(h)=1$. Such a class exists for any value of $d$. A representative of this orbit is given by $h=de_1 + f_1$ and $h^{\bot}$ is isometric to $\mathrm{U}^{\oplus 2}\oplus \langle -8 \rangle \oplus \langle -2d\rangle$.

    $(ii)$: $\mathrm{div}(h)=2$. Then, necessarily, $[(h,-)/2]=4$ in $A_{L}$. Any class of divisibility $2$ can be written as $2u+k\xi$, with $u$ in $\mathrm{U}^{\oplus 3}$ and an odd integer $k$. Then $2d=4 (u,u) - 8k^2$, and we can write $d= 4(d'-1)$ for some $d'>1$. For any such $d$,  by Eichler's criterion, up to $\Oo(L)$, we can assume that $h$ is the following class of square $2d$ and divisibility $2$: 
    \begin{equation}
        h= 2( d'e_1 + f_1) + \xi.
    \end{equation}
    A basis for $h^{\bot}\subset L$ is then given by the two vectors $4e_1+\xi, f_1 - d'e_1$ together with $e_2,f_2,e_3,f_3$, and we obtain $h^{\bot}= \mathrm{U}^{\oplus 2} \oplus \begin{psmallmatrix}
        -8 & 4\\ 4 & -2d'
    \end{psmallmatrix}$. 
    
    $(iii)$: $\mathrm{div}(h)=4$. In this case $[(h,-)/ 4] = \pm 2$; hence, classes of divisibility $4$ and fixed square form a single orbit under $\Oo(L)$. Any such class may be written as $4u + k \xi$, for some $u\in \mathrm{U}^{\oplus 3} $ and an odd integer $k$; thus, $k^2$ is congruent to $1$ modulo $4$ and we can write $2d=16 (u,u) - 8k^2=8(4d'-1)$, for some $d'>0$. For any $d=4(4d'-1)$, we consider the class 
    \begin{equation} 
    h = 4(d'e_1+f_1)+\xi
    \end{equation} 
    in $L$, which has square $2d$ and divisibility $4$.
    A basis for $h^{\bot}\subset L$ is given by the vectors $2e_1+\xi, f_1 - d'e_1$ together with $e_2,f_2,e_3,f_3$, and we find $h^{\bot}=\mathrm{U}^{\oplus 2}\oplus \begin{psmallmatrix}
        -8  & 2\\
        2 & -2d'
    \end{psmallmatrix}$.
    
    $(iv)$: $\mathrm{div}(h)=8$ and $[(h,-)/8]= \pm 1$. These classes form a single $\Oo(L)$-orbit, once the square is fixed; such a vector may be written as $8u+k\xi$, for some $u\in \mathrm{U}^{\oplus 3}$ and an integer $k$ congruent to $\pm 1$ modulo $8$, and we therefore have $d = 4(16d'-1)$ for some positive integer $d'$. For any such $d$, we consider the class 
    \begin{equation}
        h= 8(d'e_1+f_1)+\xi
    \end{equation}
    in $L$. A basis for $h^{\bot}$ is then given by $e_1+\xi, f_1 - d'e_1$ together with $e_2,f_2,e_3,f_3$, which gives $h^{\bot} =\mathrm{U}^{\oplus 2} \oplus \begin{psmallmatrix}
        -8 & 1 \\
        1 & -2d'
    \end{psmallmatrix}$. 
    
    $(iv\,')$: $\mathrm{div}(h)=8$ and $[(h,-)/8]=\pm 3$. We have a single $\Oo(L)$-orbit of these classes, once the square is fixed. Any such may be written as $8u + k\xi$ for a class $u\in \mathrm{U}^{\oplus 3}$ and an integer $k$ congruent to $\pm 3$ modulo $8$. We therefore have $d = 4(16d' - 9)$, for some positive integer $d'$. For any such $d$, consider the class 
    \begin{equation}
        h= 8(d'e_1+f_1)+3 \xi.
    \end{equation}
    A basis for $h^{\bot}$ is then given by $3e_1+\xi, f_1-d'e_1$ together with $e_2,f_2,e_3,f_3$, and we obtain $h^{\bot}=\mathrm{U}^{\oplus 2}\oplus \begin{psmallmatrix}
        -8 & 3 \\ 3  & -2d'
    \end{psmallmatrix}$.
\end{proof}

Recall that (Remark \ref{rmk:unique-overlattice}) $L=\mathrm{U}^{\oplus 3}\oplus \langle -8\rangle$ has a unique index-$2$ overlattice $\hat{L}= \mathrm{U}^{\oplus 3}\oplus \langle -2\rangle$. Given $h\in L$, we let $\hat{h}\in \hat{L}$ denote its image. 

\begin{lemma}\label{lem:divHatH}
    Let $h\in L$ be a primitive class of square $2d$ and divisibility $\gamma$. Then:
    \begin{enumerate}[label=(\roman*)]
        \item if $\gamma=1$, then $\hat{h}\in \hat{L}$ is a primitive class of divisibility $1$ and square $2d$;
        \item if $\gamma=2$, then $\hat{h}\in \hat{L}$ equals $2\cdot \hat{h}'$ for a primitive class $\hat{h}'\in \hat{L}$ of divisibility $1$ and square $d/2$;
        \item if $\gamma > 2$, then $\hat{h}\in \hat{L}$ equals $2\cdot \hat{h}'$ for a primitive class $\hat{h}'\in \hat{L}$ of divisibility $2$ and square $d/2$.
    \end{enumerate} 
\end{lemma}
\begin{proof}
    Since $\hat{L}$ is the unique overlattice of $L$, any isometry of $L$ extends to an isometry of $\hat{L}$. Thus, it is sufficient to determine the properties of $\hat{h}$ for a single representative in each $\Oo(L)$-orbit. 
    We choose the basis $e_i,f_i$, $i=1,2,3$, $\xi$ of $L$ as in the proof of Lemma \ref{lem:orbitsU^3+<-8>}, so that, moreover, $\hat{e}_i, \hat{f}_i$, $i=1,2,3$, and $\hat{\xi}/2$ form a basis of $\hat{L}$. Choosing an explicit representative $h$ as in the various cases of Lemma \ref{lem:orbitsU^3+<-8>}, it is easy to verify the claimed properties of $\hat{h}$.
\end{proof}

While any isometry of $L$ extends to an isometry of the overlattice $\hat{L}$, there are isometries of $\hat{L}$ which do not stabilize the sublattice $L$.
\begin{lemma}\label{lem:O(L)}
    Let $L=\mathrm{U}^{\oplus 3}\oplus \langle -8\rangle$ and let $\hat{L}=\mathrm{U}^{\oplus 3}\oplus \langle -2\rangle$. Let $\delta \in \hat{L}$ be a primitive class of square $-2$ and divisibility $2$ such that $L=\delta^{\bot}\oplus \langle 2\delta \rangle$. Let $f\in \Oo(\hat{L})$; we have $f(\delta)=2\alpha u +\beta\delta$ for a primitive $u\in \delta^{\bot}$ and coprime integers $\alpha,\beta$ with $\beta$ odd. Then $f$ lies in $\Oo(L)\subset \Oo(\hat{L})$ if and only if $\alpha$ is even.
\end{lemma}
\begin{proof}
    Primitive vectors of square $-2$ and divisibility $2$ in $\hat{L}$ form a single orbit under the orthogonal group, and once such a vector $\delta$ is fixed, any other is written as $2\alpha u + \beta\delta$, for a primitive $u\in \delta^{\bot}$ of square $2e$ and coprime integers $\alpha,\beta$ with $\beta$ odd and such that moreover $4\alpha^2e - \beta^2= -1$. Hence, any isometry $f\in \Oo(\hat{L})$ satisfies $f(\delta)=2\alpha u + \beta \delta$ for some $u,\alpha,\beta$ as above.

    Assume that $f\in \Oo(L)$. Then $\alpha$ must be even. Then the class $\xi\coloneqq 2\delta$ in $L$ is primitive of square $-8$ and divisibility $8$; such vectors form a single orbit under $\Oo(L)$, and any other such vector is written as $8\alpha' u' + \beta'\xi$, for a primitive $u'\in \xi^{\bot}$ of square $2e$, coprime integers $\alpha', \beta'$ with $\beta'$ odd and such that $16(\alpha')^2e-(\beta')^2=-1$. Recalling that $\xi=2\delta$, we see that if $f\in \Oo(L)$ then $f(\delta)=4\alpha' u' + \beta' \delta$, for some $u',\alpha',\beta'$ as above. Hence, $\alpha=2\alpha'$ is even. 

    Conversely, if $f(\delta)= 2\alpha u+\beta \delta$ with $\alpha$ even, there exists some $g\in \Oo(L)$ such that $g(\delta)=f(\delta)$. Hence, $g^{-1}f(\delta)=\delta$, so clearly $g^{-1}f$, and therefore also $f$, belongs to $ \Oo(L)$.
\end{proof}

We can reformulate this lemma as follows. Consider $T \coloneqq L(2)$, which is isometric to $\mathrm{U}(2)^{\oplus 3}\oplus \langle -16\rangle$, and its overlattice $\hat{T}\coloneqq \hat{L}(2)$, which is isometric to $\mathrm{U}(2)^{\oplus 3}\oplus \langle -4\rangle$. We naturally have identifications
\begin{equation}
\Oo(T) =\Oo(L),  \ \ \  \Oo(\hat{T})=\Oo(\hat{L}).  
\end{equation}
Any $f\in \Oo(T)$ extends to an isometry of $\hat{T}$, so that we have an inclusion $\Oo(T)\hookrightarrow \Oo(\hat{T})$.
Choose a primitive $\delta\in \hat{T}$ of square $-4$ and divisibility $4$ such that $\hat{T}=\mathrm{U}(2)^{\oplus 3} \oplus \langle \delta \rangle$ and $T=\mathrm{U}(2)^{\oplus 3}\oplus \langle 2\delta\rangle $.

\begin{lemma}\label{lem:extendIsometry}
     The image of $\Oo(T)$ in $\Oo(\hat{T})$ consists precisely of the isometries whose action on $A_{\hat{T}}$ stabilizes the subgroup $A_{\langle \delta\rangle}$. 
   \end{lemma}
   \begin{proof}
   By the above discussion, an isometry $f$ of $\Oo(\hat{T})$ belongs to $\Oo(T)$ if and only if $f(\delta) = 4\alpha u+\beta\delta$; equivalently, if and only if $\bar{f}\in \Oo(A_{\hat{T}})$ sends $(\delta, -)/4$ to itself up to sign, i.e., if and only if the action of $f$ on the discriminant stabilizes the subgroup $A_{\delta}$.
\end{proof}

Recall (Lemma \ref{lem:embeddingL_Km}) that there exists a unique equivalence class of primitive embeddings of $L_{\Km}$ in $\Lambda_{\mathrm{K}3^{[3]}}$ with orthogonal complement isometric to $\hat{T}= \mathrm{U}(2)^{\oplus 3}\oplus \langle -4\rangle$. Let us fix one such embedding 
\begin{equation}
    \eta\colon L_{\Km}\hookrightarrow \Lambda_{\mathrm{K}3^{[3]}},
\end{equation}
and let $\tau\colon \hat{T}\hookrightarrow \Lambda_{\mathrm{K}3^{[3]}}$ denote the inclusion of the complement. Then $\Lambda_{\mathrm{K}3^{[3]}}$ is an overlattice of $L_{\Km}\oplus \hat{T}$; as such, it corresponds to an isotropic subgroup of $A_{L_{\Km}}\oplus A_{\hat{T}}$, which must be the graph of an anti-isometry of $A_{L_{\Km}}$ with a subgroup $H_{\eta}$ of $A_{\hat{T}}$. Hence, $\eta$ determines a splitting $A_{\hat{T}}=H_{\eta}\oplus A_{\langle -4\rangle}$, with $H_{\eta}\cong A_{\mathrm{U}(2)^{\oplus 3}}$. Let $\delta\in \hat{T}$ be a primitive class of square $-4$, divisibility $4$, and such that $\delta/4$ generates the orthogonal complement of $H_{\eta}$ in $A_{\hat{T}}$; $\delta$ is unique modulo $4\hat{T}$. Let $T \cong \mathrm{U}(2)^{\oplus 3}\oplus \langle -16\rangle$ be the sublattice of $\hat{T}$ generated by $\delta^{\bot}$ and $2\delta$; obviously, $2\hat{T}\subset T$.

\begin{lemma}\label{lem:extendIsometryToM}
    Let $i\colon \Lambda_{\mathrm{K}3}^{[3]} \hookrightarrow \widetilde{\Lambda}_{\mathrm{K}3}$ be a primitive embedding. Let $v$ be a primitive vector of square $4$ which generates the orthogonal complement. Let $T\subset \hat{T}$ be as above, and consider the sublattice \begin{equation} 
    M\coloneqq \mathrm{im}(i\circ \eta)^{\bot} = (\mathrm{im}(i\circ \tau) \oplus \langle v\rangle)^{\mathrm{sat}}.
    \end{equation}
    Then, $M$ is isometric to $\mathrm{U}(2)^{\oplus 3}\oplus \mathrm{U}$, and an isometry $f\in \Oo(\hat{T})$ extends to an isometry of $M$ if and only if it stabilizes $T\subset \hat{T}$.
\end{lemma}
\begin{proof}
    Since $M$ is the orthogonal complement to a primitive copy of $L_{\Km}$ in $\widetilde{\Lambda}_{\mathrm{K}3}$, by Remark \ref{rmk:A.10}, it is isometric to $\mathrm{U}(2)^{\oplus 3}\oplus \mathrm{U}$. Moreover, $(i\circ \eta) (L_{\Km}) \oplus \langle v\rangle$ is a saturated sublattice of $\widetilde{\Lambda}_{\mathrm{K}3}$, since its complement is isometric to $\hat{T}$. Hence, the overlattice $\widetilde{\Lambda}_{\mathrm{K}3}$ of $(\iota\circ \tau)(\hat{T})\oplus ((\iota\circ \eta)(L_{\Km})\oplus \langle v\rangle)$ corresponds to the graph of an anti-isometry between $A_{\hat{T}}$ and $A_{L_{\Km}}\oplus A_{\langle v\rangle}$. This is nothing but the splitting $H_{\eta}\oplus A_{\delta}$ discussed above. 

    It follows that $\langle (i\circ \tau)(\delta), v\rangle^{\mathrm{sat}}$ is a hyperbolic plane. 
    The overlattice $M$ of $(i\circ \eta)(\hat{T})\oplus \langle v\rangle$ is thus given by an isotropic subgroup which is the graph of an anti-isometry $A_{\langle v\rangle}\to A_{\delta}\subset A_{\hat{T}}$. Hence, an isometry $f\in \Oo(\hat{T})$ extends to an isometry of $M$ if and only if the action $\bar{f}$ on $A_{\hat{T}}$ stabilizes $A_{\delta}$ and acts as $\pm 1$ on it. But $\Oo(A_{\delta})=\{\pm 1\}$, so the last condition is automatically satisfied. By Lemma \ref{lem:extendIsometry}, an isometry $f\in \Oo(\hat{T})$ extends to an isometry of $M$ if and only if $f$ belongs to $\Oo(T)$.    
\end{proof}

\begin{lemma}\label{lem:NS-hyperKummer}
Let $\eta\colon L_{\Km}\hookrightarrow \Lambda_{\mathrm{K}3^{[3]}}$ be a primitive embedding with complement isometric to $\hat{T}$; consider the sublattice $T\subset \hat{T}$ as above. Let $h$ be a primitive class in $\Lambda_{\mathrm{K}3^{[3]}}$ orthogonal to the image of $L_{\mathrm{Km}}$, of square $4d>0$. Then $2h\in T$, and: 
\begin{enumerate}[label=(\roman*)]
\item if $2h$ is not primitive in $T$, then necessarily $\langle h\rangle \oplus \eta(L_{\Km})$ is not saturated in $\Lambda_{\mathrm{K}3^{[3]}}$; its saturation is an overlattice of index $2$, isometric to the unique even lattice $L_{4d}$ of signature $(1,16)$ and discriminant form isometric to $A_{\mathrm{U}(2)^{\oplus 2}\oplus \langle 4d\rangle}$. 
\item if $2h$ is primitive in $T$, then $\langle h\rangle\oplus \eta(L_{\Km})$ is saturated in $\Lambda_{\mathrm{K}3^{[3]}}$. 
\end{enumerate}
\end{lemma}
\begin{proof}
    We choose a primitive embedding $i\colon \Lambda_{\mathrm{K}3^{[3]}}\hookrightarrow \widetilde{\Lambda}_{\mathrm{K}3}$, with complement generated by a primitive vector $v$ of square $4$. We let $M$ be the saturation of the sublattice of $\Lambda_{\mathrm{K}3^{[3]}}$ spanned by $(i\circ \tau)(\hat{T})$ and $v$. As in the above proof, we find that $M$ is isometric to $\mathrm{U}(2)^{\oplus 3}\oplus \mathrm{U}$, and that $\langle (i\circ \tau)(\delta), v\rangle^{\mathrm{sat}}\cong \mathrm{U}$ is a hyperbolic plane,
    with a standard basis given by (up to sign choices)
    \begin{equation}  u\coloneqq \frac{v+\delta}{4}, \ \ w\coloneqq \frac{v-\delta}{2}.
    \end{equation} 
    A primitive class $h\in \hat{T} = \mathrm{U}(2)^{\oplus 3}\oplus \langle (2u-w)\rangle$ can be written as $h= \alpha h' + \beta (2u-w)$ for some $h'\in \mathrm{U}(2)^{\oplus 3}$ and coprime integers $\alpha, \beta$.
    Now $2h$ is primitive in $T$ if and only if $ \beta$ is odd, i.e., if and only if $\mathrm{div}_{M}(i(h)) = 1$, while $2h$ is not primitive in $T$ if and only if $\beta$ is even, i.e., if and only if $\mathrm{div}_M(i(h))=2$. By Lemma \ref{lem:orbitsPrimitiveVectors}, $\langle h\rangle \oplus \eta(L_{\Km})$ is saturated in the first case, and $(\langle h\rangle \oplus \eta(L_{\Km}))^{\mathrm{sat}}$ is isometric to $L_{4d}$ in the second case.
\end{proof}

\begin{remark}\label{rmk:possible-invariants}
    For any $d>0$ there exists a primitive class $h \in \eta(L_{\Km})^{\bot}$ of square $4d$ and divisibility $2$ such that $\eta(L_{\Km})\oplus \langle h\rangle$ is saturated, and if $d=4d'-1$, there exists a primitive class $h \in \eta(L_{\Km})^{\bot}$ of square $4d$ and divisibility $4$ such that $\eta(L_{\Km})\oplus \langle h\rangle$ is saturated.
    Indeed, recalling that $T=(\mathrm{U}^{\oplus 3} \oplus \langle - 8 \rangle) (2)$, by Lemma \ref{lem:orbitsU^3+<-8>} we have:
    \begin{itemize}
        \item for any $e'>1$ there exists a primitive class $h'\in T$ with $(h')^2 = 16 (e'-1)$ and $\mathrm{div}_{T}(h')=4$, and then $h=h'/2$ is a primitive class of square $4(e'-1)$ and divisibility $2$ in $\hat{T}$;
        \item for any $e'>0$, there exists a primitive class $h'\in T$ of divisibility $8$ and square $(h')^2 = 16(4e'-1)$, and then $h=h'/2$ is a primitive class in $\hat{T}$ of divisibility $4$ and square $4(4e'-1)$;
    \end{itemize}
\end{remark}

Recall that the lattice $\hat{T}=\mathrm{U}(2)^{\oplus 3}\oplus \langle -4\rangle$ also naturally arises as orthogonal complement of the Barnes--Wall lattice in $\Lambda_{\mathrm{K}3^{[3]}}$, as there exists a primitive embedding 
\begin{equation}
    \iota\colon \mathrm{BW}_{16}\hookrightarrow \Lambda_{\mathrm{K}3^{[3]}},
\end{equation}
unique up to isometry, with complement isometric to $\hat{T}$.

\begin{lemma}\label{lem:extendIsometryComplementBW}
    Let $\iota$ be as above and let $\tau\colon \hat{T}\hookrightarrow \Lambda_{\mathrm{K}3^{[3]}}$. Any isometry of $\hat{T}$ can be extended to an isometry of $\Lambda_{\mathrm{K}3^{[3]}}$.
\end{lemma}
\begin{proof}
    Let $j\colon \Lambda_{\mathrm{K}3^{[3]}}\hookrightarrow \widetilde{\Lambda}_{\mathrm{K}3}$ be a primitive embedding, with orthogonal complement generated by a primitive vector $v$ of square $4$. We obtain a primitive embedding $j\circ \iota$ of $\mathrm{BW}_{16}$ in the Mukai lattice, with complement $K$ necessarily isometric to $\mathrm{U}(2)^{\oplus 4}$. Notice that $K$ is the saturation of the sublattice generated by $\hat{T}\oplus \langle v\rangle$. We can find an isometry $\hat{T}\cong \mathrm{U}(2)^{\oplus 3}\oplus \langle -4\rangle$ such that, if $\epsilon$ is the primitive generator of the last summand, then $\langle \epsilon, v\rangle^{\mathrm{sat}}$ is a copy of $\mathrm{U}(2)$, generated by 
       \begin{equation} 
    u\coloneqq \frac{v+\epsilon}{2}, \ \ w\coloneqq \frac{v-\epsilon}{2}.
    \end{equation}
   Therefore, the isotropic subgroup $H\subset A_{\hat{T}}\oplus A_{\langle v\rangle}$ corresponding to the overlattice $K$ is the order $2$ subgroup generated by $(2\cdot \tfrac{(\epsilon, -)}{4},2\cdot \tfrac{(v, -)}{4})$, i.e., the graph of an anti-isometry $\gamma\colon H_{\langle v\rangle} \to H_{\hat{T}}$ where $H_{\langle v\rangle}\cong \ZZ/2\ZZ$ is the subgroup of $A_{\langle v\rangle}$ generated by $2\cdot \tfrac{(v,-)}{4}$, and $H_{\hat{T}}\cong \ZZ/2\ZZ$ is the subgroup of $A_{\hat{T}}$ generated by $2\cdot \tfrac{(\epsilon,-)}{4}$. 
    
    Notice that $H_{\hat{T}}\subset A_{\hat{T}}$ is the order $2$ subgroup which is the image of multiplication by $2$ on $A_{\hat{T}}$. It follows that the action $\bar{f}$ of any $f\in \Oo(\hat{T})$ stabilizes $H_{\hat{T}}$, and in fact $\bar{f}$ is the identity on $H_{\hat{T}}$. Hence, $f$ extends to an isometry $f'$ of $K$, which is the identity on $v$. Since $\Oo(\mathrm{BW}_{16})\to \Oo(A_{\mathrm{BW}_{16}}, q_{\mathrm{BW}_{16}})$ is surjective (Remark \ref{rmk:propertiesBW}), any isometry of $K$ extends to an isometry of $\widetilde{\Lambda}_{\mathrm{K}3}$. We thus find $\tilde{f}\in \Oo(\widetilde{\Lambda}_{\mathrm{K}3})$ such that $\tilde{f}$ fixes $v$ and stabilizes $\hat{T}$, and $\tilde{f}_{|_{\hat{T}}} = f$; therefore, $\tilde{f}_{|_{v^{\bot}}}$ is an isometry of $\Lambda_{\mathrm{K}3^{[3]}}$ extending $f$.
\end{proof} 

\begin{lemma}\label{lem:MRSisHyperKummer}
    Let $\iota\colon \mathrm{BW}_{16}\hookrightarrow \Lambda_{\mathrm{K}3^{[3]}}$ be a primitive embedding and let $h\in \Lambda_{\mathrm{K}3^{[3]}}$ be a primitive class of square $4d>0$ in $\mathrm{im}(\iota)^{\bot}$. Then the sublattice $\mathrm{BW}_{16}\oplus \langle h\rangle$ is not saturated in $\Lambda_{\mathrm{K}3^{[3]}}$, and its saturation is isometric to $L_{\Km}\oplus \langle 4d\rangle$. Moreover, there exists a primitive embedding $u\colon L_{\Km}\hookrightarrow (\mathrm{BW}\oplus \langle h\rangle)^{\mathrm{sat}}$ such that the image of $u(L_{\Km})$ in $\Lambda_{\mathrm{K}3^{[3]}}$ has orthogonal complement isometric to $\mathrm{U}(2)^{\oplus 3}\oplus \langle -4\rangle$.
\end{lemma}

\begin{proof}
    By Remark \ref{rmk:BW+h}, the saturation of $\mathrm{BW}\oplus \langle h\rangle$ is isometric to $L_{\Km}\oplus \langle 4d\rangle$. Thus, there exists a primitive embedding $L_{\Km}\hookrightarrow (\mathrm{BW}\oplus \langle h\rangle)^{\mathrm{sat}}$; however, the induced primitive embedding $L_{\Km}\hookrightarrow \Lambda_{\mathrm{K}3^{[3]}}$ may not have complement isometric to $\hat{T}$. We claim that in any case there exists a primitive embedding $L_{\Km}\hookrightarrow (\mathrm{BW}\oplus \langle h\rangle)^{\mathrm{sat}}$ such that the induced primitive embedding of $L_{\Km}$ in $\Lambda_{\mathrm{K}3^{[3]}}$ has complement isometric to $\hat{T}$.
    Let $\theta\colon \hat{T}\hookrightarrow \Lambda_{\mathrm{K}3^{[3]}}$ be the inclusion of the complement of $\mathrm{im}(\iota)$. Let $4d$ be the square of $h$ and set $\gamma\coloneqq \mathrm{div}_{\hat{T}}(h)$.

   
    Choose a primitive embedding $\eta\colon L_{\Km}\hookrightarrow \Lambda_{\mathrm{K}3^{[3]}}$, with orthogonal complement isometric to $\mathrm{U}(2)^{\oplus 3}\oplus \langle -4\rangle$; let $\tau\colon \hat{T}\hookrightarrow \Lambda_{\mathrm{K}3^{[3]}}$ denote the inclusion of the complement. By Lemma \ref{lem:NS-hyperKummer} and Remark \ref{rmk:possible-invariants}, we can find a primitive class $h'\in \mathrm{im}(\tau)$ of square $4d$ and divisibility $\gamma$ such that $L_{\mathrm{Km}}\oplus \langle h'\rangle$ is saturated.
    By Remark \ref{rmk:BW+h} there exists a primitive embedding of $\mathrm{BW}$ into $L_{\Km}\oplus \langle h'\rangle$, with orthogonal complement generated by a class $l$ of square $4d$ such that $L_{\Km}\oplus \langle h'\rangle$ is the saturation of $\mathrm{BW}\oplus \langle l\rangle$, a sublattice with index $2$. Composing with the embedding of $L_{\Km}\oplus \langle h'\rangle$ in $\Lambda_{\mathrm{K}3^{[3]}}$, we obtain a primitive embedding $\mathrm{BW}\hookrightarrow \Lambda_{\mathrm{K}3^{[3]}}$; by Lemma \ref{lem:embeddingBWintoK33}, up to applying an isometry $f$ of ${\Lambda}_{\mathrm{K}3^{[3]}}$ we may assume that this primitive embedding coincides with the primitive embedding $\iota$ in the statement (replacing of course also $\eta $ with $f\circ \eta$, $l$ with $f(l)$, etc.). 
    We thus have the primitive embeddings:
    \begin{equation}\begin{split}
        \iota & \colon \mathrm{BW}\hookrightarrow \Lambda_{\mathrm{K}3^{[3]}}, \ \text{with complement } \theta\colon \hat{T}\hookrightarrow \Lambda_{\mathrm{K}3^{[3]}};\\
        \eta & \colon L_{\mathrm{Km}}\hookrightarrow \Lambda_{\mathrm{K}3^{[3]}}, \ \text{with complement } \tau\colon \hat{T}\hookrightarrow \Lambda_{\mathrm{K}3^{[3]}},
        \end{split}
    \end{equation}
    and primitive classes $t,t' \in \hat{T}$ of square $4d$, such that $l=\theta(t)$, $h'=\tau(t')$ and such that we have the equality
    \begin{equation}(\iota(\mathrm{BW})\oplus \langle \theta(t) \rangle)^{\mathrm{sat}} = \eta(L_{\Km})\oplus \langle \tau(t')\rangle
    \end{equation}
    of sublattices of $\Lambda_{\mathrm{K}3^{[3]}}$. Hence, $\theta(t^{\bot})$ coincides with $\tau({t'}^{\bot})$ in $\Lambda_{\mathrm{K}3^{[3]}}$; in particular, the sublattices $t^{\bot}$ and ${t'}^{\bot}$ of $\hat{T}$ are isometric, and it follows that $\mathrm{div}_{\hat{T}}(t) = \mathrm{div}_{\hat{T}}(t')$ (cf. Lemma \ref{lem:orbitsU^3+<-2>}). Since, by assumption, $\mathrm{div}_{\mathrm{im}(\tau)}(h') = \mathrm{div}_{\hat{T}}(t') = \gamma$, we have $\mathrm{div}_{\mathrm{im}(\theta)}(l)=\gamma$. By construction, there is a primitive embedding of $L_{\Km}$ in $(\iota(\mathrm{BW}_{16})\oplus \langle l\rangle)^{\mathrm{sat}}$ and the orthogonal complement $L_{\Km}^{\bot}\subset \Lambda_{\mathrm{K}3^{[3]}}$ is isometric to $\hat{T}$.

    We have thus found a class $l$ of square $4d$ and divisibility $\gamma$ satisfying the conclusion of the lemma.
    By Lemma \ref{lem:orbitsU^3+<-2>}, primitive vectors in $\hat{T}$ of fixed square and divisibility form a single $\Oo(\hat{T})$-orbit; in particular, there exists an isometry $g\in \Oo(\theta(\hat{T}))$ such that $h=g(l)$. By Lemma \ref{lem:extendIsometryComplementBW}, $g$ extends to an isometry $\tilde{g}$ of $\Lambda_{\mathrm{K}3^{[3]}}$. 
    Applying $\tilde{g}$ we find the desired primitive embedding of $L_{\Km}$ in $(\iota(\mathrm{BW}_{16})\oplus \langle h\rangle)^{\mathrm{sat}}$ such that $L_{\Km}^{\bot}\subset \Lambda_{\mathrm{K}3^{[3]}}$ is isometric to $\hat{T}$.
\end{proof}

\subsection{A few other lattices} 
We collect here some information about a few more lattices appearing in our paper. Consider again $L_{\Km}$ as an overlattice of $\bigoplus_{i=1}^{16} \ZZ \cdot r_i$ where the $r_i$ are pairwise orthogonal classes of square $-2$.

\begin{definition}\label{def:latticeN_S}
    We denote by $N_S$ the sublattice $\langle r_i+r_j\rangle^{\bot}\subset L_{\Km}$, for some $i\neq j$.
\end{definition}
\begin{lemma}\label{lem:embeddingN_S}
    The lattice $N_S$ does not depend on the choice of $i$ and $j$. This lattice is negative definite of rank $15$, and its discriminant group $A_{N_S}$ is isomorphic to $(\ZZ/2\ZZ)^6\oplus \ZZ/4\ZZ$. Moreover, there exists a unique equivalence class of primitive embeddings of $N_S$ into $\Lambda_{\mathrm{K}3}$, with orthogonal complement $\mathrm{U}(2)^{\oplus 3}\oplus \langle -4\rangle$.
\end{lemma}
\begin{proof}
    For any other $l\neq k$, there exists an isometry of $L_{\Km}$ sending $r_i$ to $r_l$ and $r_j$ to $r_k$.
    Indeed, there is a natural $(\ZZ/2\ZZ)^4$-action on $L_{\Km}$ via translations, which is transitive on the $16$ classes $r_{m}$. We may thus assume that $r_i=r_l$. We identify the set of points $r_m$ with a $4$-dimensional vector space $V$ over $\FF_2$, in such a way that $r_i$ corresponds to the origin. 
    Now any (affine) linear transformation of $V$ induces an isometry of $L_{\Km}$, as it sends affine hyperplanes to affine hyperplanes. Notice that this isometry will fix the origin of $V$. But we can surely find such a linear transformation sending $r_j$ to $r_k$, whenever $r_j$ and $r_k$ are not the origin of $V$. We thus have found an isometry of $L_{\Km}$ sending $r_i$ to $r_l$ and $r_j$ to $r_k$. This shows that the lattice $N_S$ does not depend on the choice of $i$ and $j$. It is clear that $N_S$ is a negative definite lattice of rank $15$.
    
    We consider now the Kummer surface $\Km(A)$ on an abelian surface $A$; then $L_{\Km}$ is primitively embedded into $H^2(\Km(A),\ZZ)$ and the composition $N_S\hookrightarrow L_{\Km}\hookrightarrow H^2(\Km(A),\ZZ)$ gives a primitive embedding of $N_S$ into the K3 lattice $\Lambda_{\mathrm{K}3}$. The orthogonal complement is the saturation of $H^2(A,\ZZ)(2)\oplus \langle r_i+r_j\rangle$. But this lattice is already saturated: identifying again the $r_m$ with the points of a $4$-dimensional $\FF_2$-vector space $V$, there exists an affine hyperplane $W$ in $V$ containing $r_i$ but not $r_j$; if $v =\alpha u + \beta (r_i+r_j)$ is a $\QQ$-linear combination with $u\in H^2(A,\ZZ)(2)$ which belongs to $H^2(\Km(A),\ZZ)$, then $(v,\tfrac{1}{2}\sum_{k\in W} r_k ) = -\beta$ must be an integer, and hence $v\in H^2(A,\ZZ)(2)\oplus \langle r_i+r_j\rangle$. 
    This shows that the orthogonal complement of $N_S$ in $\Lambda_{\mathrm{K}3}$ is isometric to $\mathrm{U}(2)^{\oplus 3 } \oplus \langle -4\rangle$. Since $\Lambda_{\mathrm{K}3}$ is unimodular, we conclude that $A_{N_S}$ is the abelian group $(\ZZ/2\ZZ)^6 \oplus \ZZ/4\ZZ$. 
    Finally, the embedding of $N_S$ in $\Lambda_{K3}$ is unique up to isometries of $\Lambda_{\mathrm{K}3}$ by Theorem \ref{thm:embeddingUnimodular}.
\end{proof}

\begin{remark} 
    The lattice $\mathrm{U}\oplus N_S$ is uniquely determined by its signature and discriminant form by Theorem \ref{thm:orthogonalGroups}. It is therefore isometric to $L_{\Km}\oplus \langle 4\rangle$. It follows that there exist primitive vectors $v$ and $v'$ of square $4$ in $N_S\oplus \mathrm{U}$ such that $v^{\bot} = L_{\Km}$ and $(v')^{\bot}=\mathrm{BW}_{16}$.
    \end{remark}
    
The following lattice was studied by Garbagnati and Sarti in \cite{GarbagnatiSarti}. The notation is as above.

\begin{definition}\label{def:latticeN_V}
We denote by $N_V$ the sublattice $\langle \sum_{i=1}^{16} r_i\rangle^{\bot}\subset L_{\Km}$.
\end{definition}

\begin{lemma}\label{lem:embeddingN_V}
The lattice $N_V$ is negative definite of rank $15$, with discriminant group isomorphic to $(\ZZ/2\ZZ)^6\oplus \ZZ/8\ZZ$. There exists a unique class of primitive embeddings of $N_V$ in $\Lambda_{\mathrm{K}3}$, with orthogonal complement $\mathrm{U}(2)^{\oplus 3}\oplus \langle -8\rangle$.  
\end{lemma}
\begin{proof}
    These results can be found in \cite{GarbagnatiSarti}. They can be obtained as in the proof of Lemma \ref{lem:embeddingN_S}, considering the natural primitive embedding of $N_V$ in $H^2(\Km(A),\ZZ)$ and computing that $N_V^{\bot} = H^2(A,\ZZ)(2)\oplus \langle \tfrac{1}{2}\sum_{i=1}^{16} r_i\rangle$ is isometric to $\mathrm{U}(2)^{\oplus 3}\oplus \langle -8\rangle$.
\end{proof}
 
There are two more negative definite lattices of rank $16$ which arise in our work as primitive sublattices of $\Lambda_{\mathrm{K}3^{[2]}}$, which, however, we have not studied in detail. 
We consider the sum $L_{\Km}\oplus \mathrm{U}$, where $\mathrm{U}$ is a hyperbolic plane with standard basis $u,v$; as above, we regard $L_{\Km}$ as an overlattice of $\bigoplus_{i=1}^{16} \ZZ\cdot r_i$ where the $r_i$ are pairwise orthogonal $(-2)$-classes.

\begin{definition}\label{def:latticeN_W}
    We let $N_{W}$ be the lattice 
    \begin{equation}
    N_W\coloneqq \bigl\langle u+v, 2(u-v)+ \frac{1}{2} \sum_{i=1}^{16} r_i\bigr\rangle^{\bot}\subset L_{\Km}\oplus \mathrm{U}.
    \end{equation}
\end{definition}

The lattice $N_W$ is negative definite of rank $16$. It can be computed that $A_{N_W}\cong (\ZZ/2\ZZ)^7\times \ZZ/16\ZZ$. 
Identifying $\Lambda_{\mathrm{K}3^{[2]}}$ as the orthogonal complement to $u+v$ in ${\Lambda}_{\mathrm{K}3}\oplus \mathrm{U}$ and choosing a primitive embedding of $L_{\Km}$ into $\Lambda_{\mathrm{K}3}$, the lattice $N_W$ is naturally a primitive sublattice of $\Lambda_{\mathrm{K}3^{[2]}}$, with complement $\mathrm{U}(2)^{\oplus 3}\oplus \langle -16\rangle$.
There exist at least two equivalence classes of primitive embeddings $N_W\hookrightarrow \Lambda_{\mathrm{K}3^{[2]}}$, one with orthogonal complement isometric to $\mathrm{U}(2)^{\oplus 3}\oplus \langle -16\rangle$, the other with orthogonal complement $\mathrm{U}(2)^{\oplus 2} \oplus \langle -2\rangle \oplus \langle 2\rangle\oplus \langle -16\rangle$.

In fact, by Theorem \ref{thm:embeddingUnimodular} there exists a unique equivalence class of primitive embeddings of $N_W$ into the Mukai lattice $\widetilde{\Lambda}_{\mathrm{K}3}$, with complement isometric to $T\coloneqq \mathrm{U}(2)^{\oplus 3}\oplus \langle -16\rangle \oplus \langle 2\rangle$. Primitive embeddings of $N_W$ in $\Lambda_{\mathrm{K}3^{[2]}}$ correspond to primitive vectors in $T$ of square $2$. Taking the class $v$ which generates the summand $\langle 2\rangle$ of $T$, we obtain a primitive embedding with complement isometric to $\mathrm{U}(2)^{\oplus 3}\oplus \langle -16\rangle$; taking instead $v+e_1$, where $e_1$ is an isotropic vector in the summand $\mathrm{U}(2)^{\oplus 3}$, gives instead a primitive embedding with orthogonal complement $\mathrm{U}(2)^{\oplus 2}\oplus \langle -2\rangle \oplus \langle 2\rangle\oplus \langle -16\rangle$.

\begin{definition}\label{def:latticeN_M}
    Fix some integer $1\leq i \leq 16$. We let $N_{M}$ be the lattice 
    \begin{equation}
    N_M\coloneqq \langle u+v, u-v+r_i\rangle^{\bot}\subset L_{\Km}\oplus \mathrm{U}.
    \end{equation}
\end{definition}

The lattice $N_M$ is negative definite of rank $16$, and $A_{N_M}\cong (\ZZ/2\ZZ)^7\times \ZZ/4\ZZ$; it does not depend on the choice of $i$. The embedding of $N_M$ into $L_{\Km}\oplus \mathrm{U}$ leads to a primitive embedding of $N_M$ in $\Lambda_{\mathrm{K}3^{[2]}}$ with orthogonal complement isometric to $\mathrm{U}(2)^{\oplus 3}\oplus \langle -4\rangle$.
As above, one can see that there exist at least two distinct equivalence classes of primitive embeddings $N_M\hookrightarrow \Lambda_{\mathrm{K}3}^{[2]}$, one with orthogonal complement isometric to $\mathrm{U}(2)^{\oplus 3}\oplus \langle -4\rangle$, the other with orthogonal complement isometric to $\mathrm{U}(2)^{\oplus 2}\oplus \langle -2\rangle \oplus \langle 2\rangle \oplus \langle -4\rangle$.


\clearpage
\addcontentsline{toc}{section}{References}
\bibliographystyle{alpha} 
\bibliography{bibliographyNoURL}{}

\medskip \medskip

\noindent Salvatore Floccari\\
\noindent{Humboldt-Universit\"at zu Berlin, Germany}\\
\medskip \noindent{\texttt{salvatore.floccari@hu-berlin.de}}

\noindent Lie Fu\\
\noindent{Universit\'e de Strasbourg, Institut de recherche mathématique avancée (IRMA),  France} \\
\medskip \noindent{\texttt{lie.fu@math.unistra.fr}}

\end{document}